\documentclass[10pt]{article}
\usepackage{amsthm}
\usepackage{amsxtra}
\usepackage{amssymb}
\usepackage{thmtools}
\usepackage{enumitem}
\usepackage[colorlinks=true,linkcolor=blue]{hyperref}
\usepackage{mathrsfs}
\usepackage[all]{xy}
\usepackage{graphics}
\usepackage{graphicx}
\usepackage{tikz}
\usepackage{turnstile}
\usepackage{verbatim}
\usepackage[retainorgcmds]{IEEEtrantools}
\usepackage{accents}
\usepackage[all]{xy}
\usepackage{ifthen}
\usepackage{lineno}
\usepackage{makeidx}
\usepackage{url}
\makeindex
\usetikzlibrary{matrix,positioning,calc,arrows,arrows.meta}
\usetikzlibrary{decorations.pathmorphing,shapes}

\declaretheoremstyle[bodyfont=\sl]{slanted}

\swapnumbers
\declaretheorem[name=Definition,style=definition,qed=$\dashv$,
numberwithin=section]{dfn}
\declaretheorem[name=Definition,style=definition,numbered=no,qed=$\dashv$]{dfn*}
\declaretheorem[name=Definition,style=definition,numbered=no]{dfnnoqed*}

\declaretheorem[name=Theorem,style=slanted,sibling=dfn]{tm}
\declaretheorem[name=Theorem,style=slanted,numbered=no]{tm*}
\declaretheorem[name=Lemma,style=slanted,sibling=dfn]{lem}
\declaretheorem[name=Corollary,style=slanted,sibling=dfn]{cor}
\declaretheorem[name=Corollary,style=slanted,numbered=no]{cor*}
\declaretheorem[name=Remark,style=definition,sibling=dfn]{rem}
\declaretheorem[name=Question,style=definition,sibling=dfn]{ques}

\swapnumbers
\declaretheoremstyle[headfont=\scshape]{claimstyle}
\declaretheorem[name=Claim,style=claimstyle]{clm}
\declaretheorem[name=Claim,style=claimstyle]{clmtwo}
\declaretheorem[name=Claim,style=claimstyle]{clmthree}

\declaretheorem[name=Claim,style=claimstyle,numbered=no]{clm*}

\declaretheorem[name=Subclaim,style=claimstyle,numberwithin=clmtwo]{sclmtwo}

\declaretheorem[name=Subclaim,style=claimstyle,numbered=no]{sclm*}

\declaretheorem[name=Subsubclaim,style=claimstyle,numberwithin=sclmtwo]{ssclmtwo
}
\declaretheorem[name=Subsubclaim,style=claimstyle,numberwithin=sclmthree]{
ssclmthree}
\declaretheorem[name=Subsubclaim,style=claimstyle,numberwithin=sclmfour]{
ssclmfour}
\declaretheorem[name=Subsubclaim,style=claimstyle,numberwithin=sclmfive]{
ssclmfive}
\declaretheorem[name=Subsubclaim,style=claimstyle,numberwithin=sclmsix]{ssclmsix
}
\declaretheorem[name=Subsubclaim,style=claimstyle,numberwithin=sclmseven]{
ssclmseven}
\declaretheorem[name=Subsubclaim,style=claimstyle,numberwithin=sclmeight]{
ssclmeight}
\declaretheorem[name=Subsubclaim,style=claimstyle,numberwithin=sclmnine]{
ssclmnine}
\declaretheorem[name=Subsubclaim,style=claimstyle,numberwithin=sclmten]{ssclmten
}
\declaretheorem[name=Subsubclaim,style=claimstyle,numbered=no]{ssclm*}

\declaretheoremstyle[headfont=\scshape]{casestyle}

\declaretheorem[name=Case,style=casestyle]{case}
\declaretheorem[name=Successor Case,style=casestyle]{casetwo}
\declaretheorem[name=Limit Case,style=casestyle]{casethree}
\declaretheorem[name=Case,style=casestyle]{casefour}

\declaretheorem[name=Subcase,style=casestyle,numberwithin=casetwo]{scasetwo}

\newcommand{\conv}{\mathrm{conv}}
\newcommand{\compmode}{1}
\newcommand{\compopt}[2]{\ifthenelse{\equal{\compmode}{0}}{#1}{#2}}
\newcommand{\almost}{\mathrm{alm}}
\newcommand{\stk}{\mathrm{st}}
\newcommand{\srsigma}{\varsigma}
\newcommand{\nrsigma}{\varrho}
\newcommand{\nrsigmabar}{\bar{\nrsigma}}
\newcommand{\wt}{\widetilde}
\newcommand{\completion}{\mathrm{complete}}
\newcommand{\LST}{\Ll_{\mathrm{LST}}}

\newcommand{\lgcd}{\mathrm{lgcd}}

\newcommand{\wcof}{\mathrm{wcof}}
\newcommand{\urho}{{\mathrm{u}\rho}}

\newcommand{\RR}{\mathbb R}

\newcommand{\PP}{\mathbb P}
\newcommand{\BB}{\mathbb B}
\newcommand{\sub}{\subseteq}
\newcommand{\cross}{\times}
\newcommand{\all}{\forall}
\newcommand{\ex}{\exists}

\newcommand{\compat}{\parallel}
\newcommand{\incompat}{\perp}
\newcommand{\inter}{\cap}
\renewcommand{\int}{\inter}

\newcommand{\om}{\omega}
\newcommand{\pow}{\mathcal{P}}
\newcommand{\OR}{\mathrm{OR}}

\newcommand{\Hull}{\mathrm{Hull}}

\newcommand{\cut}{\backslash}

\newcommand{\Tt}{\mathcal{T}}
\newcommand{\Ss}{\mathcal{S}}
\newcommand{\Uu}{\mathcal{U}}
\newcommand{\Vv}{\mathcal{V}}
\newcommand{\Ww}{\mathcal{W}}
\newcommand{\Ll}{\mathcal{L}}

\newcommand{\Ttbar}{{\bar{\Tt}}}

\newcommand{\rg}{\mathrm{rg}}
\newcommand{\dom}{\mathrm{dom}}
\newcommand{\cod}{\mathrm{cod}}
\newcommand{\ins}{\trianglelefteq}

\newcommand{\pins}{\triangleleft}

\newcommand{\crit}{\mathrm{cr}}

\newcommand{\union}{\cup}
\newcommand{\rest}{\!\upharpoonright\!}
\newcommand{\com}{\circ}

\newcommand{\lh}{\mathrm{lh}}
\newcommand{\Ult}{\mathrm{Ult}}

\newcommand{\Fbar}{{\bar{F}}}
\newcommand{\sats}{\models}

\newcommand{\J}{\mathcal{J}}

\newcommand{\AC}{\mathsf{AC}}
\newcommand{\DC}{\mathsf{DC}}
\newcommand{\HOD}{\mathrm{HOD}}
\newcommand{\HC}{\mathrm{HC}}
\newcommand{\ZFC}{\mathsf{ZFC}}
\newcommand{\ZF}{\mathsf{ZF}}

\newcommand{\Coll}{\mathrm{Col}}
\newcommand{\es}{\mathbb{E}}
\newcommand{\shortimplies}{\Rightarrow}
\newcommand{\gammabar}{\bar{\gamma}}
\newcommand{\psibar}{\bar{\psi}}

\newcommand{\eps}{\varepsilon}

\newcommand{\Qbar}{{\bar{Q}}}

\newcommand{\Ttvec}{{\vec{\Tt}}}
\newcommand{\Uuvec}{{\vec{\Uu}}}

\newcommand{\core}{\mathfrak{C}}
\newcommand{\her}{\mathcal{H}}

\newcommand{\pred}{\mathrm{pred}}

\newcommand{\dirlim}{\mathrm{dir lim}}
\newcommand{\un}{\union}

\newcommand{\id}{\mathrm{id}}

\newcommand{\sq}{\mathrm{sq}}
\newcommand{\nth}{{\textrm{th}}}
\newcommand{\conc}{\ \widehat{\ }\ }

\newcommand{\forces}{\dststile{}{}}

\newcommand{\bfSigma}{\undertilde{\Sigma}}

\newcommand{\rSigma}{\mathrm{r}\Sigma}
\newcommand{\uSigma}{\mathrm{u}\Sigma}

\newcommand{\rPi}{\mathrm{r}\Pi}

\DeclareMathOperator{\card}{card}
\DeclareMathOperator{\cof}{cof}

\DeclareMathOperator{\rank}{rank}

\newcommand{\Two}{\mathrm{II}}

\newcommand{\OD}{\mathrm{OD}}

\newcommand{\bfrSigma}{\undertilde{\rSigma}}
\newcommand{\bfuSigma}{\undertilde{\uSigma}}

\newcommand{\psub}{\subsetneq}

\newcommand{\Yy}{\mathcal{Y}}
\newcommand{\Xxvec}{\vec{\Xx}}
\newcommand{\Xx}{\mathcal{X}}

\newcommand{\Zz}{\mathcal{Z}}

\newcommand{\cHull}{\mathrm{cHull}}

\newcommand{\unsq}{\mathrm{unsq}}

\newcommand{\Mbar}{\bar{\M}}

\newcommand{\lpole}{\left\lfloor}
\newcommand{\rpole}{\right\rfloor}

\newcommand{\univ}[1]{\lpole #1\rpole}

\newcommand{\tu}{\textup}

\newcommand{\In}{\mathrm{ind}}
\newcommand{\unrvl}{\mathrm{unrvl}}

\newcommand{\lex}{{\mathrm{lex}}}

\renewcommand{\qedsymbol}{$\Box$}

\renewcommand{\cut}{\backslash}

\renewcommand{\Ss}{\mathcal{S}}

\newcommand{\M}{M}
\newcommand{\dfnemph}{\textbf}
\newcommand{\pvec}{\vec{p}}

\renewcommand{\Mbar}{\bar{M}}

\renewcommand{\Mbar}{{\bar{M}}}

\renewcommand{\hbar}{\bar{h}}

\newcommand{\udash}{\mathrm{u}\text{-}}
\newcommand{\udeg}{{\udash\mathrm{deg}}}

\newcommand{\uu}{\mathrm{u}}
\newcommand{\nutilde}{\widetilde{\nu}}
\renewcommand{\pm}{\mathrm{pm}}
\newcommand{\exit}{\mathrm{ex}}

\newcommand{\gammahat}{\widehat{\gamma}}

\newcommand{\Gg}{\mathcal{G}}

\newcommand{\eqdef}{=_{\mathrm{def}}}

\newcommand{\tembto}{\hookrightarrow}
\newcommand{\temb}{\Pi}
\newcommand{\clint}{\mathrm{clint}}
\newcommand{\ddd}{\mathrm{ddd}}
\newcommand{\dds}{\mathrm{dds}}
\newcommand{\dr}{\mathscr{D}}
\newcommand{\Pivec}{\vec{\Pi}}
\newcommand{\successor}{\mathrm{succ}}
\newcommand{\fin}{\mathrm{fin}}

\newcommand{\Qhat}{\widehat{Q}}

\newcommand{\passive}{\mathrm{pv}}
\newcommand{\dropset}{\mathscr{D}}

\newcommand{\inflatearrow}{\rightsquigarrow}
\newcommand{\strength}{\varrho}
\newcommand{\jensenurl}{https://www.math.uni-bonn.de/\textasciitilde raesch/jensen/}
\newcommand{\steelurl}{http://math.berkeley.edu/\textasciitilde steel}
\newcommand{\myurl}{https://sites.google.com/site/schlutzenberg}
\begin{document}

\setcounter{footnote}{1}
\title{Iterability for (transfinite) stacks\footnote{This is the author accepted version of an article published in Journal of Mathematical Logic, Volume 21, Number 2, 2150008 (2021),
available at \url{https://doi.org/10.1142/S0219061321500082}. However, the  references  \cite{extmax}, \cite{V=HODX} and \cite{str_comparison} have been expanded to incorporate both the preprint cited by the published version of this article, and the full bibliographic data for their published versions. Likewise,
\cite{scale_con} has been expanded to include a preprint.}}

\date{February 19, 2021}
\author{Farmer Schlutzenberg\thanks{
Editing partially funded by the Deutsche Forschungsgemeinschaft (DFG,
German Research Foundation) under Germany's Excellence Strategy EXC
2044-390685587, Mathematics M\"unster: Dynamics--Geometry--Structure.
}\ \thanks{afirstname.alastname at gmail.com, \url{\myurl}}\\
WWU M\"unster}

\maketitle

\begin{abstract}
We establish natural criteria under which normally iterable premice
are iterable for stacks of normal trees.
Let $\Omega$ be a regular uncountable cardinal. Let $m<\omega$ and $M$ be an $m$-sound premouse and $\Sigma$ be an $(m,\Omega+1)$-iteration strategy for $M$ (roughly, a normal $(\Omega+1)$-strategy).

We define a natural condensation property for iteration strategies,
\emph{inflation condensation}.
We show that if $\Sigma$ has inflation condensation
then $M$ is $(m,\Omega,\Omega+1)^*$-iterable (roughly, $M$
is iterable for length $\leq\Omega$ stacks of normal trees each of length ${<\Omega}$), and moreover, we define a specific such strategy $\Sigma^{\mathrm{st}}$
and a reduction of stacks via $\Sigma^{\mathrm{st}}$ to normal trees via $\Sigma$.
If $\Sigma$ has the Dodd-Jensen property and $\mathrm{card}(M)<\Omega$ then $\Sigma$ has inflation condensation.

We also apply some of the techniques developed
to prove that if $\Sigma$ has strong hull condensation
(introduced independently by Steel),
and $G$ is $V$-generic for an $\Omega$-cc forcing, then $\Sigma$ extends
to an $(m,\Omega+1)$-strategy $\Sigma^+$ for $M$ with strong hull condensation,
in the sense of $V[G]$. Moreover, this extension is unique.
We deduce that  if $G$ is $V$-generic for a ccc forcing then $V$ and $V[G]$ have
the same $\omega$-sound, $(\omega,\Omega+1)$-iterable premice which project to $\omega$.

Keywords: Inner model; mouse; iteration tree; iteration strategy; iterability; condensation; stack; normalization; large cardinal.

Mathematics Subject Classification 2020: 03E45, 03E55
\end{abstract}

\section{Introduction}\label{sec:intro}

Let $M$ be a normally iterable premouse.
Does it follow that $M$ is iterable for non-normal trees? We prove here the following partial positive result
in this direction, which applies to both Mitchell-Steel indexed and $\lambda$-indexed premice.
The notion \emph{inflation condensation} is a certain condensation property for iteration strategies,
defined in Definition \ref{dfn:inflation_condensation}.
Roughly, it says that if there is a normal iteration strategy
$\Sigma$ for $M$ with inflation condensation, then there is an iteration
strategy $\Sigma^*$ for $M$ for normal stacks of iteration trees;
that is, sequences of normal trees, appropriately formed:

\begin{tm*}[\ref{thm:stacks_iterability}, \ref{thm:stacks_iterability_2}]
 Let $M$ be an $m$-sound premouse, let $\Omega$ be a regular uncountable cardinal, let $\xi\in\{\Omega,\Omega+1\}$,
 let $\Sigma$ be an $(m,\xi)$-iteration strategy for $M$
 and suppose that $\Sigma$ has inflation condensation. Then:
 \begin{enumerate}[label=--]
  \item if $\xi=\Omega$ then $M$ is $(m,{<\om},\Omega)^*$-iterable, and
  \item if $\xi=\Omega+1$ then $M$ is $(m,\Omega,\Omega+1)^*$-iterable.
 \end{enumerate}
Moreover, there is an iteration strategy $\Sigma^*$ witnessing this with $\Sigma\sub\Sigma^*$.
\end{tm*}

The background theory for the above theorem is $\ZF$ (actually, it is much more local than this).
Likewise for the other results of the paper, except where indicated otherwise.

Recall here that an $(m,\alpha,\beta)^*$-iteration strategy is a winning
strategy for player II in the iteration game $\Gg(M,m,\alpha,\beta)^*$;
this is the variant, introduced in \cite{coremore}, of the iteration game
$\Gg(M,m,\alpha,\beta)$ of \cite[\S4]{outline}.
In both of these games, the players build a sequence $\left<\Tt_\gamma\right>$
of length at most $\alpha$, consisting
of normal iteration trees $\Tt_\gamma$, with $\Tt_0$ on $M$,
$\Tt_1$ on the last model $M^{\Tt_0}_\infty$ of $\Tt_0$, etc.
But in $\Gg(M,m,\alpha,\beta)^*$, if in some round $\gamma<\alpha$, a bona fide
tree $\Tt_\gamma$
of length $\beta$ is reached,
then player II automatically wins the entire game. Player I
may of course end the round earlier, with some tree of successor length ${<\beta}$.\footnote{Whereas in $\Gg_m(M,\alpha,\beta)$,
if $\beta$ is a successor ordinal and not all rounds have been played,
then the game would continue, with the next round building a tree on $M^\Tt_\infty$.
Actually, the author has always understood $\Gg(M,m,\alpha,\beta)$ and
$(m,\alpha,\beta)$-iterability
to be defined as we have just defined $\Gg(M,m,\alpha,\beta)^*$ and
$(m,\alpha,\beta)^*$-iterability,
due to misreading the definition in \cite{outline} at some point. The author thanks Gabriel Fernandes for pointing out this confusion over the definition.
The author is not aware of any use of $(m,\alpha,\beta)$-iterability beyond $(m,\alpha,\beta)^*$-iterability
in the literature.} The rules are spelled out explicitly in \S\ref{subsec:terminology}.
The two games are only distinct when $\beta$ is a successor.
For limit $\alpha$, an $(m,{<\alpha},\beta)^*$-iteration strategy
is likewise, except that if the game lasts through $\alpha$ rounds, with neither player having yet lost,
then player II wins.

We will define explicitly a specific such strategy $\Sigma^*$ from $\Sigma$,
denoted $\Sigma^\stk$. Trees via $\Sigma^\stk$ of length ${<\Omega}$ will lift
to trees via $\Sigma$ of length ${<\Omega}$.
Further, if $\Omega=\om_1$ and $M$ is countable and we code $\Sigma\rest\HC$
and $\Sigma^\stk\rest\HC$ naturally with functions $\Sigma_0,\Sigma^\stk_0$ on the reals,
then $\Sigma^\stk_0$ is $\Delta_1(\Sigma_0)$. (We do not know if one can improve on this complexity.)

The construction of $\Sigma^\stk$ breaks into two main pieces (we assume
for the purposes of this discussion that all trees
have length ${<\Omega}$
and player I does not make any artificial drops in model or degree at the
beginning of rounds; that is, $(M',m')$ is the model and degree
produced at the end of one round, then the next round
is $m'$-maximal on $M'$).
First, given a normal tree $\Tt$ via $\Sigma$ of successor length,
we define a normal strategy $\Upsilon^\Sigma_\Tt$ for $M^\Tt_\infty$,
together with a process which converts normal trees $\Uu$ on $M^\Tt_\infty$ via $\Upsilon^\Sigma_\Tt$ to normal trees $\Xx=\Ww^\Sigma_\Tt(\Uu)$ on $M$
via $\Sigma$, and produces an embedding
$\nrsigma:M^\Uu_\infty\to M^\Xx_\infty$
when $\Uu$ has successor length.
We then also have the normal strategy $\Upsilon^\Sigma_\Xx$ for $M^\Xx_\infty$.
But using $\nrsigma$ we can copy trees on $M^\Uu_\infty$ to trees on $M^\Xx_\infty$.
We can define a normal strategy $\Upsilon^\Sigma_{\Tt,\Uu}$,
as the $\nrsigma$-pullback of $\Upsilon^\Sigma_\Xx$. So $M^\Uu_\infty$ is iterable,
etc. So this first step
leads immediately to a strategy for stacks of length ${<\omega}$.
Second, given a limit $\eta$ and a stack $\Ttvec$ of length $\eta$ in which each normal component
is built using the process above (or a generalization thereof, if $\eta>\om$)
and corresponding sequence $\Xxvec$ of normal
trees, we  show that there is a natural limit $\Xx$ of $\Xxvec$, and that everything fits
together in a sufficiently commutative fashion that the direct limit $M^{\Ttvec}_\infty$ of the stack $\Ttvec$
embeds into $M^\Xx_\infty$, so we can continue through longer stacks.

This overall process we call here \emph{normal realization},
as the tree $\Xx$ is normal, but for example in the situation above, we need not have $M^\Uu_\infty=M^\Xx_\infty$.
It is often called \emph{normalization} elsewhere,
and \emph{embedding normalization} in \cite{str_comparison},
but we prefer to
reserve the  term \emph{normalization} for a tighter process that we do not
discuss here (that is, \emph{full normalization} in the terminology of
\cite{str_comparison},
where one gets a normal tree $\Xx$ for which
$M^\Uu_\infty=M^\Xx_\infty$).\footnote{Also, \cite{str_comparison}
deals with fine structural strategy mice, as well as pure $L[\es]$ mice,
whereas here we consider pure $L[\es]$ mice, and also certain coarse
structures, though many of the results adapt routinely.}

Using normal
realization, we will also prove a theorem concerning the following iteration
game:

\begin{dfn}\label{dfn:G_fin}\index{$\Gg_\fin$}Let $M$ be an $m$-sound premouse.
In
$\Gg_{\fin}(M,m,\Omega+1)$,
player I plays a \emph{finite length} putative $m$-maximal stack $\Ttvec=\left<\Tt_i\right>_{i<k}$ of \emph{finite length} trees $\Tt_i$,
player I wins if $M^{\Ttvec}_\infty$ is  illfounded, and otherwise, the players
proceed to play out the $(n,\Omega+1)$-iteration game on $M^\Ttvec_\infty$ where
$n=\deg^\Ttvec(\infty)$.
\end{dfn}

(See \S\ref{subsec:terminology} for explanations of terminology.)
Note that unlike the main theorems on normal realization mentioned above,
the
following one requires no condensation hypothesis for
$\Sigma$.\footnote{Part (i) of the
theorem was used by the author
in the presentation \emph{Fine structure from normal iterability} at the 2015
M\"unster conference,
and part (ii) provides a simplification of another fact used there.}

\begin{tm*}[\ref{thm:stacks_of_finite_trees}] Let $\Omega>\om$ be regular and $M$ be  $m$-sound $(m,\Omega+1)$-iterable. Then \tu{(}i\tu{)} player $\Two$ has a winning
strategy for $\Gg_{\fin}(M,m,\Omega+1)$. Moreover, \tu{(}ii\tu{)} let
$\Ttvec=\left<\Tt_i\right>_{i<\om}$ be an $m$-maximal
stack on $M$ consisting of \emph{finite length} trees $\Tt_i$ \tu{(}and note $\lh(\Ttvec)=\om$\tu{)}.
Then for all sufficiently large $i<\om$, $b^{\Tt_i}$ does
not drop in model or degree, player I makes no artificial drop in round $i$,
and $M^\Ttvec_\infty$ is wellfounded.
\end{tm*}

From the results in the paper we obtain the following equivalence of various forms of iterability, for countable premice.
\emph{Strong hull condensation} is another condensation property for iteration strategies,
isolated by Steel; see \ref{dfn:strong_hull_condensation}.\footnote{Strong hull condensation was defined by Steel,
and inflation condensation by the author, independently of one another,
at around the same time. Jensen also independently defined a similar condensation notion at around this time.}
 Inflation condensation and strong hull condensation have the same basic idea behind them;
indeed inflation condensation just demands that certain instances of strong hull condensation hold,
so the latter implies the  former.
The author does not know whether they are equivalent.
The implication from (weak) Dodd-Jensen to strong hull condensation,
that is, Theorem \ref{tm:wDJ_implies_cond},
was pointed out to the author by Steel in 2017.\footnote{That is,
for $\lambda$-indexed premice; for MS-indexed premice there are
additional technical considerations to deal with, as discussed here.}

 \begin{tm} Assume $\DC$.
 Let $\Omega>\om$ be regular and $M$ be a countable $m$-sound premouse. Then the following are equivalent:
 \begin{enumerate}[label=\tu{(}\alph*\tu{)}]
  \item\label{item:stacks_iter} $M$ is $(m,\Omega,\Omega+1)^*$-iterable.
  \item\label{item:norm_iter_cond} There is an $(m,\Omega+1)$-strategy for $M$ with inflation condensation.
  \item\label{item:norm_iter_shcond} There is an $(m,\Omega+1)$-strategy for $M$ with strong hull condensation.
  \item\label{item:norm_iter_wdj} There is an $(m,\Omega+1)$-strategy for $M$ with weak Dodd-Jensen.
  \item\label{item:stacks_iter_wdj} There is an $(m,\Omega,\Omega+1)^*$-strategy for $M$ with weak Dodd-Jensen.
 \end{enumerate}
\end{tm}
\begin{proof}
 \ref{item:stacks_iter} $\shortimplies$ \ref{item:stacks_iter_wdj} is by \cite{wdj} (only this implication uses $\DC$),
  \ref{item:stacks_iter_wdj} $\shortimplies$ \ref{item:norm_iter_wdj} is trivial,
   \ref{item:norm_iter_wdj} $\shortimplies$ \ref{item:norm_iter_shcond} by \ref{tm:wDJ_implies_cond},
      \ref{item:norm_iter_shcond} $\shortimplies$ \ref{item:norm_iter_cond} by \ref{lem:shcond_implies_extra_inf}, and  \ref{item:norm_iter_cond} $\shortimplies$ \ref{item:stacks_iter} by \ref{thm:stacks_iterability}.
\end{proof}

We do not know whether $\DC$ is necessary above. But in
Corollary \ref{cor:choiceless_wDJ_HOD} we do present
a construction of an iteration strategy with weak Dodd-Jensen in a specific choiceless context.

Part of the methods in the paper also yield the following
result, relating to extending a normal iteration strategy to a generic
extension $V[G]$.
While the construction of $\Sigma^\stk$ only demands inflation condensation of $\Sigma$,
its proof uses strong hull condensation:
\begin{tm*}[\ref{thm:strat_with_cond_extends_to_generic_ext}]
Let $\Omega>\om$ be regular. Let $M$ be an $m$-sound premouse.
Let $\Sigma$ be an $(m,\Omega+1)$-strategy for $M$ with strong hull condensation.
Let $\PP$ be an $\Omega$-cc forcing and $G$ be $V$-generic for $\PP$.
Then in $V[G]$, there is a unique $(m,\Omega+1)$-strategy $\Sigma'$ for $M$
such that $\Sigma\sub\Sigma'$ and $\Sigma'$ has inflation condensation.
Moreover, $\Sigma'$ has strong hull condensation.
\end{tm*}

As elsewhere, the background theory for the theorem above is $\ZF$;
the definition of \emph{$\Omega$-cc} in this general context is given in
Definition \ref{dfn:Omega-cc}.

Using the preceding results we deduce the following absoluteness facts:

\begin{cor*}[\ref{cor:wDJ_absoluteness}, \ref{cor:stack_it_absoluteness}]
Let $\Omega>\om$ be regular.
Let $M$ be a countable $m$-sound premouse
and let $e$ be an enumeration of $M$ in ordertype $\omega$.
Let $\PP$ be an $\Omega$-cc forcing and $G$ be $V$-generic for $\PP$.
Then:
\begin{enumerate}[label=--]
\item $V\sats$``There is an $(m,\Omega+1)$-strategy for $M$ with weak Dodd-Jensen with respect to $e$''
iff $V[G]$ satisfies the same statement.
\item If $\Sigma$ is an \tu{(}hence the unique\tu{)} $(m,\Omega+1)$-strategy for $M$ with weak Dodd-Jensen with respect to $e$,
and $\Sigma'$ likewise in $V[G]$, then $\Sigma\sub\Sigma'$.
\item Suppose that $V$ and $V[G]$ satisfy $\DC$.
Then $V\sats$``$M$ is $(m,\Omega,\Omega+1)^*$-iterable''
iff $V[G]$ satisfies the same statement.
\end{enumerate}
\end{cor*}

We expect that given appropriate condensation properties for $\Sigma$, one should be able to deduce
nice condensation properties for $\Sigma^\stk$. We prove one such result here.
\emph{Plus-strong hull condensation},
defined in \ref{dfn:plus-strong_hc}, is a slight technical strengthening of strong hull condensation,
and \emph{normal pullback consistency}, defined in \ref{dfn:npc},
is just pullback consistency for the normal strategy given by pullback
under iteration maps which do not drop in model or degree.

\begin{tm*}[\ref{thm:Sigma_shc_implies_stacks_npc}]
 Let $\Omega>\om$ be regular,
 and let $\Sigma$ be an $(m,\Omega+1)$-iteration strategy
 with plus-strong hull condensation.
Then $\Sigma^\stk$ is normally pullback consistent.
\end{tm*}

\begin{ques}\label{ques:some_questions}
Our results suggest the following questions:
\begin{enumerate}[label=--]
 \item
If $\Omega>\om$ is regular,
does $(n,\Omega+1)$-iterability imply $(n,\Omega,\Omega+1)^*$-iterability, at least for countable premice?
\item If $\Omega>\om_1$ is regular and $M$ is uncountable and $(n,\Omega,\Omega+1)^*$-iterable, then does $M$ have an $(n,\om_1+1)$-strategy with inflation condensation?
\end{enumerate}
\end{ques}

Why consider the methods in this paper? The author's initial motivation for
working on these ideas was
toward proving self-iterability facts in mice (particularly, non-tame mice).
This involves an extension of the methods of \cite{sile} (that paper only
applies to tame mice), using inflation condensation and arguments
related to those in the proof of Theorem
\ref{thm:strat_with_cond_extends_to_generic_ext}. We do not focus
on this method in detail in this paper, but it is
sketched in Remark \ref{rem:self-it}. The second motivation was basically
in wanting to understand the connection between normal iterability
and iterability for stacks, and their role in the proofs of the standard
fine structural properties of mice (such as solidity, etc).
The first part of the author's work on this appeared in
\cite{premouse_inheriting},
and a full proof of the fine structural properties from normal iterability
can be see in the preprint \cite{fsfni}. That proof relies heavily on Theorem
\ref{thm:stacks_of_finite_trees}. Looking forward, the key role of direct limit
systems of mice in the analysis of the $\HOD$ of determinacy models and in the
theory of Varsovian models of mice, also means that normalization can give new
information about those direct limit models $M_\infty$. In fact, the
unpublished work of Steel and the author on full normalization for infinite
stacks (which adapts  results of this paper and of \cite{str_comparison})
gives that such models $M_\infty$ are in fact typically a normal iterate of some
base mouse $M$ (not just embedded into a normal iterate).
This could be useful in understanding those models. The methods
are also useful generally in the Varsovian model analysis to appear in
\cite{vm2}. Finally, Steel's work \cite{str_comparison} on comparison of
iteration strategies (see below) provides a clear motivation for
understanding the techniques.

Other people (including Mitchell, Steel, Neeman, Sargsyan, Fuchs, Schindler,
Jensen and
Siskind) have worked on
aspects of normal realization;
for further discussion see the introduction
to \cite{str_comparison}.
Around the same time the author started this work, John Steel was working
on related calculations, as a component of \cite{str_comparison}.
Steel presented his work on normal realization (which he calls \emph{normalization}) for finite stacks of infinite
trees, at the 3rd M\"unster conference on inner
model theory, in July 2015, which the author attended.
Part\footnote{The work done prior to the M\"unster conference comprises basically of inflation $\Tt\inflatearrow\Xx$ for arbitrary normal trees $\Tt$, the notion of inflation condensation,
genericity inflation for MS-indexing, and normal realization of stacks of the form $(\Tt,\Uu)$ where $\Tt$ is
normal of finite length and $\Uu$ is normal of arbitrary length.} of the work
in this paper was done by the
author prior to being aware of Steel's work, and the remainder
afterward.\footnote{Having earlier failed to understand infinite stacks,
but motivated by Steel's suggestion during the M\"unster conference that one should be able to extend normal realization to them,
the main ideas for that extension were determined by the author during the
conference,
and details finalized shortly thereafter. So by that time, the
key material from
\S\S\ref{sec:inflation},\ref{sec:inf_comm},\ref{sec:factor_tree},
\ref{sec:normal_realization} had been developed,
and also \S\ref{sec:min_inf} excluding genericity iteration for
$\lambda$-indexing.
The main ideas for the proof of
Theorem \ref{thm:strat_with_cond_extends_to_generic_ext}
were found by the author in mid 2016, and some details corrected in August 2018.
Theorem \ref{thm:Sigma_shc_implies_stacks_npc} also came in
August 2018.}
Our approach is also different, most importantly in that we have different axiomatic starting points, and different goals.
In this paper we start with a normal iteration strategy with inflation condensation, and construct an iteration strategy for stacks from this.
Steel starts more or less with an iteration strategy for stacks, demanding
certain properties from this strategy,
and uses these toward his strategy comparison.
The notation and terminology we use is  different from Steel's (as we have not
attempted to align it with his);
this also reflects a difference in how we approach the details of normal realization.
However, many of the basic calculations and observations for dealing with finite
stacks are the same.
Around the same time we developed the methods for infinite stacks,
Steel also worked out representative cases for a somewhat different\footnote{In \S\ref{subsubsec:limit_length_stacks},
for limit $\eta$, the branches of the tree $\Yy_\eta$ are determined directly by
the given normal iteration strategy.
Steel's approach appeared to the author to be somewhat more constructive, with
branches of $\Yy_\eta$ being determined instead by
the trees $\Yy_\alpha$ for $\alpha<\eta$ and maps between these.
However, the author has not gone through the details.}
 approach to this problem. Some time after this, the authors discussed the
 two approaches together.
But
\cite{str_comparison} does not deal with infinite stacks.

Also from around this time, Ronald Jensen also developed normal realization of
finite stacks in the context of $\Sigma^*$-fine structure.
The author sent him a draft version of the present paper containing the main arguments at the end of 2017,
and Jensen then adapted the work contained here
 to infinite stacks in the $\Sigma^*$ context. His work is available in
handwritten form in \cite{jensen_norm}, and in the forthcoming
\cite{jensen_book}.

 The author would like to thank Cody Dance and Jared Holshouser for a
conversation on the topic, in roughly December 2014, during which the author
first started to consider it seriously,
and also John Steel for several conversations on the topic since July 2015.

The paper proceeds as follows.
In \S\ref{subsec:terminology} below we give a summary of basic terminology and
notation.
The results of the paper hold for iteration strategies for Mitchell-Steel (MS-)indexed and $\lambda$-indexed premice,
and many of the results hold for a fairly broad class of coarse structures, \emph{weak coarse premice} (\emph{wcpm}s).
However, iteration trees on MS-indexed premice (formed by the standard rules) are somewhat inconvenient for the main arguments.
So in \S\ref{sec:r-m} we discuss a reorganization of such iteration trees,
which allows us to treat MS-indexed and $\lambda$-indexed premice in a simpler and uniform manner
(except that then for various results we also need to give a short argument translating
between these two forms of iteration trees and corresponding strategies).
The reader who only wants to think about $\lambda$-indexed premice can safely skip this section.
In \S\ref{sec:wcpm} we define wcpms and iteration strategies for them.

The main content of the paper begins in \S\ref{sec:inflation}.
Here we introduce the key notions of the paper: \emph{tree embedding}
and \emph{inflation},
leading to \emph{inflation condensation} and \emph{strong hull condensation}.
A tree embedding $\Pi:\Tt\hookrightarrow\Xx$ embeds the structure of $\Tt$
(tree order, models and extenders) into the structure of $\Xx$ in a certain manner,
but with a key difference to the hulls of trees in the sense of \cite[\S1.6]{tale_hybrid}: Each node $\alpha<\lh(\Tt)$
is associated to a closed \emph{$\Xx$-interval} $[\gamma_\alpha,\delta_\alpha]_\Xx$,
and $M^\Tt_\alpha$ is embedded into $M^\Xx_{\gamma_\alpha}$,
and  $E^\Tt_\alpha$ is embedded into $E^\Xx_{\delta_\alpha}$, but maybe $\gamma_\alpha<\delta_\alpha$.
An inflation of $\Tt$ is an iteration tree $\Xx$
in which each extender $E$ used in $\Xx$ is considered as either \emph{copied} from $\Tt$
or as \emph{$\Tt$-inflationary}. While building an inflation $\Xx$ we keep track of various tree embeddings
from initial segments of $\Tt$ into $\Xx$.
If $E^\Xx_\delta$ is copied from $\Tt$,
then $\delta=\delta_\alpha$ for one of these tree embeddings $\Tt\rest(\alpha+2)\hookrightarrow\Xx$.
The tree embeddings are ``stretched'' by the $\Tt$-inflationary extenders used in $\Xx$.
We also introduce a lot of notation which will be needed throughout.

In \S\ref{sec:min_inf} we describe techniques analogous to comparison of mice
and genericity iteration of mice, but with mice replaced by iteration trees via a strategy
with inflation condensation; these are called \emph{comparison inflation}
and \emph{genericity inflation} respectively. The comparison technique is key
to our main results. We don't actually use the genericity inflation technique in the paper,
but it is natural and seems useful. We also describe in Theorem
\ref{tm:lambda_gen_it}
how genericity iteration for $\lambda$-indexed mice works in general.

In \S\ref{sec:inf_comm} we study the commutativity which results
when we have three iteration trees $\Xx_0$, $\Xx_1$ and $\Xx_2$,
and $\Xx_{i+1}$ is an inflation of $\Xx_i$ for $i=0,1$ (and given that
$\Xx_1$ is an \emph{$\Xx_0$-terminal} inflation of $\Xx_0$).
We show that in this situation, $\Xx_2$ is an inflation of $\Xx_0$,
and ``everything commutes'' in a natural sense. This result is essential in
our analysis of infinite stacks of iteration trees in the construction of $\Sigma^\stk$;
there we will deal with infinite sequences $\left<\Xx_\alpha\right>_{\alpha<\eta}$
in which
$\Xx_\beta$ is an inflation of $\Xx_\alpha$ for each $\alpha<\beta<\eta$.

In \S\ref{sec:gen_abs_mice}
we prove Theorem \ref{thm:strat_with_cond_extends_to_generic_ext}, on extending iteration strategies with strong hull condensation
to generic extensions.

Let $\Xx$ be an inflation of $\Tt$. With the definitions above,
one's focus tends to be on the extenders of $\Xx$ which are copied from $\Tt$
as the central objects, while the $\Tt$-inflationary extenders are in the background.
In \S\ref{sec:factor_tree} we give a second viewpoint which reverses this.
Enumerating the $\Tt$-inflationary extenders as $\left<E^\Xx_{\zeta^\alpha}\right>_{\alpha+1<\iota}$,
we define a natural \emph{factor tree} $<^{\Xx/\Tt}$,
which is an iteration tree order on $\iota$. These things arise in the construction of $\Sigma^\stk$.
Here when forming a tree $\Uu$ on $M^\Tt_\infty$,
and the associated normal tree $\Xx=\Ww^\Sigma_\Tt(\Uu)$, then $\Xx$ will be an inflation of $\Tt$,
and we will have ${<^\Uu}={<^{\Xx/\Tt}}$, and $E^\Uu_\alpha$
will embed into $E^\Xx_{\zeta^\alpha}$.
We also introduce more bookkeeping which will be needed
in the construction of $\Sigma^\stk$.

In \S\ref{sec:normal_realization} we give the construction of the stacks strategy $\Sigma^\stk$ and related proofs.

Finally in \S\ref{sec:npc} we establish some extra properties of $\Sigma^\stk$,
given certain extra properties hold of $\Sigma$. The main result here is Theorem \ref{thm:Sigma_shc_implies_stacks_npc},
on normal pullback consistency.
We also use our results to give a construction of an iteration strategy with weak Dodd-Jensen
in a certain choiceless context.

\subsection{Terminology}\label{subsec:terminology}

See the end for an index of definitions.
We give a summary here of the  basic terminology and notation we use.

\subsubsection{General}
$\univ{M}$\index{$\univ{M}$}\index{universe (of a structure)} denotes the
universe of structure $M$.

\subsubsection{Extenders and ultrapowers}\index{extender}
Given an extender $E$ over $M$, $\Ult(M,E)$\index{$\Ult$}\index{ultrapower}
denotes the ultrapower, formed from functions in $M$,
$i^M_E:M\to\Ult(M,E)$\index{$i^M_E$,$i^{M,m}_E$,$i_E$}
denotes the ultrapower embedding,
and if $M$ is an $n$-sound premouse and $E$ is short, weakly amenable and
$\crit(E)<\rho_n^M$,
then
$i^{M,n}_E:M\to\Ult_n(M,E)$
denotes the degree $n$ ultrapower embedding.
We may write $i_E$ if $M$ is not emphasized.
We write $\crit(E)=\kappa_E$\index{$\crit(E)$}\index{$\kappa_E$} for the
critical point of $E$,
$\lambda(E)=\lambda_E$\index{$\lambda(E)$} for $i_E(\crit(E))$,
$\lh(E)$\index{$\lh(E)$} for the \dfnemph{length}
of $E$ or \dfnemph{support} of $E$
(we take all extenders to be a subset of
$\pow([\theta]^{<\om})\cross[\lambda]^{<\om}$ for some ordinals
$\theta,\lambda$,
and $\lh(E)$ is the least such $\lambda$),
and $\nu(E)=\nu_E$\index{$\nu(E)$} for the strict sup of generators of $E$,
and when $E$ is used in an iteration tree $\Tt$, $\nutilde(E)$
denotes the exchange ordinal associated to $E$; this is explained further
below.\footnote{The
notation should probably literally be $\nutilde^\Tt(E)$, but $\Tt$ will be
known
from context.}
If $E$ is an extender over $V$,
also write $\varrho(E)$\index{$\varrho(E)$} for the
strength of $E$ (the largest $\alpha$ such that $V_\alpha\sub\Ult(V,E)$),
 Say an extender $E$ over $V$ is \dfnemph{suitable}\index{suitable (extender)}
iff $E$ is short and
$\lh(E)=\nu(E)=\varrho(E)$.
 So a suitable extender is coded by a subset of $2^{\crit(E)}+\varrho(E)$.

 \subsubsection{Premice}

 The term \emph{wcpm} (\emph{weak coarse premouse}) is defined in
\S\ref{sec:wcpm}.

The unqualified term \emph{premouse}\index{premouse} means either as in
\cite{zeman},
or as in \cite{outline},
except that we allow extenders of superstrong type to appear on the
extender sequence
(see \cite{operator_mice} (2.43 and 2.44 of preprint v2 on arxiv.org)
regarding this);
here given a premouse $N$ with active extender $F$,
we say that $F$ is of
\dfnemph{superstrong type}\index{superstrong type}
if $\lambda(F)$ is the largest cardinal of $N$.
The premice of \cite{zeman} we call
\dfnemph{$\lambda$-indexed},\index{$\lambda$-indexed}
and those of \cite{outline}
\dfnemph{MS-indexed}.\index{MS-indexed}\index{indexing}
So if $M$ has $\lambda$-indexing then every extender in the extender
sequence of $M$
is of superstrong type.
The \dfnemph{ISC}\index{ISC} (initial segment condition) is then as in
\cite{zeman} or
\cite{outline} respectively.

Let $M,N$ be premice, or other similar structures.
We write $M\ins N$\index{$\ins,\pins$} iff $M$ is an initial segment of $N$,
and $M\pins N$ iff $M\ins N$ and $M\neq N$.
We write $F^M$\index{$F^M$} for the active extender of $M$,
$\es^M$\index{$\es^M,\es_+^M$} denotes the extender sequence of $M$, excluding
$F^M$,
$\es_+^M$ denotes $\es^M\conc F^M$,
$M^\passive$\index{$M^\passive$} denotes the passivization of $M$ (that is, if
$M=(\J_\alpha^{\es},\es,F)$
then $M^\passive=(\J_\alpha^{\es},\es,\emptyset)$),
and given a limit ordinal $\alpha\leq\OR^N$, $N|\alpha$ denotes the $M\ins N$ such that $\OR^M=\alpha$,
and $N||\alpha$ denotes
$(N|\alpha)^\passive$.\index{$M\vert\alpha,M\vert\vert\alpha$}

If $M$ is a type 3 MS-indexed premouse,
then $M^\sq$\index{$M^\sq,M^\unsq$} denotes the squash of $M$.
If $N$ is a structure in the language of squashed premice,
then $N^\unsq$ denotes the unique such $M$ such that
$M^\sq=N$, if this exists. For other kinds of MS-indexed
premice $P$, $P^\sq=P^\unsq=P$.
(But we do not define squashing in the context
of $\udash$fine structure (\S\ref{sec:r-m})).

Given premice $M,N$ and $m,n\leq\om$ such that $M$ is $m$-sound and $N$
is $n$-sound,
we write $(M,m)\ins(N,n)$\index{$\ins,\pins$} iff either $M\pins N$ or [$M=N$
and $m\leq n$].
We write $(M,m)\pins(N,n)$ iff $(M,m)\ins(N,n)$ and $(M,m)\neq(N,n)$.
We similarly define $(M,m)\ins(N,n)$ and $(M,m)\pins(N,n)$ when $M$ is $\udash m$-sound
and $N$ is $\udash n$-sound (see \S\ref{sec:r-m}).

A \dfnemph{segmented-premouse} (\dfnemph{seg-pm})\index{segmented-premouse,
seg-pm} is a structure $N$ satisfying
all
requirements of premice (either MS-indexed or $\lambda$-indexed), except that if $F^N\neq\emptyset$ then we do not require that $N$ satisfy
the ISC (either in the sense of \cite{outline} or \cite{zeman}, as is
appropriate); we still require in this case that $N$ has a largest cardinal
$\delta$ and
\[ \Ult(N,F^N)|(\delta^+)^{\Ult(N,F^N)}=N||\OR^N,\]
and all proper segments of seg-pms must satisfy the ISC.
In particular, if $N$ is a premouse then $N$ is a
seg-pm, and if $N$ is a seg-pm then $N^\passive$ is a premouse. If $N$ is
active then $\In(F^N)$\index{$\In(E)$} (for \emph{index}) denotes $\OR^N$. We
also use ``$\In$'' for an analogous role
in connection with coarse structures; see \S\ref{sec:wcpm}.
Given a seg-pm $M$ with largest cardinal $\delta$, $\lgcd(M)$\index{$\lgcd(M)$}
denotes $\delta$.
We extend the terminology and notation for premice mentioned above
to seg-pms in the natural way.

\subsubsection{Fine structure}

We officially follow Mitchell-Steel fine structure, as simplified
in \cite{V=HODX}; however, the paper is predominantly
not particularly dependent on which version of fine structure one uses.
For ``$\udash$'' fine structure, see \S\ref{sec:r-m}.
For the notation \index{$\Hull$}\index{$\cHull$}$\Hull^M_m$ and $\cHull^M_m$
see \cite[\S1.1.3]{premouse_inheriting}.

\subsubsection{Iteration trees}

Beyond what is described here,
there is also terminology specific to iteration trees
introduced in \S\ref{subsec:it_tree_terms}.
We formally understand \emph{iteration trees} on premice and seg-pms basically
as defined in \cite{fsit}, and in particular, of the form
\[
\Tt=(<^\Tt,\dropset^\Tt,\deg^\Tt,\left<M^\Tt_\alpha\right>_{\alpha<\lambda},
\left<M^{*\Tt}_{\alpha+1},E^\Tt_\alpha\right>_{
\alpha+1<\lambda})\]
where:
\begin{enumerate}[label=--]
 \item $\lh(\Tt)=\lambda\in[1,\OR)$,\index{$\lh(\Tt)$}
 \item $<^\Tt$ is the associated tree order on
$\lambda$,\index{$<^\Tt$}
 \item $\dropset^\Tt$ is the set of all $\alpha+1<\lambda$ such that $\Tt$ drops
at $\alpha+1$,\index{$\dropset^\Tt$}
\item $\deg^\Tt:\lambda\to\om+1$\index{$\deg^\Tt$}
is a \emph{total}\footnote{The requirement of totality might differ from \cite{fsit},
depending on the reader's interpretation.} function,
\item  $M^{*\Tt}_{\alpha+1}$\index{$M^{*\Tt}_\alpha$} is the model to which
$E^\Tt_\alpha$ applies in
forming $M^\Tt_{\alpha+1}$.
\end{enumerate}
We take iteration trees on other structures with analogous formal .
We also use the following notation
(some of which is only relevant to trees on seg-pms):
\begin{enumerate}[label=--]
 \item  If $\alpha\leq^\Tt\beta$ then
$(\alpha,\beta]_\Tt$\index{$(\alpha,\beta]_\Tt$ etc}
denotes the half-open $<^\Tt$-interval,
and likewise for other such intervals.
\item  $\pred^\Tt(\alpha+1)$\index{$\pred^\Tt$} denotes the $<^\Tt$-predecessor
of $\alpha+1$ (so
$M^{*\Tt}_{\alpha+1}\ins M^\Tt_{\pred^\Tt(\alpha+1)}$),
\item If $\alpha<^\Tt\beta$ then
$\successor^\Tt(\alpha,\beta)$\index{$\successor^\Tt$} denotes
$\min((\alpha,\beta]_\Tt)$.
\item If $\Tt$ has successor length $\alpha+1$, then
$M^\Tt_\infty$\index{$M^\Tt_\infty$, $i^\Tt_{0\infty}$, etc} denotes
$M^\Tt_\alpha$,
and $\infty$ also denotes $\alpha$ in other related notation,
and $b^\Tt$ denotes $[0,\infty]_\Tt$, the last branch of $\Tt$.\index{$b^\Tt$}
\item
If $(\alpha,\beta]_\Tt\inter\dropset^\Tt=\emptyset$ then
$i^\Tt_{\alpha\beta}=i^\Tt_{\alpha,\beta}:M^\Tt_\alpha\to
M^\Tt_\beta$\index{$i^\Tt_{\alpha\beta},i^{*\Tt}_{\alpha\beta}$} is
the iteration map.
\item $i^{*\Tt}_{\alpha+1}:M^{*\Tt}_{\alpha+1}\to M^\Tt_{\alpha+1}$ denotes the ultrapower
map.\index{$i^\Tt_{\alpha\beta},i^{*\Tt}_{\alpha\beta}$}
\item $i^{*\Tt}_{\alpha+1,\beta}$
\index{$i^\Tt_{\alpha\beta},i^{*\Tt}_{\alpha\beta}$} denotes
$i^\Tt_{\alpha+1,\beta}\com i^{*\Tt}_{\alpha+1}$, when this exists.
\item $\dropset^\Tt_{\deg}$\index{$\dropset^\Tt_{\deg}$} denotes the set of all
$\alpha+1<\lambda$ such that $\Tt$ drops in either model or degree at
$\alpha+1$.
\item $\lh(\Tt)^-$\index{$\lh(\Tt)^-$} denotes the
set of all  $\beta$ such that $\beta+1<\lh(\Tt)$.
\item $\nutilde^\Tt_\alpha=\nutilde(E^\Tt_\alpha)$ is the exchange ordinal
associated to $E^\Tt_\alpha$; see below.
\item $\exit^\Tt_\alpha$\index{$\exit^\Tt_\alpha$} denotes
$M^\Tt_\alpha|\In(E^\Tt_\alpha)$.
\item \index{$\delta(\Tt)$} if $\Tt$ has limit length, then if $\Tt$ is
$k$-maximal (see below)
on a seg-pm, then $\delta(\Tt)$
denotes $\sup_{\alpha<\lh(\Tt)}\In(E^\Tt_\alpha)$,
and if $\Tt$ is normal on a wcpm,
then $\delta(\Tt)$ denotes
$\sup_{\alpha<\lh(\Tt)}\varrho^{M^\Tt_\alpha}(E^\Tt_\alpha)$.
\end{enumerate}

However, if $M^\Tt_0$ has MS-indexing then we can
have $\In(E^\Tt_\alpha)=\In(E^\Tt_{\alpha+1})$, because we allow superstrong
extenders in $\es_+(M^\Tt_0)$. Considering this, an iteration tree $\Tt$ is
\dfnemph{$k$-maximal}\index{maximal ($m$-maximal)}\index{$m$-maximal} iff
$\deg^\Tt(0)=k$ and $\Tt$ satisfies the requirements specified
in \cite[\S3.1]{outline} for $k$-maximality, except that as in
\cite{operator_mice} and \cite{premouse_inheriting}, we only require
$\In(E^\Tt_\alpha)\leq\In(E^\Tt_\beta)$ when  $\alpha+1<\beta+1<\lh(\Tt)$, not
$\In(E^\Tt_\alpha)<\In(E^\Tt_\beta)$.
\setcounter{footnote}{0}\footnote{See  \cite{operator_mice}
(2.43 and 2.44 in preprint v2 on arxiv.org)
and \cite[\S1.1.6]{premouse_inheriting} for further discussion. The algorithm
for comparison (by least disagreement) should also be slightly adjusted as in
one of those papers  (see Footnote \ref{ftn:comp_alg}; though in the proof of
Corollary \ref{cor:wDJ_absoluteness} we use  the conventional algorithm). In a
draft of this article on arxiv.org, \emph{$k$-maximal} was defined
inadvertently as in \cite{outline} (even though Remark 2.44 of
\cite{operator_mice} was also mentioned), which does not suffice.}
So if an
iteration tree is both $k$-maximal and $j$-maximal, then $k=j$.\footnote{The
definition of
\emph{iteration tree} $\Tt$ in \cite{outline} differs slightly from here and from
\cite{fsit}, in that $\deg^\Tt$ is not formally a component of $\Tt$. So in the terminology of
\cite{outline}, a tree can be both $k$-maximal and $j$-maximal, with $k\neq j$.} This helps a
little notationally. A \index{putative (tree)}\dfnemph{putative $k$-maximal
tree $\Tt$} is a system satisfying
the conditions of a $k$-maximal tree,
except that if $\Tt$ has length $\alpha+1$
where $\alpha$ is a limit, then we do not demand
that $[0,\alpha)_\Tt\inter\dropset^\Tt$
is finite
(so $M^\Tt_\alpha$ is well-defined iff $[0,\alpha)_\Tt\inter\dropset^\Tt$ is
finite),
and if $\Tt$ has length $\beta+1$ and $M^\Tt_\beta$
is well-defined, we do not demand that $M^\Tt_\beta$ be wellfounded.\footnote{In
a draft of this article on arxiv.org, the term \emph{putative} (tree)
precluded having infinitely many drops on a branch in the tree.}
See \S\ref{sec:r-m} for the particulars of (putative) $\udash m$-maximal trees.

Given an iteration tree $\Tt$ and $E=E^\Tt_\alpha$, we write
$\nutilde^\Tt_\alpha=\nutilde(E)$
\index{$\nutilde^\Tt_\alpha,\nutilde(E^\Tt_\alpha)$} for the exchange ordinal
associated
to $E$ with respect to $\Tt$.
So for $m$-maximal trees with $\lambda$-iteration rules on $\lambda$-indexed
premice, $\nutilde^\Tt_\alpha=\lambda(E)=\lgcd(\exit^\Tt_\alpha)$,
whereas for $m$-maximal trees with MS-iteration rules on MS-indexed premice,
$\nutilde^\Tt_\alpha=\nu(E)$.
However, we also deal with $\udash m$-maximal trees (see \S\ref{sec:r-m}), on
MS-indexed premice
or other seg-pms, where $\nu(E)\leq\nutilde^\Tt_\alpha\leq\lambda(E)$,
and strict inequalities are possible. And for coarse trees on wcpms,
$\nutilde^\Tt_\alpha=\strength(E)$  (see \S\ref{sec:wcpm}).

Given a $q$-sound premouse $Q$ where $q\leq\om$, a \dfnemph{$q$-maximal
stack}\index{stack}\index{maximal ($m$-maximal)}\index{$m$-maximal}
on $Q$
is a sequence $\Ttvec=\left<\Tt_\alpha\right>_{\alpha<\lambda}$ of iteration
trees such that
for some $\left<Q_\alpha,q_\alpha,M_\alpha,m_\alpha\right>_{\alpha<\lambda}$,
$\Tt_\alpha$ is an $m_\alpha$-maximal tree on $M_\alpha$,
$Q_0=Q$, $q_0=q$, $(M_\alpha,m_\alpha)\ins(Q_\alpha,q_\alpha)$,
and if
$\alpha+1<\lambda$ then $\Tt_\alpha$ has successor length
and $Q_{\alpha+1}=M^{\Tt_\alpha}_\infty$ and
$q_{\alpha+1}=\deg^{\Tt_\alpha}(\infty)$,
and for limit $\eta<\lambda$, for all sufficiently large $\alpha<\eta$,
$(M_\alpha,m_\alpha)=(Q_\alpha,q_\alpha)$,
$\Tt_\alpha$ does not drop on $b^{\Tt_\alpha}$, and
$Q_\eta=M^{\Ttvec\rest\eta}_\infty$ is the resulting direct limit
of the $M_\alpha$ for $\alpha<\eta$ under the iteration maps and
$q_\eta=\lim_{\alpha\to\eta}\deg^{\Tt_\alpha}(\infty)$.
If $\lambda$ is a limit, we define
$M^{\Ttvec}_\infty$\index{$M^{\Ttvec}_\infty$}\index{$\deg^\Ttvec(\infty)$} and
$\deg^\Ttvec(\infty)$ as the natural direct limits,
given that $\Tt_\alpha$ does not drop along $b^{\Tt_\alpha}$, etc, for all
sufficiently large $\alpha$. We say an \dfnemph{artificial
drop}\index{artificial (drop)} occurs
whenever $(M_\alpha,m_\alpha)\pins(Q_\alpha,q_\alpha)$.
An \dfnemph{optimal stack}\index{optimal stack} is one without artificial
drops.
A \dfnemph{putative $q$-maximal stack}\index{putative (stack)} is as above,
except that
if it has
length $\lambda=\alpha+1$,
then $\Tt_\alpha$ is only required to be a putative tree.
Likewise a (putative) $\udash q$-maximal stack on a $\udash q$-sound seg-pm.

The iteration game
$\Gg(M,m,\alpha,\beta)^*$\index{$\Gg(M,\ldots)^*$}\index{iteration strategy} of
\cite[p.~1202]{coremore},\footnote{In a draft of this paper
which appeared on the preprint server arxiv.org, \emph{stacks} were defined to
be what we call \emph{optimal stacks} here;
thus, no artificial drops were considered in that draft.
However, this is a more restrictive notion than Steel's
definitions in \cite{outline},
which do allow non-optimal stacks.
We likewise stated that $\Gg(M,m,\alpha,\beta)^*$
was defined without permitting artificial drops,
which is not consistent with \cite{coremore}.
Therefore, in that draft, we only constructed strategies
for optimal stacks, not more generally.
This oversight has now been amended,
primarily through Lemma \ref{lem:sub-optimal_reduces_to_optimal},
but also see the proof of Theorem \ref{thm:stacks_of_finite_trees}
in \S\ref{subsubsec:stacks_of_finite_trees}, and
the start of \S\ref{subsubsec:variants_partial},
and the proof of Theorem
\ref{tm:stacks_strat_inherits_DJ}.}
consists of $\lambda\leq\alpha$ many rounds, producing a putative $m$-maximal stack
$\left<\Tt_\gamma\right>_{\gamma<\lambda}$ on $M$,
with associated sequence
$\left<Q_\gamma,q_\gamma,M_\gamma,m_\gamma\right>_{\gamma<\lambda}$.
In round $\gamma$, given $(Q_\gamma,q_\gamma)$,
player I chooses $(M_\gamma,m_\gamma)\ins(Q_\gamma,q_\gamma)$,
and then the players
build the putative
tree $\Tt_\gamma$
($m_\gamma$-maximal, on $M_\gamma$),
of length $\leq\beta$.
If some model of $\Tt_\gamma$ is
ill-defined or illfounded then $\lambda=\gamma+1$ and
player I wins.
Having produced a bona fide tree $\Tt_\gamma\rest(\xi+1)$, where $\xi+1<\beta$,
player I may set $\Tt_\gamma=\Tt_\gamma\rest(\xi+1)$ and exit the round,
and then $\lambda>\gamma+1$ (so round $\gamma+1$ will be played).
If player I does not exit at any such stage $\xi+1<\beta$ and $\Tt_\gamma$ has wellfounded models
then $\lambda=\gamma+1$ and player II wins. Given a limit $\gamma\leq\alpha$, player II must ensure
that $M^{\Ttvec\rest\gamma}_\infty$ is well-defined and wellfounded;
given this, if $\gamma=\alpha$ then player II wins, whereas if $\gamma<\alpha$
then $\lambda>\gamma$ and play
continues.

For $\alpha$ a limit ordinal, the game
$\Gg(M,m,{<\alpha},\beta)^*$\index{$\Gg(M,\ldots)^*$} has the same
rules,
except that if all $\alpha$ rounds are played through with no player having yet lost, then player II wins automatically,
irrespective of whether $M^{\Ttvec}_\infty$ is well-defined or wellfounded.

We define the \dfnemph{optimal} variants of these games,
denoted $\Gg_{\mathrm{opt}}(M,\alpha,\beta)^{*}$
and $\Gg_{\mathrm{opt}}(M,m,{<\alpha},\beta)^*$,\index{$\Gg_{\mathrm{opt}}$}
with the same rules and payoffs as the games above, except that player I may
not make artificial drops. So the optimal variants are superficially easier for
player II. However, a straightforward copying argument,
given in Lemma \ref{lem:sub-optimal_reduces_to_optimal},
 which is much as in \cite[\S7]{mim},
 shows that if $\Sigma$
is a winning strategy for player II in $\Gg_{\mathrm{opt}}(M,m,\alpha,\beta)^*$,
then $\Sigma$ induces a canonical strategy $\Sigma'$
for II in $\Gg(M,m,\alpha,\beta)^*$. Thus, in this paper, our main focus
is on strategies for normal trees and for optimal stacks of normal trees.

\subsubsection{Embeddings}

For the definition of \emph{$n$-lifting} embedding see \cite[Definition
2.1]{premouse_inheriting}.\index{$m$-lifting}\index{lifting}

Let $\pi:P\to Q$ be an embedding between seg-pms.
We say that $\pi$ is \dfnemph{c-preserving} iff it is cardinal
preserving, in that $\alpha$ is a cardinal of $P$ iff $\pi(\alpha)$ is a cardinal of $Q$.
If $n=0$ or $P,Q$ are $(n-1)$-sound,
we say that $\pi$ is \dfnemph{$\pvec_n$-preserving} iff
$\pi(\pvec_n^P)=\pvec_n^Q$.
We say that $\pi$ is \dfnemph{nice $n$-lifting} iff $\pi$ is $n$-lifting,
c-preserving and $\pvec_n$-preserving.
Note that every near $n$-embedding is nice $n$-lifting.\index{c-preserving}
\index{$\pvec_n$-preserving}\index{nice}\index{preserving}

\section{$\udash m$-maximal iteration
strategies}\label{sec:r-m}

The paper will deal with a lot of copying of iteration trees,
requiring much associated bookkeeping. We deal with both kinds of premice
--
MS-indexed and $\lambda$-indexed --
and also weak coarse premice. Recall that the standard copying algorithm
does not quite work with type 3 MS-indexed premice. If we used here the standard
fix to this problem (inserting extra
extenders and slight modifications of tree order), we would need to integrate
that fix into our
bookkeeping, increasing notational and mental load.
Fortunately, there is an alternate path, which we will adopt, which in the end
allows us
to separate the type 3 problem from the current bookkeeping. In this section we
describe
this path.

Whenever
we say \emph{type $i$ premouse $M$}, where $i\in\{0,1,2,3\}$,
we mean that $M$ is MS-indexed.
Everything in the present section is trivial for $\lambda$-indexed premice,
and if the
reader is happy to ignore the existence of type 3  premice, then they
would have no problem
ignoring the present section, as long as they replace all later
instances of
``$\udash m$'' with
``$m$'', where $m\leq\om$, and as long as they imagine that all fine structural
embeddings $\pi:M\to N$ between premice are such that $\dom(\pi)=M$ (not just
$M^\sq$), and if $M$
is active then $\pi(\nu(F^M))=\nu(F^N)$.

\begin{rem}\index{$\udash$fine structure}\index{u-fine structure}
 The prefix ``$\mathrm{u}$'' stands for \emph{unsquashed}. It simply indicates
that we compute
 fine structure, ultrapowers, etc, at the unsquashed level, with the active
extender coded
 by the standard amenable predicate,
 just as is usually done for type $1$ or $2$
premice and $\lambda$-premice.
 Thus, for $\lambda$-premice and type $\leq 2$
premice, there is no
difference between
standard fine structure and ``$\mathrm{u}$'' fine structure.
 For type $3$, it represents a shift of $1$ degree of complexity in the Levy
hierarchy.
 However, because we also allow unsquashed ultrapowers, we also encounter
seg-pms for which the Initial Segment Condition fails.
\end{rem}

\begin{dfn}\index{type A, B}\index{u-fine structure}
 Let $n\leq\om$ and let $M$ be a segmented-premouse. We say that $M$ is
\dfnemph{$\uu$-$n$-sound} iff
either
\begin{enumerate}[label=\arabic*.,ref=\arabic*]
 \item\label{item:typeA} $M$ is an $n$-sound premouse not
of type 3, or
 \item\label{item:typeBn_geq_2}
 $n\geq 2$ and $M$ is an
 $(n-1)$-sound type 3  premouse,
 where
$\om-1=\om$, or
 \item\label{item:typeBn=1} $n=1$ and $M$ is active and letting $\nu=\nu(F^M)$,
there is an active
type 3 premouse $M'$ such that $\nu(F^{M'})=\nu$ and
$F^{M'}\rest\nu=F^M\rest\nu$ and letting
$\delta=\lgcd(M)$ and $U=\Ult(M',F^{M'})$, we have $M||\OR^M=U|(\delta^+)^U$, or
 \item\label{item:typeBn=0} $n=0$ and $\nu(F^M)\leq\lgcd(M)$.
\end{enumerate}

Suppose $M$ is $\uu$-$n$-sound. We say that $M$ is \dfnemph{type A} iff clause
\ref{item:typeA}
above holds; otherwise we say that $M$ is \dfnemph{type B}. If either $M$ is
type
A, or $M$ is type B
and $n\geq 2$, let $M^\pm=M$. If $M$ is type B and $n=1$ let
\index{$M^\pm$}$M^\pm=M'$, as
above. If $M$ is
type B, but not $\uu$-$n$-sound for any $n\geq 1$, then $M^\pm$ is not defined.

Let $M,N$ be $\uu$-$n$-sound segmented-premice and $\pi:M\to N$. Here the domain
and codomain of $\pi$
are literally (the universes of) $M,N$, not $M^\sq,N^\sq$.
We say that $\pi$ is a \dfnemph{\tu{(}near\tu{)}
$\uu$-$n$-embedding}\index{u-fine structure}\index{near u-$n$-embedding} iff
either:
\begin{enumerate}[label=\arabic*.,ref=\arabic*]
 \item $M,N$ are type A and $\pi$ is a (near) $n$-embedding, or
 \item $M,N$ are type B and and $n\geq 1$ and
\index{$\pi^\sq$} $\pi^\sq=\pi\rest(M^\pm)^\sq:M^\pm\to N^\pm$
 is a (near) $(n-1)$-embedding and $\pi$ is induced by $\pi^\sq$ and
$\pi(\lgcd(M))=\lgcd(N)$,\footnote{Recall
that the convention,
for a fine structural embedding
$\pi:M\to N$ between type 3 premice,
is that literally $\pi:M^\sq\to N^\sq$.} or
 \item $M,N$ are type B and $n=0$ and $\pi$ is a (near) $0$-embedding ($\pi$ is
a near
$0$-embedding iff $\pi$ is $\rSigma_1$-elementary in the language of
segmented-premice, and $\pi$
is
a $0$-embedding iff $\pi$ is a near $0$-embedding and is cofinal in $\OR^N$).
\end{enumerate}

The notion \dfnemph{$\udash n$-lifting embedding}
\index{u-fine structure} is defined by making
analogous
changes to the notion
\emph{$n$-lifting embedding} (defined in \cite{premouse_inheriting}).
\end{dfn}

\begin{dfn}\index{$\nutilde^M$}
For an active seg-pm $M$,
$\nutilde^M\eqdef\max(\nu(F^M),\lgcd(M))$.
\end{dfn}

Note that if $M,N$ are active and $\uu$-$n$-sound and $\pi:M\to N$ is a (near)
$\uu$-$n$-embedding
then
$\pi(\lgcd(M))=\lgcd(N)$ and $\pi(\nutilde^M)=\nutilde^N$.
For $\pi(\lgcd(M))=\lgcd(N)$
because $\pi$ respects the predicates for $F^M,F^N$. And if $M,N$ are type B
then
$\nu(F^M)\leq\lgcd(M)$ and $\nu(F^N)\leq\lgcd(N)$; therefore
$\pi(\nutilde^M)=\nutilde^N$.

\begin{dfn}\label{dfn:u-rho_n}\index{$\urho_n$}\index{u$\rho_n$}
For a $\uu$-$n$-sound seg-pm $M$, $\urho_n^M$ denotes $\rho$ where
either:
 \begin{enumerate}[label=--]
  \item $M$ is type A and $\rho=\rho_n^M$, or
  \item $M$ is type B and $n\geq 1$ and $\rho=\rho_{n-1}^{M^\pm}$, or
  \item $M$ is type B and $n=0$ and $\rho=\OR^M$.\qedhere
 \end{enumerate}
\end{dfn}

\begin{dfn}\index{$\Ult_{\udash n}$}
 Let $M$ be a $\uu$-$n$-sound seg-pm and let $E$ be a weakly amenable extender
such that
$\crit(E)<\urho_n^M$. Then $\Ult_{\uu\text{-}n}(M,E)=U$ where either:
\begin{enumerate}[label=\arabic*.,ref=\arabic*]
 \item\label{item:type_A_ult} $M$ is type A and $U=\Ult_n(M,E)$, or
 \item\label{item:type_B_sq_ult} $M$ is type B and $n\geq 2$ and
$U=\Ult_{n-1}(M,E)$, or
 \item\label{item:type_B_n_nosq_ult} $M$ is type B and $n\leq 1$ and
$U=\Ult(M,E)$ (so the ultrapower is direct; there is no
squashing).
\end{enumerate}
We also define\index{$i^{M,\udash n}_E$}
$i^{M,\udash n}_{E}:M\to U$, abbreviated $i^M_E$,
to be the (total) ultrapower map in cases \ref{item:type_A_ult}
and \ref{item:type_B_n_nosq_ult}, or the (total) map it induces in case
\ref{item:type_B_sq_ult}.
\end{dfn}

The following lemma is a standard calculation:

\begin{lem} Let $M,E,n$ be as above, and suppose that $U=\Ult_{\udash n}(M,E)$
is
wellfounded. Then $U$ is $\uu$-$n$-sound and $i^M_{E}$ is a
$\uu$-$n$-embedding.\end{lem}

\begin{rem}\label{rem:u-Sigma_1}
In the definition of $\Ult_{\udash n}(M,E)$ above,
the reader might expect that if $M$ is type B and $n=1$, it would be more
natural to define
the ultrapower using all functions which are $\bfSigma_1^M$-definable,
instead of just the functions in $M$.
We digress to show that these two ultrapowers are equivalent (the content of
this remark is not needed in the sequel).

 Let $M$ be a type B with $n=1$. Write $\uSigma_1^M$ for the definability class
over
 \[ M=(\univ{M},\es^M,\widetilde{F^M}) \]
 itself,
 not its squash.
 Here $\widetilde{F^M}$ is the standard amenable coding of $F^M$. Let
$\bfuSigma_1^M$ be the associated boldface class.
By definition we have $\urho_1^M=\rho_0^M=\nu^M$.
 In fact, $\urho_1^M$ is the least $\rho$ such that there is a $\bfuSigma_1^M$
subset of $\rho$
 not in $M$; see \cite{combinatorial} or the proof of
\cite[Lemma 2.15]{extmax}.
 Given $\eta<\OR^M$, let $M\wr\eta$ be the usual restriction
 of $M$ (with its predicates) to $M||\eta$, that is,
 \[ M\wr\eta=(M||\eta,\es^M\rest\eta,\widetilde{F^M}\inter(M||\eta)). \]
So the structures $M\wr\eta$ stratify $M$ as usual.

 Suppose $\nu^M$ is regular in $M$ but $\bfuSigma_1^M$-singular,
 in the weak sense that there is some $\gamma<\nu^M$
 and $x\in M$ such that $\Hull_{\uSigma_1}^M(\gamma\cup\{x\})$ is cofinal in
$\nu^M$. Let $\gamma$ be least such. Note  that there is $f:\gamma\to\nu^M$
which is $\bfuSigma_1^M$-definable, with $f``\gamma$ cofinal
in $\nu^M$ (that is, because of the characterization of $\urho_1^M=\nu^M$ just
mentioned, we can recover a function $f$ with domain $\gamma$).

We claim
 $\gamma=(\mu^+)^M$ where $\mu=\crit(F^M)$.
 For we have the standard cofinal monotone increasing $\uSigma_1^M$ map
$h:(\mu^+)^M\to\OR^M$ (derived from the amenable coding of $F^M$).
Given $\alpha<(\mu^+)^M$, let $D_\alpha\sub\gamma$
be the set of all $\beta<\gamma$ such that
$M\wr h(\alpha)\sats$``$f(\beta)$ is defined'',
and let $f_\alpha:D_\alpha\to\nu^M$ be the corresponding function.
So $\gamma=\bigcup_{\alpha<(\mu^+)^M}D_\alpha$
and $f=\bigcup_{\alpha<(\mu^+)^D}f_\alpha$. But $f_\alpha\in M$,
and since $\nu^M$ is $M$-regular,
therefore $\rg(f_\alpha)$ is bounded in $\nu^M$.
So defining $j:(\mu^+)^M\to\nu^M$ by $j(\alpha)=\sup\rg(f_\alpha)$,
then $j$ is $\bfuSigma_1^M$ and cofinal (and monotone increasing)
in $\nu^M$. So $\gamma\leq(\mu^+)^M$. Also
since $f\notin M$, there are cofinally many
$\alpha,\beta<(\mu^+)^M$ such that $D_\alpha\psub D_\beta$.
But then since
 $\left<D_\alpha\right>_{\alpha<(\mu^+)^M}\in M$
 and $(\mu^+)^M$ is regular in $M$,  we cannot have
 $\gamma<(\mu^+)^M$.

 Now let $E$ be a weakly amenable $M$-extender with $\kappa=\crit(E)<\nu^M$
 and $E_a\in M$ for all $a$ (because $M$ is type 3, this will be the case for
extenders $E$ applied to $M$
 in a normal iteration tree).
 We claim that $\Ult(M,E)$ is equivalent to the ultrapower formed
 using all $\bfuSigma_1^M$ functions.

 For this, let  $f:\kappa^{|a|}\to M$ be a $\bfuSigma_1^M$ function. We want to
see that there is
 $f'\in M$ and $A\in E_a$ such that $f'\rest A=f\rest A$.
 For $\eta<(\mu^+)^M$, let $f_\eta:D_\eta\to M$ be like before.
 So $f_\eta\in M$ and $f=\bigcup_{\eta<(\mu^+)^M}f_\eta$.
 If $\kappa\leq\mu$ then since $(\mu^+)^M$ is regular in $M$,
 there is $\eta$ such that $f=f_\eta$, which suffices.
 Suppose $\kappa>\mu$, so $\kappa>(\mu^+)^M$.
 Then $\left<D_\eta\right>_{\eta<(\mu^+)^M}\in M$,
 and since $\bigcup_{\eta<(\mu^+)^M}=\kappa$,
 it follows that some $D_\eta\in E_a$,
 so $f_\eta,D_\eta$ works.
\end{rem}

\begin{dfn}\label{dfn:r-k-maximal_tree}\index{$\udash
n$-maximal}\index{u-fine structure}
\index{$\udeg$}\index{u-$\deg$}
Let $k\leq\om$ and $\lambda\in\OR\cut\{0\}$ and let $M$ be a $\uu$-$k$-sound
seg-pm.
A \dfnemph{$\udash k$-maximal iteration tree $\Tt$ on $M$} of \dfnemph{length
$\lambda$} is a tuple
\index{$\dropset^\Tt$}
\index{$M^\Tt_\alpha$}
\index{$i^\Tt_{\alpha\beta},i^{*\Tt}_{\alpha\beta}$}
\index{$E^\Tt_\alpha$}
\index{$\exit^\Tt_\alpha$}
\index{$\nutilde^\Tt_\alpha,\nutilde(E^\Tt_\alpha)$}
\index{$M^{*\Tt}_\alpha$}
\[
\left(<^\Tt,\dropset, \udeg,\left<M_\alpha\right>_{\alpha<\lambda},
\left<i_{\alpha\beta},i^*_{\alpha\beta}\right>_{\alpha , \beta<\lambda},
\left<E_\alpha,\exit_\alpha,\nutilde_\alpha,M^*_{\alpha+1}\right>_{
\alpha+1<\lambda}
\right), \]
such that:\begin{enumerate}[label=\arabic*.,ref=\arabic*]
 \item $\dropset\sub\lambda$ and $\udeg:\lambda\to\{-1\}\un(\om+1)$.
 \item $<^\Tt$ is an iteration tree order on $\lambda$.
 \item $M_0=M$ and $0\notin D$ and
$\udeg(0)=k$.
 \item\label{item:M_beta_seg-premouse} For all $\beta<\lambda$, $M_\beta$ is a
$\udeg(\beta)$-sound segmented-premouse.
 \item For all $\alpha+1\leq\beta+1<\lambda$, $\emptyset\neq
E_\alpha\in\es_+^{M_\alpha}$ and
$\exit_\alpha=M_\alpha|\In(E_\alpha)$ and $\In(E_\alpha)\leq\In(E_\beta)$ and
$\nutilde_\alpha=\nutilde^{\exit_\alpha}$.
 \item For all $\alpha+1<\lambda$, letting $\kappa=\crit(E_\alpha)$:
 \begin{enumerate}
  \item $\beta=\pred^\Tt(\alpha+1)$ is the least $\xi$ such that
$\kappa<\nutilde_\xi$.
\item $M^*_{\alpha+1}=$ least $N\ins M_\beta$ with
$N=M_\beta$ or [$\exit_\beta\ins N$ and $\rho_\om^N\leq\kappa$].
\item $\alpha+1\in\dropset$ iff $M^*_{\alpha+1}\pins M_\beta$.
\item If $\alpha+1\notin\dropset$ then $\udeg(\alpha+1)=$ largest $n\leq
\udeg(\beta)$ with $\kappa<\urho_n^{M_\beta}$.
\item If $\alpha+1\in\dropset$ then $\udeg(\alpha+1)=$ largest $n<\om$ with
$\kappa<\urho_n(M^*_{\alpha+1})$.
\item Let $n=\udeg(\alpha+1)$. Then
$M_{\alpha+1}=\Ult_{\udash n}(M^*_{\alpha+1},E_\alpha)$
and
$i^*_{\alpha+1,\alpha+1}=i^{M^*_{\alpha+1},\udash n}_{E}$.
Let $\gamma\leq^\Tt\beta$ with $(\gamma,\alpha+1]\inter\dropset=\emptyset$.
Then
\[ i_{\gamma,\alpha+1}=i^*_{\alpha+1,\alpha+1}\com i_{\gamma\beta}, \]
and if $\gamma$ is a successor then
$i^*_{\gamma,\alpha+1}=i^*_{\alpha+1,\alpha+1}\com i^*_{\gamma\beta}$.
\end{enumerate}
\item\label{item:iteration_maps_commute} Let
$\alpha\leq^\Tt\gamma\leq^\Tt\beta<\lambda$ be such that
$(\alpha,\beta]_\Tt\inter\dropset=\emptyset$.
Then $i_{\alpha\beta}$ is defined and
$i_{\alpha\beta}=i_{\gamma\beta}\com i_{\alpha\gamma}$
and $\udeg(\beta)\leq\udeg(\alpha)$. (This condition follows from the others.)
\item\label{item:ev_no_drop} Let $\eta<\lambda$ be a limit.
Then there is
$\alpha<^\Tt\eta$ with $(\alpha,\eta]_\Tt\inter\dropset=\emptyset$.
Let $\alpha$ be  least such and
$m=\lim_{\beta<^\Tt\eta}\udeg(\beta)$. Then $m=\udeg(\eta)$ and
\[
M_\eta=\dirlim_{\beta\leq\gamma\in[\alpha,\eta)_\Tt}\left(M_\beta,M_\gamma,i_{
\beta\gamma}\right),\]
and for all $\beta\in[\alpha,\eta)_\Tt$, $i_{\beta\eta}$ is the associated
direct limit map,
and if also $\beta$ is a successor then $i^*_{\beta\eta}=i_{\beta\eta}\com
i^*_{\beta\beta}$.
\end{enumerate}

The \dfnemph{$\udash
k$-maximal iteration game}
$\Gg(M,\udash k,\theta)$,
\index{$(\udash k,\theta)$-iterability}\index{$\Gg(M,\udash k,\ldots)$}
\dfnemph{$(\udash
k,\theta)$-iteration strategy} and
\dfnemph{$(\udash k,\theta)$-iterability}
\index{iteration strategy} are defined in the obvious
manner.  Likewise for stacks,
such as the game $\Gg_{\mathrm{opt}}(M,\udash k,\lambda,\theta)^*$,
etc.

We say that $\Tt$ is a
\dfnemph{putative $\udash k$-maximal tree on
$M$}\index{putative (tree)} iff all
of the above properties hold, except that if
$\lambda=\lh(\Tt)$ is a successor, we do not require condition
\ref{item:M_beta_seg-premouse}
 to hold for $\beta=\lambda-1$, and if $\lambda-1$
 is a limit, we do not require that $[0,\lambda-1)_\Tt\inter\dropset^\Tt$
 is bounded in $\lambda-1$ (but if it is bounded,
 then we still define $M^\Tt_{\lambda-1}$ as before, etc).
\end{dfn}

It is routine to see that if $\Tt$ is a putative $\udash k$-maximal tree of
length
$\eta+1$ and  $M_\eta$ is well-defined and wellfounded, then $\Tt$ is a
$\udash k$-maximal tree.

Moreover, if $\beta\leq^\Tt\eta$ and
$(\beta,\eta]_\Tt$ does not drop in model then $i^\Tt_{\beta\eta}$ is a near
$\udash
m$-embedding,
and if also $\udeg^\Tt(\beta)=\udeg^\Tt(\eta)$ then $i^\Tt_{\beta\eta}$ is a
$\udash m$-embedding.
Likewise for $i^{*\Tt}_{\beta\eta}$ if $\beta$ is also a successor.

\begin{rem}[Closeness for $\uu$]\label{rem:r-closeness}\index{Closeness for
$\uu$}
 The Closeness Lemma \cite[6.1.5]{fsit} adapts easily to $\udash m$-maximal
trees on $\udash m$-sound MS-indexed seg-pms $M$.
 One key difference is that we replace the standard $\rSigma_1$ hierarchy with
$\uSigma_1$
 (see \ref{rem:u-Sigma_1}); of course, if $M$ is type ${\leq 2}$ then
$\rSigma_1^M=\uSigma_1^M$.
 Thus, we say that an extender $E$ is \dfnemph{$\mathrm{u}$-close}
 to a seg-pm $M$ iff $E$ is weakly amenable to $M$ and $E_a$ is $\bfuSigma_1^M$
for each $a\in[\nu(E)]^{<\om}$.
 By \ref{rem:u-Sigma_1}, if $M$ is a $\udash 1$-sound premouse,
 so $M$ is equivalent to some type 3 premouse $N$, then $\urho_1^M=\nu^N$.
 As in \cite{fsit}, one shows that if $E$ is $\mathrm{u}$-close to $M$
 and $\urho_1^M\leq\crit(E)$ and $U=\Ult_{\udash 0}(M,E)$,
 then $\urho_1^U=\urho_1^M$
 and every $\bfuSigma_1^U$ subset of $\crit(E)$ is $\bfuSigma_1^M$.
 As in \cite[6.1.5]{fsit}, one shows that if $\Tt$ is a $\udash m$-maximal tree
 on a $\udash m$-sound seg-pm $M$, then $E^\Tt_\alpha$ is $\mathrm{u}$-close to
$M^{*\Tt}_{\alpha+1}$ for every $\alpha+1<\lh(\Tt)$.

 The proof of \cite{fs_tame} adapts similarly, giving that
 the copying construction propagates near $\udash m$-embeddings.
\end{rem}

\begin{dfn}\label{dfn:unravelled}\footnote{In a draft of this article on
arxiv.org, there is a version of the material in
\ref{dfn:unravelled}--\ref{rem:partial_strategy_conversion} which is not quite
correct in its treatment of translations of stacks, in that
it does not restrict to the unravelled iteration game on the
$\udash$side (it also does not restrict to optimal stacks,
though it is straightforward to handle this). That version is also not general
enough to be  applied to partial strategies (in particular in
\S\ref{subsubsec:variants_partial}). The version here remedies these deficits.}
Let $M$ be a seg-pm with $M^\passive$ MS-indexed, and $k'\leq\om$
with $M$  $\udash k'$-sound. Suppose $M$ is type A or $k'\geq 1$.
If $M$ is type A  or $k'=\om$  let $k=k'$; otherwise let
$k=k'-1$. In this situation say $(M,k',k)$ is
\dfnemph{suitable}.\index{suitable}
Let $\Tt$ be a  $\udash k'$-maximal tree on $M$.

Given $\alpha<\lh(\Tt)$, say $\alpha$ is
\index{$\Tt$-special}\index{special}\dfnemph{$\Tt$-special}
iff $M^{\Tt}_\alpha$ is \dfnemph{$\Tt$-special} iff $M^{\Tt}_\alpha$ is type
B and $\udeg^{\Tt}(\alpha)=0$. Say $\alpha$ is
\index{$\Tt$-very special, $\Tt$-vs}\index{very special, vs}\dfnemph{$\Tt$-very
special} (or \dfnemph{$\Tt$-vs}) iff $\alpha$ is $\Tt$-special and
$E^{\Tt}_\alpha=F(M^{\Tt}_\alpha)$. Say $\alpha$ is a \dfnemph{transition
point}\index{transition point} of $\Tt$ iff $\alpha+1<\lh(\Tt)$ and $\alpha$ is
non-$\Tt$-special,
but $\OR((M^{\Tt}_\alpha)^\pm)<\In(E^{\Tt}_\alpha)$
(note that in this situation, $M^{\Tt}_\alpha$ is type B
and $\udash\deg^{\Tt}(\alpha)=1$, but $(M^{\Tt}_\alpha)^{\pm}\neq
M^{\Tt}_\alpha$, so $M^{\Tt}_\alpha$ is not a premouse).
Say  $\Tt$ is \index{unravelled}\dfnemph{unravelled} iff,
if $\Tt$ has successor length $\alpha+1$ then $\alpha$ is not $\Tt$-special.

The \index{unravelling}\dfnemph{unravelling}
\index{$\unrvl$} $\Ss=\unrvl(\Tt)$ of $\Tt$,
if it exists, is the unique unravelled \mbox{$\udash k'$-maximal} tree $\Ss$ on
$M$ such that $\Tt\ins\Ss$, if $\Tt$ has limit length then $\Ss=\Tt$,
and if $\lh(\Tt)=\alpha+1$ then $\alpha+i$ is $\Ss$-vs
for every $i$ such that $\alpha+i+1<\lh(\Ss)$. Note that existence just
requires wellfoundedness of the relevant models, and that if $\Ss$
exists, then $\lh(\Ss)<\lh(\Tt)+\om$, because
$\crit(E^{\Ss}_{\alpha+i+1})<\crit(E^{\Ss}_{\alpha+i})$.

Say  $\Tt$ is \index{unravelable}\index{everywhere
unravelable}\dfnemph{everywhere unravelable} iff $\unrvl(\Tt\rest\alpha)$
exists (with wellfounded models) for all $\alpha\leq\lh(\Tt)$, and for every
transition point $\alpha$ of $\Tt$, $\unrvl(\Tt')$ exists
where $\Tt'=(\Tt\rest(\alpha+1))\conc F(M^{\Tt}_\alpha)$.

If $\Ttvec=\left<\Tt_\alpha\right>_{\alpha<\lambda}$ is an optimal
$\udash k'$-maximal stack on $M$, say $\Ttvec$ is
\index{unravelled}\dfnemph{unravelled} iff
$\Tt_\alpha$ is unravelled for every $\alpha$, and say $\Ttvec$ is
\dfnemph{everywhere unravelable} iff $\Tt_\alpha$ is unravelled for each
$\alpha+1<\lambda$ and $\Tt_\alpha$ is everywhere unravelable for each
$\alpha<\lambda$.

The \index{unravelled}\dfnemph{unravelled optimal $\udash$iteration game}
$\Gg^{\unrvl}_{\mathrm{opt}}(M,\udash m,\alpha,\beta)^*$
\index{$\Gg^{\unrvl}_{\mathrm{opt}}(M,\udash m,\ldots)^*$}
is just like $\Gg_{\mathrm{opt}}(M,\udash m,\alpha,\beta)^*$,
except that player I may only round $\gamma$
with  $\Tt_\gamma$ unravelled.
This determines \dfnemph{unravelled-optimal-$(\udash
k,\alpha,\beta)^*$-iteration strategies} and \dfnemph{-iterability}.
For the corresponding definitions without the adjective \emph{optimal},
there can be artificial drops as usual, but player I must still end
rounds with unravelled trees.
\end{dfn}

\begin{dfn}
Let $(M,k',k)$ be suitable. Let $\Uu$ be a $k$-maximal tree on
$M^\pm$. Given $\alpha<\lh(\Uu)$, define
\index{$M^{+\Uu}_\alpha$}$M^{+\Uu}_\alpha$ as follows.
(We stop if we reach an illfounded model. The notation is literally ambiguous,
as it depends on $M$, whereas only $M^\pm$ is recorded in $\Uu$.)
Set $M^{+\Uu}_0=M$. Let $\alpha+1<\lh(\Uu)$
and $\beta=\pred^\Uu(\alpha+1)$. If $M^\Uu_{\alpha+1}$
is type $\leq 2$ or $\deg^\Uu(\alpha+1)=\om$  let $m'=\om$;
otherwise let $m'=\deg^\Uu(\alpha+1)+1$. If $\alpha+1\notin\dropset^\Uu$  let
$N^*=M^{+\Uu}_\beta$; otherwise let $N^*=M^{*\Uu}_{\alpha+1}$.
Now set $M^{+\Uu}_{\alpha+1}=\Ult_{\udash m'}(N^*,E^\Uu_\alpha)$.
Using the natural iteration maps
\index{$i^{+\Uu}_{\alpha\beta}$}
\[ i^{+\Uu}_{\alpha\beta}:M^{+\Uu}_\alpha\to
M^{+\Uu}_\beta \]
(defined when $(\alpha,\beta]_\Uu\inter\dropset^\Uu=\emptyset$),
take direct limits at limit stages. We say  $\Uu$ is
\index{$M$-$\udash$wellfounded}\index{u-wellfounded}
$M$-\dfnemph{$\udash$wellfounded} iff $M^{+\Uu}_\alpha$ is wellfounded for each
$\alpha<\lh(\Uu)$. Likewise for optimal $k$-maximal stacks
$\vec{\Uu}=\left<\Uu_\alpha\right>_{\alpha<\lambda}$,
where $M^{+\Uu_0}_0=M$ and  $M^{+\Uu_\alpha}_0=M^{+\vec{\Uu}\rest\alpha}_\infty$
 is the natural direct limit  for $\alpha>0$; the stack is
 \dfnemph{$M$-$\udash$wellfounded} iff
every $M^{+\Uu_\alpha}_\beta$ is wellfounded.
\end{dfn}

The following two lemmas are proved in \cite[\S4]{rule_conversion}:
\footnote{Along with proving Lemmas \ref{lem:tree_conversion}
and \ref{lem:rule_conversion}, \cite{rule_conversion} describes a translation
between $\lambda$-iteration rules and a natural version of MS-iteration rules
for $\lambda$-indexed mice. The methods for both are similar. They are related
to the proof of Theorem \ref{tm:lambda_gen_it}, and also to the methods of this
paper more generally. We have written $\Tt'\mapsto\Tt$ here,
instead of $\Tt\mapsto\Tt'$, to match better with the notation in
\cite{rule_conversion}.}

\begin{lem}\label{lem:tree_conversion}\index{correspondence of trees}
 Let $(M,k',k)$ be suitable. There is a class bijection
 \index{$\conv$}\index{conv}
\[ \Tt\mapsto\Uu=\mathrm{conv}(\Tt) \]
 from the  unravelled everywhere unravelable $\udash k'$-maximal  trees $\Tt$
on $M$ to the $M$-$\udash$wellfounded $k$-maximal  trees $\Uu$ on $M^\pm$,
such that:
\begin{enumerate}
 \item  If $\Ss\ins\Uu$ then  either
\begin{enumerate}[label=--]
 \item  $\Ss=\mathrm{conv}(\unrvl(\Tt\rest\alpha))$ for some $\alpha$, or
 \item $\Ss=\mathrm{conv}(\unrvl(\Tt\rest(\alpha+1)\conc F(M^\Tt_\alpha)))$
 for a transition point $\alpha$ of $\Tt$.
 \end{enumerate}
 \item  $\lh(\Tt)$ is a limit iff $\lh(\Uu)$ is a limit. When limits, these
lengths are equal.
\item Suppose $\lh(\Tt)=\alpha'+1$ and $\lh(\Uu)=\alpha+1$. Then:
\begin{enumerate}
 \item  $(M^{\Tt}_{\alpha'})^\pm=M^{\Uu}_\alpha$ and
$M^{\Tt}_{\alpha'}=M^{+\Uu}_\alpha$, so if $M^\Uu_\alpha$ is non-type 3 or
$\deg^\Uu(\alpha)>0$  then $M^{\Tt}_{\alpha'}=M^{\Uu}_{\alpha}$,
 \item either $M^{\Tt}_{\alpha'}=M^\Uu_\alpha$, or $(M^\Uu_\alpha)^\passive\pins
M^{\Tt}_{\alpha'}$ and $\OR^{M^\Uu_\alpha}$ is an $M^{\Tt}_{\alpha'}$-cardinal,
\item $[0,\alpha']_{\Tt}\inter\dropset^{\Tt}=\emptyset$
iff $[0,\alpha]_\Uu\inter\dropset^\Uu=\emptyset$; likewise for
$\dropset^{\Tt}_{\deg}$ and $\dropset^\Uu_{\deg}$,
 \item\label{item:ev_it_map} letting $\beta'+1\leq^{\Tt}\alpha'$ and
$\beta+1\leq^\Uu\alpha$  be least such that
 $(\beta'+1,\alpha']_\Uu$ and $(\beta+1,\alpha]_{\Tt}$ do not drop in model or
degree,  then:
 \[ (M^{*\Tt}_{\beta'+1})^\pm=M^{*\Uu}_{\beta+1}\text{ and
}
i^{*\Tt}_{\beta'+1,\alpha'}\rest((M^{*\Tt}_{\beta'+1})^\pm)^\sq=i^{*\Uu}_{
\beta+1,\alpha}, \]
and in fact if $[0,\alpha]_{\Uu}\inter\dropset_{\deg}^{\Uu}\neq\emptyset$
then $M^{*\Tt}_{\beta'+1}=M^{*\Uu}_{\beta+1}$.
\end{enumerate}
\end{enumerate}
Further, there is an analogous bijection between unravelled everywhere
unravelable optimal $k'$-maximal
stacks $\left<\Tt_\alpha\right>_{\alpha<\lambda}$ on $M$ and
$M$-$\udash$wellfounded optimal $k$-maximal
stacks $\left<\Uu_\alpha\right>_{\alpha<\lambda}$ on $M^\pm$.
Moreover, $\Uu_\alpha=\mathrm{conv}(\Tt_\alpha)$ for each $\alpha$.

The bijections are moreover uniformly definable from the parameter $M$.
If $M$ is countable then the conversion between such $\Uu,\Tt\in\HC$
is $\Delta^1_1(\{M\})$-definable in the codes, and likewise for optimal stacks.
\end{lem}

If we at times talk about the conversion of an everywhere unravelable $\udash
k'$-maximal tree $\Tt$ to a $k$-maximal tree, without assuming that $\Tt$ is
unravelled, then  one should first replace $\Tt$ with $\unrvl(\Tt)$.

\begin{lem}\label{lem:rule_conversion}\index{correspondence of strategies}
Let $\Omega>\om$ be regular  and $\Omega\leq\Xi\leq\Omega+1$.
Let $(M,k',k)$ be suitable, with $M$ a premouse. Then
 \begin{enumerate}[label=\arabic*.,ref=\arabic*]
  \item  $M$ is $(\udash k',\Xi)$-iterable iff $M$ is $(k,\Xi)$-iterable,
  \item   $M$ is unravelled-opt-$(\udash
k',\Omega,\Xi)^*$-iterable iff $M$ is opt-$(k,\Omega,\Xi)^*$-iterable.
\end{enumerate}
Moreover, there are bijections
\index{$\conv$}\index{conv}$\Sigma\mapsto\mathrm{conv}(\Sigma)$
between the sets of
\begin{enumerate}[resume*]
 \item   $(\udash k',\Xi)$-strategies and $(k,\Xi)$-strategies,
 \item unravelled-opt-$(\udash k',\Omega,\Xi)^*$-strategies and
opt-$(k,\Omega,\Xi)^*$-strategies,
 \item unravelled-opt-$(\udash k',{<\om},\Omega)^*$-strategies
 and opt-$(k,{<\om},\Omega)^*$-strategies
 \end{enumerate}
for $M$. In particular, there is a unique $(\udash k',\Xi)$-strategy for $M$
 iff there is a unique  $(k,\Xi)$-strategy for $M$.

These bijections are induced  tree-by-tree, for unravelled trees via $\Sigma$
and trees via $\Gamma=\mathrm{conv}(\Sigma)$, via the correspondence of Lemma
\ref{lem:rule_conversion}, and therefore if $\Omega=\aleph_1$ and
$\widetilde{\Sigma}$ is the natural coding of $\Sigma\rest\HC$ over $\RR$,
 and $\widetilde{\Gamma}$  likewise, then $\widetilde{\Gamma}$
is $\Delta^1_1(\widetilde{\Sigma})$ and vice versa.
\end{lem}

\begin{rem}
Note here that if $\Sigma$ is a $(\udash k,\Xi)$-strategy,
then all trees via $\Sigma$ are everywhere unravelable.
Similarly, if $\Gamma$ is a $(k,\Xi)$-strategy
for $M$, then  all trees via  $\Gamma$ are in fact $M$-$\udash$wellfounded.
(If $\Xi=\Omega+1$, then as $\Omega$ is regular, $\Sigma$
in fact extends to a $(\udash k,\Omega+\om)$-strategy.
So unravellings of trees via $\Sigma$ always exist. Similarly for $\Gamma$.)
So Lemma \ref{lem:tree_conversion} (and its proof)
is relevant to the proof of Lemma \ref{lem:rule_conversion}.
\end{rem}

At times we will also deal with partial strategies (where the trees in the
domain of the strategy  have some restricted form).

 \begin{dfn}
 A partial strategy $\Sigma$  for $\udash k'$-maximal trees/stacks
is \index{unravelable}\dfnemph{everywhere unravelable} if all trees
via $\Sigma$ are everywhere
unravelable. A partial strategy $\Gamma$ for an MS-indexed premouse  $M$  for
$k$-maximal trees/stacks is
\index{u-wellfounded}\index{$M$-$\udash$wellfounded}
\dfnemph{$M$-$\udash$wellfounded}
 if all trees via $\Gamma$ are $M$-$\udash$wellfounded.
\end{dfn}
 \begin{rem}\label{rem:partial_strategy_conversion}
Note that if $\Gamma$ (as above) is $M$-$\udash$wellfounded,
then we can define via Lemma \ref{lem:tree_conversion} a partial
$\udash$strategy $\Sigma$, where the trees via $\Sigma$ are just those which
are initial segments of trees $\Tt=\mathrm{conv}^{-1}(\Uu)$ for some $\Uu$
 via $\Gamma$ (and if a strategy for stacks, then we admit only stacks according
to the unravelled game); all putative trees via $\Sigma$ are
then true trees (and are everywhere unravelable). Likewise
conversely, if a given $\Sigma$ (as above) is everywhere unravelable,
then we can define the corresponding partial strategy $\Gamma$,
and all trees via $\Gamma$ are $M$-$\udash$wellfounded.
\end{rem}

\section{Coarse mice}\label{sec:wcpm}

The main results and methods in the paper also apply to iteration strategies for
a natural class of coarse structures.
Steel suggested to the author that the methods should go through in such a
context,
and it was indeed straightforward to verify that things go through with the same
basic ideas,
and with some simplifications. The only slight subtlety is that we seem to need
a weak form
of a coherent sequence of extenders for some of the arguments (such a notion was
already employed by Steel
in his work). The coarse case will be used by Steel and the author in the
forthcoming paper
\cite{scale_con}.

\begin{dfn}\index{weak coarse premouse}\index{wcpm}
 A \dfnemph{weak coarse premouse \tu{(}wcpm\tu{)}}
 is a transitive structure
 \index{$\es^M,\es_+^M$}\index{${<_e}$}$M=(N,\delta,\es,<_e)$ such that:
 \begin{enumerate}[label=--]
  \item $\delta\leq\OR^N=\rank(N)$, $\delta$ and $\OR^N$ are limit ordinals,
  $\card^N(V_\eta^N)<\delta$ for every $\eta<\delta$,
  $\cof^N(\delta)$ is not measurable in $N$, $N$ satisfies
$\Sigma_0$-comprehension and is rudimentarily closed,
and $N$ satisfies $\lambda$-choice for all $\lambda<\delta$.
 \item $\es,{<_e}\sub V_\delta^N$ and both are amenable to $V_\delta^N$.
 \item $\es$ is a class of $E$ such that $N\sats$``$E$ is a suitable extender''.
 \item $<_e$ is a wellorder of $\es$.
 \item if $E,F\in\es$ and $\strength^N(E)<\strength^N(F)$ then $E<_eF$.
 \end{enumerate}

 Given a wcpm $M=(N,\delta,\es,<_e)$ and $E\in\es$,
 then
 \index{$\Ult$}
 $\Ult(M,E)$ denotes $(\Ult(N,E),\delta',\es',<_e')$ where
 $\delta'=i^N_E(\delta)$,
 \[ \es'=\bigcup_{\alpha<\delta}i^N_E(\es\inter V_\alpha^N), \]
 \[ <_e'=\bigcup_{\alpha<\delta}i^N_E({<_e}\inter V_\alpha^N). \]

Given a wcpm $M$ and $E\in\es^M$, we write $\In(E)$\index{$\In(E)$} (or
$\In^M(E)$) for the ordinal rank of $E$ in $<_e^M$.

Given a wcpm $M$, we say that $M$ is \dfnemph{slightly coherent}\index{slight
coherence} iff for every $E\in\es$, letting $\varrho=\strength^M(E)$
and $U=\Ult(M,E)$, we have:
\begin{enumerate}[label=\arabic*.,ref=\arabic*]
 \item $X\eqdef\{F\in\es^M\bigm| \strength^M(F)<\varrho\}=
 \{F\in\es^U\bigm| \strength^U(F)<\varrho\}$,
 \item ${<_e^M}\rest X={<_e^U}\rest X$,
 \item for each $F\in\es^U$, if $\strength^U(F)=\varrho$ then
 $F\in\es^M$ and $F<_e^ME$.\qedhere
\end{enumerate}
\end{dfn}

\begin{rem}
 We need slight coherence for the normal realization results in
\S\ref{sec:normal_realization}
 and genericity inflation in \S\ref{sec:min_inf}. For the other results, slight
coherence is not relevant.
\end{rem}

\begin{dfn}\index{$\LST,\LST^+$}
We write $\LST$ for the language of set theory,
and $\LST^+$ for $\LST$ augmented with 1-place predicates $\es$ and $<_e$.

 Let $M,N$ be wcpms and $\pi:M\to N$.
 We say $\pi$ is a \dfnemph{coarse $0$-embedding}\index{coarse $0$-embedding}
 iff:
 \begin{enumerate}[label=--]
  \item $\pi$ is $\in$-cofinal in $N$,
 \item $\pi\rest V_{\delta^M}^M$ is $\in$-cofinal in $V_{\delta^N}^N$,
\item $\pi$ is $\Sigma_1$-elementary in $\LST$, and
\item $\pi\rest
V_{\delta^M}^M:(V_{\delta^M}^M,\es^M,<_e^M)\to(V_{\delta^N}^N,\es^N,<_e^N)$
 is $\Sigma_1$ elementary in $\LST^+$.\qedhere
 \end{enumerate}
\end{dfn}

\begin{lem}\label{new:lem_wcpm}
 Let $M$ be a wcpm.
 Then each proper segment of $<_e^M$ is in $V_{\delta^M}^M$,
 and $<_e^M$ has ordertype $\leq\delta^M$.
\end{lem}
\begin{proof}
For each $\varrho<\delta^M$,
the set $\{E\in\es^M\bigm|\varrho^M(E)\leq\varrho\}\in V_{\delta^M}^M$,
because $\card^M(V_\eta^M)<\delta^M$ for every $\eta<\delta^M$
and every $E\in\es^M$ is suitable.
Since
$<_e^M$ refines strength and by the amenability
of it and $\es^M$, this gives the lemma.
\end{proof}

\begin{lem}\label{lem:Ult_of_wcpm}
 Let $M$ be a wcpm. Let $E$ be a short $M$-extender
with $\crit(E)$ measurable in $M$ and $U=\Ult(M,E)$  wellfounded. Then:
\tu{(}i\tu{)} $\Sigma_0$-\L o\'s' Theorem holds for $\LST$,
\tu{(}ii\tu{)} $\Sigma_0$-\L o\'s' Theorem holds for $\LST^+$ with respect to
parameters in
$V_{\delta^M}^M$,
\tu{(}iii\tu{)} $U$ is a wcpm,
\tu{(}iv\tu{)} $i^M_E:M\to U$ is a coarse $0$-embedding,
\tu{(}v\tu{)} If $M$ is slightly coherent then so is $U$.
\end{lem}

\begin{dfn}\label{dfn:normal_tree_wcpm}\index{normal tree (on wcpm)}
 Let $M$ be a wcpm.  A \dfnemph{normal  iteration
tree}\index{normal tree (wcpm)} $\Tt$ on $M$ is defined in a typical manner,
 with the specific requirements that for all $\alpha+1<\lh(\Tt)$, we have:
 \begin{enumerate}[label=--]
  \item $M^\Tt_\alpha$ is a wcpm and $E^\Tt_\alpha\in\es^{M^\Tt_\alpha}$; we
write $\varrho^\Tt_\alpha=\varrho^{M^\Tt_\alpha}(E^\Tt_\alpha)$,
  \item If $\beta+1<\alpha+1$ then $\varrho^\Tt_\beta<\varrho^\Tt_\alpha$.
  \item $\pred^\Tt(\alpha+1)$ is the  least $\beta$ such that
$\crit(E^\Tt_\alpha)<\varrho^\Tt_\beta$.
 \end{enumerate}

 A \dfnemph{putative normal iteration tree}\index{putative (tree)} on $M$ is
 just like a normal tree on $M$, except that
 if $\Tt$ has successor length $\alpha+1>1$
 then we do not demand that $M^\Tt_\alpha$
 be a wcpm (nor wellfounded).
\end{dfn}

The following lemma is verified by a routine induction:
\begin{lem}\label{lem:slight_coherence_pres}
 Let $\Tt$ be a putative normal iteration tree on the wcpm $M$. If $\Tt$ has
wellfounded models, then its models are wcpms, so
  $\Tt$ is a normal  tree.

 Now suppose that $\Tt$ is a normal iteration tree.
 Then for every $\alpha<\lh(\Tt)$, writing $M_\alpha=M^\Tt_\alpha$ etc,
 \begin{enumerate}[label=\arabic*.,ref=\arabic*]
 \item If $\beta<^\Tt\alpha$ then $i^\Tt_{\beta\alpha}:M_\beta\to M_\alpha$ is
cofinal and $\Sigma_1$-elementary in $\LST$.
 \item If $\beta<^\Tt\alpha$ then
$i^\Tt_{\beta\alpha}:(V^{M_\beta}_{\delta^{M^\beta}},\es^{M_\beta},<_e^{M_\beta}
)\to(V^{M_\alpha}_{\delta^{M_\alpha}},\es^{M_\alpha},<_e^{M_\alpha})$
 is cofinal and $\Sigma_1$-elementary in $\LST^+$.
  \item Suppose $M$ is slightly coherent. Then so is $M_\alpha$,
and for $\beta<\alpha$, letting $\varrho=\varrho^\Tt_\beta$,   we have:
\begin{enumerate}[label=--]
   \item $X\eqdef\{F\in\es^{M_\beta}\bigm| \strength^{M_\beta}(F)<\varrho\}=
    \{F\in\es^{M_\alpha}\bigm| \strength^{M_\alpha}(F)<\varrho\}$,
   \item ${<_e}^{M_\beta}\rest X={<_e}^{M_\alpha}\rest X$,
 \item for each $F\in\es^{M_\alpha}$, if $\strength^{M_{\alpha}}(F)=\varrho$
then
 $F\in\es^{M_\beta}$ and $F<_e^{M_\beta}E^\Tt_\beta$.\qedhere
 \end{enumerate}
 \end{enumerate}
\end{lem}

\begin{dfn}\index{iteration strategy}\index{stack}
We define (normal) \dfnemph{$\alpha$-iteration strategies} and
\dfnemph{$\alpha$-iterability} (where $\alpha\in\OR$)
for a wcpm $M$ in the obvious manner.
Likewise \dfnemph{stacks} of normal trees,
\dfnemph{$(\lambda,\alpha)^*$-iteration strategies} and
\dfnemph{$(\lambda,\alpha)^*$-iterability}
(in which $\lambda$ is the length of the stack, and $\alpha$ the bound on the
length
of the individual normal trees;
player I may stop round before reaching a normal tree of length $\alpha$,
and otherwise the game terminates; if $\lambda$ is a limit then player II must
also ensure
that the direct limit $M^{\Ttvec}_\infty$ of the entire stack $\Ttvec$ is
wellfounded).
\end{dfn}

\begin{dfn}\index{$\es^M,\es_+^M$}
 Given a wcpm $M$, we write $\es_+(M)=\es_+^M=\es(M)=\es^M$
 (cf.~the use of $\es,\es_+$ in connection with seg-pms).
\end{dfn}

\section{Tree embeddings and inflation}\label{sec:inflation}

In this section we introduce the key concepts of the paper: \emph{tree
embeddings}, \emph{inflation},
and various kinds of condensation for iteration strategies to which these
notions lead.
These notions were introduced somewhat in \S\ref{sec:intro}.
But first we lay down some iteration tree terminology;
see \S\ref{subsec:terminology} for more.

\subsection{Iteration tree terminology}\label{subsec:it_tree_terms}

\begin{dfn}
\index{$\iota^M$, $\iota^\Tt_\alpha$}Let $M$ be an active seg-pm and
$\delta=\lgcd(M)$. We define
$\iota^M=\iota(M)$.
If $\nu(F^M)\leq\delta$ and $\delta$ is a limit cardinal of $M$ then
$\iota^M=\delta$;
otherwise $\iota^M=\OR^M$. For an iteration tree $\Tt$ and $\alpha+1<\lh(\Tt)$,
$\iota^\Tt_\alpha$ denotes $\iota(\exit^\Tt_\alpha)$.\footnote{Recall
that $\exit^\Tt_\alpha=M^\Tt_\alpha|\In(E^\Tt_\alpha)$.}
\end{dfn}

\begin{rem}
Let $\Tt$ be an $m$-maximal or $\udash m$-maximal tree (on a seg-pm with either
indexing).
Recall that $\nutilde^\Tt_\alpha$ is the exchange ordinal associated to
$E^\Tt_\alpha$.
However, note that we could have used $\iota^\Tt_\alpha$ instead, without
changing the tree order.
Moreover, in the tree copying we will do, if $\sigma:M^\Tt_\alpha\to
M^{\Tt'}_{\alpha'}$
is a copy map and $E^{\Tt'}_{\alpha'}$ is the lift of $E^\Tt_\alpha$ (under
$\sigma$)
then $\sigma\rest\iota^\Tt_\alpha$ will agree with later copy maps.
(But there will be instances where $\iota^\Tt_\alpha<\OR(\exit^\Tt_\alpha)$ but
$\sigma\rest\OR(\exit^\Tt_\alpha)$ does not agree
with later copy maps.)
\end{rem}

\begin{dfn}\label{dfn:further_tree_terminology}
 Let  $\Tt$ be an iteration tree,
 $\eta=\lh(\Tt)$ and suppose $\Tt$ is either:
\begin{enumerate}[label=\tu{(}\roman*\tu{)}]
 \item\label{item:normal} a normal tree on the wcpm $M$, or
 \item\label{item:r-m-max} a $\udash m$-maximal tree on the $\udash m$-sound
seg-pm $M$,  or
 \item\label{item:m-max} an $m$-maximal tree on the $m$-sound pm $M$.
\end{enumerate}

 \index{$(\Tt,b)$, $\Tt\conc b$}If $\eta$ is a limit and $b$ is a $\Tt$-cofinal
branch, we write $(\Tt,b)$ or
$\Tt\conc b$
 for the putative tree $\Tt'$ extending $\Tt$, of length $\eta+1$, with
$[0,\eta)_{\Tt'}=b$.

 Suppose $\eta=\beta+1$. For $E\in\es_+(M^\Tt_\beta)$,
 we say that $E$ is \dfnemph{$\Tt$-normal}\index{$\Tt$-normal}\index{normal}
iff either
 \begin{enumerate}[label=--]
  \item  \ref{item:r-m-max} or \ref{item:m-max} above holds and
$\In(E^\Tt_\alpha)\leq\In(E)$ for all $\alpha<\beta$, or
  \item \ref{item:normal} above holds and
$\varrho^\Tt_\alpha<\varrho^{M^\Tt_\beta}(E)$ for all $\alpha<\beta$.
 \end{enumerate}
 If $E$ is $\Tt$-normal, then $\Tt\conc\left<E\right>$
\index{$\Tt\conc\left<E\right>$} denotes the putative tree $\Tt'$ extending
$\Tt$,
 of length $\eta+1$, such that either (i)
 $\Tt'$ is $\udash m$-maximal, or (ii) $\Tt'$ is $m$-maximal, or (iii) $\Tt'$
is
normal,
 respectively according to the case for $\Tt$ above.\footnote{We take it that
the basic fine structural information
 regarding an iteration tree $\Uu$ is explicitly given with $\Uu$, so there can
be no ambiguity here.}
\end{dfn}

\begin{dfn}[Model dropdown]\label{dfn:dropdown}\index{dropdown}
 Let $M$ be a putative\footnote{\emph{Putative} means that $M$ satisfies the
first-order requirements of premousehood,
 but may be illfounded.} $\udash k$-sound seg-pm and $\lambda\leq\OR^M$,
 where if $M$ is illfounded then $\lambda=\OR^M$.
 The \dfnemph{extended model dropdown sequence} of $(M,\lambda)$ is the
sequence
$\left<M_i\right>_{i\leq n}$ of maximal length such that $M_0=M|\lambda$, and
given $M_i\pins M$,
$M_{i+1}$ is the least $N\ins M$ such that either (i) $N=M$ or (ii) $M_i\pins
N$
and
$\rho_\om^N<\rho_\om^{M_i}$. The \index{reverse}\dfnemph{reverse} of a sequence
$\left<N_i\right>_{i\leq n}$ (where
$n<\om$) is $\left<N_{n-i}\right>_{i\leq n}$.
\end{dfn}

\begin{dfn}[Tree
dropdown]\label{dfn:tree_dropdown}\index{dropdown}\index{$\dds$, $\ddd$}
Let $M$ be a $\udash k$-sound segmented-premouse and let $\Tt$ be a putative
$\udash k$-maximal
tree on
$M$.

For $\beta+1<\lh(\Tt)$ let
$\lambda_\beta=\In(E^\Tt_\beta)$. For $\beta+1=\lh(\Tt)$ (if $\lh(\Tt)$ is a
successor) let
$\lambda_\beta=\OR(M^\Tt_\beta)$. Let $\beta<\lh(\Tt)$. Let
$\left<M_{\beta i}\right>_{i\leq m_\beta}$ be the
\emph{reversed} extended model dropdown sequence of
$(M^\Tt_\beta,\lambda_\beta)$
(this defines $m_\beta$). Here if $M^\Tt_\beta$ is ill-defined,
set instead $\lambda_\beta=m_\beta=M^\Tt_{\beta 0}=0$.
Then
$m_\beta^{\Tt}\eqdef m_\beta$ and $M_{\beta i}^{\Tt}\eqdef M_{\beta i}$.
 Let
$\theta\leq\lh(\Tt)$.
We define the \dfnemph{dropdown domain} $\ddd^{(\Tt,\theta)}$ of $(\Tt,\theta)$
by
\[ \Delta=\ddd^{(\Tt,\theta)}\eqdef\{(\beta,i)\bigm| \beta<\theta\ \&\ i\leq
m_\beta\}, \]
and the \dfnemph{dropdown
sequence} $\dds^{(\Tt,\theta)}$ of $(\Tt,\theta)$ by
$\dds^{(\Tt,\theta)}\eqdef\left<M_{\beta
i}\right>_{(\beta,i)\in\Delta}$.

The \dfnemph{dropdown sequence} $\dds^\Tt$ of $\Tt$ is
$\dds^{(\Tt,\lh(\Tt))}$, and
the \dfnemph{dropdown domain} $\ddd^\Tt$ of $\Tt$ is $\ddd^{(\Tt,\lh(\Tt))}$.
\end{dfn}

\begin{dfn}\index{$\clint$}\index{clint}
 Let $\Xx$ be an iteration tree. Then $\clint^\Xx$ denotes the set of closed
$<^\Xx$-intervals.
\end{dfn}

\subsection{Tree embeddings}

We now define the notion of a \emph{tree embedding} $\Pi:\Tt\tembto\Xx$
between normal trees
$\Tt,\Xx$ (actually we allow $\Tt$ to be a putative tree). This is fairly
straightforward, but
there are a lot of details to keep
track of, reminiscent of iterability proofs with resurrection. We first roughly
describe the objects involved, to give an idea of what to expect. The
primary data determining the tree embedding is an embedding of the tree
structure of $\Tt$ into
that of $\Xx$.
This embedding will determine canonical copy embeddings from models in the
dropdown sequence of
$\Tt$ to initial segments of
models of $\Xx$. A natural degree of commutativity between the copy embeddings
and
iteration embeddings will be required. For each extender used in $\Tt$ there
will be a
corresponding copy of this extender used in $\Xx$. A key point is that,
corresponding to each
$\beta<\lh(\Tt)$, we will typically have not just a single corresponding node
in
$\Xx$,
but a corresponding $\Xx$-\emph{interval}
$I_\beta=[\gamma_\beta,\delta_\beta]_\Xx$.
We will have a copy embedding
\[ \pi_{\beta 0}:M^\Tt_\beta\to M^\Xx_{\gamma_\beta} \]
(with codomain $M^\Xx_{\gamma_\beta}$ sitting at the \emph{start} of $I_\beta$).
But, if $\beta+1<\lh(\Tt)$, the copy of $E^\Tt_\beta$ (in $\Xx$) will be
$E^\Xx_{\delta_\beta}$,
not $E^\Xx_{\gamma_\beta}$ (unless $\delta_\beta=\gamma_\beta$).
Here $(\gamma_\beta,\delta_\beta]_\Xx$ might actually drop in model, but it
will
not drop below the
image of $E^\Tt_\beta$.
However, if $\lh(\Tt)=\beta+1$ then
$(\gamma_\beta,\delta_\beta]_\Xx$ will not drop in model.

We will actually define a slightly more
general notion: that of a tree embedding
$(\Tt,\theta)\tembto\Xx$, where $\theta\leq\lh(\Tt)$. If $\theta=\lh(\Tt)$ or
$\theta$ is a limit,
this will be the same as a tree embedding $\Tt\rest\theta\tembto\Xx$. But if
$\theta=\beta+1<\lh(\Tt)$,
then we allow $(\gamma_\beta,\delta_\beta]_\Xx$ to drop in model, as long as it
does not
drop below the image of $E^\Tt_\beta$.

\begin{dfn}[Tree embedding]\label{dfn:tree_embedding}\index{tree embedding}
\index{$\Pi:(\Tt,\theta)\tembto\Xx$}
 Let $M$ be a $\udash k$-sound seg-pm,
 let $\Tt$ be a putative $\udash k$-maximal tree on $M$,
 let $\Xx$ be a $\udash k$-maximal tree on $M$,
 let $1\leq\theta\leq\lh(\Tt)$, and let
$\Delta=\ddd^{(\Tt,\theta)}$.

A \dfnemph{tree embedding} $\temb:(\Tt,\theta)\tembto\Xx$ from $(\Tt,\theta)$
to $\Xx$
is a system\index{$I_\beta,I_{\beta i}$}
\index{$P_\beta,P_{\beta i}$}
\index{$\pi_\beta,\pi_{\beta i}$}
\begin{equation}\label{eqn:tree_embedding_system}
\Pi=\left(\Tt,\left<I_\beta\right>_{\beta<\theta}; \left<I_{\beta
i},P_{\beta i},\pi_{\beta i}\right>_{(\beta,i)\in\Delta}\right) \end{equation}
with properties \ref{item:tree_pres}--\ref{item:embedding_agreement} below.
We will see later that  $\Pi$ is
determined by $(\Tt,\Xx,\left<I_\beta\right>_{\beta<\theta})$. While stating
\ref{item:tree_pres}--\ref{item:embedding_agreement}, we
also define various other uniquely determined objects.
We sometimes denote $(\beta,i)$ with a
single variable $x$.
For $x=(\beta,i)\in \Delta$ let $m_\beta=m_\beta^\Tt$ and $M_{\beta
i}=M_x=M_x^\Tt$.

\begin{enumerate}[label=T\arabic*.,ref=T\arabic*]

\item\label{item:tree_pres} (Preservation of tree structure) See figure
\ref{fgr:pres_tree_struc}.
\begin{figure}
\centering
\begin{tikzpicture}
 [mymatrix/.style={
    matrix of math nodes,
    row sep=0.05cm,
    column sep=0.4cm}
  ]
   \matrix(m)[mymatrix]{
{}&{}&{}&{}&{}&{}&{\bullet}&{}&{}&{}\\
{\bullet}&{}&{}&{}&{}&{}&{}&{}&{}&{}\\
{}&{}&{}&{}&{}&{}&{\bullet}&{}&{}&{}\\
{}&{}&{}&{}&{}&{}&{}&{}&{}&{\bullet}\\
{}&{}&{}&{\bullet}&{}&{}&{}&{}&{}&{\bullet}\\
{}&{}&{}&{}&{}&{}&{}&{}&{\bullet}&{}\\
{}&{}&{}&{}&{}&{}&{}&{}&{\bullet}&{}\\
{}&{}&{\bullet}&{}&{}&{}&{}&{}&{}&{}\\
{}&{}&{}&{}&{}&{}&{}&{}&{\bullet}&{}\\
{}&{}&{}&{}&{}&{}&{}&{}&{\bullet}&{}\\
{}&{}&{}&{}&{}&{}&{}&{\bullet}&{}&{}\\
{}&{\bullet}&{}&{}&{}&{}&{}&{}&{}&{}\\
{}&{}&{}&{}&{}&{}&{}&{\bullet}&{}&{}\\
{}&{\Tt}&{}&{}&{}&{}&{}&{\widetilde{\Xx}}&{}&{}\\
};
\path[-,font=\scriptsize,shorten >= -5pt, shorten <= -5pt]
(m-12-2) edge[dotted] node[pos=-0.2,right] {$\ 0$} (m-8-3)
(m-12-2) edge[dotted] node[pos=1.1,right] {$\ 1$} (m-8-3)
(m-8-3) edge[dotted] node[pos=1.1,right] {$\ 2$} (m-5-4)
(m-8-3) edge[dotted] node[pos=1.1,right] {$\ 3$} (m-2-1)
(m-13-8) edge node[pos=-0.6,right] {$\gamma_0$} (m-11-8)
(m-13-8) edge node[pos=1.5,left] {$\delta_0$} (m-11-8)
(m-12-8) edge[dotted] node[pos=1.1,right] {$\ \gamma_1$} (m-10-9)
(m-10-9) edge node[pos=1.2,left] {$\delta_1$} (m-6-9)
(m-5-10) edge node[pos=-8,right] {$\gamma_2$} (m-4-10)
(m-5-10) edge node[pos=9,right] {$\delta_2$} (m-4-10)
(m-3-7) edge node[pos=-0.5,left] {$\gamma_3$} (m-1-7)
(m-3-7) edge node[pos=1.5,left] {$\delta_3$} (m-1-7)
(m-9-9) edge[dotted] node[pos=-0.1,right] {$\eta_3$} (m-3-7)
(m-7-9) edge[dotted] node[pos=-0.4,left] {$\eta_2$} (m-5-10);
\end{tikzpicture}
\caption{Preservation of tree structure, with $\lh(\Tt)=\theta=4$.
Bullets
represent tree nodes.
Dotted lines connect nodes with their predecessors;
in particular, $\eta_i=\pred^\Xx(\gamma_i)$ for $i=2,3$. Solid lines
represent $<^\Xx$-intervals.
And $\widetilde{\Xx}$ is the restriction of $\Xx$ to
$\bigcup_{i<4}I_i$.\label{fgr:pres_tree_struc}}
\end{figure}

We have $I_\beta\in\clint^\Xx$ for each $\beta<\theta$.
Let\index{$\gamma_\beta,\gamma_{\beta i}$}
\index{$\delta_\beta,\delta_{\beta i}$}
\[ [\gamma_\beta,\delta_\beta]_\Xx\eqdef I_\beta. \]
\index{$\Gamma$}Let $\Gamma:\theta\to\lh(\Xx)$ be $\Gamma(\beta)=\gamma_\beta$.
Then:
\begin{enumerate}[label=\tu{(}\alph*\tu{)}]
\item $\gamma_0=0$,
\item $\Gamma$ preserves $<$, is continuous, sends successors (i.e. successor
ordinals)
to successors,
\item $\beta_0<^\Tt\beta_1\iff\gamma_{\beta_0}<^\Xx\gamma_{\beta_1}$.
\item\label{item:degree_match} $\udeg^\Xx(\gamma_\beta)=\udeg^\Tt(\beta)$.
\item For $\beta+1<\theta$, we have $\gamma_{\beta+1}=\delta_\beta+1$.
\item\label{item:pred_pres} For $\beta+1<\theta$, letting
$\xi=\pred^\Tt(\beta+1)$, we have
\[ \pred^\Xx(\gamma_{\beta+1})\in I_\xi\]
and
\[ \dr^\Xx\inter(\gamma_\xi,\gamma_{\beta+1}]_\Xx=\emptyset\iff \beta+1\notin
\dr^\Tt.\]
\end{enumerate}
(So (i) the $<$-intervals\footnote{Note this is ${<}$,
not ${<^\Tt}$.} $[\gamma_\beta,\delta_\beta]$ partition
$\sup_{\beta<\theta}\delta_\beta$,  (ii) for
$\xi,\zeta<\theta$,
\[ (\gamma_\xi,\gamma_\zeta]_\Xx\inter\dropset^\Xx=\emptyset\
\iff\ (\xi,\zeta]_\Tt\inter\dropset^\Tt=\emptyset,\]
and (iii) for each limit
$\beta<\theta$, we
have $\Gamma``[0,\beta)_\Tt\sub_{\text{cof}}[0,\gamma_\beta)_\Xx$.)
\item\label{item:structure_I_beta} (Structure of $I_\beta$) Let $(\beta,i)\in
\Delta$. Then:\index{$\gamma_\beta,\gamma_{\beta i}$}
\index{$\delta_\beta,\delta_{\beta i}$}
\begin{enumerate}
\item $I_{\beta i}\in\clint^\Xx$ and $I_{\beta i}\sub I_\beta$. Let
$[\gamma_{\beta
i},\delta_{\beta i}]_\Xx\eqdef I_{\beta i}$.
\item $\gamma_{\beta 0}=\gamma_\beta$ and $\delta_{\beta m_\beta}=\delta_\beta$.
\item If $(\beta,i+1)\in \Delta$ then $\gamma_{\beta,i+1}=\delta_{\beta i}$.
\end{enumerate}
(Therefore, $I_{\beta 0},\ldots,I_{\beta m_\beta}$ essentially partition
$I_\beta$ into an increasing sequence of closed
$<^\Xx$-intervals; they just overlap at their endpoints.)
\begin{enumerate}[resume*]
\item If $\gamma_{\beta i}<\delta_{\beta i}$ then let $\eps_{\beta
i}=\min(I_{\beta
i}\cut\{\gamma_{\beta i}\})$.\index{$\eps_{\beta i}$}
\item If $\gamma_{\beta 0}<\delta_{\beta 0}$ then $(\gamma_{\beta
0},\delta_{\beta 0}]_\Xx$ does
not drop in model (but may drop in degree).
\item If $i>0$ and $\gamma_{\beta i}<\delta_{\beta i}$ then
$\dr^\Xx\inter(\gamma_{\beta
i},\delta_{\beta i}]_\Xx=\{\eps_{\beta i}\}$.
\end{enumerate}
\item\label{item:model_embeddings} (Model embeddings) See figure
\ref{fgr:model_embeddings}. Let
$x=(\beta,i)\in \Delta$. Then:
\begin{enumerate}[label=\tu{(}\alph*\tu{)}]
\item $P_{\beta i}$ is a segmented-premouse and $\pi_{\beta i}:M_{\beta i}\to
P_{\beta i}$ is an
embedding. \index{$P_\beta,P_{\beta i}$}
\index{$\pi_\beta,\pi_{\beta i}$}Let $P_\beta=P_{\beta 0}$ and
$\pi_\beta=\pi_{\beta 0}$ (but maybe
$I_\beta\neq I_{\beta 0}$).
\item $P_{0}=M$ and $\pi_0=\id:M\to M$.
\item\label{item:P_beta,0_def}$P_{\beta}=M^\Xx_{\gamma_{\beta}}$ (recall
$\gamma_\beta=\gamma_{\beta 0}$).
\item\label{item:pi_beta,0_near_embedding_only} $\pi_{\beta}$ is a near
$\udash\deg^\Tt(\beta)$-embedding.
\item Suppose $i>0$. Then $P_x\pins M^\Xx_{\gamma_x}$
and
$\pi_x$ is fully
elementary. If $\gamma_x<\delta_x$ then $P_x=M^{*\Xx}_{\eps_x}$.

\begin{figure}
\centering
\begin{tikzpicture}
 [mymatrix/.style={
    matrix of  nodes,
    row sep=0.2cm,
    column sep=0.8cm}]
   \matrix(m)[mymatrix]{
{}&{}&{}&{}&{}&{}&{}&{}&{}&{}&{}\\
{}&{}&{}&{}&{}&{}&{}&{}&{}&{}&{}\\
{}&{}&{}&{}&{}&{}&{}&{}&{}&{}&{}\\
{}&{}&{}&{}&{}&{}&{}&{}&{}&{}&{}\\
{}&{}&{}&{}&{}&{}&{}&{}&{}&{}&{}\\
{}&{}&{}&{}&{}&{}&{}&{}&{}&{}&{}\\
{}&{}&{}&{}&{}&{}&{}&{}&{}&{}&{}\\
{}&{}&{}&{}&{}&{}&{}&{}&{}&{}&{}\\
{}&{}&{}&{}&{}&{}&{}&{}&{}&{}&{}\\
{}&{}&{}&{}&{}&{}&{}&{}&{}&{}&{}\\
{}&{}&{}&{}&{}&{}&{}&{}&{}&{}&{}\\
{}&{}&{}&{}&{}&{}&{}&{}&{}&{}&{}\\
{}&{}&{}&{}&{}&{}&{}&{}&{}&{}&{}\\
{}&{}&{}&{}&{}&{}&{}&{}&{}&{}&{}\\
{}&{}&{}&{}&{}&{}&{}&{}&{}&{}&{}\\
{}&{}&{}&{}&{}&{}&{}&{}&{}&{}&{}\\
{}&{}&{}&{}&{}&{}&{}&{}&{}&{}&{}\\};
\draw (m-17-1)--(m-7-1);
\path[->,font=\scriptsize,shorten >= -3pt, shorten <= -3pt]
(m-17-11) edge[-] node {} (m-1-11)
(m-17-1) edge[-] node [left,pos=(2.2/11)] {$\rho_\om^{M_{\beta1}}$} (m-7-1)
(m-17-1) edge[-] node [left,pos=(4.6/11)] {$\rho_\om^{M_{\beta2}}$} (m-7-1)
(m-17-1) edge[-] node [left,pos=(6.8/11)] {$M_{\beta2}$} (m-7-1)
(m-17-1) edge[-] node [left,pos=(9/11)] {$M_{\beta 1}$} (m-7-1)
(m-17-1) edge[-] node [pos=(12.1/11)] {$M^\Tt_\beta$} (m-7-1)
(m-17-3) edge[-] node [pos=(14.2/13)] {$M^\Xx_{\gamma_\beta}$}
(m-5-3)
(m-17-5) edge[-] node [right,pos=(6.1/15)]
{$\rho_\om(P_{\beta 1})$} (m-3-5)
(m-17-5) edge[-] node [left,pos=(13.5/15)] {$P_{\beta1}$} (m-3-5)
(m-17-5) edge[-] node [pos=(16/15)]
{$M^\Xx_{\delta_{\beta0}}$} (m-3-5)
(m-17-7) edge[-] node [right,pos=(10.6/15)]
{$\rho_\om(P_{\beta 2})$} (m-3-7)
(m-17-7) edge[-] node [pos=(16/15)]
{$M^\Xx_{\delta_{\beta1}}$} (m-3-7)
(m-17-7) edge[-] node [left,pos=(13.5/15)] {$P_{\beta 2}$} (m-3-7)
(m-17-9) edge[-] node {} (m-3-9) 
(m-17-9) edge[-] node [pos=(16.2/15)] {$M^\Xx_{\varepsilon_{\beta 2}}$}
(m-3-9)
(m-17-9) edge[-] node{} (m-3-9)
(m-17-11) edge[-] node [pos=17.9/17] {$M^\Xx_{\delta_\beta}$} (m-1-11);
\path[->,font=\scriptsize,shorten >= -2pt, shorten <= -2pt]
(m-7-1) edge[->] node[below=1mm]{$\pi_{\beta 0}$} (m-5-3)
(m-5-3) edge[->] node[below,pos=0.25]{$\ \ \ \sigma_{\beta0}$} (m-3-5)
(m-9-1) edge[->] node[below=1mm,pos=0.76]{$\pi_{\beta 1}$} (m-5-5)
(m-5-5) edge[->] node[below,pos=0.25]{$\ \ \ \sigma_{\beta1}$} (m-3-7)
(m-11-1) edge[->] node[below=1mm,pos=.86]{$\pi_{\beta 2}$} (m-5-7)
(m-5-7) edge[->] node[below,pos=0.65]{$\ \ \ \sigma_{\beta2}$} (m-1-11);
\path[->,font=\scriptsize,shorten >= -3pt, shorten <= -3pt]
(m-13-1) edge[dotted] node {} (m-7-7)
(m-15-1) edge[dotted] node {} (m-11-5);
\path[->,font=\scriptsize,shorten >= -2pt, shorten <= -3pt]
(m-9-1) edge[bend left,dashed] node {} (m-15-1)
(m-11-1) edge[bend left,dashed] node {} (m-13-1)
(m-5-5) edge[bend left,dashed] node {} (m-11-5)
(m-7-5) edge[bend left,dashed] node {} (m-9-5)
(m-5-7) edge[bend left,dashed] node {} (m-7-7);
\path[-{Straight Barb[left]},shorten >= -3pt, shorten <= -3pt]
(m-16-3) edge[dotted] node {} (m-15-4)
(m-11-5) edge[dotted] node {} (m-10-6)
(m-7-7) edge[dotted] node {} (m-6-8);
\end{tikzpicture}
\caption{Model embeddings, with $m_\beta=2$. Vertical lines represent models,
with length roughly corresponding to ordinal height.
Solid arrows represent embeddings
$\pi_{\beta i}$ and $\sigma_{\beta i}$,
with $\crit(\sigma_{\beta i})$ roughly
at the origin of a short dotted half-headed arrow.
Dotted full-headed arrows indicate certain threads under embeddings.
Dashed curved arrows point to the $\omega^\nth$
projectum of the structure at their origin. Note that
$M^\Xx_{\delta_{\beta i}}=M^\Xx_{\gamma_{\beta,i+1}}$
and for $i=0,1$.}
\label{fgr:model_embeddings}
\end{figure}

\item If $\gamma_x<\delta_x$ let
\index{$\tau_\beta,\tau_{\beta i}$}
\index{$\sigma_\beta,\sigma_{\beta i}$}
\[ \sigma_{\beta i}=\sigma_x=i^{*\Xx}_{\eps_x,\delta_x}:P_x\to
M^\Xx_{\delta_x};\]
otherwise let $\sigma_x:P_x\to P_x$ be the identity. Let
$\tau_x=\sigma_x\com\pi_x$.
\item\label{item:P_x^+,pi_x^+} Suppose $(\beta,i+1)\in \Delta$. Then
$P_{\beta,i+1}=\tau_{\beta i}(M_{\beta,i+1})$
and
$\pi_{\beta,i+1}=\tau_{\beta i}\rest M_{\beta,i+1}$.
\end{enumerate}
\item\label{item:extender_copying} (Extender copying) For $\beta+1\leq\theta$,
let
$\omega_\beta=\tau_{\beta m_\beta}$\index{$\om_\beta$}
\index{$Q_\beta,Q_{\beta i}$}
and let $Q_\beta$ be the codomain of $\omega_\beta$; that is,
\begin{enumerate}[label=--]
\item if $\gamma_{\beta m_\beta}=\delta_{\beta m_\beta}$ then $Q_\beta=P_{\beta
m_\beta}$, and
\item if $\gamma_{\beta m_\beta}<\delta_{\beta m_\beta}$ then
$Q_\beta=M^\Xx_{\delta_\beta}$.
\end{enumerate}
If $\beta+1<\theta$ then $E^\Xx_{\delta_\beta}=F^{Q_\beta}$ (so
$E^\Xx_{\delta_\beta}$ is the copy
of $E^\Tt_\beta$ under $\omega_\beta$).

\item\label{item:embedding_commutativity} (Embedding commutativity)
\begin{figure}
\centering
\begin{tikzpicture}
 [mymatrix/.style={
    matrix of  nodes,
    row sep=0.09cm,
    column sep=0.7cm}]
   \matrix(m)[mymatrix]{
{}&{}&{$M^\Tt_\xi$}&{}&{}&{}&{$P_\xi$}&{}\\
{}&{}&{}&{}&{}&{}&{}&{}\\
{}&{}&{}&{}&{}&{}&{}&{$M^\Xx_{\delta_{\beta1}}$}\\
{}&{}&{}&{}&{}&{}&{}&{}\\
{}&{}&{}&{$M^\Xx_{\gamma_{\beta1}}$}&{}&{$M^\Xx_{\tilde{\beta}}$}&{}&{}\\
{}&{}&{}&{\rotatebox[origin=c]{90}{$\pins$}}&{}&{}&{}&{}\\
{$M^\Tt_\beta$}&{}&{$P_\beta$}&{$P_{\beta1}$}&{}&{}&{}&{}\\
{\rotatebox[origin=c]{90}{$\pins$}}&{}&{\rotatebox[origin=c]{90}{$\pins$}}&{}&{}
&{}&{}&{}\\
{$M_{\beta1}$}&{}&{$P$}&{}&{}&{}&{}&{}\\};
\path[->,font=\scriptsize,shorten >= -2pt, shorten <= -2pt]
(m-9-1) edge[dotted] node[below] {$\pi_\beta\rest M_{\beta1}$} (m-9-3)
(m-9-1) edge node[left,pos=0.65] {$i^{*\Tt}_{\xi}$} (m-1-3)
(m-7-1) edge[dotted] node[below] {$\pi_\beta$} (m-7-3)
(m-9-3) edge node[below]{$\ \ \ \ \ \ \sigma_{\beta 0}\rest P$} (m-7-4)
(m-7-4) edge node[above,pos=0.65]
{$i^{*\Xx}_{\varepsilon_{\beta1}\gamma_\xi}\ \ \ \ \ \ \ $} (m-1-7)
(m-7-3) edge node[below] {$\ \ \ \ \sigma_{\beta0}$} (m-5-4)
(m-7-4) edge node[below,pos=0.7]
{$\ \ \ i^{*\Xx}_{\varepsilon_{\beta1}\tilde{\beta}}$} (m-5-6)
(m-5-6) edge node[below] {$\ \ \ i^\Xx_{\tilde{\beta}\delta_{\beta1}}$} (m-3-8)
(m-5-6) edge node[right] {$\ i^{*\Xx}_{\gamma_\xi}$} (m-1-7)
(m-1-3) edge[dotted] node[below,pos=0.4] {$\pi_\xi$} (m-1-7);
\end{tikzpicture}
\caption{Embedding commutativity part
\ref{item:embedding_commutativity}\ref{item:emb_comm_drop},
with $i=1$ and $\gamma_{\beta 1}<\widetilde{\beta}<\delta_{\beta 1}$. The
diagram commutes.
Solid arrows are iteration embeddings and their restrictions; dotted arrows are
copy embeddings
and their restrictions. In the figure, $P=\pi_\beta(M_{\beta
1})$.}\label{fgr:embedding_comm}
\end{figure}
Let $(\beta,i),(\alpha+1,0),(\xi,0)\in \Delta$ be such that
$\beta<^\Tt\alpha+1\leq^\Tt\xi$ and
$\beta=\pred^\Tt(\alpha+1)$ and $M_{\beta i}=M^{*\Tt}_{\alpha+1}$.
Then:
\begin{enumerate}[label=\tu{(}\alph*\tu{)}]
\item If $(\beta,\xi]_\Tt\inter \dr^\Tt=\emptyset$ (so $i=0$ and
$(\gamma_\beta,\gamma_\xi]_\Xx\inter\dr^\Xx=\emptyset$) then
\[ \pi_{\xi}\com
i^{\Tt}_{\beta,\xi}=i^\Xx_{\gamma_\beta,\gamma_\xi}\com\pi_{\beta} \]
and $\pred^\Xx(\gamma_{\alpha+1})\in I_{\beta 0}$.
\item\label{item:emb_comm_drop}
See figure
\ref{fgr:embedding_comm}. Suppose $\xi=\alpha+1\in \dr^\Tt$ (so $i>0$). Let
$\widetilde{\beta}=\pred^\Xx(\gamma_\xi)$. Then $\widetilde{\beta}\in
I_{\beta i}$ and:
\begin{enumerate}[label=--]
 \item If
$\widetilde{\beta}=\gamma_{\beta i}$ then $\gamma_\xi\in \dr^\Xx$ and
$M^{*\Xx}_{\gamma_\xi}=P_{\beta i}$
and
$\pi_{\xi}\com i^{*\Tt}_\xi=i^{*\Xx}_{\gamma_\xi}\com\pi_{\beta i}$;
\item If $\widetilde{\beta}>\gamma_{\beta i}$ then $\gamma_\xi\notin \dr^\Xx$
and
$\pi_{\xi}\com i^{*\Tt}_\xi=i^{*\Xx}_{\eps_x,\gamma_\xi}\com\pi_{\beta i}$.
\end{enumerate}
\end{enumerate}
\item\label{item:embedding_agreement} (Embedding agreement) For
$\beta+1<\theta$
and
$(\beta',i')\in \Delta$ with $\beta<\beta'$:
\begin{enumerate}[label=--]
\item $\omega_\beta\rest\iota^\Tt_\beta\sub\pi_{\beta' i'}$
\item $\omega_\beta(\alpha)\leq\pi_{\beta' i'}(\alpha)$ for all
$\alpha<\In(E^\Tt_\beta)$,
\item if
$\In(E^\Tt_\beta)<\OR(M_{\beta'i'})$ then
$\In(E^\Xx_{\delta_\beta})\leq\pi_{\beta' i'}(\In(E^\Tt_\beta))$,
\item if
$\In(E^\Tt_\beta)=\OR(M_{\beta' i'})$ then $\beta'=\beta+1$, $i'=0$,
$\In(E^\Xx_{\delta_\beta})=\OR(M^\Xx_{\gamma_{\beta+1}})$,
$\pi_{\beta+1}=\omega_\beta$,
$M^{*\Xx}_{\gamma_{\beta+1}}=Q_\alpha$
where $\alpha=\pred^\Tt(\beta+1)$.\footnote{It follows that
we are using MS-indexing,
$E^\Tt_\beta$ is superstrong and $M^\Tt_{\beta+1}$ is active type 2,
$E^\Xx_{\delta_\beta}$ is
superstrong and $M^\Xx_{\delta_\beta+1}$ is active type 2.}\qedhere
\end{enumerate}
\end{enumerate}
\end{dfn}

The analogue for wcpms is much simpler, as there is no dropping to consider:

\begin{dfn}[Tree embedding for wcpms]\label{cdfn:tree_embedding}\index{tree
embedding}\index{$\Pi:(\Tt,\theta)\tembto\Xx$}
 Let $M$ be a wcpm,
 let $\Tt$ be a putative  normal tree on $M$,
 let $\Xx$ be a normal tree on $M$,
 and let $1\leq\theta\leq\lh(\Tt)$.
A \dfnemph{tree embedding} $\temb:(\Tt,\theta)\tembto\Xx$
 from
$(\Tt,\theta)$ to $\Xx$
is a system $\Pi$ of form
\index{$I_\beta,I_{\beta i}$}
\index{$\pi_\beta,\pi_{\beta i}$}\index{$P_\beta,P_{\beta i}$}
$\Pi=\left(\Tt,\left<I_\beta,\pi_{\beta}\right>_{\beta<\theta}\right)$
satisfying conditions \ref{citem:tree_pres}--\ref{citem:embedding_agreement}
below.\footnote{There is no analogue
of condition \ref{dfn:tree_embedding}(\ref{item:structure_I_beta}), because
there is no dropping or degrees.}

\begin{enumerate}[label=T$_\text{c}$\arabic*.,ref=T$_\text{c}$\arabic*]

\item\label{citem:tree_pres} (Preservation of tree structure) Exactly the
assertion
of condition \ref{dfn:tree_embedding}(\ref{item:tree_pres}), minus the
references to dropping and degrees.
\setcounter{enumi}{2}
\item\label{citem:model_embeddings} (Model embeddings) See figure
\ref{fgr:model_embeddings}. For all $\beta<\theta$:
\begin{enumerate}[label=\tu{(}\alph*\tu{)}]
\item Let $P_{\beta}=M^\Xx_{\gamma_{\beta}}$.\index{$P_\beta,P_{\beta i}$}
\item $\pi_\beta:M_\beta\to P_\beta$ is a coarse
$0$-embedding.\index{$\pi_\beta,\pi_{\beta i}$}
\item $\pi_0=\id:M\to M$.
\item\index{$\sigma_\beta,\sigma_{\beta i}$}
\index{$\tau_\beta,\tau_{\beta i}$} Let
$\sigma_{\beta}=i^{\Xx}_{\gamma_\beta\delta_\beta}:P_\beta\to
M^\Xx_{\delta_\beta}$
and $\tau_\beta=\sigma_\beta\com\pi_\beta$.
\end{enumerate}
\item\label{citem:extender_copying} (Extender copying) For $\beta+1<\theta$, we
have
$E^\Xx_{\delta_\beta}=\tau_\beta(E^\Tt_\beta)$.\index{$\om_\beta$}
\index{$Q_\beta,Q_{\beta i}$}\footnote{Note that because
there is no dropping,
we do not define $\om_\beta$ and $Q_\beta$ here. The map
$\tau_\beta$ lifts $E^\Tt_\beta$ to $E^\Xx_{\delta_\beta}$ here.}

\item\label{citem:embedding_commutativity} (Embedding commutativity)
If $\beta<^\Tt\xi<\theta$  and $\alpha+1=\successor^\Tt(\beta,\xi)$
then
\[ \pi_{\xi}\com
i^{\Tt}_{\beta\xi}=i^\Xx_{\gamma_\beta\gamma_\xi}\com\pi_{\beta}.\]
\item\label{citem:embedding_agreement} (Embedding agreement) Let
$\beta+1\leq\beta'<\theta$ and $\varrho=\varrho^\Tt_\beta$. Then
\[ \tau_\beta\rest V_\varrho^{M^\Tt_\beta}\sub\pi_{\beta'}\text{ and }
\varrho^\Xx_{\delta_\beta}=\tau_\beta(\varrho)\leq\pi_{\beta'}
(\varrho).\qedhere\]
\end{enumerate}
\end{dfn}

\begin{dfn}\index{$\udash$degree}\index{u-degree}
A tree embedding $\Pi:(\Tt,\theta)\tembto\Xx$ has \dfnemph{$\udash$degree $k$}
iff $\Tt,\Xx$ are
$\udash k$-maximal. (There is a unique such $k$, since
$k=\udeg^\Tt(0)$.)\end{dfn}

\begin{dfn}\index{tree embedding}\index{$\Pi:\Tt\tembto\Xx$} A \dfnemph{tree
embedding}
$\Pi:\Tt\tembto\Xx$
from $\Tt$ to $\Xx$
 is a tree embedding $\Pi:(\Tt,\lh(\Tt))\tembto\Xx$.
\end{dfn}

Clearly if $\Pi:\Tt\tembto\Xx$ then $\Tt$ is in fact an iteration tree (it has
well-defined and wellfounded models).
We record some notation:

\begin{dfn}\label{dfn:Pi_subscript_notation}
\index{$Q_\beta,Q_{\beta i}$}\index{$i_\beta$}\index{$\gamma_{\Pi\beta}$,
$I_{\Pi\beta}$, etc}
Let $\Pi$ be a tree embedding. Fix notation as in \ref{dfn:tree_embedding}.
Define $Q_{\beta i}=\cod(\tau_{\beta i})$. That is, $Q_{\beta i}=P_{\beta i}$
if
$\gamma_{\beta
i}=\delta_{\beta i}$,
and $Q_{\beta i}=M^\Xx_{\delta_{\beta i}}$ otherwise.
Let $i_{\beta}$ be the least $i$ such that $\delta_\beta=\delta_{\beta i}$.
So $Q_\beta=Q_{\beta m_\beta}\ins Q_{\beta i_\beta}=M^\Xx_{\delta_\beta}$.

We use the subscript\footnote{The superscript position of this notation will be
used for
another purpose.} ``$\Pi$'' to indicate the objects
associated to $\Pi$.
That is, $I_{\Pi\beta}=I_\beta$ for $\beta<\theta$, and
$\Gamma_\Pi=\Gamma$,
and likewise for
$\gamma_\beta,\delta_\beta,P_\beta,\pi_\beta,Q_\beta,\omega_\beta,i_\beta$ for
$\beta<\theta$, and $I_{\beta i},P_{\beta i},\pi_{\beta i},\gamma_{\beta
i},\delta_{\beta i},\sigma_{\beta i},\tau_{\beta i},Q_{\beta i}$ for
$(\beta,i)\in \Delta$.
\end{dfn}

\begin{dfn}[$j^\Xx_{\xi\eta}$]\index{$j^\Xx_{\xi\eta}$} Let
$\Pi:(\Tt,\theta)\tembto\Xx$ be a tree embedding and
$\gamma_\beta=\gamma_{\Pi\beta}$, etc.
Let $\beta<\theta$.
Let $\xi,\eta\in I_\beta$
with $\xi\leq\eta$. Then $j^\Xx_{\xi\eta}$ denotes the embedding with domain as
large as
possible, given by composing iteration embeddings $i^\Xx_{\mu\nu}$ and
$i^{*\Xx}_{\mu\nu}$ with
$\xi\leq^\Xx\mu\leq^\Xx\nu\leq^\Xx\eta$.
That is, let $m,n$ be least such that $\xi\in I_{\beta m}$ and $\eta\in
I_{\beta
n}$ respectively.
If $m=n$ then $j^\Xx_{\xi\eta}\eqdef i^\Xx_{\xi\eta}$. If $m<n$ then letting
$\eps=\eps_{\beta
n}$ and $\delta=\delta_{\beta,n-1}=\gamma_{\beta n}$,
\[ j^\Xx_{\xi\eta}\eqdef i^{*\Xx}_{\eps\eta}\com j^\Xx_{\xi\delta},\]
where $\dom(j^\Xx_{\xi\eta})=M^\Xx_\xi$ if $m=n$, and
$\dom(j^\Xx_{\xi\eta})=j^\Xx_{\gamma_\beta\xi}(\pi_{\beta 0}(M_{\beta n}))$ if
$m<n$.
\end{dfn}

\begin{dfn}[$\pi_{\beta\kappa}:M_{\beta\kappa}\to P_{\beta\kappa}$ and
$n_{\beta\kappa}$]\label{dfn:pi_beta,kappa}\index{$\M_{\beta\kappa}$,
$\gamma_{\beta\kappa}$, $\pi_{\beta\kappa}$, etc} (Figure
\ref{fgr:pi_beta,kappa}.)
 Let
$\Pi:(\Tt,\theta)\tembto\Xx$
 be a tree embedding and $\gamma_\beta=\gamma_{\Pi\beta}$,
etc.
Let $\beta<\theta$.
Let $\kappa\in[\omega,\OR(M^\Tt_\beta))$, with $\kappa<\nutilde^\Tt_\beta$
if $\beta+1<\lh(\Tt)$,
and $\kappa\leq\OR(M^\Tt_\beta)$ if $\beta+1=\lh(\Tt)$.
We will define $i_{\beta\kappa}$, $n_{\beta\kappa}$, $M_{\beta\kappa}$,
$\gamma_{\beta\kappa}$,
$P_{\beta\kappa}$ and
$\pi_{\beta\kappa}:M_{\beta\kappa}\to P_{\beta\kappa}$.

If $\beta+1=\lh(\Tt)$ and $\kappa=\OR(M^\Tt_\beta)$ then let
$i_{\beta\kappa}=n_{\beta\kappa}=0$, $M_{\beta\kappa}=M^\Tt_\beta$,
$P_{\beta\kappa}=Q_\beta$,
$\gamma_{\beta\kappa}=\delta_\beta$ and $\pi_{\beta\kappa}=\omega_\beta$.

Now suppose either $\beta+1<\lh(\Tt)$ or $\kappa<\OR(M^\Tt_\beta)$.
Let $i_{\beta\kappa}$ be
the largest $i<\om$ such
that either $i=0$ or
$\rho_\om(M_{\beta i})\leq\kappa$. Let $i=i_{\beta\kappa}$.
Set $M_{\beta\kappa}\eqdef M_{\beta i}$. Let $n_{\beta\kappa}$ be the largest
$n<\om$ such that
\[ (M_{\beta\kappa},n)\ins(M_{\beta 0},\udeg^\Tt(\beta)) \]
and
$\kappa<\udash\rho_n^{M_{\beta\kappa}}$.

\begin{figure}
\centering
\begin{tikzpicture}
 [mymatrix/.style={
    matrix of  nodes,
    row sep=0.04cm,
    column sep=0.8cm}]
   \matrix(m)[mymatrix]{
{}&{}&{}&{}&{}&{}&{}&{}&{}\\
{}&{}&{}&{}&{}&{}&{}&{}&{}\\
{}&{}&{}&{}&{}&{}&{}&{}&{}\\
{}&{}&{}&{}&{}&{}&{}&{}&{}\\
{}&{}&{}&{}&{}&{}&{}&{}&{}\\
{}&{}&{}&{}&{}&{}&{}&{}&{}\\
{}&{}&{}&{}&{}&{}&{}&{}&{}\\
{}&{}&{}&{}&{}&{}&{}&{}&{}\\
{}&{}&{}&{}&{}&{}&{}&{}&{}\\
{}&{}&{}&{}&{}&{}&{}&{}&{}\\
{}&{}&{}&{}&{}&{}&{}&{}&{}\\
{}&{}&{}&{}&{}&{}&{}&{}&{}\\
{}&{}&{}&{}&{}&{}&{}&{}&{}\\
{}&{}&{}&{}&{}&{}&{}&{}&{}\\
{}&{}&{}&{}&{}&{}&{}&{}&{}\\
{}&{}&{}&{}&{}&{}&{}&{}&{}\\
{}&{}&{}&{}&{}&{}&{}&{}&{}\\
{}&{}&{}&{}&{}&{}&{}&{}&{}\\
{}&{}&{}&{}&{}&{}&{}&{}&{}\\
{}&{}&{}&{}&{}&{}&{}&{}&{}\\
{}&{}&{}&{}&{}&{}&{}&{}&{}\\
{}&{}&{}&{}&{}&{}&{}&{}&{}\\
{}&{}&{}&{}&{}&{}&{}&{}&{}\\};
\path[->,font=\scriptsize,shorten >= -3pt, shorten <= -3pt]
(m-23-1) edge[-] node [pos=(11.8/11)] {$M^\Tt_\beta$} (m-5-1)
(m-23-1) edge[-] node [left,pos=(8.6/11)] {$M_{\beta\kappa}=M_{\beta i}$}
(m-5-1)
(m-23-1) edge[-] node [left,pos=(6.2/11)] {$M_{\beta,i+1}$} (m-5-1)
(m-23-1) edge[-] node [left,pos=(2.4/11)] {$\kappa$} (m-5-1)
(m-23-3) edge[-] node [pos=(11.8/11)] {$M^\Xx_{\gamma_{\beta i}}$} (m-3-3)
(m-23-3) edge[-] node [left,pos=(9.3/11)] {$P_{\beta i}$} (m-3-3)
(m-23-5) edge[-] node [pos=(12/11)] {$M^\Xx_\alpha$} (m-5-5)
(m-23-7) edge[-] node [pos=(12.1/11)] {$P_{\beta\kappa}=M^\Xx_\gamma$} (m-3-7)
(m-23-9) edge[-] node [pos=(11.8/11)] {$M^\Xx_{\delta_{\beta i}}$} (m-1-9)
(m-23-9) edge[-] node [right,pos=(5.2/11)] {$\om_\beta(\kappa)$} (m-1-9)
(m-23-1) edge[-] node {} (m-5-1)
(m-23-3) edge[-] node  {} (m-3-3)
(m-23-5) edge[-] node {} (m-5-5)
(m-23-7) edge[-] node {} (m-3-7)
(m-23-9) edge[-] node  {} (m-1-9)
(m-19-1) edge[-,dotted] node{} (m-13-7)
(m-13-7) edge[-,dotted] node {} (m-13-9);
\path[->,font=\scriptsize,shorten >= -2pt, shorten <= -3pt]
(m-9-1) edge[bend left,dashed] node {} (m-21-1)
(m-13-1) edge[bend left,dashed] node {} (m-17-1);
\path[->,font=\scriptsize,shorten >= -2pt, shorten <= -2pt]
(m-9-1) edge node[above] {$\pi_{\beta i}$} (m-7-3)
(m-7-3) edge  node[above,pos=0.25] {$j$} (m-3-7)
(m-9-1) edge[transform canvas={yshift=-1.4mm}]  node[below]
{$\ \ \pi_{\beta\kappa}$} (m-3-7)
(m-3-7) edge[dashed] node[below] {$j'$} (m-1-9);
\path[-{Straight Barb[left]},shorten >= 6pt, shorten <= -3pt]
(m-21-1) edge[dotted] node {} (m-20-2)
(m-19-3) edge[dotted] node {} (m-18-4)
(m-17-5) edge[dotted] node {} (m-16-6)
(m-12-7) edge[dotted] node {} (m-11-8);
\end{tikzpicture}
\caption{A typical picture for the embedding
$\pi_{\beta\kappa}:M_{\beta\kappa}\to
P_{\beta\kappa}$, when $i=i_{\beta\kappa}$ and $\gamma_{\beta
i}<\gamma=\gamma_{\beta\kappa}<\delta_{\beta i}$.
Note $\pi_{\beta\kappa}=j\com\pi_{\beta i}$,
where $j=j^\Xx_{\gamma_{\beta i},\gamma_{\beta\kappa}}$. The long dotted path
indicates the
trajectory of $\kappa$.
Critical points are indicated by dotted half-headed arrows. (Where critical
points are shown
strictly below the image of $\kappa$ in the figure, they could in general equal
that image.) Also, $\alpha\in(\gamma_{\beta
i},\gamma_{\beta\kappa})_\Xx$ and $j'=i^\Xx_{\gamma_{\beta\kappa},\delta_{\beta
i}}$.}\label{fgr:pi_beta,kappa}
\end{figure}
Let $\gamma_{\beta\kappa}$ be the least $\gamma\in I_{\beta i}$ such that
either
$\gamma=\delta_{\beta i}$ or
\[ \crit(j^\Xx_{\gamma\delta_{\beta i}})>j^\Xx_{\gamma_{\beta
i}\gamma}\com\pi_{\beta
i}(\kappa). \]
Let $\gamma=\gamma_{\beta\kappa}$. If $\gamma=\gamma_{\beta i}$ then
$P_{\beta\kappa}\eqdef P_{\beta i}$ and $\pi_{\beta\kappa}\eqdef\pi_{\beta i}$.
If
$\gamma>\gamma_{\beta i}$ then $P_{\beta\kappa}\eqdef M^\Xx_\gamma$
and $\pi_{\beta\kappa}\eqdef j^\Xx_{\gamma_{\beta i}\gamma}\com\pi_{\beta i}$.

We write $\pi_{\Pi\beta\kappa}=\pi_{\beta\kappa}$, etc.
\end{dfn}

\begin{lem}\label{lem:intervals_I_cover_X-branches}
 Let $\Pi:(\Tt,\theta)\hookrightarrow\Xx$ be a tree embedding.
 Let
 $\alpha\in I_{\Pi\xi}$ and $\delta\leq^\Xx\alpha$.
 Then $\delta\in I_{\Pi\zeta}$ for some $\zeta\leq^\Tt\xi$.
\end{lem}

The reader will easily verify the lemma above (proceed by induction).
Part  \ref{item:deg_match} of the next lemma
ensures that when we want to extend tree
embeddings, we will not encounter any difficulties regarding condition
\ref{item:tree_pres}\ref{item:degree_match}.
\begin{lem}\label{lem:pi_beta,kappa_deg}
  Let
$\Pi:(\Tt,\theta)\tembto\Xx$
 be a tree embedding and let $\gamma_\beta=\gamma_{\Pi\beta}$,
etc. Let $\pi=\pi_{\beta\kappa}$ and $n=n_{\beta\kappa}$ and
$\gamma=\gamma_{\beta\kappa}$.
Then
\begin{enumerate}[label=\arabic*.,ref=\arabic*]
 \item $\pi$ is a near $\udash n$-embedding,
 \item $(P_{\beta\kappa},n)\ins^\uu(M^\Xx_{\gamma},\udeg^\Xx(\gamma))$ and
\item\label{item:deg_match} if
$(P_{\beta\kappa},n)\pins^\uu(M^\Xx_{\gamma},\udeg^\Xx(\gamma))$
and $\kappa<\OR(M_{\beta\kappa})$  then
$\pi_{\beta\kappa}(\kappa)\geq\udash\rho_{n+1}(P_{\beta\kappa})$.
\end{enumerate}
\end{lem}
\begin{proof}[Proof Sketch]
Part \ref{item:deg_match}:
 Let $i=i_{\beta\kappa}$. If
$i=0$,
use that $\pi_\beta$ is a near
$\udash(n+1)$-embedding
(see condition
\ref{item:model_embeddings}\ref{item:pi_beta,0_near_embedding_only}),
as are the relevant iteration maps; in the other case it is similar,
but $\pi_{\beta i}$ is fully elementary.
\end{proof}

\begin{dfn}\index{$E^\Pi_\gamma$}\index{bounding}
Let $\Pi:(\Tt,\theta)\tembto\Xx$. Let $\beta\in\theta\inter\lh(\Tt)^-$ and
$\gamma\in
I_{\Pi\beta}$. Then $E^\Pi_\gamma$ denotes the copy $F$ of $E^\Tt_\beta$ in
$\es_+(M^\Xx_\gamma)$.
(That is, letting $k=j^\Xx_{\gamma_\beta\gamma}\com\pi_\beta$, if
$E^\Tt_\beta\in\dom(k)$ then
$F=k(E^\Tt_\beta)$, and otherwise $F=F(M^\Xx_\gamma)$). We say that $\Pi$ is
\dfnemph{bounding}
iff $\In(E^\Xx_\gamma)\leq\In(E^\Pi_\gamma)$ for all such $\beta,\gamma$.
\end{dfn}

We will only really be interested in bounding tree embeddings,
and in this case we have the following easy observation:

\begin{lem}\label{lem:Pi-bounded_reasonable_copy}Let
$\Pi:(\Tt,\theta)\tembto\Xx$ be bounding.
Suppose $\theta=\beta+1<\lh(\Tt)$. Then $E^\Pi_{\delta_{\Pi\beta}}$ is
$\Xx\rest(\delta_{\Pi\beta}+1)$-normal.\end{lem}

We now consider the existence and uniqueness of tree embeddings.

\begin{dfn}\label{dfn:trivial_tree_embedding}\index{trivial (tree embedding)}
Let $\Tt,\Xx$ be putative $\udash
k$-maximal trees on $M$, with $\Xx$ an iteration tree.
 The \dfnemph{trivial} tree embedding $\Pi:(\Tt,1)\tembto\Xx$ is the
unique one
such that $I_{\Pi 0}=[0,0]$; that is,
\[
\Pi=\left(\Tt,\left<[0,0]\right>;\left<I_{0i},P_{0i},\pi_{0i}\right>_{(0,
i)\in\Delta}\right) \]
where $\Delta=\ddd^{(\Tt,1)}$ and $I_{0 i}=[0,0]$ and $P_{0 i}=M_{0 i}^\Tt$ and
$\pi_{0
i}=\id$.
\end{dfn}

We will give two lemmas describing how we can propagate tree embeddings via
ultrapowers.
The first of these involves copying an extender. Part of this is a natural
variant of
the fact that the copying construction propagates near embeddings (see
\cite{fs_tame}), and to
state this we need the following definition:

\begin{dfn}\index{$*$-tree embedding}\index{tree
embedding}\index{$\Pi:\Tt\tembto^*\Xx$}
Let $\Tt,\Xx,\theta$ be as in \ref{dfn:tree_embedding}. A \dfnemph{$*$-tree
embedding $\Pi$ from
$(\Tt,\theta)$ to $\Xx$}, denoted $\Pi:(\Tt,\theta)\tembto^*\Xx$, is a system
as
in
\ref{dfn:tree_embedding}, but replacing
\ref{item:model_embeddings}\ref{item:pi_beta,0_near_embedding_only} with the
requirement that
$\pi_{\beta}$ be $\rSigma_n$-elementary where $n=\udeg^\Tt(\beta)$.
\end{dfn}

\begin{lem}\label{lem:tree_embedding_copy}\index{uniqueness, tree embeddings}
Let $\Pi':(\Tt,\theta)\tembto\Xx'$ have $\udash$degree $k$.
Let $\gamma'_\alpha=\gamma_{\Pi'\alpha}$, etc.
Suppose that $\theta=\alpha+1<\lh(\Tt)$ and $\lh(\Xx')=\delta'_\alpha+1$ and
$E^{\Pi'}_{\delta'_\alpha}$ is $\Xx$-normal.
Suppose that the putative $\udash k$-maximal
tree $\Xx''\eqdef\Xx'\conc\left<E^{\Pi'}_{\delta'_\alpha}\right>$ has
wellfounded last model.

Then \tu{(}i\tu{)} $M^\Tt_{\alpha+1}$ is
wellfounded, \tu{(}ii\tu{)} there is a unique pair $(\Xx,\Pi)$
such that:
\begin{enumerate}[label=--]
\item $\Xx$ is a $\udash k$-maximal tree extending $\Xx'$ with
$\lh(\Xx)=\lh(\Xx')+1$,
\item $\Pi:(\Tt,\theta+1)\tembto^*\Xx$, and
\item $\Pi'\sub\Pi$,
\end{enumerate}
\tu{(}iii\tu{)}
$\Xx=\Xx''$, \tu{(}iv\tu{)}
$\Pi:(\Tt,\theta+1)\tembto\Xx$, \tu{(}v\tu{)} if $\theta+1<\lh(\Tt)$ then
$E^{\Pi}_{\delta'_\alpha+1}$ is
$\Xx$-normal, and \tu{(}vi\tu{)} if $\Pi'$ is bounding then so is $\Pi$.
\end{lem}

Before we prove the lemma, we state two easy consequences:

\begin{cor}
 Every $*$-tree embedding is a tree embedding.
\end{cor}
\begin{cor}\label{cor:tree_embedding_uniqueness}
Let $\Pi:(\Tt,\theta)\tembto\Xx$ and $\Pi':(\Tt,\theta)\tembto\Xx$ be tree
embeddings
such that
$\delta_\beta^\Pi=\delta_\beta^{\Pi'}$ for all $\beta<\theta$. Then
$\Pi=\Pi'$.
\end{cor}

\begin{proof}[Proof of \ref{lem:tree_embedding_copy}]
We first exhibit $(\Xx,\Pi)$ as in (ii);
then (i) follows. We then prove the uniqueness of $(\Xx,\Pi)$,
and then (iv), and leave the
rest to the reader.

We use $\Xx=\Xx''$.
 In defining the components of $\Pi$, we will write $I_\beta=I_{\Pi\beta}$,
etc.
Most of
$\Pi$ is already determined by the requirement that $\Pi'\sub\Pi$, so we just
define the rest. Let
\[ I_{\alpha+1}=I_{\alpha+1,0}=[\delta'_\alpha+1,\delta'_\alpha+1] \]
and $P_{\alpha+1,0}=M^{\Xx}_{\delta'_\alpha+1}$.
It just remains to define
$\pi_{\alpha+1,0}:M^{\Tt}_{\alpha+1}\to M^{\Xx}_{\delta_\alpha+1}$,
and we claim that we can do this using the Shift Lemma.

For let $E=E^{\Tt}_\alpha$ and $\kappa=\crit(E)$ and
$\beta=\pred^{\Tt}(\alpha+1)$ and
$n=\udeg^{\Tt}(\alpha+1)$.
Note $M^\Tt_{\beta\kappa}=M^{*\Tt}_{\alpha+1}$ and $n_{\beta\kappa}=n$
and
$P_{\beta\kappa}=M^{*\Xx}_{\delta_\alpha+1}$ and
\[ \pow(\kappa)\inter M^\Tt_{\beta\kappa}=\pow(\kappa)\inter
M^\Tt_\alpha|\In(E), \]
\[
\pi_{\beta\kappa}
\rest\pow(\kappa)=\om_\beta\rest\pow(\kappa)=\omega_\alpha\rest\pow(\kappa) \]
and by \ref{lem:pi_beta,kappa_deg}, $n=\udeg^{\Xx}(\delta_\alpha+1)$.

So we apply the (proof of the) Shift Lemma to $\pi_{\beta\kappa}$ and
$\omega_\alpha$, defining
$\pi_{\alpha+1,0}$, a weak $\udash n$-embedding. The embedding commutativity
and
agreement conditions are
satisfied. So $M^{\Tt}_{\alpha+1}$ is wellfounded and
$\Pi:(\Tt,\theta)\tembto^*\Xx$ and
$\Pi'\sub\Pi$.

The definitions we made were in fact
the only ones possible; in the case of
$\pi_{\alpha+1}$, this is because if
$\pi:M^{\Tt}_{\alpha+1}\to M^{\Xx}_{\delta_\alpha+1}$
is $\rSigma_n$-elementary and satisfies the commutativity and agreement
conditions,
then $\pi$ is just as defined in the proof of the Shift Lemma. This gives
uniqueness.

For (v), it remains to see that $\pi_{\alpha+1}$ is a near
$\udash n$-embedding. This is proved almost as in \cite{fs_tame}; we give a
sketch so as to indicate the main difference.

For $\zeta+1<\lh(\Tt)$, we say that \dfnemph{strong closeness at
$\zeta$}\index{strong closeness} holds
iff
for
each $a\in[\nu(E^{\Tt}_\zeta)]^{<\om}$ there is a $\uSigma_1$ formula
$\varphi_a$
and
$q_a\in M^{*\Tt}_{\zeta+1}$ such that
\[ (E^{\Tt}_\zeta)_a=\{x\in M^{*\Tt}_{\zeta+1}\bigm|
M^{*\Tt}_{\zeta+1}\sats\varphi_a(q_a,x)\},\]
and letting $\beta=\pred^{\Tt}(\zeta+1)$ and $\mu=\crit(E^{\Tt}_\zeta)$,
so $M^{*\Xx}_{\gamma_{\zeta+1}}=P_{\beta\mu}$,
\[(E^{\Xx}_{\delta_\zeta})_{\omega_\zeta(a)}=\{x\in P_{\beta\mu}\bigm|
P_{\beta\mu}\sats\varphi_a(\pi_{\beta\mu}(q_a),x)\}.\]

For $\eps<\lh(\Tt)$, we say \dfnemph{translatability at
$\eps$}\index{translatability} holds
iff, letting $m=\udash\deg^{\Tt}(\eps)$, for all $(x,\varphi,\zeta+1)$
such that $x\in M^{\Tt}_\eps$ and $\varphi$ is
$\uSigma_{m+1}$ and  $\zeta+1\leq^{\Tt}\eps$ and
$(\zeta+1,\eps]_{\Tt}$ does not drop in model or degree, there is
$(x',\varphi')$ such that $x'\in M^{*\Tt}_{\zeta+1}$ and
$\varphi'$ is $\uSigma_{m+1}$, and for all
$\gamma<\mu\eqdef\crit(E^{\Tt}_\zeta)$, we have
\[ M^{\Tt}_\eps\sats\varphi(x,\gamma)\ \iff\
M^{*\Tt}_{\zeta+1}\sats\varphi'(x',\gamma), \]
and letting $\beta=\pred^{\Tt}(\zeta+1)$,
for all $\gamma<\pi_{\beta\mu}(\mu)=\crit(E^{\Xx}_{\delta_\zeta})$, we have
\[ M^{\Xx}_{\gamma_\eps}\sats\varphi(\pi_\eps(x),\gamma)\ \iff\
M^{*\Xx}_{\gamma_{\zeta+1}}\sats\varphi'(\pi_{\beta\mu}(x'),\gamma). \]
(Recall that $\gamma_{\beta\mu}=\pred^{\Xx}(\gamma_{\zeta+1})$ and
$P_{\beta\mu}=M^{*\Xx}_{\gamma_{\zeta+1}}$.)

One proves strong closeness at $\zeta$ and translatability at $\eps$, by
simultaneous
induction on $\max(\zeta+1,\eps)$. This is basically as in
\cite{fs_tame}, so the reader should refer there for the full argument,
but there are a few extra details which arise here, which we explain.
Fix $\zeta$ and consider the proof of strong closeness at $\zeta+1$.
Let $\beta=\pred^{\Tt}(\zeta+1)$  and suppose $\beta<\zeta$. Let
$E=E^{\Tt}_\zeta$ and $\kappa=\crit(E)$ and $F=E^{\Xx}_{\delta_\zeta}$. Suppose
$(\kappa^+)^{\exit^{\Tt}_\beta}<\OR^{\exit^{\Tt}_\beta}$ but
$E_a\notin\exit^{\Tt}_\beta$ for some $a\in[\nu_E]^{<\om}$. Then as in
\cite[6.1.5]{fsit},
$E=F(M^{\Tt}_\zeta)$ and $\beta<^{\Tt}\zeta$ and letting
$\xi+1=\successor^\Tt(\beta,\zeta)$,
we have
\begin{enumerate}[label=--]
 \item  $(\xi+1,\zeta]_{\Tt}$ does not drop,
$\udeg^{\Tt}(\zeta)=\udeg^\Tt(\xi+1)=0$ and
 \item $\kappa<\mu\eqdef\crit(i^{*\Tt}_{\xi+1,\zeta})$.
\end{enumerate} Now $j\eqdef i^{\Xx}_{\gamma_\zeta\delta_\zeta}$
exists because $E=F(M^{\Tt}_\zeta)$, and note
$\crit(F)<\crit(E^{\Xx}_{\delta_\xi})<\crit(j)$
(where $\crit(j)=\infty$ if $j=\id$).
So for $a\in[\nu_E]^{<\om}$, letting $\Fbar=F(M^{\Xx}_{\gamma_\zeta})$, we have
\[ F_{\omega_\zeta(a)}=F_{j(\pi_\zeta(a))}=\Fbar_{\pi_{\zeta}(a)}.\]
So using translatability at $\zeta$ as in \cite{fs_tame}, we get $(\varphi,q)$
such that $\varphi$ is
$\uSigma_1$ and $q\in M^{*\Tt}_{\xi+1}=M^\Tt_{\beta\mu}$ and
\[ \varphi(q,\cdot)\text{ defines }
E_a\text{ over }M^{*\Tt}_{\xi+1}=M^\Tt_{\beta\mu}, \]
\[ \varphi(\pi_{\beta\mu}(q),\cdot)\text{ defines }F_{\omega_\zeta(a)}\text{
over }M^{*\Xx}_{\gamma_{\xi+1}}=P_{\beta\mu}.\]
So if $\gamma_{\beta\kappa}=\gamma_{\beta\mu}$ then we get strong closeness at
$\zeta+1$ as in \cite{fs_tame}.
Suppose instead that $\gamma_{\beta\kappa}<\gamma_{\beta\mu}$.
Let $k=j^\Xx_{\gamma_{\beta\kappa}\gamma_{\beta\mu}}$, so
$\pi_{\beta\kappa}(M^\Tt_{\beta\mu})\ins\dom(k)$,
\[ \pi_{\beta\mu}=k\com\pi_{\beta\kappa}\rest M^\Tt_{\beta\mu}, \]
\[ \crit(k)>\pi_{\beta\kappa}(\kappa)=\om_\beta(\kappa)=\crit(F).\]
So if $M^{\Tt}_{\beta\mu}=M^\Tt_{\beta\kappa}$ then
\[ \varphi(\pi_{\beta\kappa}(q),\cdot)\text{ defines }F_{\omega_\zeta(a)}\text{
over }M^{*\Xx}_{\gamma_{\zeta+1}}=P_{\beta\kappa},\]
as required. And if $M^\Tt_{\beta\mu}\pins M^\Tt_{\beta\kappa}$
then we get a natural $\uSigma_1$ formula $\varphi''$ such that
\[ \varphi''((q,M^\Tt_{\beta\mu}),\cdot)\text{
defines }E_a\text{ over }M^\Tt_{\beta\kappa},\]
\[ \varphi''(\pi_{\beta\kappa}(q,M^\Tt_{\beta\mu}),\cdot)\text{ defines
}F_{\om_\zeta(a)}\text{ over }P_{\beta\kappa},\]
again as required.

The second detail is as follows. Consider again strong closeness at $\zeta+1$.
Let $\beta,\kappa,E,F$ be as before and suppose   $\beta<\zeta$,
but now with
$(\kappa^{++})^{\exit^{\Tt}_\beta}=\OR^{\exit^{\Tt}_\beta}$,
and $E_a\in\exit^{\Tt}_\beta$ for every $a\in[\nu_E]^{<\om}$. Then
$\gamma_{\beta\kappa}=\delta_\beta$ and
$\iota(\exit^\Tt_\beta)=\OR^{\exit^\Tt_\beta}$, so
\[
\pi_{\beta\kappa}
\rest\iota(\exit^\Tt_\beta)=\omega_\beta\rest\OR\sub\omega_\zeta, \]
which implies that $\pi_{\beta\kappa}(E_a)=F_{\omega_\zeta(a)}$ for each $a$.
This easily gives strong closeness in this case.

There are also similar considerations in other cases of strong closeness.

The proof of translatability at a successor $\eps=\xi+1$ also involves an extra
detail,
with respect to $\zeta+1<^\Tt\xi+1$. Let $\delta=\pred^\Tt(\xi+1)$, so
$\zeta+1\leq^\Tt\delta$
and $(\zeta+1,\delta]_\Tt$ does not drop in model or degree.
Let $\kappa=\crit(E^\Tt_\zeta)$ and $\mu=\crit(E^\Tt_\xi)$, so $\kappa<\mu$.
Since we have translatability at $\delta$,
it suffices to see that
for each $(\varphi,q)$ there is $(\varphi',q')$ such that for all
$\alpha<\kappa$,
\[ M^\Tt_{\xi+1}\sats\varphi(q,\alpha)\iff
M^\Tt_\delta\sats\varphi'(q',\alpha), \]
and all $\alpha<\crit(E^\Xx_{\delta_\zeta})=\om_\zeta(\kappa)$,
\[
M^\Xx_{\gamma_{\xi+1}}\sats\varphi(\pi_{\xi+1}(q),\alpha)\iff M^\Xx_{
\gamma_\delta}\sats\varphi'(\pi_\delta(q'),\alpha). \]
Now
$\gamma_\delta\leq^\Xx\gamma_{\delta\mu}=\pred^\Xx(\gamma_{\xi+1})\in I_\delta$
and $(\gamma_\delta,\gamma_{\xi+1}]_\Xx$ does not drop in model or degree as
$(\delta,\xi+1]_\Tt$ does not.
Fix $(\varphi,q)$.
As usual, using strong closeness at $\xi$,
we can choose $(\varphi',q')$ such that
for all $\alpha<\mu$,
\[
M^\Tt_{\xi+1}\sats\varphi(q,\alpha)\iff M^\Tt_\delta\sats\varphi'(q',\alpha),
\]
and all $\alpha<\crit(E^\Xx_{\delta_\xi})=\om_\xi(\kappa)$,
\[
M^\Xx_{\gamma_{\xi+1}}\sats\varphi(\pi_{\xi+1}(q),\alpha)\iff M^{*\Xx}_{\gamma_
{\xi+1}}=M^\Xx_{\gamma_{\delta\mu}}\sats\varphi'(\pi_{ \delta\mu}(q'),\alpha).
\]
But $\pi_{\delta\mu}=i^\Xx_{\gamma_\delta\gamma_{\delta\mu}}\com\pi_\delta$
and $\crit(E^\Xx_{\delta_\zeta})<\crit(E^\Xx_{\delta_\xi})$, so
for all $\alpha<\crit(E^\Xx_{\delta_\zeta})$, we have
\[ M^\Xx_{\gamma_{\delta\mu}}\sats\varphi'(\pi_{\delta\mu}(q'),\alpha)\iff
M^\Xx_{\gamma_\delta}\sats\varphi'(\pi_\delta(q'),\alpha), \]
so $(\varphi',q')$ is as desired.

We leave the remaining details to the reader.
\end{proof}

\begin{dfn}\label{dfn:tree_embedding_extensions}\index{one-step copy extension}
 Let $\Pi':(\Tt,\theta)\tembto\Xx'$ and
$(\Xx,\Pi)$
be as in \ref{lem:tree_embedding_copy} (so $\theta<\lh(\Tt)$). Then we say that
$(\Xx,\Pi)$ is the
\dfnemph{one-step copy extension} of $(\Xx',\Pi')$.
\end{dfn}

The second manner of propagating tree embeddings involves the use of an
extender
in the upper
tree $\Xx'$ which is \emph{not} (considered as) copied from $\Tt$. We will call
such extenders
\emph{$\Tt$-inflationary}.\index{inflationary}\index{$\Tt$-inflationary} In
this case we can just give the
definition directly, as it is clear that it works.

\begin{dfn}\label{dfn:inflationary_extender}
Let $\Pi:(\Tt,\theta)\tembto\Xx$ be bounding and $k=\udeg(\Pi)$.
Let $\gamma_\alpha=\gamma_{\Pi\alpha}$, etc.
Suppose $\lh(\Xx)=\xi+1$. Let
$E\in\es_+(M^\Xx_\xi)$ be $\Xx$-normal.
Suppose that the putative $\udash
k$-maximal tree $\Xx'=\Xx\conc\left<E\right>$ has wellfounded last model, and
let
$\eta=\pred^{\Xx'}(\xi+1)$.
Suppose that $\eta\in I_\beta$ and
if $\beta\in\lh(\Tt)^-$ then $E$ is total over
$M^\Xx_\eta|\In(E^\Pi_\eta)$,
and otherwise $E$ is total over $M^\Xx_\eta$.

The \dfnemph{$E$-inflation of
$(\Xx,\Pi)$}\index{$E$-inflation}\index{inflation} is
$(\Xx',\Pi')$, where
$\Pi':(\Tt,\beta+1)\tembto\Xx'$
is the unique tree embedding such that
$I_{\Pi'\beta}=(I_\beta\inter\eta+1)\un\{\xi+1\}$ and $I_{\Pi'\alpha}=I_\alpha$
for every
$\alpha<\beta$.
\end{dfn}

\begin{rem}\label{rem:tree_embedding_existence}
The uniqueness of the $E$-inflation is by \ref{cor:tree_embedding_uniqueness},
and
existence is easy. We have
$P'_{\alpha i}=P_{\alpha i}$ and $\pi'_{\alpha i}=\pi_{\alpha i}$ for
$(\alpha,i)\leq_\lex(\beta,0)$.
Because $\Pi$ is bounding,
$E^{\Pi'}_{\xi+1}$ is $\Xx'$-normal, and if
$\In(E)\leq\In(E^\Pi_\xi)$ or $\eta<\xi$
then $\Pi'$ is also bounding.
\end{rem}

\begin{dfn}\index{almost tree embedding}\index{$\Pi:\Tt\tembto_{\almost}\Xx$}
Let $M,\Tt,\Xx,k$ be as in
\ref{dfn:tree_embedding}.
An \dfnemph{almost tree embedding} $\Pi$ from $\Tt$ to $\Xx$, denoted
$\Pi:\Tt\hookrightarrow_{\almost}\Xx$,
is a system $\Pi$ satisfying the requirements of a tree embedding,
except that (letting $\Gamma$ be as in \ref{dfn:tree_embedding})
we drop the requirement that $\Gamma$ be continuous at limits
(but $\Gamma$ must still send limits to limits etc).
\end{dfn}

\begin{rem}
 Note here that if $\pred^\Tt(\beta+1)=\alpha$ and $\alpha$ is a limit,
 then $\pred^\Xx(\gamma_{\beta+1})\in I_{\Pi\alpha}$, and in particular,
 $\pred^\Xx(\gamma_{\beta+1})\geq\gamma_\alpha$, by the requirements of tree
embeddings;
 this remains a requirement of almost tree embeddings, even when $\Gamma$ is
discontinuous at $\alpha$.

Note that given a tree $\Xx$, the requirements of almost tree embeddings from
countable
$\Tt\hookrightarrow_{\almost}\Xx$ are closed in the natural topology,
so we can form a tree (in the descriptive set theoretic sense)
which searches for a countable $\Tt$ and almost tree embedding
$\Pi:\Tt\hookrightarrow_{\almost}\Xx$.
\end{rem}

\begin{lem}\label{lem:ate_to_te}
Let $\Pi:\Tt\hookrightarrow_{\almost}\Xx$ be an almost tree embedding. Write
$\gamma_\alpha=\gamma_{\Pi\alpha}$ etc.
Then there is a unique tree embedding $\Pi':\Tt\hookrightarrow\Xx$ such that
$\delta'_\alpha=\delta_\alpha$ for all $\alpha$,
where $\delta'_\alpha=\delta_{\Pi'\alpha}$, etc.
Hence, for limit $\alpha$,
\[
\gamma'_\alpha=\sup_{\beta<\alpha}\gamma_\beta=\sup_{\beta<\alpha}\delta'_\beta,
\]
whereas for successor $\alpha$,
$\gamma'_\alpha=\gamma_\alpha$ \tu{(}and $\gamma'_0=0=\gamma_0$\tu{)}. Moreover,
for each $\alpha$, we have $\omega_\alpha=\omega'_\alpha$
and
$\pi_\alpha=i^\Xx_{\gamma'_\alpha\gamma_\alpha}\com\pi'_\alpha$.
\end{lem}
\begin{proof}[Proof Sketch]
This is straightforward; we just mention the key facts.
Uniqueness is by \ref{cor:tree_embedding_uniqueness}.
 Fix a limit $\alpha<\lh(\Tt)$. The main point is that
 \[ B_\alpha\eqdef\{\gamma_\beta\bigm|
\beta<^\Tt\alpha\}\sub[0,\gamma_\alpha)_\Xx, \]
so
$\gamma'_\alpha=\sup(B_\alpha)\leq^\Xx\gamma_\alpha$.
Moreover, for  each $\beta<^\Tt\alpha$,
\[
(\beta,\alpha]_\Tt\inter\dropset^\Tt=\emptyset\iff(\gamma_\beta,\gamma_\alpha]
_\Xx\inter\dropset^\Xx=\emptyset; \]
therefore, $(\gamma'_\alpha,\gamma_\alpha]_\Xx\inter\dropset^\Xx=\emptyset$.
Because of commutativity requirements of (almost) tree embeddings, we have
\[ \pi_\alpha\com
i^\Tt_{\beta\alpha}=i^\Xx_{\gamma_\beta\gamma_\alpha}\com\pi_\beta \]
for sufficiently large $\beta<^\Tt\alpha$. Likewise with
$\pi'_\alpha,\gamma'_\alpha$ replacing $\pi_\alpha,\gamma_\alpha$.
It follows that
$\pi_\alpha=i^\Xx_{\gamma'_\alpha\gamma_\alpha}\com\pi'_\alpha$
(and note $\sigma'_\alpha=\sigma_\alpha$).
As remarked above, if $\pred^\Tt(\beta+1)=\alpha$ and
$\xi=\pred^\Xx(\gamma_{\beta+1})$ then $\xi\in I_\alpha$,
hence, $\gamma_\alpha\leq^\Xx\xi$, and $\delta'_\beta=\delta_\beta$, so
\[
\delta(\Tt\rest\gamma_\alpha)\leq\crit(E^\Xx_{\delta_\beta}
)=\sigma_\beta(\crit(E^\Tt_\beta))=\sigma'_\beta(\crit(E^\Tt_\beta)), \]
so everything agrees appropriately in producing $M^\Tt_{\beta+1}$ and
$M^\Xx_{\gamma'_{\beta+1}}$,
with regard to the tree embedding $\Pi'$.
\end{proof}

\subsection{Inflation}

We now proceed to the definition of an \emph{inflation} of a normal iteration
tree $\Tt$.
This will be a normal tree $\Xx$ which can be interpreted as being produced by
using extenders
which are either (i) copied from $\Tt$, or (ii) $\Tt$-inflationary. Certain
nodes $\alpha<\lh(\Xx)$
will correspond to nodes
$f(\alpha)<\lh(\Tt)$, in that there will be a natural tree embedding \[
\Pi_\alpha:(\Tt,f(\alpha)+1)\tembto\Xx\rest(\alpha+1),\]
with $\delta_{\Pi_\alpha f(\alpha)}=\alpha$. The set of all such
$\alpha$ will be denoted by $C$. For successor $\alpha\in C$, $\Pi_\alpha$ will
be
produced through one of the two methods we have just described.
We will take natural direct limits at limit ordinals $\alpha$.
The set $C^-$ will consist of those $\alpha\in C$ such that
$ f(\alpha)+1<\lh(\Tt)$,
and note that at such $\alpha$, we have $Q_{\Pi_\alpha f(\alpha)}\ins
M^\Xx_\alpha$,
and the active extender of $Q_{\Pi_\alpha f(\alpha)}$ is a copy of
$E^\Tt_{ f(\alpha)}$.

\begin{dfn}[Inflation]\label{dfn:inflation}\index{inflation}
Let either (i) $M$ be a $\udash k$-sound seg-pm and $\Tt,\Xx$ be $\udash
k$-maximal trees on
$M$, or (ii) $M$ be a wcpm and $\Tt,\Xx$ be normal trees on $M$.
We say that $\Xx$ is an \dfnemph{inflation} of $\Tt$ iff there is a
tuple\index{$t$}\index{$C,C^-$}\index{$f$}\index{$\Pi_\alpha$}
\[ \left(t,C,C^-,f,\left<\Pi_\alpha\right>_{\alpha\in C}\right) \]
with the following properties (which will unique the tuple); we will also
define
further
notation:
\begin{enumerate}[label=\arabic*.,ref=\arabic*]
\item
\index{type}\index{$\Tt$-copying}\index{copying}\index{$\Tt$-inflationary}
\index{inflationary}We have $t:\lh(\Xx)^-\to\{0,1\}$. The value of $t(\alpha)$
indicates the
\dfnemph{type} of
$E^\Xx_\alpha$, either \dfnemph{$\Tt$-copying} (if $t(\alpha)=0$) or
\dfnemph{$\Tt$-inflationary} (if
$t(\alpha)=1$).
\item\label{item:C_nature} $C\sub\lh(\Xx)$\footnote{If $M$ is a wcpm, it will
follow from the overall definition that $C=\lh(\Xx)$,
and the conditions regarding $C$ will be trivial (but $C^-$ is still
important).}
and $C\inter[0,\alpha]_\Xx$ is a closed\footnote{One could drop the closure
requirement here,
demanding only that $C\inter[0,\alpha]_\Xx$ is an initial segment of $\Xx$,
and adding to condition \ref{item:inflation_limit_ordinal} the requirement that
for limit $\alpha$, $\alpha\in C$ iff
$(\sup_{\beta<^\Xx\alpha}f(\beta))<\lh(\Tt)$.
Then if a limit $\alpha$ were such that $\alpha\notin C$ but
$[0,\alpha)_\Xx\sub
C$,
then $[0,\alpha)_\Xx$ would determine a $\Tt$-cofinal branch $b$. By demanding
that $C\inter[0,\alpha]_\Xx$ be closed,
we are demanding that such branches $b$ are already incorporated into $\Tt$.}
initial segment of $[0,\alpha]_\Xx$.
\item\label{item:def_C^-} We have $f:C\to\lh(\Tt)$ and
$C^-=\{\alpha\in C\bigm|   f(\alpha)+1<\lh(\Tt)\}$.
\item For $\alpha\in C$ we have
$\Pi_\alpha:(\Tt, f(\alpha)+1)\tembto\Xx\rest(\alpha+1)$, with
$\delta_{\alpha; f(\alpha)}=\alpha$, where
we write $\delta_{\alpha;\beta}=\delta_{\Pi_\alpha\beta}$, etc.
\index{$\gamma_{\alpha;\beta}$, $\delta_{\alpha;\beta}$, etc}
\item $0\in C$ and $f(0)=0$ and
$\Pi_0:(\Tt,1)\tembto\Xx\rest 1$
is trivial (see \ref{dfn:trivial_tree_embedding}).
\item Let $\alpha+1<\lh(\Xx)$. Then:
\begin{enumerate}[label=--]
 \item If $\alpha\in C^-$ then
$\In(E^\Xx_\alpha)\leq\In(E^{\Pi_\alpha}_\alpha)$.\footnote{This condition could
be dropped, but in our applications it will hold, and it simplifies some things.
It ensures that each $\Pi_\beta$ is bounding.}
\item $t(\alpha)=0$ iff [$\alpha\in C^-$ and
$E^\Xx_\alpha=E^{\Pi_\alpha}_\alpha$].\footnote{If we had required that $t$ be
given from the
outset (calling the pair $(\Xx,t)$ an inflation), then this condition could
also
be weakened
to say that if $t(\alpha)=0$ then $\alpha\in C^-$ and
$E^\Xx_\alpha=E^{\Pi_\alpha}_\alpha$.
But having the stronger condition also simplifies things and ensures the
canonicity of inflations.}
\end{enumerate}
\item Let $\alpha+1<\lh(\Xx)$ be such that $t(\alpha)=0$. Then we interpret
$E^\Xx_\alpha=E^{\Pi_\alpha}_\alpha$ as a copy from $\Tt$, as follows:
\begin{enumerate}[label=--]
 \item  $\alpha+1\in C$ and $f(\alpha+1)= f(\alpha)+1$.
 \item $(\Xx\rest\alpha+2,\Pi_{\alpha+1})$
is the one-step copy extension of
$(\Xx\rest\alpha+1,\Pi_\alpha)$.
\end{enumerate}
\item Let $\alpha+1<\lh(\Xx)$ be such that $t(\alpha)=1$.
We interpret $E^\Xx_\alpha$ as $\Tt$-inflationary, as follows.
Let $\eta=\pred^\Xx(\alpha+1)$. Then:
\begin{enumerate}[label=--]
\item $\alpha+1\in C$ iff [$\eta\in C$ and
if $M$ is a seg-pm then $Q_{\eta;f(\eta)}\ins M^{*\Xx}_{\alpha+1}$].
\item If $\alpha+1\in C$ then:
\begin{enumerate}[label=--]
\item $f(\alpha+1)=f(\eta)$.
\item $(\Xx\rest\alpha+2,\Pi_{\alpha+1})$
is the $E^\Xx_\alpha$-inflation of
$(\Xx\rest\alpha+1,\Pi_\eta)$.
\end{enumerate}
\end{enumerate}
\item\label{item:inflation_internal_agreement} Let $\alpha\in C$ and $\beta\in
I_{\alpha;\gamma}$ for some $\gamma\leq f(\alpha)$. Then:
\begin{enumerate}[label=--]
\item $\beta\in C$ and $f(\beta)=\gamma$.
\item  $I_{\alpha;\eps}=I_{\beta;\eps}$ for all $\eps<f(\beta)=\gamma$,
\item $I_{\beta;f(\beta)}=I_{\alpha;f(\beta)}\inter(\beta+1)$.
\end{enumerate}
\item\label{item:inflation_limit_ordinal} If $\alpha\in C$ is a
limit\footnote{Note that by condition \ref{item:C_nature},
if $\alpha$ is a limit then $\alpha\in C$ iff $[0,\alpha)_\Xx\sub C$.} then
$ f(\alpha)=\sup_{\beta<^\Xx\alpha}f(\beta)$.\qedhere
\end{enumerate}
\end{dfn}

\begin{rem}
We make some remarks regarding this definition (literally in the context of
seg-pms),
continuing with notation as above.

Note first that $\Pi_\alpha$
is bounding for each $\alpha\in C$.

Adopt the hypotheses and notation of condition
\ref{item:inflation_internal_agreement} (so
$\eps<f(\beta)=\gamma$).
Note
\[ I_{\alpha;\eps i}=I_{\beta;\eps i}\text{ and }P_{\alpha;\eps
i}=P_{\beta;\eps
i}\text{ and }\pi_{\alpha;\eps i}=\pi_{\beta;\eps i}\text{ for all }i,\]
and also $P_{\alpha;f(\beta) 0}=P_{\beta;f(\beta) 0}$ and
$\pi_{\alpha;f(\beta)}=\pi_{\beta;f(\beta)}$.
And by \ref{lem:intervals_I_cover_X-branches},
if $\widetilde{\beta}\leq^\Xx\alpha$
then $\widetilde{\beta}\in I_{\alpha;\widetilde{\gamma}}$
for some $\widetilde{\gamma}\leq f(\alpha)$, so condition
\ref{item:inflation_internal_agreement}
applies to $\widetilde{\beta},\widetilde{\gamma}$,
and therefore $f(\widetilde{\beta})=\widetilde{\gamma}\leq^\Tt f(\alpha)$.

We point out some facts regarding limit stages.
Let $\alpha\in C$ be a limit such that
$f(\alpha)>f(\beta)$ for all $\beta<^\Xx\alpha$.
Note that by condition \ref{item:inflation_internal_agreement} and the remarks
above, for $\xi<f(\alpha)$, we have
$I_{\alpha;\xi}=\lim_{\beta<^\Xx\alpha}I_{\beta;\xi}$
(where this limit exists in the eventually constant sense)
 and so likewise for $I_{\alpha;\xi
i}$, $P_{\alpha;\xi i}$ and $\pi_{\alpha;\xi i}$.
So
\[
\alpha=\left(\lim_{\xi<f(\alpha)}\gamma_{\alpha;\xi}\right)=\gamma_{
\alpha;f(\alpha)}=
\delta_{\alpha;f(\alpha)}, \]
so $I_{\alpha;f(\alpha)}=[\alpha,\alpha]$, determining
$\pi_{\alpha;f(\alpha)}$,
etc.

Now let $\alpha\in C$ be a limit such that
$f(\alpha)=f(\beta)$ for some $\beta<^\Xx\alpha$.
For such $\beta$ we have $\gamma_{\alpha; f(\alpha)}=\gamma_{\beta; f(\alpha)}$.
We also have $\delta_{\alpha; f(\alpha)}=\alpha$. This determines
the remaining objects ($I_{\alpha; f(\alpha) i}$, etc); they are just the
natural direct limits.
\end{rem}

Using \ref{cor:tree_embedding_uniqueness}, it is easily verified that $\Tt,\Xx$
determines
$(t,C,C^-,f,\Pivec)$:

\begin{lem}\label{lem:inflation_unique_objects}
Let $\Xx$ be an inflation of $\Tt$,
witnessed by
  $w=(t,C,C^-,f,\Pivec)$, and also by $w'=(t',C',(C^-)',f',\Pivec')$.
  Then $w=w'$.
\end{lem}

\begin{dfn}\index{$\Tt\inflatearrow\Xx$}\index{$f^{\Tt\inflatearrow\Xx}$,
$C^{\Tt\inflatearrow\Xx}$, etc}
\index{$E^{\Tt\inflatearrow\Xx }_\alpha$}
 Let $\Xx$ be an inflation of $\Tt$ as witnessed by
$(t,C,C^-,f,\Pivec)$. Then we write
$(t,C,C^-,f,\Pivec)^{\Tt\inflatearrow\Xx}=(t,C,C^-,f,\Pivec)$. For $\alpha\in
C^-$ we write
$E^{\Tt\inflatearrow\Xx}_\alpha\eqdef E^{\Pi_\alpha}_\alpha$.
\end{dfn}

We may freely extend inflations at successor stages, given wellfoundedness:

\begin{lem}
Let $\Xx$ be an inflation of $\Tt$, with $\lh(\Xx)=\beta+1$. Let
$C^-=(C^-)^{\Tt\inflatearrow\Xx}$. Then:
\begin{enumerate}[label=\arabic*.,ref=\arabic*]
\item\label{item:copy_is_Xx-normal} If $\beta\in C^-$ then
$E^{\Tt\inflatearrow\Xx}_\beta$ is $\Xx$-normal.

\item\label{item:extend_inflation} Let $E\in\es_+(M^\Xx_\beta)$ be
$\Xx$-normal,
with $\In(E)\leq\In(E^{\Tt\inflatearrow\Xx}_\beta)$ if $\beta\in C^-$.
Let $\Xx'$ be the putative tree $\Xx\conc\left<E\right>$, and suppose that
$\Xx'$ has wellfounded last model.
Then $\Xx'$ is an inflation of $\Tt$.
\end{enumerate}
\end{lem}
\begin{proof}
 Part \ref{item:copy_is_Xx-normal} follows from
\ref{lem:Pi-bounded_reasonable_copy},
 and part \ref{item:extend_inflation} from \ref{lem:tree_embedding_copy}
 (see \ref{dfn:tree_embedding_extensions} and \ref{dfn:inflationary_extender}).
\end{proof}

However, at limit stages,
we need to assume some condensation holds of $\Sigma$, in order to extend.
This is critical to our purposes, and we consider it next.

\subsection{Inflation condensation and strong hull condensation}

\begin{rem}\label{rem:condensation}
 Suppose that $\Xx$, of limit length $\alpha$,
 is an inflation of $\Tt$,
 as witnessed by $(C,f,\ldots)$.
Let $b$ be a wellfounded $\Xx$-cofinal branch, and $\Xx'=\Xx\conc b$.
We want to see whether $\Xx'$ is an inflation of $\Tt$.
Let $(C',f')$ be the unique candidate for $(C,f)^{\Tt\inflatearrow\Xx'}$
determined by \ref{dfn:inflation}.
Suppose that $\alpha\in C'$ and $f'(\beta)<f'(\alpha)$
for all $\beta<^{\Xx'}\alpha$
(this is the important case); in particular,
$f'(\alpha)$ is a limit.
Note that $b$ determines a $\Tt\rest f'(\alpha)$-cofinal branch $c=f``b$,
and $\Xx'$ is an inflation of $\Tt$ iff $c=[0,f'(\alpha))_\Tt$.

We first give the definition of inflation condensation for the case that
$\Sigma$ is an $(m,\Omega+1)$-strategy
for an $m$-sound $\lambda$-indexed premouse $M$, where $\Omega$ is an
uncountable regular cardinal.
After that, we define in \ref{dfn:strategy_classes} some general
kinds of (partial) iteration strategies $\Sigma$
we wish to consider, and then give the general definition of inflation
condensation
for such strategies.
\end{rem}

\begin{dfn}\index{inflation
condensation}\index{inflationary}\label{dfn:inf_con_lambda-indexed}
Let $\Omega>\om$ be regular.
Let $\Sigma$ be an $(m,\Omega+1)$-strategy for an $m$-sound $\lambda$-indexed
pm
$M$.
Then $\Sigma$ has \dfnemph{inflation condensation} or is \dfnemph{inflationary}
iff for all trees $\Tt,\Xx$, if
\begin{enumerate}[label=--]
 \item $\Tt,\Xx$ are via $\Sigma$,
 \item $\Xx$ is an inflation of $\Tt$, as witnessed by $(f,C,\ldots)$,
 \item $\Xx$ has limit length $\leq\Omega$,
 \item $b\eqdef\Sigma(\Xx)\sub C$ and $f``b$ has limit ordertype,
\end{enumerate}
then letting  $\eta=\sup f``b$, we have
$f``b=\Sigma(\Tt\rest\eta)$.
\end{dfn}

For strategies for wcpms, inflation condensation is totally analogous.
For strategies $\Lambda$ for MS-indexed mice, we must instead translate
$\Lambda$
to the corresponding $\udash$strategy $\Sigma$.
We would also like to consider partial strategies (such as a short tree
strategy).
So we next give an abstract definition of the the kinds of (partial) strategies
we wish to consider for inflation condensation
in general.

\begin{dfn}\label{dfn:strategy_classes}\index{partial strategy}\index{iteration
class}\index{putative}\index{$\mathscr{T}$-strategy}\index{via $\Sigma$}
 Let $M$ be a premouse or wcpm. An \dfnemph{iteration class \tu{(}for
$M$\tu{)}}
is a class $\mathscr{T}$ of putative trees on $M$,
 which is closed under initial segment. Let $\mathscr{T}$ be an iteration class
for $M$.
 A \dfnemph{putative partial $\mathscr{T}$-strategy \tu{(}for $M$\tu{)}} is a
class function $\Sigma$
 with $D\eqdef\dom(\Sigma)$ such that $D\sub\mathscr{T}$, and for each $\Tt\in
D$,
 $\Tt$ has limit length,
 $\Sigma(\Tt)$ is a $\Tt$-cofinal branch and
$\Tt\conc\Sigma(\Tt)\in\mathscr{T}$.
 Let $\Sigma$ be a  putative partial $\mathscr{T}$-strategy.
 For $\Tt\in\mathscr{T}$, we say that $\Tt$ is \dfnemph{via $\Sigma$}
 iff $\Tt\rest\eta\in\dom(\Sigma)$ and $[0,\eta)_\Tt=\Sigma(\Tt\rest\eta)$ for
every limit $\eta<\lh(\Tt)$.
 We say that $\Sigma$ is a \dfnemph{partial $\mathscr{T}$-strategy}
 iff every $\Tt\in\mathscr{T}$ via $\Sigma$ has wellfounded well-defined
 models.

 Given $M,\Sigma$, we say that $\Sigma$ is a \dfnemph{\tu{(}putative\tu{)}
partial pre-inflationary strategy \tu{(}for $M$\tu{)}}\index{pre-inflationary}
 iff $\Sigma$ is a (putative) partial $\mathscr{T}$-strategy (for $M$), where
for some $m\leq\om$,
 either
 \begin{enumerate}[label=\tu{(}\roman*\tu{)}]
  \item\label{item:pre-inf_wcpm}  $M$ is a wcpm, $m=0$ and $\mathscr{T}$ is the
class of putative normal trees on $M$, or
 \item\label{item:pre-inf_rm-sound}  $M$ is a $\udash m$-sound seg-pm and
$\mathscr{T}$ is the class of putative $\udash m$-maximal trees on $M$, or
 \item\label{item:pre-inf_m-sound_MS} $M$ is an $m$-sound MS-indexed pm
 and $\mathscr{T}$ is the class of putative $m$-maximal trees on $M$.
 \end{enumerate}
\index{convenient}\index{inconvenient}We say that $\Sigma$ is
\dfnemph{conveniently pre-inflationary} iff either
\ref{item:pre-inf_wcpm} or \ref{item:pre-inf_rm-sound} above hold,
 and \dfnemph{inconveniently pre-inflationary} iff
\ref{item:pre-inf_m-sound_MS}
holds.\footnote{\label{ftn:Sigma_partial_pre-inf}Let $\Sigma$ be
partial pre-inflationary.
Note that there can be limit length trees $\Tt\in\mathscr{T}$ which are
via $\Sigma$, but with $\Tt\notin D$.
And if  \ref{item:pre-inf_rm-sound}
or \ref{item:pre-inf_m-sound_MS} holds and
 $\Tt\in\mathscr{T}$ is via $\Sigma$ of
successor length, then
 $M^\Tt_\infty$
is well-defined and wellfounded, and all
$\udash m$-maximal/$m$-maximal putative trees $\Tt'$
such that $\Tt\ins\Tt'$ and $\lh(\Tt')<\lh(\Tt)+\om$,
are also  in $\mathscr{T}$ and via $\Sigma$, and hence
have wellfounded models. Therefore if \ref{item:pre-inf_rm-sound}
holds where $M$ is MS-indexed then $\Tt$ is everywhere unravelable, and if
\ref{item:pre-inf_m-sound_MS}
holds then $\Tt$ is $M$-$\udash$wellfounded.}

 Let $\Sigma$ be pre-inflationary and $\mathscr{T},M,m$ be as in the preceding
paragraph.
 We say that $\Sigma$ is \index{regularly $\Xi$-total}\dfnemph{regularly
$\Xi$-total} iff there is an
regular
uncountable $\Omega$
 such that $\Omega\leq\Xi\leq\Omega+1$ and $\Sigma$ is either a normal
$\Xi$-strategy for a wcpm (and $m=0$),
 a $(\udash m,\Xi)$-strategy, or an $(m,\Xi)$-strategy for an MS-indexed $M$.
 In this case we write $m^\Sigma=m$.
\end{dfn}

We can now give the general definition of inflation condensation.
\begin{dfn}[Inflation
condensation]\label{dfn:inflation_condensation}\index{inflation
condensation}\index{convenient}\index{inconvenient}\index{inflationary}
Let $\Sigma$ be a conveniently pre-inflationary partial strategy.
Then $\Sigma$ has \dfnemph{convenient inflation condensation} or is
\dfnemph{conveniently inflationary} iff for all trees $\Tt,\Xx$, if
\begin{enumerate}[label=--]
 \item $\Tt,\Xx$ are via $\Sigma$,
 \item $\Xx$ is an inflation of $\Tt$, as witnessed by $(f,C,\ldots)$,
 \item $\Xx$ has limit length and $\Xx\in\dom(\Sigma)$,
 \item $b\eqdef\Sigma(\Xx)\sub C$ and $f``b$ has limit ordertype,
\end{enumerate}
then letting  $\eta=\sup f``b$, we have $\Tt\rest\eta\in\dom(\Sigma)$ and
$f``b=\Sigma(\Tt\rest\eta)$.

Let $\Lambda$ be an inconveniently pre-inflationary partial strategy.
Let $\Sigma$ be the partial $\udash$strategy corresponding to
$\Lambda$ (as in Remark \ref{rem:partial_strategy_conversion},
which applies by Footnote \ref{ftn:Sigma_partial_pre-inf}).
Then
$\Lambda$ has \dfnemph{inconvenient inflation condensation} or is
\dfnemph{inconveniently inflationary} iff
$\Sigma$ is conveniently inflationary.

In general, we say that $\Sigma$ (or $\Lambda$) \dfnemph{has inflation
condensation} or \dfnemph{is inflationary}
iff $\Sigma$ (or $\Lambda$) has convenient or inconvenient inflation
condensation.
\end{dfn}

Immediately from the definition, inflations via inflationary $\Sigma$ can be
continued at limit stages:

\begin{lem}\label{lem:inflationary_limit_continuation}
Let $\Sigma$ be a conveniently inflationary partial strategy.
Let $\Tt,\Xx$ be such that $\Xx$ is via $\Sigma$, $\Xx$ is an inflation of
$\Tt$,
as witnessed by $(f,C,\ldots)$, and
\[ \lh(\Tt)=\sup_{\alpha\in C}(f(\alpha)+1).\]
Then $\Tt$ is via $\Sigma$.

Suppose also that $\Xx$ has limit length $\lambda$ and $\Xx\in\dom(\Sigma)$,
and let $\Xx'=(\Xx,\Sigma(\Xx))$.
Then there is $\Tt'$ via $\Sigma$ such that $\Tt\ins\Tt'$ and $\Xx'$ is an
inflation of $\Tt'$, as witnessed by $(C',f',\ldots)$.
 Moreover,  we may take $\Tt'$ such that either:
 \begin{enumerate}[label=--]
 \item $\Tt'=\Tt$ and if $\lambda\in C'$ then $f'(\lambda)<\lh(\Tt)$, or
 \item $\Tt$ has limit length $\bar{\lambda}$,
  $\Tt'=(\Tt,\Sigma(\Tt))$, $\lambda\in C'$, $f'(\lambda)=\bar{\lambda}$
 and $\gamma'_{\lambda;\bar{\lambda}}=\lambda$.
 \end{enumerate}
 Further, the choice of $\Tt'$ is uniqued by adding these requirements.
\end{lem}

We also immediately have:
\begin{lem}\label{lem:characterize_inflation}
Let $\Sigma$ be a conveniently inflationary partial strategy and $\Tt,\Xx$ be
via $\Sigma$.
Then
$\Xx$ is an inflation of $\Tt$ iff:
\begin{enumerate}[label=--]
 \item  $\Xx$ satisfies the bounding requirements on extender indices imposed
by
$\Tt$;
that is, for each $\alpha+1<\lh(\Xx)$, if $\Xx\rest(\alpha+1)$ is an inflation
of $\Tt$ and $\alpha\in(C^-)^{\Tt\inflatearrow\Xx\rest(\alpha+1)}$ then
$\In(E^\Xx_\alpha)\leq\In(E^{\Tt\inflatearrow(\Xx\rest\alpha+1)}_\alpha)$,
and
\item if $\Tt$ has limit length then $\Xx$ does not determine a $\Tt$-cofinal
branch; that is,
if $\eta<\lh(\Xx)$ is a limit and $\Xx\rest\eta$ is an inflation of $\Tt$ and
$(f,C)=(f,C)^{\Tt\inflatearrow\Xx\rest\eta}$
and $[0,\eta)_\Xx\sub C$
then $\lh(\Tt)>\sup_{\alpha<^\Xx\eta}f(\alpha)$.
\end{enumerate}
\end{lem}

\begin{dfn}\index{$\completion^\Sigma(\Tt)$}\index{complete}
Let $\Sigma$ be a partial iteration strategy and $\Tt$ be via $\Sigma$,
with $\Tt$ either of successor length or $\Tt\in\dom(\Sigma)$.
Then $\completion^\Sigma(\Tt)$ denotes $\Tt$ if $\lh(\Tt)$ is a successor,
and denotes $\Tt\conc\Sigma(\Tt)$ otherwise.
\end{dfn}
We easily have:
\begin{lem}
Let $\Sigma$ be a conveniently inflationary partial strategy.
Let $\Tt,\Xx$ be via  $\Sigma$, with $\Xx$  an inflation of $\Tt$,
$\Xx$ of limit length, $\Xx\in\dom(\Sigma)$.
Then either $\completion^\Sigma(\Xx)$ is an inflation of $\Tt$,
or $\Tt$ has limit length, $\Tt\in\dom(\Sigma)$ and $\completion^\Sigma(\Xx)$
is an inflation of $\completion^\Sigma(\Tt)$.
\end{lem}

Steel uses the following notion of strategy condensation in
\cite{str_comparison} (however,
note we also allow partial
strategies).
It easily implies inflation condensation; we do not know whether the converse
holds.

\begin{dfn}\label{dfn:strong_hull_condensation}\index{strong hull
condensation}\index{convenient}\index{inconvenient}
Let $\Sigma$ be a conveniently pre-inflationary partial strategy.
We say that $\Sigma$ has \dfnemph{convenient strong hull condensation}
iff whenever $\Xx$ is via $\Sigma$ and $\Pi:\Tt\hookrightarrow\Xx$ is a tree
embedding,
then $\Tt$ is also via $\Sigma$.

Let $\Lambda$ be an inconveniently pre-inflationary partial strategy.
We say that $\Lambda$ has \dfnemph{inconvenient strong hull condensation}
iff whenever $\Sigma$ has convenient strong hull condensation, where $\Sigma$
is the partial $\udash$strategy corresponding to $\Lambda$.

We say that a pre-inflationary partial strategy has \dfnemph{strong hull
condensation}
iff it has either convenient or inconvenient strong hull condensation.
\end{dfn}

A third condensation notion, also a consequence of strong hull condensation,
we will make use of in \S\ref{sec:gen_abs_mice} in our generic absoluteness
argument.
For our normal realization results we only require inflation condensation.

\begin{dfn}\label{dfn:extra_inf}\index{extra
inflationary}\index{convenient}\index{inconvenient}
Let $\Sigma$ be a conveniently pre-inflationary partial strategy.
We say that $\Sigma$ is \dfnemph{conveniently extra inflationary} iff $\Sigma$
is conveniently inflationary
and for all sufficiently large $\theta\in\OR$,
for all countable transitive $X$ and elementary
\[ \pi:X\to\her_\theta \]
and $\bar{\Tt}\in X$ such that $\pi(\bar{\Tt})$ is via $\Sigma$,
(so $\bar{\Tt}$ is on $\bar{M}$ where $\pi(\bar{M})=M$),
then $\pi\bar{\Tt}$ (the copy of $\bar{\Tt}$ to $M$ via $\pi$) is via $\Sigma$.

We then define \dfnemph{inconveniently extra inflationary},
and \dfnemph{extra inflationary}, as before.
\end{dfn}

\begin{lem}\label{lem:shcond_implies_extra_inf}
If $\Sigma$ has strong hull condensation then $\Sigma$ is extra inflationary.
\end{lem}
\begin{proof}
 The fact that $\Sigma$ has inflation condensation is immediate (inflation
condensation
 just requires the $\Sigma$ condenses under the tree embeddings which arise
from
inflation).

 So let $\pi:X\to\her_\theta$ and $\bar{\Tt}\in X$ be as in \ref{dfn:extra_inf}
and $\Tt=\pi(\bar{\Tt})$.
 We will observe that $\pi\bar{\Tt}$ is via $\Sigma$.
 We define an almost tree embedding
 \[ \Pi:\pi\bar{\Tt}\hookrightarrow_{\almost}\Tt,\]
 by setting
$I_\alpha=[\gamma_\alpha,\delta_\alpha]_\Tt=[\pi(\alpha),\pi(\alpha)]_\Tt$.
 One verifies by a straightforward induction on $\theta\leq\lh(\bar{\Tt})$ that
\[ \Pi\rest(\theta+1):(\pi\bar{\Tt}\rest(\theta+1))\hookrightarrow_{\almost}\Tt
\]
is an almost tree embedding, with associated maps $\pi_\alpha$ and
$\omega_\alpha=\pi_\alpha\rest\exit^\Tt_\alpha$,
and letting
$\varrho_\alpha:M^{\bar{\Tt}}_\alpha\to M^{\pi\bar{\Tt}}_\alpha$
be the copy map induced by $\pi:\bar{M}\to M$, that
$\pi_\alpha\com\varrho_\alpha=\pi\rest M^{\bar{\Tt}}_\alpha$,
and hence,
\[
E^\Tt_{\delta_\alpha}=E^\Tt_{\pi(\alpha)}=\pi(E^{\bar{\Tt}}
_\alpha)=\pi_\alpha(\varrho_\alpha(E^{\bar{\Tt}}_\alpha))=
\pi_\alpha(E^{\pi\bar{\Tt}}_\alpha).\]
This is routine and we leave it to the reader.

By \ref{lem:ate_to_te} and strong hull condensation, it follows that $\Ttbar$
is
according to $\Sigma$.
\end{proof}

We now give some important examples of strategies with strong hull condensation.

\begin{lem}\label{lem:unique_strat_has_inf_cond}
Let $\Sigma$ be a regularly $\Xi$-total pre-inflationary strategy for $M$,
and suppose that $\Sigma$ is the unique such strategy for $M$.
Then $\Sigma$ has strong hull condensation.
\end{lem}
\begin{proof}
We leave the wcpm case to the reader.
Consider the fine case. It suffices then to consider the case that $\Sigma$ is
convenient,
by the 1-1 correspondence between $\udash$strategies and standard strategies
for
MS-indexed premice
(see \ref{lem:rule_conversion}).

Let $\Pi:(\Tt,c)\hookrightarrow(\Xx,d)$ be a tree embedding, with $(\Xx,d)$ via
$\Sigma$,
$\Tt,\Xx$ of limit length, $c$ is $\Tt$-cofinal.
We may assume that $\Tt$ is via $\Sigma$ and $\Pi$ is cofinal in $\lh(\Xx)$.
We must show that $c=\Sigma(\Tt)$.
Let $\eta=\lh(\Tt)$.

If $\eta=\Omega\eqdef\Omega^\Sigma$ this holds because $\cof(\Omega)>\om$. So
suppose $\eta<\Omega$.
Then $\lh(\Xx)<\Omega$ because $\Pi$ is cofinal.
And $k\eqdef\udeg^\Xx(d)=\udeg^\Tt(c)$.
By the uniqueness of $\Sigma$,
it suffices to see that the phalanx $\Phi(\Tt, c)$  is $(\udash
k,\Xi)$-iterable.
But using the embeddings given by $\Pi$,
we can copy $\udash k$-maximal trees on $\Phi(\Tt, c)$ to $\udash k$-maximal
trees on $\Phi(\Xx,d)$.\footnote{Use the one-step copy extension at successor
stages
and form direct limits at limit stages.}
Since $(\Xx, d)$ is via $\Sigma$ and $\lh(\Xx)<\Omega$,
this suffices.
\end{proof}

\begin{rem}\label{rem:wDJ_implies_cond}The previous lemma can be adapted to
wcpms in the obvious manner. However, we do not see how to adapt the following
theorem to wcpms,
 because it relies on a comparison argument.
 Recall from \cite{wdj} or \cite{outline} the \emph{weak Dodd-Jensen property}
for an iteration strategy $\Sigma$
for a countable premouse $M$. John Steel pointed out the following theorem (or
something very similar,
and in the case that $M$ is $\lambda$-indexed) to the author in 2017.
We note that a
variant of its proof shows that if $\Omega>\om$ is regular,
and $e$ an enumeration of $M$ in ordertype $\om$,
there is at most one $(m,\Omega+1)$-strategy $\Sigma$ for $M$ with weak
Dodd-Jensen with respect to $e$.
We often abbreviate \emph{Dodd-Jensen} with \emph{DJ}.\index{DJ}
   \end{rem}

\begin{tm}\label{tm:wDJ_implies_cond}
Let $\Omega>\om$ be regular.
Let $M$ be an $m$-sound premouse with $\card(M)<\Omega$. Let $\Sigma$ be an
$(m,\Omega+1)$-strategy for $M$ such that either
$\Sigma$ has the DJ property, or $M$ is countable and $\Sigma$ has weak DJ.
Then $\Sigma$ has strong hull condensation.
\end{tm}
\begin{proof}
We literally assume that $M$ is countable and $\Sigma$ has weak DJ; otherwise
it
is almost the same but slightly simpler.

We consider first the case that $M$ is $\lambda$-indexed. Thus, $m$- and
$\udash
m$- fine structure
are equivalent.
Suppose the theorem fails in this case, and let
\[ \Pi:(\Tt,c)\hookrightarrow(\Xx,d) \]
be a tree embedding,
with properties as before. Let $b=\Sigma(\Tt)$ and suppose that $b\neq c$.
We have $\lh(\Tt),\lh(\Xx)<\Omega$ as $\Omega>\om$ is regular.

Let $\Gamma$ be the $(\Omega+1)$-strategy for $\Phi(\Tt,c)$ induced by
lifting
to $\Phi(\Xx,d)$.
 Let $\Sigma'$ be the $(\Omega+1)$-strategy for $\Phi(\Tt,b)$ induced by
$\Sigma$.
 Because $M$ and $\Tt$ have cardinality $<\Omega$, we get a successful
comparison $(\Uu,\Vv)$ extending
 $((\Tt,b),(\Tt,c))$,
 according to $\Sigma,\Gamma$;
 here $\Uu,\Vv$ are $m$-maximal trees on $M$. (Note that $\ZF$ suffices here;
 although the standard proof the comparison terminates
 involves taking a hull of $V$, we can do this part working inside $L[X]$ where
 $X\sub\OR$ codes the comparison.) Let $\Ww$ be the tree extending $\Xx$,
 which is the lift of $\Vv$. Let $\pi_\infty:M^\Vv_\infty\to M^\Ww_\infty$ be
the final lifting map.

 If $M^\Vv_\infty\pins M^\Uu_\infty$ then $b^\Vv$ does not drop, so $\Uu$ and
$i^\Vv_\infty:M\to M^\Vv_\infty$
 contradicts weak DJ for $\Sigma$; likewise if $b^\Uu$ drops in model or degree
but $b^\Vv$ does not.
 If $M^\Uu_\infty\pins M^\Vv_\infty$ then $\pi_\infty(M^\Uu_\infty)\pins
M^\Ww_\infty$, so $\Ww$ and $\pi_\infty\com i^\Uu_\infty$ contradicts weak DJ;
 likewise if $b^\Vv$ drops in model or degree (and hence $b^\Ww$ drops
correspondingly) but $b^\Uu$ does not.
 So $M^\Uu_\infty=M^\Vv_\infty$ and neither $b^\Uu$ nor $b^\Vv$ drops in model
or degree.

 We claim that $i^\Uu=i^\Vv$. For suppose not. Let $\left<x_i\right>_{i<\om}$
be
our enumeration of $M$
 relative to which $\Sigma$ has the weak DJ property. Let $k$ be least such
that
$i^\Uu(x_k)\neq i^\Vv(x_k)$.
 Since $i^\Vv$ is a near $n$-embedding,
 and $\Uu$ is according to $\Sigma$, weak DJ gives $i^\Uu(x_k)<i^\Vv(x_k)$.
 But since $b^\Vv$ does not drop, $\pi_\infty$ is also a near $n$-embedding,
 so $\pi_\infty\com i^\Uu$ is likewise, as is $i^\Ww$, and
$i^\Ww=\pi_\infty\com
i^\Vv$.
 Therefore $\pi_\infty(i^\Uu(x_k))<\pi_\infty(i^\Vv(x_k))=i^\Ww(x_k)$,
 so we contradict weak DJ with $\Ww$ (which is according to $\Sigma$) and
$\pi_\infty\com i^\Uu$.

 So $i^\Uu=i^\Vv$. Using this, standard fine structural calculations yield a
contradiction.
 Here is a reminder. One first shows that $b^\Uu$ extends $b$ and $b^\Vv$
extends $c$.
 Then, let $\gamma=\max(b\inter c)$, so $\gamma<\lh(\Tt)$.
 Let $\nu=\sup_{\alpha<\gamma}\nu(E^\Tt_\alpha)$.
 Then
 \[ M^\Tt_\gamma=\cHull^{M^\Uu_\infty}_{n+1}(\rg(i^\Uu)\cup\nu), \]
 and $i^\Uu_{\gamma\infty}$ is just the uncollapse map.
 Likewise with $\Vv$ replacing $\Uu$. But then
$i^\Uu_{\gamma\infty}=i^\Vv_{\gamma\infty}$,
 which contradicts the fact that $\gamma=\max(b\inter c)$. This completes
the
proof in this case.

Now suppose instead that $M$ is MS-indexed. Thus, the statement that $\Sigma$
has
strong hull condensation
literally means that $\Sigma'$ has inflation condensation,
where $\Sigma'$ is the $\udash$strategy corresponding to $\Sigma$.
Suppose the theorem fails in this case,
and let
\[ \Pi':(\Tt',c')\hookrightarrow(\Xx',d'),\]
etc, be a counterexample as before, and $b'=\Sigma'(\Tt')$.
Again $\lh(\Tt'),\lh(\Xx')<\Omega$.

 Let $\Gamma_{c'}'$ be the $(\Omega+1)$-strategy for $\Phi(\Tt',c')$
induced by
lifting to $\Phi(\Xx',d')$.
 Let $\Gamma_{b'}'$ be the $(\Omega+1)$-strategy for $\Phi(\Tt',b')$ induced by
$\Sigma'$.

 Now let $\Tt,\Xx,c,\Gamma_c$, etc, be the canonical conversions of all
of
these
objects to standard MS-premice
 given by \ref{lem:rule_conversion} and
\ref{lem:tree_conversion}. (For
 $c,\Gamma_{c}$,  proceed as follows.
First define a $\uu$-strategy  $\widetilde{\Sigma}'$ for $M$,
by just following $\Sigma'$,
except that $\widetilde{\Sigma}'(\Tt')=c'$ and
$\widetilde{\Sigma}'$ proceeds according to $\Gamma_{c'}'$ for trees
extending
$\Tt'$.
Then let $\widetilde{\Sigma}$ be the $m$-maximal strategy for $M$
corresponding
to $\widetilde{\Sigma}'$,
 given by \ref{lem:rule_conversion}. Finally let $c,\Gamma_c$ be determined by
$\widetilde{\Sigma}$.)
 Then $\lh(\Tt)=\eta'=\lh(\Tt')$
 (as $\eta'$ is a limit), $c\neq b$,
 and $\Gamma_c,\Gamma_b$ are $(\Omega+1)$-strategies for
$\Phi(\Tt,c),\Phi(\Tt,b)$.

We get a successful
comparison\footnote{\label{ftn:comp_alg}Here we adjust the algorithm for
comparison, for example as in \cite{premouse_inheriting}, by minimizing on
$\nu(E)$ before using an extender $E$. That is, if at stage $\alpha$ of the
comparison, the least disagreement consists in two non-empty extenders $E,F$,
and $\nu(E)\neq\nu(F)$,
then we use $E$ if $\nu(E)<\nu(F)$, padding on the other side, and vice versa
if $\nu(F)<\nu(E)$. This avoids the uncomfortable situation of using extender
$E$ on side 1 and $F$ on side 2 at stage $\alpha$, where $E$ is superstrong and
$F$ type 2, and then using the same $F$ on side 1 at stage $\alpha+1$.}
$(\Uu,\Vv)$ extending $((\Tt,b),(\Tt,c))$, according to $\Gamma_b,\Gamma_c$;
 (here $\Uu,\Vv$ are $m$-maximal trees on $M$). Let $\Uu',\Vv'$ be the
corresponding $\udash$trees (in the sense of \ref{lem:tree_conversion}),
 so $(M^{\Uu'}_\infty)^\pm=M^\Uu_\infty$ and
$(M^\Uu_\infty)^\passive\ins(M^{\Uu'}_\infty)^\passive$, and if $M^\Uu_\infty$
is type 3 and $\deg^\Uu(\infty)<\om$ then
$\udeg^{\Uu'}(\infty)=\deg^\Uu(\infty)+1$, and otherwise
$M^{\Uu'}_\infty=M^\Uu_\infty$
 and $\udeg^{\Uu'}(\infty)=\deg^\Uu(\infty)$. Likewise for $\Vv,\Vv'$.
 In particular, $\infty$ is non-$\Uu'$-special and non-$\Vv'$-special.

 Let $\Ww'$ be the tree extending $\Xx'$,
 which is the lift of $\Vv'$; thus, $\Ww'$ is according to $\Sigma'$.
 Then $M^{\Vv'}_\infty,M^{\Ww'}_\infty$ have the same type
 and $\udeg^{\Vv'}(\infty)=\udeg^{\Ww'}(\infty)$, because cofinally many
extenders used in $\Ww'$ are copied from $\Vv'$ (note this includes the case
that $\Vv'=(\Tt',b')$).
 Thus, $\infty$ is non-$\Ww'$-special. So letting $\Ww$ be the standard MS-tree
corresponding to $\Ww'$,
 then $\Ww$ is according to $\Sigma$ and
$(M^{\Ww'}_\infty)^\pm=M^\Ww_\infty$ and
$\deg^\Ww(\infty)=\deg^\Vv(\infty)$.
 Let $\pi'_\infty:M^{\Vv'}_\infty\to M^{\Ww'}_\infty$ be the final copy map.
 Let $\pi_\infty=\pi'_\infty\rest((M^{\Vv'}_\infty)^\pm)^\sq$.
 Then $\pi_\infty:(M^\Vv_\infty)^\sq\to(M^\Ww_\infty)^\sq$
 is a near $\deg^\Vv(\infty)$-embedding (as $\pi'_\infty$ is a near
$\udeg^{\Vv'}(\infty)$-embedding).

Because we have $\pi_\infty$, weak DJ gives that $M^\Vv_\infty=M^\Ww_\infty$ as
usual.
 We have that $[0,\infty]_\Uu$ drops iff $[0,\infty]_{\Uu'}$ drops
 (by Lemma \ref{lem:tree_conversion}),
 and if non-dropping, that $i^\Uu=i^{\Uu'}\rest M^\sq$; likewise for $\Vv,\Vv'$
and $\Ww,\Ww'$.
 Also, $b^\Vv$ drops iff $b^\Ww$ drops, and if non-dropping, then
$\pi'_\infty\com i^{\Vv'}=i^{\Ww'}$ and $\pi_\infty\com i^\Vv=i^\Ww$.

With these facts, the usual
 weak DJ argument  leads to contradiction.
\end{proof}

\subsection{Further inflation terminology}

\begin{dfn}
 Let $\Tt$ be an iteration tree, either $\udash m$-maximal or $m$-maximal,
 or normal on a wcpm. We say that $\Tt$ is \dfnemph{terminally non-dropping}
iff
$\lh(\Tt)$ is a successor and if $\Tt$ is $\udash m$-maximal or $m$-maximal
then
$b^\Tt$ does not drop in model or degree.
\end{dfn}

\begin{dfn}\label{dfn:inflation_notation}
Let $\Xx,\Tt$ be on $M$, with $\Xx$ an inflation of $\Tt$. Let
\[ (t,C,C^-,f,\Pivec)=(t,C,C^-,f,\Pivec)^{\Tt\inflatearrow\Xx} \]
and let $\gamma_{\alpha;\beta}$, etc, be as in
\ref{dfn:inflation}.
Suppose that $\Xx$ has successor length $\alpha+1$.

We say that $\Xx$ is:
\begin{enumerate}[label=--]
 \item \dfnemph{\tu{(}$\Tt$\tu{)}-pending} iff $\alpha\in C^-$.\index{pending}
 \item \dfnemph{non-\tu{(}$\Tt$\tu{)}-pending} iff $\alpha\notin C^-$.
\item \dfnemph{\tu{(}$\Tt$\tu{)}-terminal} iff $\Tt$ has successor length and
$\Xx$ is non-$\Tt$-pending.\index{terminal}
\end{enumerate}

Suppose that $\Xx$ is $\Tt$-terminal. We say that $\Xx$ is:
\begin{enumerate}[label=--]\index{terminally-(non)-(model)-dropping}
\item \dfnemph{$\Tt$-terminally-non-model-dropping} iff $\alpha\in C$ (hence,
$f(\alpha)+1=\lh(\Tt)$),
\item  \dfnemph{$\Tt$-terminally-non-dropping} iff $\alpha\in C$  and
$\udeg^\Xx(\alpha)=\udeg^\Tt(f(\alpha))$,
\item  \dfnemph{$\Tt$-terminally-model-dropping} iff $\alpha\notin C$,
\item \dfnemph{$\Tt$-terminally-dropping} iff $\alpha\notin C$ or
$\udeg^\Xx(\alpha)<\udeg^\Tt(f(\alpha))$.
\end{enumerate}

Suppose $\Xx$ is $\Tt$-terminally-non-\emph{model}-dropping and let
$\alpha+1=\lh(\Xx)$ and
$\beta=f(\alpha)$ and $\gamma=\gamma_{\alpha;\beta}$. Then we
define\index{$\pi_\infty^{\Tt\inflatearrow\Xx}$}
\[ \pi_\infty^{\Tt\inflatearrow\Xx}:M^\Tt_\beta\to M^\Xx_\alpha \]
by $\pi_\infty^{\Tt\inflatearrow\Xx} =
i^\Xx_{\gamma\alpha}\com\pi_{\alpha;\beta}$.
\end{dfn}

\begin{rem}\label{rem:non-drop_inf_comm}Suppose $\Xx$ is
$\Tt$-terminally-non-\emph{model}-dropping  and $\Tt,\Xx$ are $\udash
m$-maximal. Note that $\pi_\infty=\pi_\infty^{\Tt\inflatearrow\Xx}$ is a near
$\udash n$-embedding, where $n=\udeg^\Xx(\infty)$.
If $\Xx$ is $\Tt$-terminally-non-\emph{dropping} and $\Tt$ is
terminally non-dropping, then note that $\Xx$ is terminally non-dropping, so
$n=m$, $\pi_\infty$ is a $\udash m$-embedding and $\pi_\infty\com i^\Tt=i^\Xx$.
\end{rem}

\section{Comparison inflation, genericity inflation}\label{sec:min_inf}

In this section we prove a comparison result for iteration trees, analogous to
comparison of premice.
The process we call \emph{comparison inflation}.\footnote{In a
preprint of this paper on arxiv.org, it was called \emph{minimal
\tu{(}simultaneous\tu{)} inflation}, but this would be in conflict with another
notion of \emph{minimal inflation} to be used in the context of (full)
normalization.} We will
need this result both in the construction of an iteration strategy for stacks
of
limit length,
and in the extension of an iteration strategy with inflation condensation to a
sufficiently small generic extension.
We also introduce \emph{genericity inflation}, an inflation analogue to
genericity iteration.

\subsection{Comparison inflation}\label{subsec:min_inf}

\begin{dfn}\label{dfn:min_inf}\index{comparison inflation}
Let $\Omega>\om$ be regular
and let $\Sigma$ be a regularly $(\Omega+1)$-total
conveniently inflationary
strategy for $M$.
Let $\mathscr{T}$ be a set of trees according to $\Sigma$,
each of length $\leq\Omega+1$,
and such that there is no surjection
$\mathscr{T}\to\Omega$.\footnote{Since $\Omega$
is regular, it follows that there is no cofinal map
$\mathscr{T}\to\Omega$. Note there is no
restriction on
$\card(M)$,
but $\lh(\Tt)\leq\Omega+1$ for each $\Tt\in\mathscr{T}$.} The
\dfnemph{comparison inflation} of $\mathscr{T}$
is the tree $\Xx$ on $M$ with the following properties:
\begin{enumerate}[label=--]
 \item $\Xx$ is according to $\Sigma$,
 \item $\Xx$ is an inflation of each $\Tt\in\mathscr{T}$; we write
$t^\Tt=t^{\Tt\inflatearrow\Xx}$, etc,
 for $\Tt\in\mathscr{T}$,
 \item for each $\alpha+1<\lh(\Xx)$ there is $\Tt\in\mathscr{T}$ such that
$t^{\Tt}(\alpha)=0$,
 \item $\Xx$ has successor length $\leq\Omega+1$,
 \item if $\lh(\Xx)=\alpha+1<\Omega$ then for every $\Tt\in\mathscr{T}$, we
have
$\alpha\notin(C^-)^{\Tt}$.\qedhere
\end{enumerate}
\end{dfn}
\begin{lem}[Comparison inflation]\label{lem:min_inf}
Let $\Omega,\mathscr{T},\Sigma$ be as in \ref{dfn:min_inf}. Then there is a
unique comparison inflation $\Xx$ of $\mathscr{T}$.
Moreover, there is $\Tt\in\mathscr{T}$ such that, with
$\Tt'=\completion^\Sigma(\Tt)$, we have
\begin{enumerate}[label=--]
 \item  $\Xx$ is $\Tt'$-terminally-non-dropping, and
 \item if $\lh(\Xx)=\Omega+1$ then $\lh(\Tt')=\Omega+1$.
\end{enumerate}
\end{lem}

\begin{proof}
We first verify uniqueness. Given $\alpha<\lh(\Xx)\inter\Omega$, we have
$\alpha+1<\lh(\Xx)$
iff $\alpha\in(C^-)^{\Tt}$ for some $\Tt\in\mathscr{T}$.
And if $\alpha+1<\lh(\Xx)$ then
\[ \In(E^\Xx_\alpha)=\min(\{\In(E^{\Tt\inflatearrow\Xx}_\alpha)\bigm|
\alpha\in(C^-)^{\Tt}\}) \]
as $\Xx$ is an inflation of every $\Tt\in\mathscr{T}$. So there is no freedom
in
the choice of extenders,
and since $\Xx$ is via $\Sigma$, $\Xx$ is therefore unique.

Existence is by the proof of uniqueness and because inflations can be freely
extended (as $\Sigma$ has inflation condensation).

We now verify the ``moreover'' clause.

Suppose first that $\lh(\Xx)=\Omega+1$.
For every $\beta$ such that $\beta+1<^\Xx\Omega$,
there is $\Tt\in\mathscr{T}$ such that $t^\Tt(\beta)=0$.
Since there is no surjection $\card(\mathscr{T})\to\Omega$ and $\Omega$ is
regular,
we may fix $\Tt\in\mathscr{T}$ such that $t^\Tt(\beta)=0$ for cofinally many
$\beta+1<^\Xx\Omega$.
Let $\Tt'=\completion^\Sigma(\Tt)$.
It follows that $\Omega\in C^{\Tt'}$, and in fact, $\Omega=f^{\Tt'}(\Omega)$
and $\Omega=\gamma^{\Tt'}_\Omega$, so $\Xx$ is
$\Tt'$-terminally-non-dropping.\footnote{Clearly
this reflection argument uses only the regularity of $\Omega$, no $\AC$.}

Next suppose $\lh(\Xx)=\beta+2=\alpha+1$ for some $\beta$.
Then letting $\Tt\in\mathscr{T}$ be such that $t^\Tt(\beta)=0$,
we have $\alpha=\beta+1\in C^{\Tt}$. Since $\alpha\notin(C^-)^\Tt$,
it follows that $\Tt$ has successor length and $\Xx$ is
$\Tt$-terminally-non-dropping.

Finally suppose that $\lh(\Xx)=\alpha+1<\Omega$ and $\alpha$ is a limit.
Let $\beta<^\Xx\alpha$ be such that
$(\beta,\alpha)_\Xx\inter\dropset_{\deg}^\Xx=\emptyset$.
Fix $\Tt\in\mathscr{T}$ such that $t^\Tt(\beta)=0$, so $\beta\in C^\Tt$.
Let $\Tt'=\completion^\Sigma(\Tt)$, so $\Xx$ is also an inflation of $\Tt'$.
Moreover, $\alpha\in C^{\Tt'\inflatearrow\Xx}$ because
$(\beta,\alpha)_\Xx\inter\dropset^\Xx=\emptyset$.
But then $f^{\Tt'}(\alpha)+1=\lh(\Tt')$, since $\alpha\notin(C^-)^\Tt$.
Since also $(\beta,\alpha)_\Xx$ does not drop in model or degree and
$f^{\Tt'}(\beta)+1<\lh(\Tt')$,
it follows that $\Xx$ is $\Tt'$-terminally-non-dropping, as required.
\end{proof}

\subsection{Genericity inflation}
Like for comparison, there is also an inflation analogue of genericity
iteration, which we describe next.
We won't actually use the technique in this paper, but it is easy to describe
and worth noting, and the author has used it in other unpublished work,
for the purposes mentioned in \ref{rem:self-it} below.
\setcounter{footnote}{0}
\footnote{This technique and its application to self-iterability of mice
was the author's first main motivation for considering inflation.}
Analogous results hold for slightly coherent wcpms and fine mice of both
indexings (paired with their
standard iteration rules).
We first give the full proof for $\udash m$-sound seg-pms with MS-indexing
(with MS-iteration rules).
The proof adapts easily to the wcpm version, and we leave this to the reader;
slight coherence
ensures that the tree produced is normal.
We will then explain how genericity iteration works for $\lambda$-indexed mice
with $\lambda$-iteration rules,
and finally sketch genericity inflation for this case.
We state the results for the $\delta$-generator extender algebra,
but the versions for the $\om$-generator extender algebra are an easy corollary.

\begin{dfn}\label{dfn:BB_delta}\index{$\BB_\delta$}\index{extender algebra}
We write $\BB_\delta$ for the $\delta$-generator extender algebra at $\delta$.
When working inside a seg-pm or wcpm $M$, we only use extenders $E\in\es^M$
such
that $\nu_E$ is an $M$-cardinal
to induce extender algebra axioms (one can also require that $\nu_E$ is
inaccessible in $M$,
etc, as desired). Let $\kappa=\crit(E)$. Recall here that the axioms have the
form
\[
\bigvee_{\alpha<\nu_E}\varphi_\alpha\implies\bigvee_{\alpha<\kappa}
\varphi_\alpha \]
where $\vec{\varphi}=\left<\varphi_\alpha\right>_{\alpha<\nu_E}\in M$,
$\varphi_\alpha\in M|\kappa$ for all $\alpha<\kappa$,
$\varphi_\alpha\in M|\nu_E$ for all $\alpha<\nu_E$,
and
$\vec{\varphi}=i^M_E(\vec{\varphi})\rest\nu_E$. (So $\vec{\varphi}\in\Ult(M,E)$,
so $\vec{\varphi}\in M|\In(E)$.)
We use this definition independent of indexing.

Given an extender $G$ and $A\sub\OR$, we say that $G$ is
\dfnemph{$A$-bad}\index{bad}\index{$A$-bad}
iff $G$ induces a $\delta$-generator extender algebra axiom not satisfied by
$A$
(equivalently,
by $A\inter\nu_G$).
\end{dfn}

\begin{dfn}[Genericity inflation for MS-indexing and slightly coherent
wcpms]\label{dfn:gen_inf}\index{genericity inflation}
Let $\Omega$ be regular uncountable.
Let $M,\Sigma$ be such that either:
\begin{enumerate}[label=--]
 \item $M$ is a $\udash m$-sound MS-indexed seg-pm
and $\Sigma$ is a conveniently inflationary $(\udash m,\Omega+1)$-strategy for
$M$, or
\item $M$ is a slightly coherent wcpm and $\Sigma$ is an inflationary
$(\Omega+1)$-strategy for $M$,
\end{enumerate}
and suppose that $\card(M)<\Omega$ (here if $M$ is a wcpm,
which might not satisfy $\AC$,
we mean that $M$ is coded by some set $X\sub\eta<\Omega$).
Let $\Tt$ be according to $\Sigma$, of limit length $\leq\Omega$, and
$\Tt'=\Tt\conc\Sigma(\Tt)$.
Let $A\sub\Omega$. The \dfnemph{$A$-genericity inflation} of $\Tt$ is the tree
$\Xx$ such that:
\begin{enumerate}[label=--]
\item $\Xx$ is a $\Tt'$-terminally-non-dropping inflation of $\Tt'$ (hence of
successor length), according to $\Sigma$;
write $C^{\Tt'}=C^{\Tt'\inflatearrow\Xx}$, etc.
\item For every $\alpha+1<\lh(\Xx)$, we have
 $\alpha\in(C^-)^{\Tt'}$, and letting
$\xi=\In(E^{\Tt'\inflatearrow\Xx}_\alpha)$, then
 $\In(E^\Xx_\alpha)$ is the least $\gamma$ such that either $\gamma=\xi$, or:
\begin{enumerate}[label=--]
\item  $F\eqdef\es_\gamma(M^\Xx_\alpha)$ is $A$-bad, and
 \item if $M$ is MS-indexed then $\nu_F$ is a cardinal of $M^\Xx_\alpha|\xi$
(hence $F$ is total over $M^\Xx_\alpha|\xi$).\qedhere
\end{enumerate}
\end{enumerate}
\end{dfn}

\begin{rem}
 Note that if $\Xx$ is the $A$-genericity inflation of $\Tt$ and
$\lh(\Xx)=\alpha+1$,
 then $\alpha$ is least such that
$f^{\Tt\inflatearrow\Xx}(\alpha)=\lh(\Tt)=\lh(\Tt')-1$.
 So $\Sigma(\Tt)$ is not relevant to the construction of $\Xx$; we need only
$M,\Sigma,\Tt,A$.
But $\Xx$ determines $\Sigma(\Tt)$, hence $\Tt'$, by inflation condensation.
\end{rem}

\begin{tm}\label{lem:MS_gen_inf}
Let $\Omega,\Sigma,\Tt,A$ be as in \ref{dfn:gen_inf}.
Then there is a unique $A$-genericity inflation $\Xx$ of $\Tt$ via $\Sigma$,
and  $\lh(\Xx)=\Omega+1$ iff $\lh(\Tt)=\Omega+1$.
\end{tm}
\begin{proof}
The choice of extenders in $\Xx$ is clearly uniqued. The minimality of
$\lh(\Xx)$ and the requirement that it be via $\Sigma$, therefore determines
$\Xx$ uniquely.

Now consider existence.
Let us first verify that given a segment $\Xx\rest(\eps+1)$ which is normal and
such that $\Xx\rest\eps$ satisfies the properties stated above,
then either $\eps\in(C^-)^{\Tt'}$ or $\Xx=\Xx\rest(\eps+1)$ is as desired.
If $\eps=0$ this is trivial and if $\eps$ is a limit it holds by induction
(if $\eps\in C^{\Tt'}\cut(C^-)^{\Tt'}$ then we are finished).
If $M$ is a wcpm it is also automatic.
So suppose $\eps=\beta+1$ and $M$ is MS-indexed.
We may assume that $t^{\Tt'}(\beta)=1$. Let $\alpha=\pred^\Xx(\eps+1)$.
Then by induction, $\alpha\in(C^-)^{\Tt'}$. We may also assume that
$t^{\Tt'}(\alpha)=1$.
Then $\crit(E^\Xx_\beta)<\nu(E^\Xx_\alpha)$
and $E^\Xx_\beta$ is total over $M^\Xx_\alpha|\nu(E^\Xx_\alpha)$, but then
$E^\Xx_\beta$ is total over
\[ K\eqdef M^\Xx_\alpha|\In(E^{\Tt'\inflatearrow\Xx}_\alpha), \]
because, by construction, $\nu(E^\Xx_\alpha)$
is a cardinal of $K$. So $\eps\in(C^-)^{\Tt'}$ as required.

Now $\Xx$ is monotone index-increasing; that is, if $\alpha+1<\beta+1<\lh(\Xx)$
then
$\In(E^\Xx_\alpha)\leq\In(E^\Xx_\beta)$.
For suppose not and let $(\alpha,\beta)$ be least such. Suppose $M$ is
MS-indexed.
Since $\Xx\rest(\beta+1)$ is an inflation of $\Tt$,
\[ \xi\eqdef\In(E^{\Tt\inflatearrow\Xx}_\beta)\geq\In(E^\Xx_\alpha). \]
Since $\In(E^\Xx_\beta)<\In(E^\Xx_\alpha)$, note
$E^\Xx_\beta$ is $A$-bad and
$\nu(E^\Xx_\beta)\leq\nutilde(E^\Xx_\alpha)<\In(E^\Xx_\alpha)$
and $\nu(E^\Xx_\beta)$ is a cardinal
of $M^\Xx_\beta|\xi$. But then by coherence, $E^\Xx_\beta\in\es(M^\Xx_\alpha)$
and $\nu(E^\Xx_\beta)$
is a cardinal of $M^\Xx_\alpha|\In(E^\Xx_\alpha)$,
which implies that we should have used $E^\Xx_\beta$ at stage $\alpha$,
contradiction.
If instead $M$ is a wcpm then one uses slight coherence for a similar argument.

It remains to see that if we reach $\Xx$ of length $\Omega+1$,
then $\Omega+1=\lh(\Tt')$ and $\Omega=f^{\Tt'}(\Omega)$.
We may assume $\ZFC$, by noting that the entire construction
takes place in $L[X,\Sigma,\Tt,A]$ where $X\sub\eta<\Omega$ codes $M$.
We have $\Omega\in C^{\Tt'}$, and moreover, it suffices to see that
$t^{\Tt'}(\beta)=0$
for cofinally many $\beta+1<^\Xx\Omega$.
Let $\eta$ be large and $\pi:H\to V_\eta$ be elementary
with $\pi(\mu)=\Omega$ where $\crit(\pi)=\mu$, and everything relevant in
$\rg(\pi)$.
Then by the usual calculations, letting $\beta+1=\successor^\Xx(\mu,\Omega)$,
$E^\Xx_\beta$ coheres $A$ through $\nu(E^\Xx_\beta)$,
and hence, $E^\Xx_\beta$ is not $A$-bad.
Therefore $t^{\Tt'}(\beta)=0$. So by elementarity, we are done.
\end{proof}

\begin{rem}\label{rem:self-it}\index{self iterability}
Note that to construct $\Xx$,
the information we actually need is $M,\Tt,A$,
and the sequence of branches actually used in forming $\Xx$.
Moreover, from $\Xx$ we can compute $\Sigma(\Tt)$.
Thus, genericity inflations (and variants thereof)
provide a natural method to
attempt to compute $\Sigma(\Tt)$,
if we know how to compute $\Sigma(\Xx)$ for enough trees $\Xx$:
one builds such an $\Xx$ into which $\Tt$ is embedded.
An application of this is some unpublished work of the author's,
showing that
a (in the more interesting case, non-tame) premouse $M$
computes some fragment of its own iteration strategy;
an instance of this method will also be used in \cite{vm2}.
This application incorporates and generalizes the methods of \cite{sile}, which
covers
a large part of the technical issues, but is limited to tame mice.
In this context, $\Tt$ is some tree on $M$ or a segment thereof, $\Tt\in M$,
and $A=\es^M$. One uses P-constructions/$*$-translations to compute
$\Sigma(\Xx\rest\eta)$ for limits $\eta$
(see  \cite{*-trans}, \cite{sile}, augmented with \cite{*-trans_add}). Note
that
 because the computation of $C^{\Tt}$ and $E^{\Tt\inflatearrow\Xx}_\alpha$
is local, at non-trivial limit stages $\eta$ of the genericity inflation,
with $\delta=\delta(\Xx\rest\eta)$, we get that $\Xx\rest\eta$
is definable from parameters over $M|\delta$ (to arrange this, one might need
to
insert short linear iterations into the genericity inflation,
to ensure that the $*$-translations of the Q-structures determining earlier
branch choices are
proper segments of $M|\delta$; such arguments appear in \cite{odle}).
Because we have also made $M|\delta$ generic, we have the necessary base for
forming P-constructions/$*$-translations.
For example, we might want to use this method to prove that $M\sats$``My
countable
proper segments are $(\om,\om_1)$-iterable''
(maybe above some $\alpha<\om_1^M$).
For arbitrary non-tame mice, there seem to be subtleties
in proving that the genericity inflation process terminates
prior to $\om_1^M$ in $M$.
But in typical ``$\varphi$-minimal'' mice (for example,
the sharp for the least proper class mouse satisfying ``There is a superstrong
extender''), it does.
\end{rem}

We now discuss the version for $\lambda$-indexing and $\lambda$-iteration rules.
We first describe how standard genericity iteration works for $\lambda$-indexed
mice with $\lambda$-iteration rules.
\footnote{The methods
here are related to those used by the author in \cite{rule_conversion}
to translate between different iteration rules
for $\lambda$-indexed mice.}
The main difference between this and standard genericity iteration (for
MS-indexing with
MS-iteration rules)
is that we will allow drops in model to appear at intermediate stages of the
iteration.
We will thus need to be a little careful to ensure that the eventual main
branch
is non-dropping.
In our original attempted proof, we had ignored the fact that the collection of
extenders used to induce extender algebra axioms
are not cohered by extenders $E$ through $\lambda(E)$. We thank Stefan
Miedzianowski for pointing  this issue out.
Fortunately a fix was available for this problem.

 \begin{tm}[Genericity iteration for
$\lambda$-indexing]\label{tm:lambda_gen_it}\index{genericity iteration}
 Let $\Omega>\om$ be regular.
 Let $M$ be a $\lambda$-indexed pm with $\card(M)<\Omega$.
 Let $\Sigma$ be a $(0,\Omega+1)$-strategy  for $M$ \tu{(}for
$\lambda$-iteration rules\tu{)}.
 Let $\delta\in\OR^M$ be such that $M\sats$``$\delta$ is Woodin as witnessed by
$\es$''.
 Let $A\sub\Omega$.

 Then there is $\Tt$ on $M$ via $\Sigma$, of length $\alpha+1<\Omega$,
 such that $[0,\alpha)_\Tt$ does not drop in model,
 and $A\inter\delta'$ is $M^\Tt_\alpha$-generic for
$\BB_{\delta'}(M^\Tt_\alpha)$,
 where $\delta'=i^\Tt_{0\alpha}(\delta)$.
 \end{tm}
 \begin{proof}
    We form $\Tt$ as follows. Suppose we have defined $\Tt\rest(\alpha+1)$,
 but it doesn't yet witness the theorem. We (attempt to) define a sequence
$\left<M_{\alpha i}\right>_{i\leq k_\alpha}$,
 with $k_\alpha<\om$, and with $M_{\alpha i}$ an active segment of $M_\alpha$
and $M_{\alpha,i+1}\pins M_{\alpha i}$.
 Let $M_{\alpha 0}$, if it exists,
 be the least $N\ins M_\alpha$ such that $N$ is active and letting $G=F^N$,
either
 \begin{enumerate}[label=--]
  \item $[0,\alpha]_\Tt$ drops in model and $N=M_\alpha$, or
  \item $\nu_G$ is a cardinal\footnote{If one only forms extender algebra
axioms
with extenders $E$
  with $\nu_E$ inaccessible, then one could also assume here that $\nu_G$ is
inaccessible in $M^\Tt_\alpha$.}
  of $M_\alpha$, $G$ is
$A$-bad\footnote{This makes sense even if
$[0,\alpha]_\Tt$ drops,
  as the requirements are local.} and if $[0,\alpha]_\Tt$ does not drop in
model
then $\In(G)<i^\Tt_{0\alpha}(\delta)$.
 \end{enumerate}
 If $M_{\alpha 0}$ does not exist then we terminate the process,
 setting $\Tt=\Tt\rest\alpha+1$.

 Suppose that $M_{\alpha i}$ exists where $i<\om$.
 Then $M_{\alpha,i+1}$, if it exists, is the
least $N\pins M_{\alpha i}$
 such that $N$ is active with $G=F^N$, $\nu_G$
 is a cardinal of $M_{\alpha i}$ and $G$
 is $A$-bad.
 If $M_{\alpha,i+1}$ does not exist then set $k_\alpha=i$ and
 $E^\Tt_\alpha=F(M_{\alpha k_\alpha})$.

 We claim that this works. Suppose not. For each $\alpha+1<\lh(\Tt)$ (hence,
$M_{\alpha
0}$ exists)
 let $\nu_{\alpha i}=\nu(F(M_{\alpha i}))$.

 \begin{clm}\label{clm:active}
  Let $\alpha+1<\lh(\Tt)$ with $M_{\alpha 0}=M_\alpha$. Then $k_\alpha=0$, so
$E^\Tt_\alpha=F^{M_\alpha}$.
 \end{clm}
\begin{proof}
 Otherwise $M_{\alpha 1}$ would contradict the minimality of the choice of
$M_{\alpha 0}$.
\end{proof}

 \begin{clm}\label{clm:dropdown}
  Let $\alpha+1<\lh(\Tt)$ with $M_{\alpha 0}\pins M_\alpha$.
  Then:
  \begin{enumerate}[label=\arabic*.,ref=\arabic*]
   \item\label{item:proj_is_nu} $\rho_1(M_{\alpha i})=\rho_\om(M_{\alpha
i})=\nu_{\alpha i}$.
   \item\label{item:nus_increasing} $\nu_{\alpha 0}<\ldots<\nu_{\alpha
k_\alpha}$.
   \item\label{item:dropdown}
The reverse model dropdown sequence of $(M_\alpha,\In(E^\Tt_\alpha))$ is
$\left<M_{\alpha i}\right>_{i\leq k_\alpha}$.
  \end{enumerate}
 \end{clm}
\begin{proof}
Part \ref{item:proj_is_nu}: For any active premouse $N$, $\rho_1^N\leq\nu(F^N)$.
But $\rho_\om(M_{\alpha i})\geq\nu_{\alpha i}$, because either:
\begin{enumerate}[label=--]
 \item $i=0$ and $\nu_{\alpha 0}$ is a cardinal of $M_\alpha$ and $M_{\alpha
0}\pins M_\alpha$, or
 \item $i>0$ and $\nu_{\alpha i}$ is a cardinal of $M_{\alpha,i-1}$ and
$M_{\alpha i}\pins M_{\alpha,i-1}$.
\end{enumerate}

Part \ref{item:nus_increasing}: Suppose $\nu_{\alpha,i+1}\leq\nu_{\alpha i}$.
Then we contradict the minimality of $M_{\alpha i}$.
That is, if $i=0$, then  $\nu_{\alpha,i+1}$ is a cardinal of
$M_\alpha$,
but $M_{\alpha,i+1}\pins M_{\alpha i}\ins M_\alpha$, so we should have chosen
$M_{\alpha,i+1}$ over $M_{\alpha i}$.
It is similar if $i>0$.

Part \ref{item:dropdown}: Because $\nu_{\alpha,i+1}$ is a cardinal in
$M_{\alpha
i}$, and $\nu_{\alpha 0}$ a cardinal in $M_\alpha$,
this follows from the previous parts.
\end{proof}

 \begin{clm}\label{clm:drop_picture}
  Let $\beta<\lh(\Tt)$ be such that $[0,\beta]_\Tt$ drops in model.
  Then $M_\beta$ is active. Moreover,
 let $\gamma+1\leq^\Tt\beta$ be such that $\gamma+1\in\dropset^\Tt$
 and $(\gamma+1,\beta]_\Tt$ does not drop in model, and let
$\alpha=\pred^\Tt(\gamma+1)$.
 Then $M_{\alpha 0}\pins M_\alpha$
 and $M^{*\Tt}_{\gamma+1}=M_{\alpha i}$ for some $i\leq k_\alpha$,
 and
$F^{M_\beta}\rest\nu_{\alpha i}=F(M_{\alpha i})\rest\nu_{\alpha i}$.
 \end{clm}
\begin{proof}
 Because $\gamma+1\in\dropset^\Tt$, $M^{*}_{\gamma+1}$ is in the
$(M_\alpha,\In(E_\alpha))$-dropdown,
 so by Claims \ref{clm:active} and \ref{clm:dropdown}, $M_{\alpha 0}\pins
M_\alpha$ and $M^{*}_{\gamma+1}=M_{\alpha i}$ for some $i$.
 Therefore $M_\beta$ is active. But also by Claim \ref{clm:dropdown},
$\nu_{\alpha i}=\rho_\om(M_{\alpha i})\leq\crit(E_\gamma)$,
and so $F^{M_\beta}\rest\nu_{\alpha i}\sub F(M_{\alpha i})$.
\end{proof}

\begin{clm}
 $\Tt$ is normal.
\end{clm}
\begin{proof}
We just need to see that $\In(E_\alpha)<\In(E_\beta)$ for $\alpha<\beta$.
But otherwise, letting $(\alpha,\beta)$ be the least counterexample,
then since $(\exit^\Tt_\alpha)^\passive=(M_{\alpha k_\alpha})^\passive$ is a
cardinal segment of $M_\beta$,
we easily get that $M_{\beta 0}\pins M_{\alpha k_\alpha}$,
and reach a contradiction to the maximality of $k_\alpha$ (that is,
$M_{\alpha,k_\alpha+1}$ exists, a contradiction).
\end{proof}

By Claim \ref{clm:drop_picture} and by construction, if $\Tt$ terminates in
length $\alpha+1<\Omega$,
then $[0,\alpha]_\Tt$ does not drop, so we are done. So it suffices to prove:

\begin{clm}$\Tt$ terminates with length $<\Omega$.
\end{clm}
\begin{proof} We may assume $\ZFC$, by working in $L[M,\Sigma,A]$, where we
have
$\Tt$.
Suppose that we reach $\Tt$ of length $\Omega+1$.
Let $\pi:N\to V_\eta$ be elementary, where $\eta$ is large and $N$
is transitive with $\card(N)<\Omega$, $\crit(\pi)=\kappa$
and $\pi(\kappa)=\Omega$,
and the relevant objects are in $\rg(\pi)$.
Then as usual, $M^\Tt_\kappa\in N$, $i^\Tt_{\kappa\Omega}\sub\pi$,
and $\pi(A\inter\kappa)=A$. Let $\beta+1<^\Tt\Omega$
with $\pred^\Tt(\beta+1)=\kappa$.
By the usual argument that genericity iterations for MS-indexing terminate,
$E_\beta$ is not $A$-bad,
so $[0,\beta]_\Tt$ drops in model and $E_\beta=F(M_\beta)$.
So by Claim \ref{clm:drop_picture} there is $\alpha<^\Tt\beta$ and $i\leq
k_\alpha$
such that $F(M_{\alpha i})$
\emph{is} $A$-bad and
$F(M_{\alpha i})\rest\nu_{\alpha i}=E_\beta\rest\nu_{\alpha i}$.
But then again, the usual argument gives a contradiction.\end{proof}

This completes the proof.
\end{proof}

Finally, genericity inflation for $\lambda$-iteration rules is just a
straightforward
combination of the preceding methods:

\begin{dfn}[Genericity inflation for
$\lambda$-indexing]\label{dfn:lambda_gen_inf}\index{genericity inflation}
Let $\Omega>\om$ be regular.
Let $M$ be an $m$-sound $\lambda$-indexed premouse with $\card(M)<\Omega$.
Let $\Sigma$ be an inflationary $(m,\Omega+1)$-strategy for $M$ (for
$\lambda$-iteration rules).
Let $\Tt$ be according to $\Sigma$, of limit length $\leq\Omega$, and
$\Tt'=\Tt\conc\Sigma(\Tt)$.
Let $A\sub\Omega$. The \dfnemph{$A$-genericity inflation} of $\Tt$ is the tree
$\Xx$ such that:
\begin{enumerate}[label=--]
\item $\Xx$ is a $\Tt'$-terminally-non-dropping inflation of $\Tt'$ (hence of
successor length), according to $\Sigma$;
write $C^{\Tt'}=C^{\Tt'\inflatearrow\Xx}$, etc.
\item If $\alpha+1\leq\lh(\Xx)$ and $\Xx\rest\alpha+1$ is
$\Tt'$-terminally-non-dropping
then $\alpha+1=\lh(\Xx)$.
\item Let $\alpha+1<\lh(\Xx)$. We define $\xi_\alpha$, $k_\alpha<\om$,
$\left<M_{\alpha i}\right>_{i\leq k_\alpha}$ and $E^\Xx_\alpha$ as follows:
\begin{enumerate}[label=--]
 \item If $\alpha\in(C^-)^{\Tt'}$ then
$\xi_\alpha=\In(E^{\Tt'\inflatearrow\Xx}_\alpha)$.
 \item If $\alpha\notin(C^-)^{\Tt'}$ then $\xi_\alpha=\OR(M^\Xx_\alpha)$;
 in this case, $[0,\alpha]_\Xx$ drops in model
 and $M^\Xx_\alpha$ is active.
 \item $M_{\alpha 0}$ is the least $N\ins M^\Xx_\alpha|\xi_\alpha$ such that
either $N=M^\Xx_\alpha|\xi_\alpha$ or $N$ is active with $G=F^N$,
$\nu_G$ is an $M^\Xx_\alpha|\xi_\alpha$-cardinal and $G$ is $A$-bad.
\item $k_\alpha$ and $\left<M_{\alpha i}\right>_{0<i\leq k_\alpha}$
are determined from  $M_{\alpha 0}$ as in the proof of \ref{tm:lambda_gen_it}.
\item $E^\Xx_\alpha=F(M_{\alpha k_\alpha})$.\qedhere
\end{enumerate}
\end{enumerate}
\end{dfn}

A straightforward combination of the proofs of \ref{tm:lambda_gen_it}
and \ref{lem:MS_gen_inf} gives:

\begin{tm}
Let $\Omega,\Sigma,\Tt,A$ be as in \ref{dfn:lambda_gen_inf}.
Then there is a unique $A$-genericity inflation $\Xx$ of $\Tt$ via $\Sigma$,
and  $\lh(\Xx)=\Omega+1$ iff $\lh(\Tt)=\Omega+1$.
\end{tm}

\section{Commutativity of inflation}\label{sec:inf_comm}

We will later show that a normal iteration strategy with
inflation condensation induces a strategy for stacks $\Ttvec$ of normal trees.
The latter strategy will be such that we can embed the last model of $\Ttvec$
into
the last model of a
normal tree
$\Xx$. The tree $\Xx$ will be produced by inflation; for example, if
$\Ttvec=(\Tt_0,\Tt_1)$ where
each $\Tt_i$ is normal, then $\Xx$ will be an inflation of $\Tt_0$.
For infinite stacks,
we will produce an infinite sequence of trees
$\left<\Xx_\alpha\right>_{\alpha<\eta}$,
with $\Xx_\beta$ an inflation of $\Xx_\alpha$ for each $\alpha<\beta$.
In this section we establish a key commutativity lemma which helps us
understand
this situation.
We will also use the lemma in \S\ref{sec:gen_abs_mice}, when we extend an
iteration strategy with inflation condensation to a sufficiently small generic
extension.
We state the coarse version of the lemma first, as it contains the main points,
and then state and prove the fine version.
A key point to note is that the commutativity lemmas hold for arbitrary trees
and inflations
(satisfying certain conditions);
we do not assume that the trees are via a strategy with condensation.

\begin{figure}
\centering
\begin{tikzpicture}
 [mymatrix/.style={
    matrix of  nodes,
    row sep=1.2cm,
    column sep=1.2cm}]
   \matrix(m)[mymatrix]{
 {Coarse}&{}&{$M^{\Xx_2}_{\alpha_2}$}& {Fine}&{}&{$M^{\Xx_2}_{\alpha_2}$}\\
 {}&{$M^{\Xx_1}_{\alpha_1}$}&{$M^{\Xx_2}_{\hat{\gamma}}$}&
{}&{$M^{\Xx_1}_{\alpha_1}$}&{$M^{\Xx_2}_{\hat{\gamma}}$}\\
  {$M^{\Xx_0}_{\alpha_0}$}&{$M^{\Xx_1}_{\bar{\gamma}}$}&{$M^{\Xx_2}_\gamma$}&
{$M^{\Xx_0}_{\alpha_0}$}&{$M^{\Xx_1}_{\bar{\gamma}}$}&{$M^{\Xx_2}_\gamma$}\\};
\path[->,font=\scriptsize,shorten >= -1pt, shorten <= -1pt]
(m-3-1) edge node[below] {$\pi^{01}_{\alpha_1;\alpha_0}$} (m-3-2)
(m-3-1) edge[bend right] node[below] {$\pi^{02}_{\alpha_2;\alpha_0}$} (m-3-3)
(m-3-1) edge node[below] {$\ \ \ \ \ \ \tau^{01}_{\alpha_1;\alpha_0}$} (m-2-2)
(m-3-1) edge[bend left] node[left] {$\tau^{02}_{\alpha_2;\alpha_0}\ \ $}
(m-1-3)
(m-3-2) edge node[below] {$\pi^{12}_{\alpha_2;\bar{\gamma}}\ \ $} (m-3-3)
(m-3-2) edge node {} (m-2-2)
(m-3-3) edge node {} (m-2-3)
(m-2-2) edge node[below] {$\pi^{12}_{\alpha_2;\alpha_1}$} (m-2-3)
(m-2-2) edge node[below] {$\ \ \ \ \ \ \tau^{12}_{\alpha_2;\alpha_1}$} (m-1-3)
(m-2-3) edge node {} (m-1-3)
(m-3-4) edge node[below] {$\pi^{01}_{\alpha_1;\alpha_0}$} (m-3-5)
(m-3-4) edge[bend right] node[below] {$\pi^{02}_{\alpha_2;\alpha_0}$} (m-3-6)
(m-3-4) edge[dotted] node[below] {$\ \ \ \ \ \ \om^{01}_{\alpha_1;\alpha_0}$}
(m-2-5)
(m-3-4) edge[dotted,bend left] node[left]
{$\om^{02}_{\alpha_2;\alpha_0}\ \ $} (m-1-6)
(m-3-5) edge node[below] {$\pi^{12}_{\alpha_2;\bar{\gamma}}\ \ $} (m-3-6)
(m-3-5) edge[dotted] node {} (m-2-5)
(m-3-6) edge[dotted] node {} (m-2-6)
(m-2-5) edge node[below] {$\pi^{12}_{\alpha_2;\alpha_1}$} (m-2-6)
(m-2-5) edge[dotted] node[below] {$\ \ \ \ \ \ \om^{12}_{\alpha_2;\alpha_1}$}
(m-1-6)
(m-2-6) edge[dotted] node {} (m-1-6);
\end{tikzpicture}
\caption{Commutativity of inflation,
 coarse and fine. In both diagrams,
$\alpha_2\in C^{02}$,
$\alpha_1=f^{12}(\alpha_2)$, $\alpha_0=f^{02}(\alpha_2)=f^{01}(\alpha_1)$,
$\bar{\gamma}=\gamma^{01}_{\alpha_1;\alpha_0}$,
$\gamma=\gamma^{02}_{\alpha_2;\alpha_0}=\gamma^{12}_{\alpha_2;\bar{\gamma}}$
and $\hat{\gamma}=\gamma^{12}_{\alpha_2;\alpha_1}$.
Note $\alpha_2=\delta^{02}_{\alpha_2;\alpha_0}=\delta^{12}_{\alpha_2;\alpha_1}$
and $\alpha_1=\delta^{01}_{\alpha_1;\alpha_0}$
and $\bar{\gamma}\leq^{\Xx_1}\alpha_1$ and
$\gamma\leq^{\Xx_2}\hat{\gamma}\leq^{\Xx_2}\alpha_2$.
Solid arrows indicate total embeddings, and dotted arrows
indicate partial embeddings (the domain and codomain
are initial segments of the models in the figure). The vertical arrows are
(partial) iteration embeddings.
Both diagrams commute, after restricting to common domains in the fine diagram.
 For example,
$\dom(\om^{12}_{\alpha_2;\alpha_1}\com\om^{01}_{\alpha_1;\alpha_0})\sub\dom(\om^
{02}_{\alpha_2;\alpha_0})$
and these maps agree over the smaller domain. Note that in the fine diagram,
while the maps $\om^{k\ell}_{\alpha_\ell;\alpha_k}$ are the only ones displayed
mapping directly between segments of $M^{\Xx_k}_{\alpha_k}$ and
$M^{\Xx_\ell}_{\alpha_\ell}$,
there could be maps $\tau^{k\ell}_{\alpha_\ell;\alpha_k i}$ mapping between
larger segments thereof,
and these also commute with the rest of the
diagram.}\label{fgr:inflation_commutativity}
\end{figure}

\begin{lem}[Commutativity of inflation
(coarse)]\label{clem:inflation_commutativity}\index{commutativity of inflation}
Let $M$ be a wcpm and $\Xx_0$, $\Xx_1$, $\Xx_2$ be normal on $M$,
$\Xx_{i+1}$ an inflation of $\Xx_i$, with $\Xx_1$ being non-$\Xx_0$-pending
\tu{(}but $\Xx_2$ could be $\Xx_1$-pending\tu{)}.
Then $\Xx_2$ is an inflation of $\Xx_0$, and things commute in a reasonable
fashion.
That is, let
\[ (t^{ij},C^{ij},(C^-)^{ij},f^{ij},\left<\Pi^{ij}_\alpha\right>_{\alpha\in
C^{ij}})
=(t,C,\ldots)^{\Xx_i\inflatearrow\Xx_j}\]
for $i<j$; we also use analogous notation for other associated objects.
\tu{(}Note that $C^{ij}=\lh(\Xx_j)$ for each $i,j$, because $M$ is a wcpm.\tu{)}
Let $\alpha_2<\lh(\Xx_2)$ and $\alpha_k=f^{k2}(\alpha_2)$.
Then \tu{(}cf.~Figure \ref{fgr:inflation_commutativity}\tu{)}:

\begin{enumerate}[label=\arabic*.,ref=\arabic*]
\item\label{citem:f_comm}
$\alpha_0=f^{02}(\alpha_2)=f^{01}(f^{12}(\alpha_2))=f^{01}(\alpha_1)$.
 \item\label{citem:t_equiv} Suppose $\alpha_2+1<\lh(\Xx_2)$ and let
$E_2=E^{\Xx_2}_{\alpha_2}$. Then:
 \begin{enumerate}[label=--]
  \item $E_2$ is the $\Xx_0\inflatearrow\Xx_2$-copy of
  an extender $E_0$ \tu{(}so
$E_0=E^{\Xx_0}_{\alpha_0}$\tu{)}
 \end{enumerate}
 iff
 \begin{enumerate}[label=--]
  \item $E_2$ is the $\Xx_1\inflatearrow\Xx_2$-copy of  an extender
  $E_1$ \tu{(}so
$E_1=E^{\Xx_1}_{\alpha_1}$\tu{)} and
  \item $E_1$ is the $\Xx_0\inflatearrow\Xx_1$-copy of $E_0$.
 \end{enumerate}
 That is,  $\alpha_2\in (C^-)^{02}\text{ and }t^{02}(\alpha_2)=0$ iff
\[ \alpha_2\in (C^-)^{12}\text{ and }t^{12}(\alpha_2)=0\text{ and }
\alpha_1\in (C^-)^{01}\text{ and }t^{01}(\alpha_1)=0.\]
\setcounter{enumi}{3}
\item\label{citem:gamma_comm,internal_coverage}
We have:
\begin{enumerate}[label=\tu{(}\alph*\tu{)}]
 \item\label{citem:gamma_comm_in} If $\beta\leq\alpha_0$ and
$\gamma=\gamma^{01}_{\alpha_1;\beta}$
then
$\gamma^{02}_{\alpha_2;\beta}=\gamma^{12}_{\alpha_2;\gamma}$ and
$\pi^{02}_{\alpha_2;\beta}=\pi^{12}_{\alpha_2;\gamma}\com\pi^{01}_{
\alpha_1;\beta}$.
\item\label{citem:interval_coverage_in}
$\bigcup_{\beta\leq
\alpha_0}I^{02}_{\alpha_2;\beta}\sub\bigcup_{\beta\leq \alpha_1}
I^{12}_{\alpha_2;\beta}$.
\item\label{citem:internal_coverage_in_2} If $\beta\leq\alpha_0$ and $\gamma\in
I^{02}_{\alpha_2;\beta}$ then
$f^{12}(\gamma)\in I^{01}_{\alpha_1;\beta}$.
\end{enumerate}
\item\label{citem:extended_comm_at_end}
Let $\gamma^{02}=\gamma^{02}_{{\alpha_2};\alpha_0}$ and
$\gamma^{01}=\gamma^{01}_{{\alpha_1};\alpha_0}$ and
$\gamma^{12}=\gamma^{12}_{{\alpha_2};{\alpha_1}}$ \tu{(}maybe
$\gamma^{12}\neq\gamma^{12}_{\alpha_2;\gamma^{01}}$\tu{)}.
Note that for $k<\ell\leq 2$, we have
\[
\tau^{k\ell}_{{\alpha_\ell};\alpha_k}=j^{\Xx_\ell}_{\gamma^{k\ell}{\alpha_\ell}}
\com\pi^{k\ell}_{{\alpha_\ell};\alpha_k}\colon
M^{\Xx_k}_{\alpha_k i^{k\ell}}\to
M^{\Xx_\ell}_{\alpha_\ell}.\]
Then
$\tau^{02}_{{\alpha_2};\alpha_0}
=\tau^{12}_{{\alpha_2};\alpha_1}\com\tau^{01}_{\alpha_1;\alpha_0}$.
Therefore if $\lh(\Xx_2)=\alpha_2+1$ and $\lh(\Xx_1)=\alpha_1+1$, then
$\pi^{02}_\infty=\pi^{12}_\infty\com\pi^{01}_\infty$.
\end{enumerate}
\end{lem}

\begin{lem}[Commutativity of inflation
(fine)]\label{lem:inflation_commutativity}\index{commutativity of inflation}
Let $M$ be $\udash m$-sound, let $\Xx_0,\Xx_1,\Xx_2$ be $\udash m$-maximal on
$M$,
$\Xx_{i+1}$ an inflation of $\Xx_i$, with $\Xx_1$ non-$\Xx_0$-pending \tu{(}but
$\Xx_2$ could be $\Xx_1$-pending\tu{)}.
Then $\Xx_2$ is an inflation of $\Xx_0$, and things commute in a reasonable
fashion.
That is, let
\[ (t^{ij},C^{ij},(C^-)^{ij},f^{ij},\left<\Pi^{ij}_\alpha\right>_{\alpha\in
C^{ij}})
=(t,C,\ldots)^{\Xx_i\inflatearrow\Xx_j}\]
for $i<j$; we also use analogous notation for other associated objects.
Let $\alpha_2<\lh(\Xx_2)$. If $k<2$ and $\alpha_2\in C^{k2}$ let
$\alpha_k=f^{k2}(\alpha_2)$.
Then \tu{(}cf.~Figure \ref{fgr:inflation_commutativity}, which depicts a key
case of the lemma\tu{)}:

\begin{enumerate}[label=\arabic*.,ref=\arabic*]
\item\label{item:f_comm} If $\alpha_2\in C^{02}$ then
$\alpha_2\in C^{12}$, $\alpha_1\in C^{01}$ and
\[ \alpha_0=f^{02}(\alpha_2)=f^{01}(f^{12}(\alpha_2))=f^{01}(\alpha_1).\]
 \item\label{item:t_equiv} Suppose $\alpha_2+1<\lh(\Xx_2)$ and let
$E_2=E^{\Xx_2}_{\alpha_2}$. Then:
 \begin{enumerate}[label=--]
  \item $E_2$ is the $\Xx_0\inflatearrow\Xx_2$-copy of an extender $E_0$
\tu{(}so $E_0=E^{\Xx_0}_{\alpha_0}$\tu{)}
 \end{enumerate}
 iff
 \begin{enumerate}[label=--]
  \item $E_2$ is the $\Xx_1\inflatearrow\Xx_2$-copy of an extender $E_1$
\tu{(}so $E_1=E^{\Xx_1}_{\alpha_1}$\tu{)}, and
  \item $E_1$ is the $\Xx_0\inflatearrow\Xx_1$-copy of $E_0$.
 \end{enumerate}
 That is,  $\alpha_2\in (C^-)^{02}\text{ and }t^{02}(\alpha_2)=0$ iff
\[ \alpha_2\in (C^-)^{12}\text{ and }t^{12}(\alpha_2)=0\text{ and }
\alpha_1\in (C^-)^{01}\text{ and }t^{01}(\alpha_1)=0.\]
\item\label{item:gamma_in_C} Suppose $\alpha_2\in C^{12}$ and $\alpha_1\in
C^{01}$.\footnote{This does not imply that $\alpha_2\in C^{02}$,
so $\alpha_0$ might not be defined, although $f^{01}(\alpha_1)$ is.} Then:
\begin{enumerate}[label=\tu{(}\alph*\tu{)}]
\item If $\alpha_1+1=\lh(\Xx_1)$ then $\alpha_2\in C^{02}$.
\item\label{item:gamma_gamma_in_C^02} If $\beta\leq f^{01}(\alpha_1)$ and
$\xi\in I^{01}_{\alpha_1;\beta}$ then
$\gamma^{12}_{\alpha_2;\xi}\in C^{02}$.
\item If $\beta<f^{01}(\alpha_1)$ and
$\xi=\delta^{01}_{\alpha_1;\beta}$ then $\delta^{12}_{\alpha_2;\xi}\in C^{02}$.
\end{enumerate}

\item\label{item:gamma_comm,internal_coverage}
Suppose $\alpha_2\in C^{02}$. Then:
\begin{enumerate}[label=\tu{(}\alph*\tu{)}]
 \item\label{item:gamma_comm_in} If $\beta\leq\alpha_0$ and
$\gamma=\gamma^{01}_{\alpha_1;\beta}$
then
$\gamma^{02}_{\alpha_2;\beta}=\gamma^{12}_{\alpha_2;\gamma}$ and
$\pi^{02}_{\alpha_2;\beta}=\pi^{12}_{\alpha_2;\gamma}\com\pi^{01}_{
\alpha_1;\beta}$.
\item\label{item:interval_coverage_in}
$\bigcup_{\beta\leq
\alpha_0}I^{02}_{\alpha_2;\beta}\sub\bigcup_{\beta\leq \alpha_1}
I^{12}_{\alpha_2;\beta}\sub C^{12}$.
\item\label{item:internal_coverage_in_2} If $\beta\leq\alpha_0$ and $\gamma\in
I^{02}_{\alpha_2;\beta}$ then
$f^{12}(\gamma)\in I^{01}_{\alpha_1;\beta}$.
\end{enumerate}
\item\label{item:extended_comm_at_end} Suppose $\alpha_2\in C^{02}$.
Let $\gamma^{02}=\gamma^{02}_{{\alpha_2};\alpha_0}$ and
$\gamma^{01}=\gamma^{01}_{{\alpha_1};\alpha_0}$ and
$\gamma^{12}=\gamma^{12}_{{\alpha_2};{\alpha_1}}$ \tu{(}maybe
$\gamma^{12}\neq\gamma^{12}_{\alpha_2;\gamma^{01}}$\tu{)} and
$i^{02}=i^{02}_{{\alpha_2};\alpha_0}$ and
$i^{01}=i^{01}_{{\alpha_1};\alpha_0}$ and
$i^{12}=i^{12}_{{\alpha_2};{\alpha_1}}$.
Note that
\[ \tau^{k\ell}_{{\alpha_\ell};\alpha_k
i^{k\ell}}=j^{\Xx_\ell}_{\gamma^{k\ell},{\alpha_\ell}}
\com
\pi^{k\ell}_{{\alpha_\ell};\alpha_k}
\rest M^{\Xx_k}_{\alpha_k i^{k\ell}}
\colon
M^{\Xx_k}_{\alpha_k i^{k\ell}}\to
M^{\Xx_\ell}_{\alpha_\ell}\]
for $k<\ell\leq 2$. Then we have:
\begin{enumerate}[label=\tu{(}\alph*\tu{)}]
\item\label{item:comm_inf_5a}$i^{01}\leq i^{02}$ \tu{(}so $M^{\Xx_0}_{\alpha_0
i^{02}}\ins
M^{\Xx_0}_{\alpha_0 i^{01}}$, with equality iff $i^{02}=i^{01}$\tu{)}.
\item $i^{01}+i^{12}=i^{02}$.
\item $i^{01}=i^{02}_{\gamma^{12};\alpha_0}$; that is, $i^{01}$ is the least
$i'$ such that $\gamma^{12}\in
I^{02}_{{\alpha_2};\alpha_0 i'}$.
\item\label{item:taus_comm} if $i=i^{01}=i^{02}$ \tu{(}which holds iff
$i^{12}=0$ iff
$(\gamma^{12},\alpha_2]_{\Xx_2}\inter\dropset^{\Xx_2}=\emptyset$\tu{)}
then
\[
\tau^{02}_{{\alpha_2};\alpha_0 i} =\tau^{12}_{{\alpha_2};\alpha_1
0}\com\tau^{01}_{\alpha_1;\alpha_0 i}. \]
\item\label{item:taus_comm_when_extra_drop} Suppose $i^{01}<i^{02}$ \tu{(}which
holds iff $i^{12}>0$ iff
$(\gamma^{12},{\alpha_2}]_{\Xx_2}\inter\dropset^{\Xx_2}\neq\emptyset$ iff
$M^{\Xx_0}_{\alpha_0 i^{02}}\pins M^{\Xx_0}_{\alpha_0 i^{01}}$\tu{)}. Then
$M^{\Xx_1}_{\alpha_1 i^{12}}=\tau^{01}_{\alpha_1;\alpha_0
i^{01}}(M^{\Xx_0}_{\alpha_0 i^{02}})$ and
\[ \tau^{02}_{{\alpha_2};\alpha_0
i^{02}}=\tau^{12}_{{\alpha_2};\alpha_1i^{12}}\com
(\tau^{01}_{\alpha_1;\alpha_0 i^{01}}\rest
M^{\Xx_0}_{\alpha_0 i^{02}}).\]
\end{enumerate}
Therefore if also $\lh(\Xx_2)=\alpha_2+1$ and $\lh(\Xx_1)=\alpha_1+1$ \tu{(}so
$\alpha_0+1=\lh(\Xx_0)$ and $i^{02}=i^{01}=0=i^{12}$,
because $\Xx_1$ is non-$\Xx_0$-pending\tu{)}, then
\[ \pi^{02}_\infty=\pi^{12}_\infty\com\pi^{01}_\infty. \]
\end{enumerate}
\end{lem}

We literally only prove the fine version; the coarse version is easier.

\begin{proof}[Proof of Lemma \ref{lem:inflation_commutativity}]
By induction on $\lh(\Xx_2)$. Fix $\alpha_2+1<\lh(\Xx_2)$ and suppose that the
lemma holds with respect
to $\Xx_2\rest(\alpha_2+1)$. We consider three cases.

\begin{case}\label{case:copy_copy}
$\alpha_2$ is an $\Xx_1$-copying stage of $\Xx_2$,
and $\alpha_1$ is an $\Xx_0$-copying stage of $\Xx_1$
(that is, $\alpha_2\in(C^-)^{12}$ and $t^{12}(\alpha_2)=0$ and
$\alpha_1\in(C^-)^{01}$ and $t^{01}(\alpha_1)=0$).

We first verify that $\alpha_2\in C^{02}$, and establish some other facts. Let
$\alpha'_0=f^{01}(\alpha_1)$.
(We don't yet know $\alpha_2\in C^{02}$, so we don't yet write $\alpha_0$.)
We have
$\delta^{12}_{\alpha_2;\alpha_1}=\alpha_2$ and
$\delta^{01}_{\alpha_1;\alpha'_0}=\alpha_1$.
Let $\bar{\gamma}=\gamma^{01}_{\alpha_1;\alpha'_0}$ and
$\gamma=\gamma^{12}_{\alpha_2;\bar{\gamma}}$ and
$\hat{\gamma}=\gamma^{12}_{\alpha_2;\alpha_1}$.
Since $\bar{\gamma}\leq^{\Xx_1}\alpha_1$, we have
$\gamma\leq^{\Xx_2}\hat{\gamma}$.
And $\hat{\gamma}\in C^{02}$ by property
\ref{item:gamma_in_C}\ref{item:gamma_gamma_in_C^02} (applied with
$\beta=\alpha'_0$ and $\xi=\alpha_1$),
so $[0,\hat{\gamma}]_{\Xx_2}\sub C^{02}$, so $\gamma\in C^{02}$.
By \ref{dfn:inflation}(\ref{item:inflation_internal_agreement}),
$\gammabar\in C^{01}$ and $\alpha_0'=f^{01}(\gammabar)$
and
$\gammabar=\gamma^{01}_{\gammabar;\alpha_0'}$,
 and likewise,
$\gamma\in C^{12}$ and $\gammabar=f^{12}(\gamma)$
and
$\gamma=\gamma^{12}_{\gamma;\gammabar}$.
Since $\gamma\in C^{02}$, therefore by induction with property
\ref{item:gamma_comm,internal_coverage}\ref{item:gamma_comm_in} (applied with
$\gamma$ replacing $\alpha_2$),
we have
\[ \alpha_0'=f^{02}(\gamma)=f^{01}(f^{12}(\gamma))
\text{ and
}\gamma=\gamma^{02}_{\gamma;\alpha_0'}=\gamma^{12}_{\gamma;\gammabar},\]
\[
\pi^{01}_{\alpha_1;\alpha_0'}=\pi^{01}_{\gammabar;\alpha_0'}:M^{\Xx_0}_{
\alpha_0'}\to
M^{\Xx_1}_{\gammabar},
\]
\[
\pi^{12}_{\alpha_2;\gammabar}=\pi^{12}_{\gamma;\gammabar}:M^{\Xx_1}_{\gammabar}
\to
M^{\Xx_2}_{\gamma}, \]
\[ \pi^{02}_{\gamma;\alpha_0'}:M^{\Xx_0}_{\alpha_0'}\to M^{\Xx_2}_{\gamma}, \]
\begin{equation}\label{eqn:pi_gamma_com}
\pi^{02}_{\gamma;\alpha_0'}=\pi^{12}_{\alpha_2;\gammabar}
\com\pi^{01}_{\alpha_1;\alpha_0'}. \end{equation}
We have $t^{02}(\xi)=1$ for all $\xi+1\in(\gamma,\alpha_2]_{\Xx_2}$.
For otherwise, by induction (property \ref{item:t_equiv}),
\[ \xi\in(C^-)^{12}\text{ and }t^{12}(\xi)=0\text{ and }
\zeta=f^{12}(\xi)\in(C^-)^{01}\text{ and }t^{01}(\zeta)=0.\]
So $\xi+1=\gamma^{12}_{\alpha_2;\zeta+1}$ and
$\zeta+1\in(\gammabar,\alpha_1]_{\Xx_1}$.
But $[\gammabar,\alpha_1]_{\Xx_1}=I^{01}_{\alpha_1;\alpha_0'}$,
so then $t^{01}(\zeta)=1$, contradiction. Let $Q_0=\exit^{\Xx_0}_{\alpha_0'}$
and $\Qbar=\pi^{01}_{\alpha_1;\alpha_0'}(Q_0)$. So to
verify $\alpha_2\in C^{02}$ we just need to see that
$(\gamma,\alpha_2]_{\Xx_2}$
does not drop strictly
below the
iteration image of
\[ Q\eqdef\pi^{02}_{\gamma;\alpha_0'}(Q_0)=\pi^{12}_{\alpha_2;\gammabar}
\com\pi^{01}_{\alpha_1;\alpha_0'}(Q_0)=\pi^{12}_{\alpha_2;\gammabar}(\Qbar). \]

Note that $j^{\Xx_2}_{\gammahat\alpha_2}$ is defined,
as $[\gammahat,\alpha_2]_{\Xx_2}=I^{12}_{\alpha_2;\alpha_1}$
(we only defined such embeddings for such intervals),
and $\dom(j^{\Xx_2}_{\gammahat\alpha_2})$ is in the dropdown sequence
of
$(M^{\Xx_2}_{\gammahat},\pi^{12}_{\alpha_2;\alpha_1}(\exit^{\Xx_1}_{\alpha_1}
))$.
Likewise, $j^{\Xx_1}_{\gammabar\alpha_1}$ is defined, with
$A=\dom(j^{\Xx_1}_{\gammabar\alpha_1})$ in the dropdown sequence of
$(M^{\Xx_1}_{\gammabar},\pi^{01}_{\alpha_1;\alpha_0'}(\exit^{\Xx_0}_{\alpha_0'}
))$;
in fact for each $\beta\in[\gammabar,\alpha_1]_{\Xx_1}$,
$\dom(j^{\Xx_1}_{\gammabar\beta})$
is in this dropdown sequence.
Let $A'=\pi^{12}_{\alpha_2;\gammabar}(A)$ (where $A'=M^{\Xx_2}_\gamma$ if
$A=M^{\Xx_1}_{\gammabar}$) and
\[ k^{\Xx_2}_{\gamma\gammahat}:A'\to M^{\Xx_2}_{\gammahat} \]
be the composition of iteration maps along $(\gamma,\gammahat]_{\Xx_2}$.
This makes sense and we get
\begin{equation}\label{eqn:tree_emb_comm}\pi^{12}_{\alpha_2;\alpha_1}\com
j^{\Xx_1}_{\gammabar\alpha_1}=k^{\Xx_2}_{\gamma\gammahat}\com\pi^{12}_{
\alpha_2;\gammabar} \end{equation}
by the commutativity of tree embedding maps with iteration maps, and
preservation of dropping segments under tree embedding maps.
Since $t^{01}(\alpha_1)=0$ and $\alpha_1=\delta^{01}_{\alpha_1;\alpha_0'}$,
\[ \exit^{\Xx_1}_{\alpha_1}= Q_1\eqdef j^{\Xx_1}_{\gammabar\alpha_1}(\Qbar).\]
Since $t^{12}(\alpha_2)=\alpha_1$ and
$\alpha_2=\delta^{12}_{\alpha_2;\alpha_1}$, letting
$\Qhat=\pi^{12}_{\alpha_2;\alpha_1}(Q_1)$,
\[ \exit^{\Xx_2}_{\alpha_2}=Q_2\eqdef j^{\Xx_2}_{\gammahat\alpha_2}(\Qhat), \]
and in particular, $(\gammahat,\alpha_2]_{\Xx_2}$ does not drop below the
iteration image of $\Qhat$.
But by line (\ref{eqn:tree_emb_comm}),
\[
\Qhat=k^{\Xx_2}_{\gamma\gammahat}(\pi^{12}_{\alpha_2;\gammabar}(\Qbar))=k^{\Xx_2
}_{\gamma\gammahat}(Q). \]

So $[\gamma,\alpha_2)_{\Xx_2}$ does not drop below the image of $Q$, as desired.

So $\alpha_2\in C^{02}$, so by induction, properties
\ref{item:f_comm}, \ref{item:gamma_comm,internal_coverage} and
\ref{item:extended_comm_at_end}
hold for $\alpha_2$, and in particular, $\alpha_0=f^{02}(\alpha_2)=\alpha_0'$.
Since $\alpha_1\in(C^-)^{01}$, we have $\alpha_0+1<\lh(\Xx_0)$,
so $\alpha_2\in(C^-)^{02}$. And since $t^{01}(\alpha_1)=0$ and
$t^{12}(\alpha_2)=0$, property \ref{item:extended_comm_at_end}
(note in particular its parts \ref{item:taus_comm}
and \ref{item:taus_comm_when_extra_drop})
implies
$E^{\Xx_2}_{\alpha_2}=E^{\Xx_0\inflatearrow\Xx_2}_{\alpha_2}$, so
$t^{02}(\alpha_2)=0$, completing the proof of property
\ref{item:t_equiv}. The same property also gives
\begin{equation}\label{eqn:omega_comm}
\omega^{02}_{\alpha_2;\alpha_0}=\omega^{12}_{\alpha_2;\alpha_1}\com\omega^{01}_{
\alpha_1;\alpha_0}
\end{equation}
(including that these maps have the same domain and codomain).
And note that $\alpha_2+1\in C^{02}\inter C^{12}$ and $\alpha_1+1\in C^{01}$,
and properties \ref{item:f_comm} and \ref{item:gamma_in_C} at $\alpha_2+1$
follow immediately.

We now  verify property \ref{item:gamma_comm,internal_coverage}
for $\alpha_2+1$. Now
$\gamma^{02}_{\alpha_2+1;\alpha_0+1}=\alpha_2+1=\gamma^{12}_{
\alpha_2+1;\alpha_1+1}$ and
$\gamma^{01}_{\alpha_1+1;\alpha_0+1}=\alpha_1+1$,
by definition of the one-step copy extension.
So because of the agreement
between $\Pi^{02}_{\alpha_2+1}$ and $\Pi^{02}_{\alpha_2}$, etc,
and by induction, it easily suffices to see that
\begin{equation}\label{eqn:pi_succ_comm}
\pi^{02}_{\alpha_2+1;\alpha_0+1}=\pi^{12}_{\alpha_2+1;\alpha_1+1}\com\pi^{01}_{
\alpha_1+1;\alpha_0+1}. \end{equation}
Let $\xi_0=\pred^{\Xx_0}(\alpha_0+1)$ and
$\kappa_0=\crit(E^{\Xx_0}_{\alpha_0})$.
So $M^{*\Xx_0}_{\alpha_0+1}=M^{\Xx_0}_{\xi_0\kappa_0}$.
As $t^{01}(\alpha_1)=0$ (recall the definitions of
$\gamma_{\Pi\xi\kappa},P_{\Pi\xi\kappa},\pi_{\Pi\xi\kappa}$
from \ref{dfn:pi_beta,kappa}),
\[ \pred^{\Xx_1}(\alpha_1+1)=\xi_1\eqdef\gamma^{01}_{\alpha_1;\xi_0\kappa_0}\in
I^{01}_{\alpha_1;\xi_0},\]
\[ M^{*\Xx_1}_{\alpha_1+1}= P^{01}_{\alpha_1;\xi_0\kappa_0}.\]
Let $\pi^{01}=\pi^{01}_{\alpha_1;\xi_0\kappa_0}$ and
$\kappa_1=\pi^{01}(\kappa_0)=\crit(E^{\Xx_1}_{\alpha_1})$.
We have
\[
\pred^{\Xx_2}(\alpha_2+1)=\xi_2\eqdef\gamma^{02}_{\alpha_2;\xi_0\kappa_0}
=\gamma^{12}_{\alpha_2;\xi_1\kappa_1}\in I^{02}_{\alpha_2;\xi_0}\inter
I^{12}_{\alpha_2;\xi_1}, \]
\[
M^{*\Xx_2}_{\alpha_2+1}=P^{02}_{\alpha_2;\xi_0\kappa_0}=P^{12}_{
\alpha_2;\xi_1\kappa_1}, \]
with the equalities holding because
$t^{02}(\alpha_2)=t^{12}(\alpha_2)=t^{01}(\alpha_1)=0$ and inflations can be
freely extended.
Let $\pi^{12}=\pi^{12}_{\alpha_2;\xi_1\kappa_1}$ and
$\pi^{02}=\pi^{02}_{\alpha_2;\xi_0\kappa_0}$,
so $\pi^{02}(\kappa_0)=\crit(E^{\Xx_2}_{\alpha_2})=\pi^{12}(\kappa_1)$.
Using part \ref{item:extended_comm_at_end} (with $\xi_2$ in place of
$\alpha_2$;
note that $\xi_2\in C^{02}$), it is now easy to verify that
$\pi^{02}=\pi^{12}\com\pi^{01}$. But
$\pi^{02}_{\alpha_2+1;\alpha_0+1}$, etc, are defined as in the proof of the
Shift Lemma
from $\pi^{02}$ and $\om^{02}_{\alpha_2;\alpha_0}$, etc. So
line (\ref{eqn:pi_succ_comm}) follows from this commutativity and line
(\ref{eqn:omega_comm}).

Finally note that part \ref{item:extended_comm_at_end} for $\alpha_2+1$
follows immediately by induction and from part
\ref{item:gamma_comm,internal_coverage}, because for the new ordinal
$\alpha_2+1$,
with notation as in part \ref{item:extended_comm_at_end}, we have
$\gamma^{02}=\alpha_2+1$, $i^{02}=0$, etc,
so $\tau^{02}_{{\alpha_2+1};\alpha_0+1,0}=\pi^{02}_{\alpha_2+1;\alpha_0+1}$,
etc.

This completes the induction step in this case.
\end{case}

\begin{case}\label{case:Xx_1-inflationary} $\alpha_2$ is $\Xx_1$-inflationary
(that is,
$t^{12}(\alpha_2)=1$).

Then $t^{02}(\alpha_2)=1$, so part \ref{item:t_equiv} holds.
For if $\alpha_2\in(C^-)^{02}$ then by induction,
$\alpha_2\in C^{12}$ and $\alpha_1\in C^{01}$ and $f^{01}(\alpha_1)=\alpha_0$,
hence $\alpha_1\in(C^-)^{01}$, but then since $\Xx_1$ is non-$\Xx_0$-pending,
$\alpha_1+1<\lh(\Xx_1)$, so $\alpha_2\in(C^-)^{12}$ and
$\In(E^{\Xx_1}_{\alpha_1})\leq\In(E^{\Xx_0\inflatearrow\Xx_1}_{\alpha_1})$,
so (as $t^{12}(\alpha_2)=1$)
$\In(E^{\Xx_2}_{\alpha_2})<\In(E^{\Xx_1\inflatearrow\Xx_2}_{\alpha_2})\leq\In(E^
{\Xx_0\inflatearrow\Xx_2}_{\alpha_2})$
 by commutativity. Let $\xi_2=\pred^{\Xx_2}(\alpha_2+1)$.

 Part \ref{item:f_comm}:
 Suppose $\alpha_2+1\in C^{02}$. Then
$\xi_2\in C^{02}$; let $\xi_0=f^{02}(\xi_2)$ and $\xi_1=f^{12}(\xi_2)$,
so also $\xi_1\in C^{01}$ and $\xi_0=f^{01}(\xi_1)$.
And $E^{\Xx_2}_{\alpha_2}$ is total over $Q^{02}_{\xi_2;\xi_0}$.
But if $\xi_1+1<\lh(\Xx_1)$ then
$\exit^{\Xx_1}_{\xi_1}\ins Q^{01}_{\xi_1;\xi_0}$
and if $\xi_1+1=\lh(\Xx_1)$ then (because $\Xx_1$ is non-$\Xx_0$-pending)
$\xi_0+1=\lh(\Xx_0)$ and
$M^{\Xx_1}_{\xi_1}=Q^{01}_{\xi_1;\xi_0}$.
So $Q^{12}_{\xi_2;\xi_1}\ins Q^{02}_{\xi_2;\xi_0}$.
So $E^{\Xx_2}_{\alpha_2}$ is total over $Q^{12}_{\xi_2;\xi_1}$.
So $\alpha_2+1\in C^{12}$ and
$f^{12}(\alpha_2+1)=f^{12}(\xi_2)=\xi_1\in C^{01}$.
Likewise $f^{02}(\alpha_2+1)=\xi_0$, giving part \ref{item:f_comm}.

Parts \ref{item:gamma_in_C} and \ref{item:gamma_comm,internal_coverage}
are easy by induction.

Part \ref{item:extended_comm_at_end}: Suppose $\alpha_2+1\in C^{02}$
and continue with the notation above.
Now $\Pi^{i2}_{\alpha_2+1}$ is the
$E^{\Xx_2}_{\alpha_2}$-inflation of $\Pi^{i2}_{\xi_2}$ for $i=0,1$.
But then property \ref{item:extended_comm_at_end}
at $\alpha_2+1$ follows easily from the same property at $\xi_2$;
we get the instance of Figure \ref{fgr:inflation_commutativity} at stage
$\alpha_2+1$,
from that at stage $\xi_2$, by simply adding one further step of iteration
above
$M^{*\Xx_2}_{\alpha_2+1}\ins M^{\Xx_2}_{\xi_2}$ (at the top of the diagram).
(This possibly inflicts a drop in model, but because $\alpha_2+1\in C^{02}$,
hence also $\alpha_2+1\in C^{12}$, we do not drop too far;
the integer $i^{01}$ is not modified, and the integers $i^{02}$ and $i^{12}$
are
modified by the same amount.)
 \end{case}

 \begin{case}\label{case:copy_inflation}
 $\alpha_2$ is $\Xx_1$-copying but $\alpha_1$ is $\Xx_0$-inflationary
(that is, $\alpha_2\in(C^-)^{12}$ and $t^{12}(\alpha_2)=0$ but
$t^{01}(\alpha_1)=1$).

We have $\alpha_2+1\in C^{12}$ and $f^{12}(\alpha_2+1)=\alpha_1+1$
and $\gamma^{12}_{\alpha_2+1;\alpha_1+1}=\alpha_2+1$.
And $t^{02}(\alpha_2)=1$ for reasons much as before, giving part
\ref{item:t_equiv}.
Let $\xi_i=\pred^{\Xx_i}(\alpha_i+1)$ for $i=1,2$.
Then $\xi_2\in C^{12}$ and $f^{12}(\xi_2)=\xi_1$.
By commutativity at stage $\xi_2$, we easily have $\alpha_2+1\in C^{02}$
iff $\alpha_1+1\in C^{01}$; and if $\alpha_2+1\in C^{02}$ then, letting
$\xi_0=f^{02}(\xi_2)=f^{01}(\xi_1)$,
we have $f^{02}(\alpha_2+1)=\xi_0=f^{01}(\alpha_1+1)$, since
$t^{02}(\alpha_2)=t^{01}(\alpha_1)=1$.
So part \ref{item:f_comm} holds.

Parts \ref{item:gamma_in_C} and \ref{item:gamma_comm,internal_coverage} are
again easy.
(In part \ref{item:gamma_in_C}\ref{item:gamma_gamma_in_C^02},
for $\alpha_2+1$
and $\beta=\xi_0$ and $\xi=\alpha_1+1\in I^{01}_{\alpha_1+1;\xi_0}$,
we have
$\gamma^{12}_{\alpha_2+1;\alpha_1+1}=\alpha_2+1\in C^{02}$,
as required.)
And part \ref{item:extended_comm_at_end} is again straightforward by
induction;
we obtain the diagram at stage $\alpha_2+1$ by adding a commuting square
to the top of diagram from stage $\xi_2$, applying the extenders
$E^{\Xx_1}_{\alpha_1}$ and $E^{\Xx_2}_{\alpha_2}$\
to $M^{*\Xx_1}_{\alpha_1+1}$ and $M^{*\Xx_2}_{\alpha_2+2}$ respectively;
in the new diagram the upper triangle collapses.
\end{case}

This completes the successor case. The limit case is a simplification thereof.
Suppose that the lemma holds with regard to $\Xx_2\rest\eta$, where $\eta$ is a
limit,
and we want to prove it for $\Xx_2\rest\eta+1$.
There are again three cases, analogous to those in the successor case.
For an inflation $\Tt\inflatearrow\Xx$, with associated objects $C,f$, and a
limit $\eta<\lh(\Xx)$,
say that $\eta$ is a \dfnemph{$(\Tt,\Xx)$-limit}\index{$(\Tt,\Xx)$-limit} iff
$\eta\in C$ and
$f(\alpha)<f(\eta)$ for all $\alpha<^\Xx\eta$.
Then either:
\begin{enumerate}[label=\arabic*.,ref=\arabic*]
 \item  $\eta$ is an $(\Xx_0,\Xx_2)$-limit.
Then easily by induction, $\eta$ is also an $(\Xx_1,\Xx_2)$-limit and
$f^{12}(\eta)$
is an $(\Xx_0,\Xx_1)$-limit. This is analogous to Case \ref{case:copy_copy}
(an $\Xx_0$-copying (and $\Xx_1$-copying) stage of $\Xx_2$).
\item $\eta$ is not an $(\Xx_1,\Xx_2)$-limit. So $\eta$ is also not an
$(\Xx_0,\Xx_2)$-limit.
(Analogous to Case \ref{case:Xx_1-inflationary}, an $\Xx_1$-inflationary stage
of $\Xx_2$.)
\item $\eta$ is an $(\Xx_1,\Xx_2)$-limit, but not an $(\Xx_0,\Xx_2)$-limit.
Then $f^{12}(\eta)$ is not an $(\Xx_0,\Xx_1)$-limit. (Analogous to Case
\ref{case:copy_inflation},
an $\Xx_1$-copying, $\Xx_0$-inflationary stage of $\Xx_2$.)
\end{enumerate}
In each case, the properties follow easily from the commutativity given by
induction. We leave the details to the reader.
\end{proof}

An easy consequence is:

\begin{cor}\label{cor:terminal_dropping_equiv} Let $\Xx_0,\Xx_1,\Xx_2$ be as in
\ref{lem:inflation_commutativity}.
Suppose that $\Xx_2$ is $\Xx_1$-terminal and $\Xx_1$ is
$\Xx_0$-terminal.
Then $\Xx_2$ is $\Xx_0$-terminal. Moreover, $\Xx_2$ is
$\Xx_0$-terminally-\tu{(}model-\tu{)}dropping
iff either $\Xx_1$ is $\Xx_0$-terminally-\tu{(}model-\tu{)}dropping
or $\Xx_2$ is $\Xx_1$-terminally-\tu{(}model-\tu{)}dropping.
\end{cor}

The author was initially focused on inflation (as opposed to tree embeddings
more generally),
and did not notice that the preceding lemma has the following natural variant,
until it was pointed out by Jensen.
It follows from part of the proof of
\ref{lem:inflation_commutativity}:

\begin{lem}[Composition of tree
embeddings]\label{lem:comp_tree_emb}\index{composition (tree embeddings)}
 Let $\Xx_i$ be $\udash m$-maximal trees for $i=0,1,2$.
 Let
$\Pi_{i,i+1}:\Xx_i\hookrightarrow\Xx_{i+1}$ be a tree embedding, for
$i=0,1$.
 Then
  $\Pi_{02}:\Xx_0\hookrightarrow\Xx_2$
 is a tree embedding, where writing
$\gamma^{ij}_\alpha=\gamma_{\Pi_{ij}\alpha}$, etc, we have
 \[ \gamma^{02}_\alpha=\gamma^{12}_{\gamma^{01}_\alpha}\text{ and
}\delta^{02}_\alpha=\delta^{12}_{\delta^{01}_\alpha}. \]
 for each $\alpha<\lh(\Xx_0)$. Moreover, for each $\alpha<\lh(\Xx_0)$ we have
 \[ \pi^{02}_\alpha=\pi^{12}_{\gamma^{01}_\alpha}\com\pi^{01}_\alpha\text{ and
}
\om^{02}_\alpha=\om^{12}_{\delta^{01}_\alpha}\com\om^{01}_\alpha.\]
\end{lem}

\section{Generic absoluteness of iterability}\label{sec:gen_abs_mice}

We establish in this section some general theorems on the absoluteness of
iterability
under forcing. Let $M$ be an $m$-sound premouse.
Let $\Omega>\om$ be regular and let $V[G]$ be a generic extension of $V$
via an $\Omega$-cc forcing.
In the main result (Theorem \ref{thm:strat_with_cond_extends_to_generic_ext}),
assuming that $\Sigma$ is an $(m,\Omega+1)$-strategy for $M$ with strong hull
condensation,
we extend $\Sigma$ to $\Sigma'$, such that in $V[G]$, $\Sigma'$ is an
$(m,\Omega+1)$-strategy with strong hull condensation.
(We do not know whether the analogous statement can be proved for inflation
condensation.) This holds for both wcpms and seg-pms, of arbitrary cardinality.
If $M$ is a countable premouse and $e$ an $\om$-enumeration of $M$
and $\Sigma$ has weak DJ with respect to $e$, then so does $\Sigma'$.
We also use the result to obtain a universally Baire representation for
$\Sigma\rest\HC$,
assuming that $M$ is also countable (see \S\ref{subsec:univ_baire}).
In the other direction (Corollary \ref{cor:wDJ_absoluteness}), assume that $M$
is
countable in $V$ and $\Sigma'$ has weak DJ in $V[G]$
with respect to some enumeration $e\in V$;
then $\Sigma=\Sigma'\rest V\in V$. The proof involves standard kinds of
arguments and is probably part of the folklore,
but we give it.
Thus, if $M$ is a countable premouse and $e\in V$ an $\om$-enumeration of $M$,
then the existence of an $(m,\Omega+1)$-strategy for $M$ with weak DJ with
respect to $e$
is absolute between $V$ and $V[G]$.
Combined with the results later in the paper, we will also get that if
$V\sats\ZFC$ and $M$ is countable,
then the existence of an $(m,\Omega+1)$-strategy for $M$ with strong hull
condensation
is absolute between $V$ and $V[G]$; this is because under $\DC$, given such a
strategy and an enumeration $e$,
we can construct a strategy with weak DJ with respect to $e$.

\subsection{Extending strategies to generic extensions}

The background theory here, as elsewhere,
is $\ZF$. Thus, we specify exactly what we mean by the $\Omega$-chain condition:

\begin{dfn}\label{dfn:Omega-cc}\index{pre-antichain}\index{chain condition, cc}
Let $\PP$ be a poset
and $\lambda\in\OR$.
A \dfnemph{$\lambda$-pre-antichain} of $\PP$ is a partition
$\left<A_\alpha\right>_{\alpha<\lambda}$ of some set $A\sub\PP$
such that each $A_\alpha\neq\emptyset$, and $p\incompat q$
whenever $p\in A_\alpha$ and $q\in A_\beta$ for some $\alpha<\beta<\lambda$.
We say that $\PP$ has the \dfnemph{$\lambda$-cc} iff there is no
$\lambda$-pre-antichain of $\PP$.
\end{dfn}

\begin{rem}
Clearly the above definition agrees with the usual definition of $\lambda$-cc
under $\ZFC$.
The usual $\ZFC$ argument easily adapts to show under $\ZF$ that
if $\lambda$ is regular then forcing with a $\lambda$-cc forcing preserves the
regularity of $\lambda$.
\end{rem}

\begin{tm}\label{thm:strat_with_cond_extends_to_generic_ext}
Let $\Omega>\om$ be regular.
Let $\PP$ be an $\Omega$-cc forcing and $G$ be $V$-generic for $\PP$.
Let $M$ be an $\ell$-sound premouse, or let $M$ be a wcpm and $\ell=0$.
Let $\Gamma$ be an $(\ell,\Omega+1)$-strategy\footnote{Recall that if $M$ is a
wcpm,
this just means an $(\Omega+1)$-strategy.} for $M$ with strong hull
condensation.
Then:

\begin{enumerate}[label=\arabic*.,ref=\arabic*]
\item\label{item:Gamma'_exists_unique} In $V[G]$ there is a unique
$(\ell,\Omega+1)$-strategy $\Gamma'$ such that $\Gamma\sub\Gamma'$ and
$\Gamma'$
has inflation condensation.
\item\label{item:Gamma'_shc} In $V[G]$, $\Gamma'$ has strong hull condensation.
\item\label{item:Gamma,Gamma'_DJ} Suppose $M\in\HC$ is a premouse \tu{(}not a
wcpm\tu{)}
and let $e$ be an enumeration of $M$ in ordertype $\om$.
Then:
\begin{enumerate}[label=--]
\item $\Gamma$ has Dodd-Jensen iff $\Gamma'$ has Dodd-Jensen in $V[G]$.
\item $\Gamma$ has weak Dodd-Jensen with respect to $e$
iff $\Gamma'$ has weak Dodd-Jensen with respect to $e$ in $V[G]$.
\end{enumerate}
\end{enumerate}
Further, let $\Sigma$ be the $\udash$strategy corresponding to $\Gamma$
and $m=m^\Sigma$.\footnote{See \ref{dfn:strategy_classes}. So if $M$ is not
MS-indexed
then $\Gamma=\Sigma$ and $m=\ell$.} Then:
\begin{enumerate}[resume*]
\item\label{item:Sigma'_exists_unique} In $V[G]$ there is a unique $(\udash
m,\Omega+1)$-strategy $\Sigma'$ such that $\Sigma\sub\Sigma'$ and $\Sigma'$ has
inflation condensation.
\item\label{item:Sigma'_shc} In $V[G]$, $\Sigma'$ has strong hull condensation.
\item\label{item:Sigma'_corresp_Gamma'} If $M$ is MS-indexed then in $V[G]$,
$\Sigma'$ is the $\udash$strategy corresponding to $\Gamma'$.
\item\label{item:trees_via_Sigma'_embed} For every tree $\Tt\in V[G]$ via
$\Sigma'$, there is a $\Tt$-terminally-non-dropping inflation $\Xx$ of $\Tt$
such that $\Xx\in V$ and $\Xx$ is via $\Sigma$. Moreover,
if $\lh(\Tt)<\Omega$ then we can take $\lh(\Xx)<\Omega$.
\end{enumerate}
\end{tm}

\begin{proof}
We just prove the fine-structural variants; the version for wcpms is a slight
simplification.
 (The key point here is that we do not need to form any standard comparison of
premice in the argument,
 although we do use comparison inflation.)
We will first prove parts
\ref{item:Sigma'_exists_unique}, \ref{item:Sigma'_shc} and
\ref{item:trees_via_Sigma'_embed};
this automatically yields parts \ref{item:Gamma'_exists_unique},
\ref{item:Gamma'_shc}
and \ref{item:Sigma'_corresp_Gamma'}, by the correspondence of convenient and
inconvenient strategies.

Work in $V[G]$.
Let $\Sigma'$ be the set of all pairs $(\Tt,b)$
such that $\Tt$ is a $\udash m$-maximal tree on $M$
of length $\leq\Omega$ and $b$ is $\Tt$-cofinal and there is a limit length tree
$\Xx\in V$ and $\Xx$-cofinal branch $c\in V$ with $(\Xx,c)$ via
$\Sigma$,
and there is a tree embedding
$\Pi:(\Tt,b)\hookrightarrow(\Xx,c)$;
equivalently by \ref{lem:ate_to_te}, there is an almost tree embedding
$\Pi:(\Tt,b)\hookrightarrow_{\almost}(\Xx,c)$.
We will verify that
$\Sigma'$ is a $(\udash m,\Omega+1)$-strategy for $M$, with strong hull
condensation.
Actually, for each such $(\Tt,b)$, we will find a witnessing
$(\Xx,c)\in V$
which is a terminally-non-dropping inflation of $(\Tt,b)$.

We start by showing that $\Sigma'$ is a function.

\begin{clm}\label{clm:uniqueness_of_b}
Let $\Xx,\Xx'\in V$ be via $\Sigma$.
Work in $V[G]$. Let $\Pi:(\Tt,b)\hookrightarrow\Xx$ and
$\Pi':(\Tt,b')\hookrightarrow\Xx'$
be tree embeddings.
Then $b=b'$.
\end{clm}
\begin{proof}
 Suppose not and fix $\Xx,\Xx'$.
 Let $S$ be the tree of attempts to build (a code for)
 a tuple $(\Tt,b,b',\Pi,\Pi')$
 such that $\Tt$ is a countable limit length (potential, that is, satisfies the
relevant first order requirements, but without demanding
that $\lh(\Tt)$ be wellfounded or that $\Tt$ have wellfounded models)
iteration tree on $M$ and
 $b,b'$ are distinct $\Tt$-cofinal branches,
 $\Pi:(\Tt,b)\hookrightarrow_{\almost}\Xx$
 and
 $\Pi':(\Tt,b')\hookrightarrow_{\almost}\Xx'$
 (and hence, $(\Tt,b)$ and $(\Tt,b')$ are in fact true iteration trees).
 Here we can and do take $S$ as a tree on some $\lambda\in\OR$.
 We can do this because
 an element $s$ of $S$ specifies some finite iteration tree $\bar{\Tt}_s$ on
$M$,
 with domain some finite set $D_s\sub\om$, with $D_{s'}\sub D_s$ for $s'\ins s$,
 specifies how each $\bar{\Tt}_{s'}$ fits as a subtree of $\bar{\Tt}_s$,
 and specifies $b\inter D$ and $b'\inter D$ and $\Pi\rest D$ and $\Pi'\rest D$
(the latter
 meaning just $\gamma_\alpha,\delta_\alpha,\gamma'_\alpha,\delta'_\alpha$ for
$\alpha\in D$).
 Here $\bar{\Tt}_s$ can be specified by a finite sequence of ordinals because
recall that
 in the coarse (wcpm) case, although $M$ need not model $\ZFC$,
 we do demand that the extenders used come from $\es^M$,
 which is a wellordered set.

 Now because of our contradictory assumption, $S$ is illfounded in
$V^{\Coll(\om,\gamma)}$ for sufficiently large $\gamma$,
 and therefore $S$ is illfounded in $V$.
 But then (as $S$ is on $\lambda$)  we get some such $\Tt,b,b',\Pi,\Pi'\in V$,
contradicting strong hull
condensation.
\end{proof}

As mentioned earlier, whenever $b=\Sigma'(\Tt)$, we will actually find a
$(\Tt,b)$-terminally-non-dropping
inflation $(\Xx,c)$ of $(\Tt,b)$,  with $(\Xx,c)\in V$ and via $\Sigma$.
We can actually prove the uniqueness of such $b$ using only inflation
condensation,
and we give this proof next.
However, this uniqueness is not enough for the overall proof; we seem to need
the stronger uniqueness of the claim above,
which relied on strong hull condensation.
So we just include the next claim for interest, and in case
one might be able to improve on its proof,
so as to replace the use of strong hull condensation in the theorem with
inflation condensation.
\footnote{In an earlier draft of this paper,
which was available on the author's website for a short period of time,
we had actually stated the theorem with inflation condensation instead of
strong
hull condensation,
but there was a gap in that putative proof.}

\begin{clm}\label{clm:Sigma'_is_function}
Let $\Tt,b_0,\Xx_0,c_0,b_1,\Xx_1,c_1$ be such that
$(\Xx_i,c_i)\in V$, according to $\Sigma$, is a
$(\Tt,b_i)$-terminally-non-dropping inflation of $(\Tt,b_i)$.
Then $b_0=b_1$, assuming only inflation condensation for $\Sigma$.
\end{clm}
\begin{proof} Let $f^0=f^{(\Tt,b_0)\inflatearrow(\Xx_0,c_0)}$, etc.
By minimizing $\lh(\Xx_i)$, we may assume $\Xx_i$ is also an inflation of
$\Tt$,
as witnessed by $\widetilde{f}^0=f^{\Tt\inflatearrow\Xx_0}$, etc
(otherwise replace $(\Xx_0,c_0)$
with $\Xx_0\rest\eta+1$
for the least $\eta$ where $f^0(\eta)=\lh(\Tt)$).
Then  $\widetilde{C}^i=C^i\inter\eta_i$ where $\eta_i=\lh(\Xx_i)$,
$\widetilde{f}^i=f^i\rest\widetilde{C}^i$,
$\eta_i\in C^i$ and $f^i(\eta_i)=\lh(\Tt_i)$.

In $V$, let $\Xx$, of length $\lambda+1$, be the least initial segment of the
comparison inflation (\S\ref{subsec:min_inf}) of $(\Xx_0,c_0)$
and
$(\Xx_1,c_1)$
where for some $i\in\{0,1\}$,
\[ \lambda\in C^{(\Xx_i,c_i)\inflatearrow\Xx}
\text{ and } f^{(\Xx_i,c_i)\inflatearrow\Xx}(\lambda)=\lh(\Xx_i).\]
We may assume $i=0$.
Then $\Xx$ is an $(\Xx_0,c_0)$-terminally-non-dropping
inflation of $(\Xx_0,c_0)$.
Note $\lambda$ is a limit, and  by
Corollary \ref{cor:terminal_dropping_equiv},
$\Xx$ is a $(\Tt_0,b_0)$-terminally-non-dropping inflation of $(\Tt_0,b_0)$.
Let $\hat{C}^0$, etc, be the witnesses to the latter.
Then by Lemma \ref{lem:inflation_commutativity}, $\lambda\in\hat{C}^0$
and $\hat{f}^0(\lambda)=\lh(\Tt)$,
so $b^\Xx$ determines $b_0$ via this inflation.
By the minimality of $\lambda$, $\hat{f}^0(\alpha)<\lh(\Tt)$ for each
$\alpha\in\lambda\inter\hat{C}^0$.
So note that $\Xx\rest\lambda$ is an inflation of $\Tt$,
as witnessed by $\hat{C}^0\inter\lambda$, $\hat{f}^0\rest\lambda$, etc.
(The branch $b_0$ is irrelevant because $\lh(\Tt)\notin\hat{f}^0``\lambda$.)

Now because $\Xx$ is also an inflation of $(\Xx_1,c_1)$,
by \ref{lem:inflation_commutativity},
$\Xx$ is also an inflation of $(\Tt,b_1)$,
as witnessed by $\hat{C}^1$, etc,
and again by minimality of $\lambda$, we have
$\lh(\Tt)\notin\hat{f}^1``\lambda$.
So $\hat{C}^1\inter\lambda=\hat{C}^0\inter\lambda$ and
$\hat{f}^0\rest\lambda=\hat{f}^1\rest\lambda$
etc. But then $\hat{C}^0=\hat{C}^1$ and $\hat{f}^0=\hat{f}^1$ etc,
because the extensions are determined by the common restrictions to $\lambda$
and $\Tt$ and $b^\Xx$.
So $\lambda\in\hat{C}^1$ and $\hat{f}^1(\lambda)=\lh(\Tt)$
and since $\Xx$ is an inflation of $(\Tt,b_1)$, $b^\Xx$ determines $b_1$.
But $b^\Xx$ determines $b_0$, so $b_0=b_1$. This gives the claim.
\end{proof}

We now verify that $\Sigma'$ produces wellfounded models
and is total.

\begin{clm}\label{clm:Sigma'_is_total}
Let $\Tt\in V[G]$ be a putative tree via $\Sigma'$.
If $\lh(\Tt)$ is a successor
then there is $\Xx\in V$  via $\Sigma$
and such that $\Xx$ is a $\Tt$-terminally-non-dropping inflation of $\Tt$,
and if $\lh(\Tt)<\Omega$ then we can take $\lh(\Xx)<\Omega$;
so  every such $\Tt$ is a true iteration tree.
If $\lh(\Tt)$ is a limit $\leq\Omega$ then
$\Tt\in\dom(\Sigma')$.
\end{clm}
\begin{proof}
We prove the claim by induction on $\lh(\Tt)$. Suppose we have a tree $\Tt$ of
length $\eta+1\leq\Omega+1$,
and the claim holds for $\Tt$. Fix $\Xx$ witnessing this.
Then for trees $\Tt'$ normally extending $\Tt$ of length ${<\eta+\om}$,
we may extend $\Xx$ to a $\Tt'$-terminally-non-dropping inflation $\Xx'$ of
$\Tt'$,
by simply copying the finite remainder of $\Tt'$ up,
and since $\Xx\in V$ is via $\Sigma$, so is $\Xx'$.

So fix $\Tt$ of limit length $\leq\Omega$.
We will find some $\Tt$-cofinal $b\in V[G]$ and a
$(\Tt,b)$-terminally-non-dropping inflation $(\Xx,c)$ of $(\Tt,b)$,
with $(\Xx,c)\in V$ via $\Sigma$.

For this, working in $V$, we form a Boolean valued comparison
inflation of various candidates for $\Tt$.
Fix $p_0\in\PP$ forcing that $\dot{\Tt}$ is as above. We will define the
(Boolean valued comparison) inflation relative to $p_0$,
producing a tree $(\Xx,c)$, and show that there is $q\leq p_0$ such that $q$
forces that it works for some $\dot{\Tt}$-cofinal branch $\dot{b}$.
This is enough by density.

So, we define a tree $\Xx$ on $M$, using extenders $E^\Xx_\alpha$ with indices
$\xi_\alpha$,
as follows. Let $E^\Xx_0$ be the least $E\in\es_+^M$ such that some $q\leq p_0$
forces that $E^{\dot{\Tt}}_0=E$. This gives $\Xx\rest 2$.

Now suppose we have $\Xx\rest\alpha+1$. If $\alpha$ is a limit and there is
some
$q\leq p_0$
forcing ``$\Xx\rest\alpha+1$ is a $(\dot{\Tt},b)$-terminally-non-dropping
inflation of $(\dot{\Tt},b)$ for some $\Tt$-cofinal $b$'',
then we stop the construction (with success).
Now suppose otherwise. If $\alpha=\Omega$ then we stop (with failure due to
long
tree). Suppose otherwise. Let $E^\Xx_\alpha$ be the least
$E\in\es_+(M^\Xx_\alpha)$
such that some $q\leq p_0$ forces that $\Xx\rest\alpha+1$ is an inflation of
$\dot{\Tt}$,
with $\alpha\in C^-$, and $E=E^{\dot{\Tt}\inflatearrow\Xx\rest\alpha+1}_\alpha$,
if such an $E$ exists; otherwise we stop (with failure due to dropping).

At limit stages $\eta$, we extend $\Xx\rest\eta$ using $\Sigma$.

This completes the definition of $\Xx$. We next verify that the construction
stops with success.

Now $p_0$ forces that $\Xx$ is an inflation of $\dot{\Tt}$.
This follows from the minimality of $\In(E^\Xx_\beta)$ for each $\beta$
together
with Claim \ref{clm:uniqueness_of_b}.
That is, if $\eta<\lh(\Xx)$ is a limit
and $p_0$ forces that $\Xx\rest\eta$ is an inflation of $\dot{\Tt}$,
then $p_0$ forces that $\Xx\rest(\eta+1)$ is also an inflation of $\dot{\Tt}$.
For otherwise there are $q,\lambda$
such that $q\leq p_0$ and $q$ forces ``$\lambda<\lh(\dot{\Tt})$
and there is $b\neq[0,\lambda)_{\dot{\Tt}}$ and a tree embedding
\[ \Pi:(\dot{\Tt}\rest\lambda,b)\hookrightarrow\Xx\rest(\eta+1),\text{''}\]
contradicting Claim \ref{clm:uniqueness_of_b}.\footnote{Note that Claim
\ref{clm:Sigma'_is_function}
does not suffice here, because we need to rule out the possibility
of having a limit $\lambda<\lh(\Tt)$ and some limit $\eta$
such that $\Xx\rest\eta$ is an inflation of $\Tt$,
but $\Xx\rest(\eta+1)$ is not, because $[0,\eta)_\Xx$ induces some
$\Tt$-maximal
branch
which is not $\Tt$-cofinal. Claim \ref{clm:Sigma'_is_function} does not suffice
to rule this out.}

Now suppose the construction stops with failure due to dropping, giving tree
$\Xx=\Xx\rest\alpha+1$ (so $\alpha<\Omega$).
Note $p_0$ forces ``$\alpha\notin C^-$''.
Now $\alpha$ is a limit, because if $\alpha=\beta+1$
then some $q\leq p_0$ forces ``$E^\Xx_\beta$ is copied from $\dot{\Tt}$'',
so $q$ forces
  ``$\alpha=\gamma^{\dot{\Tt}\inflatearrow\Xx}_{\alpha;f(\alpha)}$,
so $\alpha\in C^-$ (as $\dot{\Tt}$ has limit length)'';
contradiction. So let $\beta<^\Xx\alpha$ be such that $(\beta,\alpha)_\Xx$ does
not drop.
Some $q\leq p_0$ forces ``$\beta\in C^-$''. We claim $q$ forces
``$\alpha\in C^-$'',
a contradiction. For suppose not, and let $\gamma\in(\beta,\alpha]_\Xx$ be
least
such that
some $r\leq q$ forces  ``$\gamma\notin C^-$'', and fix $s\leq r$
such that $s$ decides the values $\lambda=\sup_{\xi<^\Xx\gamma}f(\xi)$ and
$\lambda'=\lh(\dot{\Tt})$.
Because $(\beta,\alpha)_\Xx$ does not drop, $\gamma$ is a limit ordinal.
But then if $\lambda<\lambda'$, note that $s$ forces $\gamma\in C$ (recall
$p_0$
forces that $\Xx$ is an inflation of $\dot{\Tt}$,
so $s$ forces that $[0,\gamma)_\Xx$ determines $[0,\lambda)_{\dot{\Tt}}$)
and hence $\gamma\in C^-$, a contradiction. So $\lambda=\lambda'$, but then
the construction stops with success at stage $\gamma$, as witnessed by $s$ and
the $\dot{\Tt}$-cofinal
branch determined by $[0,\gamma)_\Xx$, a contradiction.

So finally suppose that the process stops with failure due to a long tree,
so we get $\Xx$ of length $\Omega+1$.
If $q\leq p_0$ and $q$ forces that cofinally many extenders used along
$[0,\Omega)_\Xx$
are $\dot{\Tt}$-copying, then because $\Omega$ is regular in $V[G]$
and $\lh(\dot{\Tt})\leq\Omega$, $q$ forces that $\lh(\dot{\Tt})=\Omega$
and the process ends successfully at $\alpha=\Omega$, contradiction.
But if there is no such $q$, then by $\Omega$-cc-ness, there is some
$\alpha<\Omega$
such that $p_0$ forces that every extender used along $(\alpha,\Omega)_\Xx$ is
$\dot{\Tt}$-inflationary.
But this is impossible, as by construction, for every extender $E$ used in
$\Xx$,
there is $q\leq p_0$ forcing that $E$ is $\dot{\Tt}$-copying.

So the construction stops with success, as witnessed by $\alpha\leq\Omega$ and
$q\leq p_0$
(so $\lh(\Xx)=\alpha+1$). Finally, to complete the proof of the claim,
we show that if $\alpha=\Omega$
then $q$ forces that $\lh(\dot{\Tt})=\Omega$. But by the minimality of $\alpha$
(that is, there is no $\alpha'<\alpha$ such that the construction stopped with
success at stage $\alpha'$),
$q$ forces that ``$f(\alpha)<\lh(\dot{\Tt})$ for all $\alpha<^\Xx\Omega$,
and $f(\Omega)=\lh(\dot{\Tt})$, and
$f(\Omega)=\sup_{\alpha<^\Xx\Omega}f(\alpha)$'',
but $\Omega$ is regular in $V[G]$,
so $q$ forces $\lh(\dot{\Tt})=\Omega$.
\end{proof}

\begin{clm} $\Sigma'$ has strong hull condensation.\end{clm}
\begin{proof}
Work in $V[G]$. Let $\Pi:\Tt\hookrightarrow\Uu$ where $\Uu$ is via $\Sigma'$.
We may assume that $\Uu\in V$ is via $\Sigma$, by \ref{lem:comp_tree_emb}
and Claim \ref{clm:Sigma'_is_total}. We claim that $\Tt$ is via $\Sigma'$. For
let $\eta<\lh(\Tt)$ be a limit
and $b=\Sigma'(\Tt\rest\eta)$. Then using a restriction of $\Pi$
 and Claim \ref{clm:uniqueness_of_b}, we have $b=[0,\eta)_\Tt$.\footnote{One
can
alternatively
 use an absoluteness argument like the proof of Claim
\ref{clm:uniqueness_of_b};
this argument
 does not use \ref{lem:comp_tree_emb}.
 Fix some trees $\Xx,\Vv\in V$ via $\Sigma$, and consider the tree of attempts
 to build trees $\Tt$ and $\Uu$ together with $\Tt$-cofinal branches $b\neq c$
 and almost tree embeddings $\Pi_b:(\Tt,b)\hookrightarrow_{\almost}\Xx$ and
 $\Pi_c:(\Tt,c)\hookrightarrow_{\almost}\Uu$ and
$\Pi:\Uu\hookrightarrow_{\almost}\Vv$.
 Given objects of this form, then by \ref{lem:ate_to_te} and strong hull
condensation in $V$,
 $(\Tt,b)$ is via $\Sigma$, and $\Uu$ is via $\Sigma$, but therefore also
$(\Tt,c)$ is via $\Sigma$,
 so $b=c$. So the tree is wellfounded, which suffices.}
 \end{proof}

\begin{clm}
 In $V[G]$, $\Sigma'$ is the unique $(\udash m,\Omega+1)$-strategy
 with inflation condensation which
extends $\Sigma$.
\end{clm}
\begin{proof}
In $V[G]$, let $\Sigma''$ be such a strategy.
Let $\Tt$ be of limit length $\leq\Omega$,
according to both $\Sigma'$ and $\Sigma''$, and let $b'=\Sigma'(\Tt)$ and
$b''=\Sigma''(\Tt)$.
We need to see that $b'=b''$.
Let $(\Xx,c)\in V$, according to $\Sigma$, be a $(\Tt,b')$-terminal inflation
of
$(\Tt,b')$,
of minimal possible length. Since $\Sigma\sub\Sigma''$, $(\Xx,c)$ is also
according to $\Sigma'$,
so by inflation condensation for $\Sigma''$, we have $b''=b'$, as required.
\end{proof}

This completes the proof of parts \ref{item:Sigma'_exists_unique},
\ref{item:Sigma'_shc} and \ref{item:trees_via_Sigma'_embed}.
Finally consider part \ref{item:Gamma,Gamma'_DJ}:

\begin{clm}\label{clm:DJ_for_Sigma'} Suppose $M$ is a premouse (not wcpm) and
countable in $V$. Then  $\Gamma$ has DJ iff $\Gamma'$ has DJ in $V[G]$.
Likewise for weak DJ with respect to $e$.
\end{clm}
\begin{proof}
We just discuss DJ; weak DJ is almost the same.

If $\Gamma$ fails DJ then since $\Gamma\sub\Gamma'$, clearly $\Gamma'$ fails DJ
in $V[G]$.
So suppose $\Gamma$ has DJ, but $\Gamma'$ does not in $V[G]$.
Let $\Tt\in V[G]$ be a successor length tree according to $\Gamma'$, witnessing
this,
via some $Q\ins M^\Tt_\infty$ and $\pi:M\to Q$.

Assume for now that
$M$ has $\lambda$-indexing.
Let $\Xx\in V$, via $\Gamma$, be a $\Tt$-terminally-non-dropping inflation of
$\Tt$.
Let $\sigma:M^\Tt_\infty\to M^\Xx_\infty$ be the final inflation copying map.
So $\sigma$ is a near $\deg^\Tt(\infty)$-embedding, and by
\ref{rem:non-drop_inf_comm}, if $\Tt$ is terminally-non-dropping
then so is $\Xx$ and $\sigma\com i^\Tt=i^\Xx$.
So by considering $\sigma\com i^\Tt$ and $\sigma(Q)$ if $Q\pins M^\Tt_\infty$,
we may in fact assume that $\Tt\in V$ is via $\Gamma$.
But then since $M$ is countable in $V$, the existence of $\pi\in V[G]$ and
absoluteness
yields some $\pi'\in V$
which gives a counterexample to DJ in $V$, contradiction.

Now suppose instead that $M$ has MS-indexing. Note by minimizing on $\lh(\Tt)$,
we get $\lh(\Tt)<\Omega$ (for otherwise consider $\Tt\rest(\alpha+1)$ for
sufficiently large $\alpha<^\Tt\Omega$).
Let $\widetilde{\Tt}$ be the tree according to $\Sigma'$, corresponding to
$\Tt$,
so (by \ref{lem:tree_conversion})
$(M^{\widetilde{\Tt}}_\infty)^\pm=M^\Tt_\infty$ and if $M^\Tt_\infty$ is
type 3
then
$\udeg^{\widetilde{\Tt}}(\infty)=\deg^\Tt(\infty)+1>0$.
Let $\widetilde{\Xx}\in V$, via $\Sigma$, be a $\Tt$-terminally-non-dropping
inflation of $\widetilde{\Tt}$.
Let
\[ \widetilde{\sigma}:M^{\widetilde{\Tt}}_\infty\to M^{\widetilde{\Xx}}_\infty
\]
be the final copying map,
so $\widetilde{\sigma}$ is a near $\udeg^{\widetilde{\Tt}}(\infty)$-embedding.
Let $\Xx$ be the tree according to $\Gamma$, corresponding to $\widetilde{\Xx}$.
So $M^\Xx_\infty=(M^{\widetilde{\Xx}}_\infty)^\pm$.

Now if $Q\pins M^\Tt_\infty$ then note that either
\[ \widetilde{\sigma}(Q)\pins M^\Xx_\infty
\text{ or
}\widetilde{\sigma}(Q)\pins\Ult(M^\Xx_\infty|(\mu^+)^{M^\Xx_\infty},F)\]
where $F=F(M^\Xx_\infty)$ and $\mu=\crit(F)$,\footnote{
Here if $Q\notin(M^\Tt_\infty)^\sq$
and $\widetilde{\sigma}(\nu(F(M^\Tt_\infty)))>\nu(F(M^\Xx_\infty))$
then use $\Xx\conc F$ instead of just $\Xx$.}
and so either from
$(\Xx,\pi(Q),\widetilde{\sigma}\com\pi)$
or $(\Xx\conc F,\pi(Q),\widetilde{\sigma}\com\pi)$,
and absoluteness,
we get a contradiction to DJ for $\Gamma$ in $V$.

So suppose $Q=M^\Tt_\infty$. Because $\widetilde{\sigma}$ is a near
$\udeg^{\widetilde{\Tt}}(\infty)$-embedding,
\[
\sigma=\widetilde{\sigma}
\rest(M^\Tt_\infty)^\sq:(M^\Tt_\infty)^\sq\to(M^\Xx_\infty)^\sq \]
is a near $\deg^\Tt(\infty)$-embedding $M^\Tt_\infty\to M^\Xx_\infty$,
and we have
\[ \udeg^{\widetilde{\Xx}}(\infty)=\udeg^{\widetilde{\Tt}}(\infty)\text{ and }
\deg^\Xx(\infty)=\deg^\Tt(\infty)\geq n. \]
If $\Tt$ drops in model on $b^\Tt$ then so do
$\widetilde{\Tt},\widetilde{\Xx},\Xx$,
and $\sigma\com\pi:M\to M^\Xx_\infty$ is a near $\deg^\Tt(\infty)$-embedding,
so by absoluteness we have a contradiction.
So $\Tt$ does not drop in model, hence nor in degree,
and likewise for $\Xx,\widetilde{\Tt},\widetilde{\Xx}$.
So $\widetilde{\sigma}\com i^{\widetilde{\Tt}}=i^{\widetilde{\Xx}}$,
so by \ref{lem:tree_conversion}, $\sigma\com i^\Tt=i^\Xx$.
And  $\pi(\alpha)<i^\Tt(\alpha)$ for some $\alpha\in\OR^M$, so
$\sigma(\pi(\alpha))<\sigma(i^\Tt(\alpha))=i^\Xx(\alpha)$,
so again
we have a contradiction.
\end{proof}

This completes the proof of the theorem.
\end{proof}

\subsection{Universally Baire strategies}\label{subsec:univ_baire}

The following corollaries on universally Baire representations for iteration
strategies were motivated by related work of Steel.
Given an iteration strategy $\Sigma$ on a countable premouse $M$,
let $\widetilde{\Sigma}$ be the natural coding of $\Sigma\rest\HC$ over the
reals.
Note that without $\AC$, it seems that the trees $S,T$ in the following
corollary might not be trees on ordinals.
However, the only non-ordinal information is specified by $\Xx=y(0)$.
In Corollary \ref{cor:uB_on_OR} we prove a version where we do get trees $S,T$
on
ordinals.

\begin{cor}\label{cor:uB_strat}
Let $\Omega,\Gamma,M$ be as in
Theorem \ref{thm:strat_with_cond_extends_to_generic_ext},
with $M$ countable.
Then $\widetilde{\Gamma}\rest\RR$ is $\Omega$-universally Baire.
In fact, there are trees $S,T$ on $\om\cross\her_\Omega$
such that letting $G$ be $V$-generic for $\Coll(\om,{<\Omega})$, then $S,T$
project to complements in $V[G]$,
and
\[ p[T]^{V[G]}=\widetilde{\Gamma'}\rest\RR,\]
where $\Gamma'$ is the extension of $\Gamma$ given by
Theorem \ref{thm:strat_with_cond_extends_to_generic_ext}.
\end{cor}
\begin{proof}
Let $\Sigma$ be the $\udash$strategy corresponding to $\Gamma$, as in
Theorem \ref{thm:strat_with_cond_extends_to_generic_ext}.

Let $T$ be the tree of attempts to build $(x',(x,y))$, where
$x,x'\in{^\om}\om$,
$x$ codes a pair $(\Tt,b)$,
 where $\Tt$ is a (potential) countable limit length $\udash m$-maximal tree on
$M$ and $b$ is a  $\Tt$-cofinal branch,
$x'$ codes the corresponding (potential) $m$-maximal tree $(\Tt',b')$,
 and $y\in{^\om(\her_\Omega)}$ specifies $y(0)=\Xx$ is some tree on $M$ via
$\Sigma$, of length $<\Omega$,
 and $y$ codes an almost tree embedding
 $\Pi:(\Tt,b)\hookrightarrow_{\almost}\Xx$.

 Let $S$ be natural tree for the complement. That is, $S$ builds
$(\wt{x'},(\wt{x},\wt{y}))$
 such that either $\wt{x'}$ codes garbage information,
 or $\wt{x},\wt{x'}$ code $(\Tt,\wt{b}),(\Tt',\wt{b'})$ as above,
 and $\wt{y}$ codes a tuple $(x',(x,y))\in[T]$,
 and $x'$ codes the pair $(\Tt',c)$ with $c\neq\wt{b'}$.

 Now $\Coll(\om,{<\Omega})$ is $\Omega$-cc, so
Theorem \ref{thm:strat_with_cond_extends_to_generic_ext} applies.
But clearly by strong hull condensation we have $p[S]\inter p[T]=\emptyset$ in
both $V$ and $V[G]$.
\footnote{If $\her_\Omega$ is not wellordered in $V$, then we can't quite use
the usual argument here to deduce that $V[G]\sats$``$p[T]\inter
p[S]=\emptyset$'',
given that $V\sats$``$p[T]\inter p[S]=\emptyset$'',
However, one could note that for any given tree $\Xx$ as a choice of $y(0)$,
the sub-trees $S_\Xx$ and $T_\Xx$ can be taken on ordinals.
So if $V[G]\sats$``$p[T]\inter p[S]\neq\emptyset$'',
then we could fix a specific $\Xx$ and $\Yy$ with $V[G]\sats$``$p[T_\Xx]\inter
p[S_\Yy]\neq\emptyset$''
and deduce that $V\sats$``$p[\Tt_\Xx]\inter p[S_\Yy]\neq\emptyset$'',
a contradiction.}
And by the proof of \ref{thm:strat_with_cond_extends_to_generic_ext},
note that $p[T]^{V[G]}=\widetilde{\Gamma'}\rest\RR$,
and $S,T$ project to complements in $V[G]$.\end{proof}

If $\Omega$ is inaccessible, we can improve the conclusion;
in the following proof,
the trees we form are analogous to those formed by direct limits of mice used
by
Steel.

\begin{cor}\label{cor:uB_on_OR}
Adopt the hypotheses and notation of Corollary \ref{cor:uB_strat}.
Suppose also that for no $\alpha<\Omega$ is $\Omega$ the surjective image of
$\pow(\alpha)$.
Then there are trees $S,T\in\OD_{\Gamma,M}$ witnessing
Corollary \ref{cor:uB_strat} with
$S,T$ on $\om\cross\Omega$.
\end{cor}
\begin{proof}
Let $\Sigma$ be as before. By the proof of \ref{cor:uB_strat}, it suffices to
show that for each $\Tt$ via $\Sigma$ of length ${<\Omega}$,
 there is some $\Xx\in\OD_{\Gamma,M}$ such that $\Xx$ is via $\Sigma$, of
length
${<\Omega}$,
 and is a $\Tt$-terminally-non-dropping inflation of $\Tt$. For by the
 largeness assumption of $\Omega$
 (including regularity),
we can  enumerate all such $\Xx$ in ordertype $\Omega$\footnote{Enumerate
those of length $\alpha$ before those of length $\beta$,
when $\alpha<\beta$.}
in an $\OD_{\Gamma,M}$ fashion, leading to an
$\OD_{\Gamma,M}$ tree $T$ on $\om\cross\Omega$.

So fix $\chi<\Omega$ and let $\mathscr{T}$ be the set of all trees
via $\Sigma$ of
length
${<\chi}$. We define $\lambda<\Omega$ and a partition
$\vec{\mathscr{T}}=\left<\mathscr{T}_\alpha\right>_{\alpha<\lambda}$
of $\mathscr{T}$ and a sequence
$\vec{\Xx}=\left<\Xx_\alpha\right>_{\alpha<\lambda}$ of trees $\Xx_\alpha$ via
$\Sigma$,
such that for each $\alpha<\lambda$, we have:
(i) $\vec{\mathscr{T}},\vec{\Xx}$ are $\OD_{\Gamma,M}$,
(ii) $\mathscr{T}_\alpha\neq\emptyset$,
 (iii) $\lh(\Xx_\alpha)<\Omega$, and
 (iv) $\Xx_\alpha$ is a $\Tt$-terminally-non-dropping inflation of each
$\Tt\in\mathscr{T}_\alpha$.
Clearly this suffices.

So suppose we have defined $\left<\mathscr{T}_\alpha\right>_{\alpha<\eta}$
and $\left<\Xx_\alpha\right>_{\alpha<\eta}$ satisfying the requirements so far,
and suppose that
$\mathscr{T}'=\mathscr{T}\cut\bigcup_{\alpha<\eta}\mathscr{T}
_\alpha\neq\emptyset$;
otherwise we are done.

We set $\Xx_\eta$ to be the comparison inflation of $\mathscr{T}'$.
This exists and has length ${<\Omega}$. For otherwise via
comparison inflation,
we reach a tree $\Xx$ of length $\Omega+1$. Each extender used along
$[0,\Omega)_\Xx$
is copied from some $\Tt\in\mathscr{T}'$. But each $\Tt\in\mathscr{T}'$ has
length ${<\Omega}$,
and since $\Omega$ is regular, it follows that there is $\alpha_\Tt<^\Xx\Omega$
such that no extender used in $(\alpha_\Tt,\Omega]_\Xx$ is copied from $\Tt$.
But then $\Tt\mapsto\alpha_\Tt$ is cofinal in $\Omega$, and since $\Omega$ is
regular,
this gives a surjection $\pow(\alpha)\to\Omega$, a contradiction.

Now by \ref{lem:min_inf}, there is some $\Tt\in\mathscr{T}'$ such that
$\Xx_\eta$
is $\Tt$-terminally-non-dropping.
So letting $\mathscr{T}_\eta$ be the set of all such $\Tt$, we are done.
\end{proof}

\subsection{Restricting weak Dodd-Jensen strategies from $V[G]$}

The proof of the following corollary involves standard comparison of premice,
and
the author does not see a version for wcpms. We will prove an extension of the
corollary
in \S\ref{sec:npc}, once we have Theorem \ref{thm:stacks_iterability} at our
disposal. Recall that strategies with weak DJ with respect to
a given enumeration $e$ (in ordertype $\om$)
are unique; see Remark \ref{rem:wDJ_implies_cond}.

\begin{cor}\label{cor:wDJ_absoluteness}
Let $\Omega>\om$ be regular.
Let $M$ be a countable $m$-sound premouse.
Let $e$ be an enumeration of $M$ in ordertype $\omega$.
Let $\PP$ be an $\Omega$-cc forcing and $G$ be $V$-generic for $\PP$.
Then:
\begin{enumerate}[label=\arabic*.,ref=\arabic*]
\item\label{item:V,V[G]_agree_wDJ} $V\sats$``There is an
$(m,\Omega+1)$-strategy
for $M$ with weak DJ with respect to $e$''
iff $V[G]$ satisfies the same statement.
\item\label{item:wDJ_strat_extends} If $\Sigma$ is an \tu{(}hence the
unique\tu{)} $(m,\Omega+1)$-strategy for $M$ with weak DJ with respect to $e$,
and $\Sigma'$ likewise in $V[G]$, then $\Sigma\sub\Sigma'$.
\end{enumerate}
\end{cor}
\begin{proof}
\setcounter{clm}{0}
The forward direction and the fact that $\Sigma\sub\Sigma'$
is by Theorem \ref{thm:strat_with_cond_extends_to_generic_ext}.

So suppose that in $V[G]$, $\Sigma'$ is an $(m,\Omega+1)$-strategy
for $M$ with weak DJ with respect to $e\in V$.
Let $p_0\in\PP$ force this fact. Let $\Sigma=\Sigma'\rest V$.
It suffices to see that $\Sigma\in V$, as then $\Sigma$ has weak DJ with
respect
to $e$ in $V$,
and part \ref{item:wDJ_strat_extends} follows from the uniqueness
of
this
strategy (see Remark \ref{rem:wDJ_implies_cond}).

So let $\Tt\in V$ be via $\Sigma'$,  of limit length
$\leq\Omega$,
and $b=\Sigma'(\Tt)$.
Let $\dot{\Sigma}',\dot{b}$ be names for $\Sigma',b$.
By the following claim, $b\in V$ and $p_0$ forces ``$b=\dot{\Sigma}'(\Tt)$'',
which clearly suffices.

\begin{clmtwo}For each $\alpha<\lh(\Tt)$, $p_0$ decides
the truth of ``$\alpha\in\dot{b}$''.\end{clmtwo}
\begin{proof}
Suppose not. We form a Boolean-valued comparison of generic phalanxes
$\Phi(\Tt,\dot{b})$.
Inductively on stages $\alpha\leq\Omega$,
we define a monotone increasing sequence
$\left<\xi_\alpha\right>_{\alpha<\Omega}$ of ordinals
and a sequence $\left<N_\alpha\right>_{\alpha\leq\Omega}$ of premice (in $V$)
and a sequence $\left<\dot{\Tt}_\alpha\right>_{\alpha\leq\Omega}$ of names for
padded iteration trees on $M$.
In fact, $\dot{\Tt}_\alpha$ is just  the name for the padded tree via
$\dot{\Sigma}'$,
extending $(\Tt,\dot{b})$, of length $\lh(\Tt)+\alpha+1$,
which uses extenders with indices $\left<\xi_\beta\right>_{\beta<\alpha}$ (where
we pad when there is no extender indexed at $\xi_\beta$)
\footnote{We are using the conventional algorithm for comparison
by least disagreement, not modified as in \cite{operator_mice}
or \cite{premouse_inheriting}.},
and $p_0$ will force that $N_\alpha=M^{\dot{\Tt}_\alpha}_\infty||\xi_\alpha$.
Given a $\PP$-name $\sigma$, we write $\sigma^0$
for the $\PP\cross\PP$-name for $\sigma^{\dot{G}_0}$,
where $\dot{G}_0$ is the $\PP\cross\PP$-name for the projection of the
$\PP\cross\PP$-generic on the left coordinate;
likewise for $\sigma^1$ and the right coordinate.

We begin with $\dot{\Tt}_0=(\check{\Tt},\dot{b})$.

Given $\dot{\Tt}_\alpha$, where $\alpha<\Omega$, let $\xi_\alpha$ be the least
ordinal $\xi$
such that $p_0$ forces ``$\xi\leq\OR(M^{\dot{\Tt}_\alpha}_\infty)$''
and for some $p,q\leq p_0$, we have
\[ (p,q)\forces_{\PP\cross\PP}\text{``}M^{\dot{\Tt}^0_\alpha}_\infty|\xi\neq
M^{\dot{\Tt}^1_\alpha}_\infty|\xi\text{'',}\]
if there is such a $\xi$; otherwise $\xi_\alpha$ is undefined and we stop the
construction.
Assuming $\xi_\alpha$ is defined, this determines $\dot{\Tt}_{\alpha+1}$, and
note that $p_0$ decides the value of
$N_\alpha\eqdef M^{\dot{\Tt}_\infty}_\alpha||\xi_\alpha$.

Given $\dot{\Tt}_\alpha$ for all $\alpha<\eta$, for a limit $\eta$,
note that $\dot{\Tt}_\eta$ is determined.

Clearly $\xi_\alpha\leq\xi_\beta$ for $\alpha<\beta$ (with
$\xi_\alpha=\xi_\beta$
only if $\beta=\alpha+1$ and we are using MS-indexing and the usual
superstrong/type 2 situation occurs).

\begin{sclmtwo} $p_0$ forces ``$\xi_\alpha$ exists for every $\alpha<\Omega$''.
\end{sclmtwo}

\begin{proof}
Suppose not and let $\alpha$ be least such. Write $\dot{\Uu}=\dot{\Tt}_\infty$.
Let $\xi$ be least such that for some $q\leq p_0$, $q$ forces
``$\xi=\OR(M^{\dot{\Uu}}_\infty)$''. Note that $p_0$ determines $N\eqdef
M^{\dot{\Uu}}_\infty|\xi$ (so $N\in V$) and $q$  forces
``$M^{\dot{\Uu}}_\infty=N$''.

Now $p_0$ forces ``$M^{\dot{\Uu}}_\infty=N$''. For otherwise we can fix $r\leq
p_0$ forcing ``$N\pins M^{\dot{\Uu}}_\infty$''. Then $N$ is fully sound, so
$q$ forces ``$b^{\dot{\Uu}}$ does not drop in model or degree, so
$i^{\dot{\Uu}}:M\to N$ exists and is an $m$-embedding'', so by absoluteness,
$r$ forces ``There is an $m$ embedding $\pi:M\to N$'', which contradicts weak
DJ below $r$.

Similar considerations using weak DJ now also give that either
\begin{enumerate}[label=\tu{(}\roman*\tu{)}]
 \item\label{item:N_unsound}  $p_0$ forces ``$b^{\dot{\Uu}}$ drops in model or
degree'', or
\item\label{item:N_sound} $p_0$ forces
``$b^{\dot{\Uu}}$ does not drop in model or degree'',
\end{enumerate}
and moreover, if \ref{item:N_sound} holds then there is $\pi\in V$ such that
$p_0$ forces ``$i^{\dot{\Uu}}=\pi$'', and of course if \ref{item:N_sound}
holds then letting $\pi:\core_{n+1}(N)\to N$ be the core map, where $n<\om$ is
largest such that $N$ is $n$-sound, then $p_0$ forces ``there is $\beta$ such
that $\beta+1\in b^{\dot{\Uu}}$ and $i^{*\dot{\Uu}}_{\beta+1,\infty}=\pi$''.

We can now recover the sequence of extenders $E$ forced to be used in
$\dot{\Uu}$ to form $\pi$, in the usual manner (cf.~\cite[4.3 \text{ and
Remark}]{cmip}). (So this sequence is in $V$.) But  by the Zipper Lemma,
$p_0$ forces that there is some such $E$ which is not used in $\Tt$. Let $E$ be
the least such.

By the rules of comparison, there cannot be a single $\gamma$ such that $p_0$
forces ``$E^{\dot{\Uu}}_\gamma=E$'', but since $\In(E)$ is fixed, therefore
$M$ is MS-indexed, $E$ is type 2 and there is $\gamma$ such that  $p_0$ forces
``either \begin{enumerate}[label=(\alph*)]
 \item\label{item:E_option_1} $E^{\dot{\Uu}}_\gamma=E$ and
$E^{\dot{\Uu}}_{\gamma+1}=\emptyset$, or
\item\label{item:E_option_2} $E^{\dot{\Uu}}_\gamma$ is superstrong
and $E^{\dot{\Uu}}_{\gamma+1}=E=F(M^{\dot{\Uu}}_{\gamma+1})$
and $\crit(E^{\dot{\Uu}}_\gamma)=\lgcd(M^{*\dot{\Uu}}_{\gamma+1})$'',
\end{enumerate}
and moreover both options get forced by some condition $\leq p_0$.

Note that $P=M^{\dot{\Uu}}_{\gamma+1}$ of option \ref{item:E_option_2} is in
$V$, and $P$ is non-sound, so in option \ref{item:E_option_1},
$P=M^{\dot{\Uu}}_\gamma$. But now arguing like we did for $N,\pi$ and the
sequence of extenders above, it follows that there is a superstrong extender
$F\in V$ such that $p_0$ forces that $F$ is used on the branch leading to $P$,
but then in fact $p_0$ forces ``$E^{\dot{\Uu}}_\gamma=F$
and $E^{\dot{\Uu}}_{\gamma+1}=E$'' a contradiction.
\end{proof}

By the subclaim, we reach $\dot{\Uu}\eqdef\dot{\Tt}_\Omega$, of length
$\Omega+1$.
Let $\dot{c}=b^{\dot{\Uu}}$. Since $\PP$ is $\Omega$-cc, we get a club
$B\sub\Omega$
such that $p_0$ forces ``$B\sub\dot{c}$ and $[\alpha,\Omega)_{\dot{\Uu}}$ does
not drop and $i^{\dot{\Uu}}_{\alpha\Omega}(\alpha)=\Omega$
for each $\alpha\in B$''. Let $N=\mathrm{stack}_{\alpha<\Omega}N_\alpha$.
Note $p_0$ forces that $N=M(\dot{\Uu})$ and
$M^{\dot{\Uu}}_\alpha||(\alpha^+)^{M^{\dot{\Uu}}_\alpha}=N|(\alpha^+)^N$
for each $\alpha\in B$.

We will now run a version of the standard proof of termination of comparison.
We first define a strictly increasing sequence
$\left<\alpha_n\right>_{n<\om}\sub
B$
and a $\sub$-increasing sequence $\left<X_n\right>_{n<\om}$ with
$X_n\sub\Omega^+$
and $\card(X_n)<\Omega$.

Let $X_0=\emptyset$ and $\alpha_0=\min(B)$. Suppose we have $X_n,\alpha_n$,
and
\[ p_0\forces\text{``}X_n\inter(\Omega^+)^{M^{\dot{\Uu}}_\Omega}\sub
i^{\dot{\Uu}}_{\alpha_n\Omega}[(\alpha_n^+)^N]\text{''}.\]
Let
$X_{n+1}=X_n\cup\{\beta<\Omega^+\bigm| \exists q\leq p_0\text{ s.t. }q\text{
forces ``}\beta\in i^{\dot{\Uu}}_{\alpha_n\Omega}[(\alpha_n^+)^N]\text{''}\}$.
Clearly then $X_{n+1}\sub\Omega^+$, $\card(X_{n+1})<\Omega$ by the $\Omega$-cc
and since $(\alpha_n^+)^N<\Omega$, and
\[ p_0\forces\text{``} i^{\dot{\Uu}}_{\alpha_n\Omega}[(\alpha_n^+)^N]\sub
X_{n+1}\inter(\Omega^+)^{M^{\dot{\Uu}}_\Omega}\text{''}.\]
Now let $\alpha_{n+1}$ be the least $\alpha\in B$ such that $\alpha>\alpha_n$
and
\[ p_0\forces\text{``} X_{n+1}\inter(\Omega^+)^{M^{\dot{\Uu}}_\Omega}\sub
i^{\dot{\Uu}}_{\alpha\Omega}[(\alpha^+)^N]. \]
By the $\Omega$-cc and since $\card(X_{n+1})<\Omega$, $\alpha_{n+1}$ exists.

Now let $\alpha=\sup_{n<\om}\alpha_n$ and $X=\bigcup_{n<\om}X_n$.
So $\alpha\in B$, and note that
\[ p_0\forces
\text{``}X\inter(\Omega^+)^{M^{\dot{\Uu}}_\Omega}=i^{\dot{\Uu}}_{\alpha\Omega}
[(\alpha^+)^N]\text{''}. \]
So
$p_0\forces \text{``}X\inter(\Omega^+)^{M^{\dot{\Uu}}_\Omega}\text{ is
cofinal
in }(\Omega^+)^{M^{\dot{\Uu}}_\Omega}
\text{ and has
ordertype }(\alpha^+)^N\text{''}$.
It follows that $p_0$ decides the value of $(\Omega^+)^{M^{\dot{\Uu}}_\Omega}$,
and decides $i^{\dot{\Uu}}_{\alpha\Omega}\rest(\alpha^+)^N$.

Now repeat the preceding construction, starting with $\alpha'_0>\alpha$,
and producing a limit $\alpha'$.
Then note that $p_0$ decides $i^{\dot{\Uu}}_{\alpha\alpha'}\rest(\alpha^+)^N$.
But $p_0$ decides $N|((\alpha')^+)^N$,
and hence decides
$i^{\dot{\Uu}}_{\alpha\alpha}\rest(N|(\alpha^+)^N)$ (not just the restriction
to
the ordinals;
one needs to know that the codomains match before being able to deduce
the agreement of the embeddings).
This contradicts comparison much as before, proving the claim.
\end{proof}
This completes the proof of the corollary.
\end{proof}

\begin{cor}\label{cor:om-mouse_abs}
 Let $\Omega>\om$ be regular.
 Let $G$ be $V$-generic for an $\Omega$-cc forcing.
 Let $M\in V[G]$ be an $\om$-sound premouse with $\rho_\om^M=\om$. Then:
 \begin{enumerate}[label=--]
\item $V[G]\sats$``$M\text{ is }(\om,\Omega+1)\text{-iterable''}$ iff $M\in V$
and $M\text{ is }(\om,\Omega+1)\text{-iterable}$.
\item If $\Sigma$ is an \tu{(}hence the unique\tu{)} $(\om,\Omega+1)$-strategy
for $M$,
and $\Sigma'$ likewise in $V[G]$, then $\Sigma\sub\Sigma'$.
\end{enumerate}
\end{cor}
\begin{proof}
Recall first that for an $\om$-sound premouse $N$ with $\rho_\om^N=\om$,
if $\Sigma$ is an $(\om,\Omega+1)$-strategy for $N$,
then $\Sigma$ has DJ, and hence, weak DJ with respect to any enumeration $e$ of
$N$,
and hence, strong hull condensation, by Theorem
\ref{tm:wDJ_implies_cond}.
So if $M\in V$, then the conclusions of the theorem
follow from Theorem
\ref{thm:strat_with_cond_extends_to_generic_ext}
and Corollary \ref{cor:wDJ_absoluteness}. So we only need to see that if $M$
is
$(\om,\Omega+1)$-iterable in $V[G]$ then $M\in V$.
Suppose not. Let $\dot{M}$ be a name for $M$ and $p_0$ force the facts we have
about $M$,
and letting $\lambda=\OR^M$, such that $p_0$ forces ``$\OR^{\dot{M}}=\lambda$''.
Then $p_0$ decides $\es_+^M$, so $M\in V$. For if not then we can form a
Boolean-valued comparison of generic interpretations of $\dot{M}$, below $p_0$.
This is almost the same as the proof of Corollary \ref{cor:wDJ_absoluteness},
and
leads to contradiction as there.
We leave the details to the reader.
 \end{proof}

 In Corollaries  \ref{cor:stack_it_absoluteness} and
\ref{cor:ZFC_iter_forcing_ab}, later in the paper, when we have some more
 results at our disposal, we will be able to deduce further generic absoluteness
 results under some choice assumptions.

\begin{ques} One can of course ask to what extent other types of
forcings preserve iterability.
Schindler and the author have found a couple of interesting counterexamples,
which are yet to appear; these include a model
of $\ZFC+$``$M_1^\#$ exists and is
$(\om,\om_1+1)$-iterable,
but there is a $\sigma$-distributive
forcing which forces that $\check{M}_1^\#$ is not $(\om,\om_1+1)$-iterable''.

The following questions, for example, seem to be open.
Let $V[G]$ be an
$\om$-closed forcing extension of $V$.
Is every $\om$-mouse of $V$ also an $\om$-mouse of $V[G]$?
Is every $\om$-mouse of $V[G]$ also an $\om$-mouse of $V$, or at least,
$\om_1$-iterable in $V$?
(Clearly every $\om_1$-iterable premouse of $V$ is also $\om_1$-iterable in
$V[G]$.)
Can $\Coll(\om_1,\kappa)$, for some $\kappa\geq\om_1$, consistently kill the
$(\om,\om_1+1)$-iterability of $M_1^\#$?\end{ques}

\begin{rem} Let $\Omega$ be regular uncountable and $M$ be an
$(\om,\Omega+1)$-iterable $\om$-mouse, as witnessed by $\Sigma$,
such that $M$ has 2 measurable cardinals, and suppose that $\PP$ is \emph{not}
$\Omega$-cc. Then the method of extending $\Sigma$ to $V[G]$
used for $\Omega$-cc forcing, fails for $\PP$. For let
$\left<p_\alpha\right>_{\alpha<\Omega}\sub\PP$
be an antichain. Let $\mu_0<\mu_1$ be measurables of $M$.
Define the $\PP$-name $\dot{\Tt}$ where below $p_\alpha$, $\dot{\Tt}$ is the
length $\alpha$ linear iteration of $M$ using
a measure on $\mu_0$ and its images, followed by a measure on the image of
$\mu_1$. Letting $\Xx$ be the comparison inflation of the $\dot{\Tt}$
through length $\Omega+1$, clearly $\Xx$ is just the length $\Omega$ linear
iteration of $M$ at $\mu_0$.
So the process is forced to fail.\end{rem}

\section{The factor tree $\Xx/\Tt$}\label{sec:factor_tree}

In this section we give a second perspective on inflation $\Tt\inflatearrow\Xx$,
in which we shift the focus from the $\Tt$-copied extenders to the
$\Tt$-inflationary extenders.
From this perspective, a natural analogy arises:
$\Xx$ induces what can be considered an iteration tree $\wt{\Xx}$ \emph{on
$\Tt$},
which consists of a sequence of (standard) \emph{iteration trees} $\Xx^\alpha$
on $M$
(instead of a sequence of models)
and whose extenders are just the $\Tt$-inflationary extenders of $\Xx$.
We will also define various tree embeddings
$\Pi^{\alpha\beta}:\Xx^\alpha\hookrightarrow\Xx^\beta$,
in the right circumstances, analogous to iteration maps, and introduce more
bookkeeping.
Benjamin Siskind has recently (in 2018) developed this perspective formally,
proving
 versions of the Shift Lemma and so forth in this context.
(We do not use any of Siskind's work here, however.)

\subsection{The factor tree order $<^{\Xx/\Tt}$}

We begin by describing an iteration tree order $<^{\Xx/\Tt}$ determined
by an inflation $\Tt\inflatearrow\Xx$.
When we reduce an (appropriate) stack of two normal trees $(\Tt,\Uu)$
to a
single normal tree $\Xx$ (via an iteration strategy with inflation
condensation),
then $\Xx$ will be an inflation of $\Tt$, and $\Uu$
will have tree order ${<^\Uu}={<^{\Xx/\Tt}}$.

\begin{dfn}\label{dfn:Tt-unravelling}
Let $\Xx$ be an inflation of $\Tt$ and $t=t^{\Tt\inflatearrow\Xx}$.

If $\alpha<\lh(\Xx)$ then the \dfnemph{$\Tt$-unravelling}\index{unravelling}
of $\Xx\rest(\alpha+1)$, if it exists, is the unique non-$\Tt$-pending
inflation
$\Ww$
of $\Tt$ such that $\Ww$ extends $\Xx\rest(\alpha+1)$ and
$t^{\Tt\inflatearrow\Ww}(\beta)=0$ for every $\beta\geq\alpha$.
Existence just depends on the models being wellfounded. We say that $\Xx$
is \dfnemph{$\Tt$-good} (or just \dfnemph{good}) iff the $\Tt$-unravelling of
$\Xx\rest(\alpha+1)$
exists for every $\alpha<\lh(\Xx)$.\index{good (inflation)}
\end{dfn}
\begin{dfn}\label{dfn:superscript_notation}\index{$\zeta^\alpha$}
\index{$\lambda^\alpha$}\index{$L^\alpha$}\index{$\eta_\delta$} Let $\Xx$ be a
good inflation of $\Tt$ and $t=t^{\Tt\inflatearrow\Xx}$.
Let $\left<\lambda^\alpha\right>_{\alpha<\iota}$ enumerate
in increasing order all $\lambda<\lh(\Xx)$
such that either $\lambda=0$, or $\lambda=\zeta+1$
where $t(\zeta)=1$, or $\lambda$ is a limit of such ordinals.
If $\alpha+1<\iota$ then let $\zeta^\alpha+1=\lambda^{\alpha+1}$
and $L^\alpha=[\lambda^\alpha,\zeta^\alpha]$,
and if $\alpha+1=\iota$ then let $L^\alpha=[\lambda^\alpha,\lh(\Xx))$.
So the intervals $L^\alpha$ are disjoint and partition $[0,\lh(\Xx))$.
For $\delta<\lh(\Xx)$ let $\eta_\delta$ be the $\eta<\iota$ such that
$\delta\in
L^\eta$.

We write $\Xx^\alpha$\index{$\Xx^\alpha$, $t^\alpha$, etc} for the
$\Tt$-unravelling of $\Xx\rest(\lambda^\alpha+1)$,
with associated objects $(t^\alpha,C^\alpha,\ldots)$.
If $\lambda^\alpha\in C^\alpha$ then also let
$\theta^\alpha=f^\alpha(\lambda^\alpha)$;\index{$\theta^\alpha$}
otherwise $\theta^\alpha$ is not defined. Then either
\begin{enumerate}[label=--]
 \item $\lambda^\alpha\notin C^\alpha$ and $\lh(\Xx^\alpha)=\lambda^\alpha+1$,
or
 \item $\lambda^\alpha\in C^\alpha$ and
$\lh(\Xx^\alpha)=\lambda^\alpha+(\lh(\Tt)-\theta^\alpha)$.
\end{enumerate}

Let
$(\lambda^\alpha,\zeta^\alpha,L^\alpha,\Xx^\alpha,t^\alpha,\ldots)^{
\Tt\inflatearrow\Xx}\eqdef
(\lambda^\alpha,\zeta^\alpha,L^\alpha,\Xx^\alpha,t^\alpha,\ldots)$.
If $\lambda^\alpha\in C^\alpha$, then for $\xi<\lh(\Tt)$
we set\index{$(I^\alpha_\xi)^{\Tt\inflatearrow\Xx}$,
$(\pi^\alpha_\xi)^{\Tt\inflatearrow\Xx}$, etc}\index{$\Tt\inflatearrow\Xx$}
\[
(I^\alpha_{\xi})^{\Tt\inflatearrow\Xx}\eqdef\lim_{\lambda\to\lh(\Xx^\alpha)}I_{
\lambda;\xi}^{\Tt\inflatearrow\Xx^\alpha}, \]
\[ (\pi^\alpha_{\xi
i})^{\Tt\inflatearrow\Xx}\eqdef\lim_{\lambda\to\lh(\Xx^\alpha)}\pi_{\lambda;\xi
i}^{\Tt\inflatearrow\Xx^\alpha},\]
etc. Note here that because $\Xx^\alpha\rest[\lambda^\alpha,\lh(\Xx^\alpha))$
is
formed
by copying,
this makes sense and
$(I^\alpha_{\xi})^{\Tt\inflatearrow\Xx}=I_{\lambda;\xi}^{
\Tt\inflatearrow\Xx^\alpha}$, etc,
for all sufficiently large $\lambda<\lh(\Xx^\alpha)$. Of course
if $\lh(\Tt)$ is a successor, then
$\lh(\Xx^\alpha)$ is a successor $\lambda+1$ and
$(I^\alpha_{\xi})^{\Tt\inflatearrow\Xx}=I_{\lambda;\xi}^{
\Tt\inflatearrow\Xx^\alpha}$, etc; this is the main case of interest.
\end{dfn}

\begin{dfn}\label{dfn:<^Xx/Tt}\index{$<^{\Xx/\Tt}$}\index{factor tree} Let
$\Xx$ be a good
inflation of $\Tt$. Adopt notation as in
\ref{dfn:superscript_notation}.
Then $<^{\Xx/\Tt}$ denotes
the order on $\iota$ defined recursively as follows: for each $\alpha<\iota$,
\[
[0,\alpha)_{\Xx/\Tt}=\bigcup_{\delta<^\Xx\lambda^\alpha}[0,\eta_\delta]_{\Xx/\Tt
}.\qedhere \]
\end{dfn}

We remark that
$\{\eta_\delta\bigm| \delta<^\Xx\lambda^\alpha\}$
need not be closed downward under $<^{\Xx/\Tt}$. We will verify soon that
$<^{\Xx/\Tt}$ is an iteration tree order, but first we have the following
approximation:
\begin{lem}\label{lem:Xx/Tt_simple_facts}
Let $\Xx$ be a good inflation of $\Tt$. Adopt notation as above.
Let $b_\alpha=[0,\alpha)_{\Xx/\Tt}$.
Then
\tu{(}i\tu{)} $<^{\Xx/\Tt}$ is transitive,
\tu{(}ii\tu{)} $b_\alpha\sub\alpha$,
\tu{(}iii\tu{)} if $\alpha=\gamma+1$ then
$\eta=\max(b_\alpha)$ and $[0,\eta]_{\Xx/\Tt}\sub b_\alpha$ where
$\eta$ is least such that
$\crit(E^\Xx_{\zeta^\gamma})<\iota(\exit^\Xx_{\zeta^\eta})$,
\footnote{In fact, by \ref{lem:<^Xx/Tt} below, $b_\alpha=[0,\eta]_{\Xx/\Tt}$.}
and \tu{(}iv\tu{)} if $\alpha$ is a limit then $b_\alpha$ is cofinal in
$\alpha$.
\end{lem}
\begin{proof}
 Transitivity is a straightforward induction, and the other facts follow easily
from the definitions.
\end{proof}

\begin{dfn}\index{$\Tt^{\geq\alpha}$}\index{${<^{(\alpha)}}$}
For an iteration tree $\Vv$ and $\alpha<\lh(\Vv)$, let
$\Vv^{\geq\alpha}=\Vv\rest[\alpha,\lh(\Vv))$,
considered as an iteration tree on the phalanx $\Phi(\Vv\rest\alpha+1)$.
Given an iteration tree order $<_0$ and
$\alpha<\lh(<_0)$, let
${{<}_0^{(\alpha)}}={{<}_0}\rest\{\delta\bigm|\delta\geq_0\alpha\}$.
\end{dfn}

\begin{lem}\label{lem:<^Xx/Tt_above_easy}
Let $\Xx$ be a good inflation of $\Tt$. Adopt notation as above.
Let $\alpha<\lh(\Xx/\Tt)$.
Then:
\begin{enumerate}[label=--]
 \item
$\lambda^\alpha\in(C^-)^\alpha$ iff $\lambda^\alpha+1<\lh(\Xx^\alpha)$.
\end{enumerate}
Suppose $\lambda^\alpha\notin(C^-)^\alpha$. Then:
\begin{enumerate}[label=--]
\item $L^\alpha=[\lambda^\alpha,\lambda^\alpha]$,
\item if $\lambda\geq_\Xx\lambda^\alpha$ then there is $\delta$ such that
$\lambda=\lambda^\delta$, and moreover, $\lambda^\delta\notin(C^-)^\delta$,
\item if $\lambda^\delta\geq_\Xx\lambda^\alpha$ and $\lambda^\delta\in
C^\delta$
then
$\lambda^\alpha\in C^\alpha$, and
\item the map $\xi\mapsto\lambda^\xi$ restricts to an isomorphism between
\mbox{$(<^{\Xx/\Tt})^{(\alpha)}$} and \mbox{$(<^{\Xx})^{(\lambda^\alpha)}$}.
\end{enumerate}
\end{lem}
\begin{proof}
This is straightforward; the last clause
 is by induction on $\lh(\Xx/\Tt)$.
\end{proof}

\begin{lem}\label{lem:<^Xx/Tt}
Let $\Xx$ be a good inflation of $\Tt$. Adopt notation as above.
Then
\begin{enumerate}[label=\arabic*.,ref=\arabic*]
 \item\label{item:<^Xx/Tt_is_it_tree_order} $<^{\Xx/\Tt}$ is an iteration tree
order on $\lh(\Xx/\Tt)$.
   \item\label{item:<^Xx/Tt_respects_<^Xx} For all
$\mu<^\Xx\lambda<\lh(\Xx)$, we have
$\eta_\mu\leq^{\Xx/\Tt}\eta_\lambda$.
\end{enumerate}
Moreover, let $\alpha\leq^{\Xx/\Tt}\beta<\lh(\Xx/\Tt)$ with $\lambda^\beta\in
C^\beta$
\tu{(}so
 $\lambda^\alpha\in C^\alpha$ by \ref{lem:<^Xx/Tt_above_easy}\tu{)}.
Then:
\begin{enumerate}[resume*]
 \item\label{item:xi_class_is_Xx/Tt_below}
$\gamma^\alpha_{\theta\kappa}\in[\lambda^\alpha,\lh(\Xx^\alpha))\cup\bigcup_{
\delta<^{\Xx/\Tt}\alpha}L^\delta$
for all $(\theta,\kappa)$.
\item\label{item:gamma_agmt_along_branch} Suppose $\alpha<\beta$. Let
$\xi+1=\successor^{\Xx/\Tt}(\alpha,\beta)$
and $\gamma=\pred^{\Xx}(\lambda^{\xi+1})$. Then:
\begin{enumerate}[label=\tu{(}\alph*\tu{)}]
\item  $\gamma\in L^\alpha$ and $\theta^\alpha\leq\theta\eqdef
f^\alpha(\gamma)\leq\theta^\beta$.
 \item For each $\theta'<\theta$ and $\kappa$ we have
$I^\alpha_{\theta'}=I^\beta_{\theta'}\sub\gamma$
and $\gamma^\alpha_{\theta'\kappa}=\gamma^\beta_{\theta'\kappa}<\gamma$,
 \item $I^\alpha_{\theta}\psub I^\beta_{\theta}$. In fact,
$\gamma^\alpha_{\theta}=\gamma^\beta_{\theta}$
 but $\delta^\alpha_{\theta}=\gamma<^\Xx\delta^\beta_{\theta}$.
 \item If $\theta+1<\lh(\Tt)$ then for each $\kappa<\iota(\exit^\Tt_{\theta})$,
 either:
 \begin{enumerate}[label=--]
\item $\pi^\alpha_{\theta\kappa}(\kappa)<\crit(E^\Xx_{\zeta^\xi})$ and
$\gamma^\alpha_{\theta\kappa}=\gamma^\beta_{\theta\kappa}$,
 \item $\pi^\alpha_{\theta\kappa}(\kappa)\geq\crit(E^\Xx_{\zeta^\xi})$ and
$\gamma^\alpha_{\theta\kappa}=\gamma<^\Xx\gamma^\beta_{\theta\kappa}$.
\end{enumerate}
\end{enumerate}
\end{enumerate}
\end{lem}
\begin{proof}
By induction on $\lh(\Xx)$.

Part \ref{item:<^Xx/Tt_is_it_tree_order}: By induction, we may assume that
$\lh(\Xx)=\lambda^\alpha+1$.
By \ref{lem:Xx/Tt_simple_facts} and induction, it suffices to verify that
$b_\alpha=[0,\alpha)_{\Xx/\Tt}$
is linearly ordered by $<^{\Xx/\Tt}$ and closed below $\alpha$.
Let $\delta,\varepsilon\in b_\alpha$.
We can fix $\delta',\varepsilon'<^{\Xx}\lambda^\alpha$
with  $\delta\leq^{\Xx/\Tt}\eta_{\delta'}$ and
$\varepsilon\leq^{\Xx/\Tt}\eta_{\varepsilon'}$.
We may assume that $\delta'\leq^{\Xx}\varepsilon'$. Note
$\eta_{\delta'}\leq\eta_{\eps'}<\alpha$.
By induction with part \ref{item:<^Xx/Tt_respects_<^Xx} then,
$\eta_{\delta'}\leq^{\Xx/\Tt}\eta_{\varepsilon'}$.
So if $\delta=\eta_{\varepsilon'}$ or $\varepsilon=\eta_{\varepsilon'}$
then by transitivity, we are done, and otherwise, use the inductive hypothesis
that $[0,\eta_{\varepsilon'})_{\Xx/\Tt}$
is linearly ordered by $<^{\Xx/\Tt}$.
Finally, $b_\alpha$ is closed below $\alpha$,
by induction and because if $\alpha$ is a limit then $b_\alpha$ is unbounded in
$\alpha$ and linearly ordered by $<^{\Xx/\Tt}$.

For parts \ref{item:<^Xx/Tt_respects_<^Xx}--\ref{item:gamma_agmt_along_branch},
we consider a few cases:

\begin{casefour} $\lh(\Xx/\Tt)=\iota=1$.

 This case is trivial.
\end{casefour}

\begin{casefour} $\iota=\xi+2$.

We may assume that $\Xx=\Xx_{\xi+1}$. Let $\alpha=\pred^{\Xx/\Tt}(\xi+1)$
and
$\gamma=\pred^\Xx(\lambda^{\xi+1})$,
so $\gamma\in L^\alpha$. Let $\theta=f^\alpha(\gamma)$.
Then $\gamma=\delta^{\alpha}_\theta$
and $\Pi^{\xi+1}\rest(\theta+1)$ is the
$E^\Xx_\xi$-inflation
of $\Pi^\alpha\rest(\theta+1)$,
and then $\Xx_{\xi+1}$ is given by then copying $\Tt\rest[\theta,\lh(\Tt))$,
starting from $\Pi^{\xi+1}\rest(\theta+1)$ (via the one-step extension
at successor stages, and copying at limits).

Consider part \ref{item:<^Xx/Tt_respects_<^Xx}.
If $\lambda=\lambda^{\xi+1}$ then the property holds
for $\lambda$ (and corresponding $\mu<^\Xx\lambda$) directly by definition of
$<^{\Xx/\Tt}$.
We now proceed by a sub-induction through
$\lambda>\lambda^{\xi+1}$.
By the sub-induction, we may  assume that
$\lambda=\eps+1$ and $\pred^{\Xx}(\eps+1)=\upsilon\notin
L^{\xi+1}$. Now  $E^\Xx_{\eps}$ is
copied from $\Tt$. Let $\theta'$ be such that
$\theta'-f^{\xi+1}(\theta)=\eps-\lambda^{\xi+1}$,
so $E^\Xx_{\eps}$ is the copy of $E^\Tt_{\theta'}$.
Let $\kappa=\crit(E^\Tt_{\theta'})$.
Then $\upsilon=\gamma^\alpha_{\theta'\kappa}$,
and by induction with parts  \ref{item:xi_class_is_Xx/Tt_below} and
\ref{item:gamma_agmt_along_branch},
therefore $\eta_\upsilon<^{\Xx/\Tt}\alpha=\eta_{\lambda^\alpha}$.
But now if $\mu<^\Xx\lambda$ then $\mu\leq^\Xx\upsilon$,  so by
induction,
$\eta_\mu\leq^{\Xx/\Tt}\eta_\upsilon$, so
$\eta_\mu<^{\Xx/\Tt}\eta_\lambda=\xi+1$.

Parts \ref{item:xi_class_is_Xx/Tt_below} and \ref{item:gamma_agmt_along_branch}
are straightforward consequences of how $\Pi^{\xi+1}$ is produced from
$\Pi^\alpha$.
\end{casefour}

\begin{casefour}
$\iota=\alpha+1$ where $\alpha$ is a limit,
and $\lambda^\alpha$ is a $(\Tt,\Xx)$-limit.

Consider part \ref{item:xi_class_is_Xx/Tt_below}.
 If $\theta>f^\alpha(\lambda^\alpha)=\theta^\alpha$
then $\gamma^\alpha_{\theta\kappa}>\lambda^\alpha$, so
$\gamma^\alpha_{\theta\kappa}\in[\lambda^\alpha,\lh(\Xx^\alpha))$.
And if $\theta=\theta^\alpha$
then by the case hypothesis,
$\lambda^\alpha=\gamma^\alpha_\theta=\delta^\alpha_\theta$,
and $\gamma^\alpha_{\theta\kappa}=\lambda^\alpha$ for
each
$\kappa$.
So  suppose $\theta<\theta^\alpha$ and fix $\kappa$. Then
$\gamma^\alpha_{\theta\kappa}=\gamma_{\mu;\theta\kappa}$
for all sufficiently large $\mu<^\Xx\lambda^\alpha$. Fix such
$\mu$;\footnote{In
an earlier
draft of this paper, $\mu$ was supposedly
chosen independent of $\kappa$.
But this need not actually be possible.} we may
choose $\mu$ with
$f(\mu)>\theta$ where $f=f^{\Tt\inflatearrow\Xx}$.
We have $\gamma_{\mu;\theta\kappa}=\gamma^{\eta_\mu}_{\theta\kappa}$.
But then
\[
\gamma^\alpha_{\theta\kappa}=\gamma_{\mu;\theta\kappa}=
\gamma^{\eta_\mu}_{\theta\kappa}
\in[0,\eta_\mu]_{\Xx/\Tt}\sub[0,\alpha)_{\Xx/\Tt} \]
by induction and definition, which suffices.

Part \ref{item:<^Xx/Tt_respects_<^Xx} is proved much like
in the successor case, combined with considerations as above.
Part \ref{item:gamma_agmt_along_branch}
is easy (note the ``$\alpha$'' there is not the $\alpha$
of the case hypothesis).
\end{casefour}

\begin{casefour} $\alpha$ is a limit but  $\lambda^\alpha$ is not a
$(\Tt,\Xx)$-limit.

Part \ref{item:xi_class_is_Xx/Tt_below}:
For $\theta\neq\theta^\alpha=f^\alpha(\lambda^\alpha)$, it is basically as
before.
Consider $\theta=\theta^\alpha$.
Fix
$\kappa<\OR(M^\Tt_{\theta^\alpha})$,
with $\kappa<\iota^\Tt_\alpha$ if $\theta^\alpha\in\lh(\Tt)^-$.
Choose
$\mu<^\Xx\lambda^\alpha$ large enough
that $(\mu,\lambda^\alpha)_\Xx$ does not drop,
$f(\mu)=\theta^\alpha$ and
\[\text{if
}\pi^\alpha_{\theta^\alpha\kappa}(\kappa)<\delta(\Xx\rest\lambda^\alpha)\text{
then
}\pi^\alpha_{\theta^\alpha\kappa}(\kappa)<\crit(i^\Xx_{\mu\lambda^\alpha}).\]
Then if $\gamma^\alpha_{\theta^\alpha\kappa}<\lambda^\alpha$
then $\gamma^\alpha_{\theta^\alpha\kappa}=\gamma_{\mu;\theta^\alpha\kappa}$,
so the property for $(\theta^\alpha,\kappa)$ follows by induction as before.

Part \ref{item:<^Xx/Tt_respects_<^Xx} again follows by combining
such considerations with the argument from the
successor case. Part \ref{item:gamma_agmt_along_branch}
is again easy.\qedhere
\end{casefour}
\end{proof}

\subsection{Tree embeddings of the factor tree}

\begin{dfn}
\label{lem:theta^alpha,beta}\index{$\lambda^{\alpha\beta}$,
$\kappa^{\alpha\beta}$, etc (inflation)}
Let $\Xx$ be a good inflation of $\Tt$. Adopt notation from
Lemma \ref{lem:<^Xx/Tt}(\ref{item:gamma_agmt_along_branch}).
Then $\lambda^{\alpha\beta}$ denotes $\gamma$,
$\theta^{\alpha\beta}$ denotes $f^\alpha(\lambda^{\alpha\beta})$,
and $\kappa^{\alpha\beta}$ denotes the least $\kappa$ such that
$\pi^{\alpha}_{\theta\kappa}(\kappa)\geq\crit(E^\Xx_{\zeta^\xi})$ where
$\theta=\theta^{\alpha\beta}$
(because $\lambda^\beta\in C^\beta$, this makes sense and holds of
$\kappa=\In(E^\Tt_\theta)$ if $\theta+1<\lh(\Tt)$,
and holds of $\kappa=\OR(M^\Tt_\theta)$ otherwise).
\end{dfn}

\begin{dfn}\label{dfn:intermediate_tree_embeddings}
\index{$\Pi^{\alpha\beta}:\Xx^\alpha\tembto\Xx^\beta$ (inflation)}
\index{$\gamma^{\alpha\beta}_\lambda$, etc (inflation)}
Let $\Xx$ be a good inflation of $\Tt$. Adopt notation as before.
Let $\alpha\leq^{\Xx/\Tt}\beta<\lh(\Xx/\Tt)$ with $\lambda^\beta\in C^\beta$.
We define a (putative, verified in \ref{lem:Pi^alpha,beta_is_tree_embedding})
tree embedding $\Pi^{\alpha\beta}:\Xx^\alpha\hookrightarrow\Xx^\beta$ as
follows.
Write $\gamma^{\alpha\beta}_\lambda=\gamma_{\Pi^{\alpha\beta}\lambda}$ etc. It
suffices to specify
$I^{\alpha\beta}_\lambda=[\gamma^{\alpha\beta}_\lambda,\delta^{\alpha\beta}
_\lambda]$ for each $\lambda<\lh(\Xx^\alpha)$.
We set:
\begin{enumerate}[label=--]
\item $I^{\alpha\beta}_\lambda=[\lambda,\lambda]$ if $\alpha=\beta$ or
$\lambda<\lambda^\alpha$.
\item
$I^{\alpha\beta}_{\lambda^\alpha}=[\lambda^\alpha,\delta^{\beta}_{\theta^\alpha}
]_{\Xx^\beta}$ if $\alpha<\beta$.
\item $I^{\alpha\beta}_\lambda=I^\beta_{f^\alpha(\lambda)}$ if $\alpha<\beta$
and $\lambda>\lambda^\alpha$.\qedhere
\end{enumerate}
\end{dfn}

\begin{lem}
Let $\Tt,\Xx,\alpha,\beta$ be as in \ref{dfn:intermediate_tree_embeddings},
and $\lambda\geq\lambda^\alpha$.
Let
\begin{enumerate}[label=--]
 \item
$\eps$ be the supremum of $\alpha$ and all $\xi+1\leq^{\Xx/\Tt}\beta$
such that $\theta^{\xi+1}<f^\alpha(\lambda)$,
\item $\eps'$ be the supremum of all $\xi\in[\alpha,\beta]_{\Xx/\Tt}$
such that $\theta^\xi\leq f^\alpha(\lambda)$.
 \end{enumerate}
Then:
\begin{enumerate}[label=--]
\item $\left<\theta_\xi\right>_{\xi\in[\alpha,\beta]_{\Xx/\Tt}}$ is continuous,
monotone increasing,
so $\theta^\eps\leq\theta^{\eps'}\leq f^\alpha(\lambda)$.
\item
$\gamma^{\alpha\beta}_{\lambda}=\lambda^\eps+(f^\alpha(\lambda)-\theta^\eps)$,
so
$f^{\eps}(\gamma^{\alpha\beta}_\lambda)=f^\alpha(\lambda)=f^\beta(\gamma^{
\alpha\beta}_\lambda)$.
\item
$\delta^{\alpha\beta}_\lambda=\lambda^{\eps'}+(f^\alpha(\lambda)-\theta^{\eps'}
)$,
so
$f^{\eps'}(\delta^{\alpha\beta}_\lambda)=f^\alpha(\lambda)=f^\beta(\delta^{
\alpha\beta}_\lambda)$.
\end{enumerate}
\end{lem}
\begin{proof}
By induction on $\beta$. When $\beta=\alpha$ it is trivial, and for successor
$\beta$ it follows directly from the definitions.
So suppose $\beta$ is a limit. If $\theta^\xi=\theta^\beta$ for some
$\xi<^{\Xx/\Tt}\beta$ then it is just like in the successor case,
so suppose otherwise, that is, $\lambda^\beta$ is a $(\Tt,\Xx)$-limit.
Then note that
\[
\theta^\beta=f^\beta(\lambda^\beta)=\sup_{\lambda<^\Xx\lambda^\beta}f^{
\Tt\inflatearrow\Xx}(\lambda)=\sup_{\eta<^{\Xx/\Tt}\beta}\theta^\eta \]
and for $\theta<\theta^\beta$,
\[
I^\beta_\theta=\lim_{\lambda<^\Xx\lambda^\beta}I^{\Tt\inflatearrow\Xx}_{
\lambda;\theta}=\lim_{\eta<^{\Xx/\Tt}\beta}I^\eta_\theta. \]
So for $\lambda\in[\lambda^\alpha,\lambda^\alpha+(\theta^\beta-\theta^\alpha))$
the result follows by induction,
and for larger $\lambda$ it is easy.
\end{proof}

During the course of the proof of the following lemma we will specify notation
for various embeddings which will also be needed later.

\begin{lem}\label{lem:Pi^alpha,beta_is_tree_embedding}
 Let $\Tt,\Xx,\alpha,\beta$ be as in \ref{dfn:intermediate_tree_embeddings}.
 Then $\Pi^{\alpha\beta}:\Xx^\alpha\hookrightarrow\Xx^\beta$ is a bounding tree
embedding.
\end{lem}
\begin{proof}
Write $\Pi=\Pi^{\alpha\beta}$. We will show that
\[
\Pi\rest(\delta^\alpha_\theta+1):(\Xx^\alpha,
\delta^\alpha_\theta+1)\hookrightarrow\Xx^\beta \]
is a bounding tree embedding, by induction
on
 $\theta<\lh(\Tt)$.
 We simultaneously
define embeddings $\pi^{\alpha\beta}_\theta$,
$\om^{\alpha\beta}_\theta$ and $\pi^{\alpha\beta}_{\theta\kappa}$ for
$\kappa\leq\OR(M^\Tt_\theta)$, as follows,
and verify that
\begin{enumerate}[label=\arabic*.,ref=\arabic*]
 \item\label{item:int_pi_comm}
$\gamma^\Pi_{\gamma^\alpha_\theta}=\gamma^\beta_{\theta}$
(and $P^\alpha_\theta=M^{\Xx^\alpha}_{\gamma^\alpha_\theta}$ and
$P^\beta_\theta=M^{\Xx^\beta}_{\gamma^\beta_\theta}=P^\Pi_{\gamma^\alpha_\theta}
$) and defining $\pi^{\alpha\beta}_\theta$ by
\[ \pi^{\alpha\beta}_\theta=\pi^\Pi_{\gamma^\alpha_\theta}:P^\alpha_\theta\to
P^\beta_\theta, \]
we have
$\pi^{\alpha\beta}_\theta\com\pi^\alpha_\theta=\pi^\beta_\theta$.
\item\label{item:int_om_comm}
$\delta^\Pi_{\delta^\alpha_\theta}=\delta^\beta_{\theta}$ (and
$Q^\alpha_\theta=\exit^{\Xx^\alpha}_{\delta^\alpha_\theta}$ and
$Q^\beta_\theta=\exit^{\Xx^\beta}_{\delta^\beta_\theta}=Q^\Pi_{
\delta^\alpha_\theta}$),
and defining $\om^{\alpha\beta}_{\theta}$ by
\[ \om^{\alpha\beta}_{\theta}=\om^\Pi_{\delta^\alpha_\theta}:Q^\alpha_\theta\to
Q^\beta_\theta,\]
we have
$\om^{\beta}_\theta=\om^{\alpha\beta}_\theta\com\om^\alpha_\theta$.
\item\label{item:int_pi_kappa_comm} for each $\kappa$, letting
$\psi_\theta(\kappa)=\pi^\alpha_{\theta\kappa}(\kappa)$, we have
\[
\gamma^\Pi_{\gamma^\alpha_{\theta\kappa}\psi_\theta(\kappa)}=\gamma^\beta_{
\theta\kappa}\text{
and
}P^\alpha_{\theta\kappa}=M^{\Xx^\alpha}_{\gamma^\alpha_{\theta\kappa}
\psi_\theta(\kappa)}\text{
and
}P^\beta_{\theta\kappa}=P^\Pi_{\gamma^\alpha_{\theta\kappa}\psi_\theta(\kappa)},
\]
and defining
$\pi^{\alpha\beta}_{\theta\kappa}$ by
\[
\pi^{\alpha\beta}_{\theta\kappa}=\pi^\Pi_{\gamma^\alpha_{\theta\kappa}
\psi_\theta(\kappa)}:
P^\alpha_{\theta\kappa}\to
P^\beta_{\theta\kappa}, \]
we have
$\pi^{\beta}_{\theta\kappa}=\pi^{\alpha\beta}_{\theta\kappa}\com\pi^\alpha_{
\theta\kappa}$.
\end{enumerate}

Let $\eta+1=\successor^{\Xx/\Tt}(\alpha,\beta)$
and let $\mu=\pred^\Xx(E^\Xx_{\zeta^\eta})$, so
$f^\alpha(\mu)=\theta^{\eta+1}$.
For $\theta<\theta^{\eta+1}$,
 everything is
trivial,
as  $\Xx^\alpha\rest(\mu+1)=\Xx^\beta\rest(\mu+1)$ and
$\Pi\rest\mu=\id$
and $I^\alpha_\theta=I^\beta_\theta$
and for each $\kappa$,
we have
$\gamma^\alpha_{\theta\kappa}=\gamma^\beta_{\theta\kappa}$,
so $P^\alpha_\theta=M^{\Xx}_{\gamma^\alpha_\theta}=P^\beta_\theta$,
$\pi^\alpha_\theta=\pi^\beta_\theta$,
$Q^\alpha_\theta=\exit^{M^\Xx_{\delta^\alpha_\theta}}=Q^\beta_\theta$, etc, and
$\pi^{\alpha\beta}_\theta$,
$\om^{\alpha\beta}_\theta$ and
$\pi^{\alpha\beta}_{\theta\kappa}$ are just the identity maps.

Now consider $\theta=\theta^{\eta+1}$. We have
$\gamma^\alpha_\theta\leq^\Xx\mu=\delta^\alpha_\theta$.
By definition of $\Pi$, if $\mu=\lambda^\alpha$
then
\[
\gamma^\Pi_{\gamma^\alpha_\theta}
=\gamma^\beta_\theta=\gamma^\alpha_\theta\leq^\Xx\lambda^\alpha\text{ and }
\delta^\alpha_\theta=\lambda^\alpha<^\Xx\delta^\Pi_{\gamma^\alpha_\theta}
=\delta^\beta_\theta,\]
and if $\mu>\lambda^\alpha$ then
\[ \gamma^\alpha_\theta=\delta^\alpha_\theta=\gamma^\beta_\theta=\mu\text{ and
}I^\Pi_\mu=I^\beta_\theta=[\mu,\delta^\beta_\theta]_{\Xx}.\]
So in either case,
$\gamma^\alpha_\theta=\gamma^\beta_\theta=\gamma^\Pi_{\gamma^\alpha_\theta}$,
and $\pi^\alpha_\theta=\pi^\beta_\theta$ and $\pi^{\alpha\beta}_\theta=\id$, so
property \ref{item:int_pi_comm} is trivial.
Also, $E^{\Xx^\alpha}_\mu=E^{\Tt\inflatearrow\Xx_\alpha}_\alpha$,
which is the iteration image of $\pi^\alpha_\theta(E^\Tt_\theta)$,
and since $I^\beta_\theta$ does not drop below the image of
$\pi^\beta_\theta(E^\Tt_\theta)$,
therefore $I^\Pi_\mu$ does not drop below the image of $E^{\Xx^\alpha}_\mu$.
Therefore, $\Pi\rest(\mu+1)$ is a tree embedding. Similarly, $\Pi\rest(\mu+1)$
is bounding.
Properties \ref{item:int_om_comm} and \ref{item:int_pi_kappa_comm} follow
directly from this and property \ref{item:int_pi_comm}.

Now suppose we have the induction hypotheses for
$\Pi\rest(\delta^\alpha_\theta+1)$,
where $\theta\geq\mu$. We have
$\gamma^\alpha_{\theta+1}=\delta^\alpha_\theta+1$. Using the commutativity
given
by this, and the fact that tree embeddings can be freely extended (in this case
by copying),
we get that $\pi^{\alpha\beta}_{\theta+1}=\pi^{\Pi}_{\gamma^\alpha_{\theta+1}}$
is well-defined,
and property \ref{item:int_pi_comm} holds. So like before,
$I^\Pi_{\gamma^\alpha_\theta+1}=I^\beta_{\theta+1}$
does not drop below the image
\[
\pi^\Pi_{\gamma^\alpha_{\theta+1}}(E^{\Xx^\alpha}_{\gamma^\alpha_{\theta+1}})=
 \pi^\beta_{\theta+1}(E^\Tt_{\theta+1})
\]
(assuming that $\theta+1<\lh(\Tt)$; otherwise there is no drop in model at all),
and $\Pi\rest(\delta^\alpha_{\theta+1}+1)$ is a bounding tree embedding.
Again, properties \ref{item:int_om_comm} and \ref{item:int_pi_kappa_comm}
follow.

For limit $\theta$, everything fits together easily by commutativity.
This completes the proof.
\end{proof}
\begin{dfn}\label{dfn:Pi^alpha,beta_embs}\index{$\pi^{\alpha\beta}_\theta$,
$\pi^{\alpha\beta}_{\theta\kappa}$}
\index{$\om^{\alpha\beta}_\theta$}
 Let $\Tt,\Xx,\alpha,\beta$ be as in \ref{dfn:intermediate_tree_embeddings}.
 Then for $\theta<\lh(\Tt)$
 and $\kappa\leq\OR(M^\Tt_\theta)$ we define
$\pi^{\alpha\beta}_\theta,\om^{\alpha\beta}_\theta,\pi^{\alpha\beta}_{
\theta\kappa}$
 as in the proof of \ref{lem:Pi^alpha,beta_is_tree_embedding}.
 \end{dfn}

 \begin{lem}\label{lem:int_tree_embs_int_comm_extra_agmt}
 Let $\Tt,\Xx,\alpha,\beta$ be as in \ref{dfn:intermediate_tree_embeddings} and
let $\gamma\in[\alpha,\beta]_{\Xx/\Tt}$.
 Then:
 \begin{enumerate}[label=\arabic*.,ref=\arabic*]
  \item\label{item:int_tree_embs_int_comm}
$\pi^{\alpha\beta}_\theta=\pi^{\gamma\beta}_\theta\com\pi^{\alpha\gamma}
_\theta$
and
$\om^{\alpha\beta}_\theta=\om^{\gamma\beta}_\theta\com\om^{\alpha\gamma}
_\theta$
and
$\pi^{\alpha\beta}_{\theta\kappa}=\pi^{\gamma\beta}_{\theta\kappa}\com\pi^{
\alpha\gamma}_{\theta\kappa}$.
 \item\label{item:pure_copy_extra_agmt} If
$\theta^\beta\leq\theta<\theta'<\lh(\Tt)$ and $\kappa'\leq\OR(M^\Tt_{\theta'})$
then\footnote{Recall that in general for tree embeddings
$\Pi:\Uu\hookrightarrow\Vv$ we have for example
$\om^\Pi_\xi\rest\iota(\exit^{\Uu}_\xi)\sub\pi^\Pi_{\xi'}$
for $\xi<\xi'<\lh(\Uu)$; here we get a little more agreement.}
\[
\om^{\alpha\beta}_{\theta}\sub\pi^{\alpha\beta}_{\theta'},\om^{\alpha\beta}_{
\theta'},\pi^{\alpha\beta}_{\theta'\kappa'}, \]
and if $f^\alpha(\lambda)=\theta=f^\beta(\lambda')$
and $\gamma=\In(E^{\Xx^\alpha}_{\lambda})<\OR(M^{\Xx^\alpha}_{\lambda+1})$
then
\[
\pi^{\alpha\beta}_{\theta'}(\gamma)=\om^{\alpha\beta}_{\theta'}(\gamma)=\pi^{
\alpha\beta}_{\theta'\kappa'}(\gamma)=\In(E^{\Xx^\beta}_{\lambda'}).
\]
\end{enumerate}
\end{lem}
\begin{proof}
Part \ref{item:int_tree_embs_int_comm} is proved much like the commutativity in
\ref{lem:Pi^alpha,beta_is_tree_embedding}.

Part \ref{item:pure_copy_extra_agmt} holds because
$\Xx^\beta\rest[\lambda^\beta,\lh(\Xx^\beta))$ is the copy of
$\Xx^\alpha\rest[\lambda,\lh(\Xx^\alpha))$,
where $f^\alpha(\lambda)=\theta^\beta$, under the base copy maps
$\om^{\alpha\beta}_{\theta^\beta}$
and $\pi^{\alpha\beta}_\theta,\pi^{\alpha\beta}_{\theta\kappa}$ for
$\theta\leq\theta^\beta$.
\end{proof}

\section{Iterability for stacks via normal
realization}\label{sec:normal_realization}

In this section we will prove the main result of the paper:

\begin{tm}\label{thm:stacks_iterability}
 Let $\Omega>\om$ be regular.
 Let $\Sigma$ be a regularly $(\Omega+1)$-total strategy for $M$ with inflation
condensation,
 where if $M$ is a wcpm then $M$ is slightly coherent. Then
 $\Sigma$ extends to a strategy $\Sigma^*$ for stacks of length $\Omega$.
More precisely, letting $m=m^\Sigma$:\footnote{See \ref{dfn:strategy_classes}.}
\begin{enumerate}[label=--]
 \item if $M$ is a wcpm then $M$ is $(\Omega,\Omega+1)^*$-iterable,
 \item if $\Sigma$ is a $\udash$strategy then $M$ is $(\udash
m,\Omega,\Omega+1)^*$-iterable, and
 \item if $\Sigma$ is an $(m,\Omega+1)$-strategy then $M$ is
$(m,\Omega,\Omega+1)^*$-iterable,
\end{enumerate}
as witnessed by some $\Sigma^*$ with $\Sigma\sub\Sigma^*$.
\end{tm}

\begin{rem}
The proof will in fact give an explicit construction of a specific such strategy
$\Sigma^*$ from $\Sigma$, and we denote this $\Sigma^*$ by $\Sigma^{\stk}$ (see
Definition \ref{dfn:Sigma^stk}).\index{$\Sigma^\stk$}
If $\Sigma$ is conveniently inflationary, then
for each stack $\Ttvec$ via $\Sigma^{\stk}$ of length ${<\Omega}$, we will
produce a tree $\Xx$ via $\Sigma$,
and, roughly, lifting maps from $\Ttvec$ into $\Xx$.
For MS-indexed $M$ we must also translate through $\udash$iteration strategies.
We write $\Ww^\Sigma(\Ttvec)=\Xx$.\index{$\Ww^\Sigma$} This and other
notation is also recorded
later in Definitions \ref{dfn:W(T,U)_etc} and \ref{dfn:W(Tvec)_etc}.
In \S\ref{sec:npc} we will verify some extra properties of
$\Sigma^{\stk}$,
given that $\Sigma$ satisfies some stronger properties itself.

We also prove the following variant (the relevant definitions are in
\S\ref{subsec:terminology}).
 \end{rem}

\begin{tm}\label{thm:stacks_iterability_2}
 Let $\Omega>\om$ be regular.
 Let $\Sigma$ be a regularly $\Omega$-total strategy for $M$ with inflation
condensation,
 where if $M$ is a wcpm then $M$ is slightly coherent. Then
 $\Sigma$ extends to a strategy $\Sigma^*$ for stacks of length $<\om$.
More precisely, letting $m=m^\Sigma$:
\begin{enumerate}[label=--]
 \item if $M$ is a wcpm then $M$ is $({<\om},\Omega)^*$-iterable,
 \item if $\Sigma$ is a $\udash$strategy then $M$ is $(\udash
m,{<\om},\Omega)^*$-iterable, and
 \item if $\Sigma$ is an $(m,\Omega)$-strategy $M$ is
$(m,{<\om},\Omega)^*$-iterable,
\end{enumerate}
as witnessed by some $\Sigma^*$ with $\Sigma\sub\Sigma^*$.
\end{tm}

Recall that by Theorem \ref{tm:wDJ_implies_cond}, $(n,\Omega+1)$-iteration
strategies
with the DJ property for premice $M$ with $\card(M)<\Omega$,
or with weak DJ when $M$ is countable, have strong hull condensation, hence
inflation condensation,
so Theorem \ref{thm:stacks_iterability} applies in this case.
In particular:

\begin{cor}\index{$\om$-mice}
Let $\Omega$ be regular uncountable.
Let $M$ be $\om$-sound, $(\om,\Omega+1)$-iterable, with $\rho_\om^M=\om$.
Then $M$ is $(\om,\Omega,\Omega+1)^*$-iterable.
\end{cor}
\begin{proof}
 The unique $(\om,\Omega+1)$-strategy for $M$ has DJ.
\end{proof}

We will also prove a variant of Theorem \ref{thm:stacks_iterability},
which applies to length $\om$ (not just length ${<\om}$) stacks  of
\emph{finite} normal trees, assuming only normal iterability,
without any condensation assumption.
It is used in \cite{fsfni} in the proof of solidity, etc, from normal
iterability. In
order to state the result we need the
following
definition.
Recall that \emph{\tu{(}putative\tu{)} $m$-maximal stack} was defined in
\S\ref{subsec:terminology},
and $\Gg_\fin(M,m,\Omega+1)$ in
Definition \ref{dfn:G_fin}. We extend this naturally as follows:

\begin{dfn}\index{$\Gg_\fin$}
For $\udash m$-sound $M$, we define $\Gg_{\fin}(M,\udash m,\Omega+1)$,
and for wcpms $M$, define $\Gg_\fin(M,\Omega+1)$, analogously
to $\Gg_\fin(M,m,\Omega+1)$.
\end{dfn}

If player $\Two$ has a winning strategy for $\Gg_{\fin}(M,m,\Omega+1)$ where
$\Omega\geq\om$, then clearly every putative
$m$-maximal stack $\Ttvec$ as in Definition \ref{dfn:G_fin} (of finite length,
consisting
of finite length trees) is a true stack (has  wellfounded
models).
By a proof very similar (but simpler) to that for Theorem
\ref{thm:stacks_iterability},
we also prove
the following.
It needs no strategy condensation hypothesis
because the relevant trees have finite length.

\begin{tm}\label{thm:stacks_of_finite_trees}
Let $\Omega>\om$ be regular.
Let $\Sigma$ be a regularly $(\Omega+1)$-total pre-inflationary\footnote{Recall
that \emph{pre-inflationary} does not involve
 any actual condensation assumption!} strategy for $M$
 and $m=m^\Sigma$, where if $M$ is a wcpm then $M$ is slightly coherent.
Then player $\Two$ has a winning
strategy for $\Gg_{\fin}(M,m,\Omega+1)$, $\Gg_{\fin}(M,\udash m,\Omega+1)$, or
$\Gg_\fin(M,\Omega+1)$
accordingly.
Moreover, let $\Ttvec=\left<\Tt_i\right>_{i<\om}$ be an $m$-maximal, $\udash
m$-maximal, or normal,
stack on $M$ respectively.
Then for all sufficiently large $i<\om$, $b^{\Tt_i}$ does
not drop in model or degree, and $M^\Ttvec_\infty$ is wellfounded.
\end{tm}

\begin{rem}
In considering the proofs to come, the reader should make one observation.
The definition of $\Xx=\Ww_\Sigma(\Ttvec)$ will depend on $\Ttvec$ and the
restriction of $\Sigma$
to the segments of $\Xx$. We are presently assuming that $\Sigma$ is total,
but if $\Sigma$ were instead a partial strategy (with inflation condensation),
then everything would work as long as the segments of $\Xx$ remain in the
domain
of $\Sigma$.
We will use this observation later to deduce
\ref{thm:stacks_iterability_partial},
which is a variant of \ref{thm:stacks_iterability}
for partial strategies. Its statement depends on the definition of
$\Ww_\Sigma(\Ttvec)$,
which is spelled out in the proof, and the statements are somewhat inconvenient,
so we postpone them for later (the reader who wants to know what we intend to
prove
in this regard in advance should consult \ref{thm:stacks_iterability_partial}).
\end{rem}

\subsection{Proof of Theorems \ref{thm:stacks_iterability},
\ref{thm:stacks_iterability_2} and \ref{thm:stacks_of_finite_trees}: The stacks
strategy $\Sigma^\stk$}

We first observe that it suffices to construct (appropriately definable)
strategies for optimal stacks:

\begin{lem}\label{lem:sub-optimal_reduces_to_optimal}
 Let $\Omega>\om$ be regular and $M$ be either \tu{(}i\tu{)},\tu{(}ii\tu{)}
$\udash m$-sound, or
\tu{(}iii\tu{)}  MS-indexed and $m$-sound.
 Let $\Gamma$ be a strategy for player II in  the
\begin{enumerate}[label=\tu{(}\roman*\tu{)}]
 \item\label{item:sorto_part_i} $\Gg_{\mathrm{opt}}(M,\udash
m,\Omega,\Omega+1)^*$-iteration
game,
 or
 \item\label{item:sorto_part_ii} $\Gg^{\unrvl}_{\mathrm{opt}}(M,\udash
m,\Omega,\Omega+1)^*$-iteration game, or
 \item\label{item:sorto_part_iii}
$\Gg_{\mathrm{opt}}(M,m,\Omega,\Omega+1)^*$-iteration game,
 \end{enumerate}
 respectively.
Then there is
 a strategy $\widehat{\Gamma}$ for player II
 in the
 \begin{enumerate}[label=\tu{(}\roman*\tu{)}]
 \item $\Gg(M,\udash m,\Omega,\Omega+1)^*$-iteration game,
 or
 \item $\Gg^{\unrvl}(M,\udash m,\Omega,\Omega+1)^*$-iteration game,
 or
 \item
 $\Gg(M,m,\Omega,\Omega+1)^*$-iteration game,
 \end{enumerate}
 respectively.
 Moreover, stacks via $\widehat{\Gamma}$ lift canonically to (optimal) stacks
via $\Gamma$,
 and if $\Omega=\om_1$ and $M\in\HC$ and $\Gamma'\sub\RR$
and $\widehat{\Gamma}'\sub\RR$ code  $\Gamma\rest\HC$ and
$\widehat{\Gamma}\rest\HC$
in a natural manner then
$\widehat{\Gamma}'$
is $\Delta^1_1(\Gamma')$, uniformly in $\Gamma$.

The analogous facts also hold for deriving
\begin{enumerate}[label=\tu{(}\roman*\tu{)}]
 \item  $\Gg(M,\udash m,{<\om},\Omega)^*$-strategies from
 $\Gg_{\mathrm{opt}}(M,\udash m,{<\om},\Omega)^*$-strategies,
and
\item  $\Gg^{\unrvl}(M,\udash m,{<\om},\Omega)^*$-strategies from
 $\Gg^{\unrvl}_{\mathrm{opt}}(M,\udash m,{<\om},\Omega)^*$-strategies,
and
\item $\Gg(M,m,{<\om},\Omega)^*$-strategies from
$\Gg_{\mathrm{opt}}(M,m,{<\om},\Omega)^*$-strategies.
\end{enumerate}
\end{lem}
\begin{proof}[Proof Sketch]
Part \ref{item:sorto_part_i}: This is just by a standard copying construction,
particularly
 because we are dealing with $\udash$strategies (so there are no type 3
problems); it is in particular a simplification of the construction in
\cite[\S7]{mim}. We officially assume that $M$ is $\lambda$-indexed, so may
drop the  ``$\uu$'', but the MS-indexed case is likewise. The strategy
$\widehat{\Gamma}$ is defined
recursively as follows.
Suppose $\vec{\mathscr{S}}$ is via $\widehat{\Gamma}$,
of length $\gamma<\Omega$, and
$(R,r)=(M^{\vec{\mathscr{S}}}_\infty,\deg^{\vec{\mathscr{S}}}
(\infty))$.
Then we
will have a corresponding optimal
stack $\vec{\Uu}$ via $\Gamma$, with last model/degree $(N,n)$,
and some $R'=R^*_\gamma\ins N$, with $(R',r)\ins(N,n)$, and an
$r$-lifting embedding
$\sigma_\gamma:R\to R'$.
Let $\Gamma_{R,r}$ be the $(r,\Omega+1)$-strategy for $R$
given by lifting to a tree $\Uu$ via $\Gamma_{\vec{\Uu}}$ with $\pi_\gamma$.
Note that $\Uu$ will be an $n$-maximal tree on $N$,
as opposed to a $r$-maximal tree on $R'$.

That is, let $\Tt$ be a $r$-maximal tree on $R$ via $\Gamma_{R,r}$,
and $\Uu$ the lift,
and let $R_\alpha=M^\Tt_\alpha$, $N_\alpha=M^\Uu_\alpha$,
and if $[0,\alpha]_\Uu$ does not drop below the iteration image $R''$ of
$R'$, then
set $R'_\alpha=R''$,
and otherwise set $R'_\alpha=N_\alpha$. Let $r_\alpha=\deg^\Tt(\alpha)$
and $n_\alpha=\deg^\Uu(\alpha)$. Then we will have
$(R'_\alpha,r_\alpha)\ins(N_\alpha,n_\alpha)$,
and a $r_\alpha$-lifting embedding $\pi_\alpha:R_\alpha\to R'_\alpha$,
where $\pi_0=\sigma_\gamma$,
and the sequence of models and copy maps have typical commuting and
agreement properties.
If $\alpha+1<\lh(\Tt)$ then  $E^\Uu_\alpha=\pi_\alpha(E^\Tt_\alpha)$
(where $\pi_\alpha(F(R_\alpha))=F(R'_\alpha)$), and we proceed basically as
usual, except that we can have $\alpha=\pred^\Tt(\beta+1)$
and $[0,\beta+1]_\Tt\inter\dropset^\Tt_{\deg}=\emptyset$
and $(R'_\alpha,r)\pins(M^{*\Uu}_{\beta+1},\deg^\Uu(\beta+1))$,
even when $\beta+1\in\dropset^\Uu_{\deg}$.  It can also
be that
$[0,\alpha]_\Tt\inter\dropset^\Tt_{\deg}\neq\emptyset$
but $(R'_\alpha,r_\alpha)\pins(N_\alpha,n_\alpha)$.

Now suppose that at the beginning of round $\gamma$,
player I plays $(S,s)\ins(R,r)$. Let $\Gamma_{S,s}$ be the
$(s,\Omega+1)$-strategy for $S$ given by lifting trees $\mathscr{S}$ on $S$
to $r$-maximal trees $\Tt$ on $R$ via the identity map $S\to S\ins R$.
This lifting is just just  like the preceding one (except that maybe $R\neq R'$
and $\sigma_\gamma\neq\id$ above),
and letting $S_\alpha=M^{\mathscr{S}}_\alpha$
and $R_\alpha=M^\Tt_\alpha$ and $S'=S$, we get $S'_\alpha\ins R_\alpha$
and copy maps $\varrho_\alpha:S_\alpha\to S'_\alpha$.
Then player II plays out round $\gamma$ using $\Gamma_{S,s}$, producing
tree $\mathscr{S}_\gamma$, and
composing the two lifts, we produce the $n$-maximal tree $\Uu_\gamma$ on $N$.
If $\lh(\mathscr{S}_\gamma)=\alpha+1$ and we reach round $\gamma+1$,
then
we produce $R^*_{\gamma+1}$ in the natural way, and set
\[ \sigma_{\gamma+1}=\pi_\alpha\com\varrho_\alpha:S_\alpha\to
R^*_{\gamma+1}\ins N_\alpha\]
(if $S'_\alpha\pins R_\alpha$ then
$R^*_{\gamma+1}=\pi_\alpha(S'_\alpha)\pins R'_\alpha$,
and otherwise $R^*_{\gamma+1}=R'_\alpha$).

If $\eta$ is a limit ordinal and we have defined $\vec{\mathscr{S}}$ of length
$\eta$ and $\vec{\Uu}$ of length $\eta$ as above, then since
$M^{\vec{\Uu}}_\infty$ is well-defined and wellfounded,
it is easy to see that $M^{\vec{\Tt}}_\infty$ is also (including
that player I eventually stopped artificially dropping) and we
define $R^*_\eta$ and $\sigma_\eta$ via direct limit.

Part \ref{item:sorto_part_i} and the corresponding definability clause now
follow easily.

Part \ref{item:sorto_part_ii} is almost
the same as part \ref{item:sorto_part_i}.
With notation as there,
suppose $\Tt$ (played in round $\gamma$)
has length $\alpha+1$ (with $\Tt$ unravelled),
and $\Uu$ is its lift.
We have $(R'_\alpha,r_\alpha)\ins(N_\alpha,n_\alpha)$.
If $R'_\alpha=N_\alpha$ then $r_\alpha\neq n_\alpha$,
and note then that $\Uu$ is unravelled.
If instead $R'_\alpha\pins N_\alpha$ and $\Uu$
is not unravelled, then first replace $\Uu$
with $\unrvl(\Uu)$ before continuing.

Part \ref{item:sorto_part_iii}: Fix $\Gamma$ as in \ref{item:sorto_part_iii}.
If $M$ is type 3 then let $m'=m+1$,
and otherwise let $m'=m$. Let $\Sigma$ be the corresponding
$\Gg^{\unrvl}_{\mathrm{opt}}(M,\udash m',\Omega,\Omega+1)^*$-strategy
(see Lemma \ref{lem:rule_conversion}).
Let $\widehat{\Sigma}$ be defined as above. Now define $\widehat{\Gamma}$
as follows. Suppose we have defined $\Ttvec$ of length $\gamma$ via
$\widehat{\Gamma}$. Then
$(M_\gamma,m_\gamma)=(M^{\Ttvec}_\infty,\deg^{\Ttvec}(\infty))$ will be
well-defined, and we will have a corresponding stack
$\Uuvec$ via $\widehat{\Sigma}$, of length $\gamma$,
and letting
$(M'_\gamma,m'_\gamma)=(M^{\Uuvec}_\infty,\udeg^{\Uuvec}(\infty))$,
we will have
$M_\gamma=(M'_\gamma)^\pm$,
and $m_\gamma,m'_\gamma$ are related according to the type of $M_\gamma$
as $m,m'$ are. Suppose player I plays
$(Q_\gamma,q_\gamma)\ins(M_\gamma,m_\gamma)$.
If this is not an artificial drop, then also set
$(Q'_\gamma,q'_\gamma)=(M'_\gamma,m'_\gamma)$, and then form
$(\Tt_\gamma,\Uu_\gamma)$
with $\Uu_\gamma$ according to $\widehat{\Sigma}_{\Uuvec}$
and $\Tt_\gamma$ its translation.
If there is an artificial drop, then let $q'_\gamma=q_\gamma+1$
if $Q_\gamma$ is type 3, and $q'_\gamma=q_\gamma$ otherwise,
$Q'_\gamma=Q_\gamma$ (recalling that if $M'_\gamma\neq M_\gamma$
then $m_\gamma=0$, so $Q_\gamma\pins M_\gamma$
and since $M_\gamma=(M'_\gamma)^\pm$, therefore $Q_\gamma\pins M'_\gamma$),
and noting that $(Q_\gamma,q'_\gamma)\pins(M'_\gamma,m'_\gamma)$,
now play $\Tt_\gamma,\Uu_\gamma$ as before, but on $(Q_\gamma,q_\gamma)$
and $(Q'_\gamma,q'_\gamma)$. Note then that by Lemma \ref{lem:tree_conversion},
$\Ttvec\conc\Tt_\gamma$
and $\Uuvec\conc\Uu_\gamma$ again satisfy the inductive requirements.

Finally, if we have $\Ttvec,\Uuvec$ of limit length, then because
$M^{\Uuvec}_\infty$ is well-defined and wellfounded, and because
of the correspondence of iteration maps given by Lemma
\ref{lem:tree_conversion},
$M^{\Ttvec}_\infty$ is also well-defined and wellfounded,
and the inductive hypotheses  hold.

The lemma easily follows.
\end{proof}

So by the lemma, in order to prove Theorems \ref{thm:stacks_iterability} and
\ref{thm:stacks_iterability_2}, we just need to construct
appropriate strategies for optimal stacks.
In the construction we work with
conveniently
inflationary strategies, and directly construct a
convenient strategy (for optimal stacks),
and then derive from this inconvenient strategies
(also for optimal stacks).
This derivation is is quickly dispensed with and we deal with it first.
 Consider the case of \ref{thm:stacks_iterability}.
 Suppose $M$ is MS-indexed.
 We have the normal strategy $\Sigma$ for $M$.
 Let $\ell=m+1$ if $M$ is type 3; otherwise let $\ell=m$.
 Let $\Gamma$ be the $(\udash\ell,\Omega+1)$-strategy for $M$ corresponding to
$\Sigma$ (see \ref{lem:rule_conversion}).
 By definition, $\Gamma$ has inflation condensation.
 Suppose that the theorems hold with respect to convenient strategies (hence
for
$\Gamma$).
 Let $\Gamma^*$ be a $(\udash\ell,\Omega,\Omega+1)^*$-strategy for $M$
 such that $\Gamma\sub\Gamma^*$.
 Let $\Sigma^*$ be the $(m,\Omega,\Omega+1)^*$-strategy
 for $M$ determined by $\Gamma^*$ (that is, by restricting
 $\Gamma^*$ to unravelled stacks, we get an unravelled
 strategy, and this corresponds to $\Sigma^*$).
 Then $\Sigma\sub\Sigma^*$, so we are done. For \ref{thm:stacks_iterability_2}
it is completely analogous.

We now consider convenient strategies.
We only literally give the proof for $\udash$strategies, as the coarse case is
mainly a simplification thereof,
but we will point out where we use slight coherence.
So fix $\Omega$ and a $\udash m$-strategy $\Sigma$ for $M$ as in
\ref{thm:stacks_iterability} or \ref{thm:stacks_iterability_2}.
We will construct an appropriate stacks strategy $\Sigma^*$ for $M$, extending
$\Sigma$.
We first give a sketch of the process. For the purposes of this sketch,
we consider literally the case of \ref{thm:stacks_iterability}, so $\Sigma$ is
an $(\udash m,\Omega+1)$-strategy
(but in either case, the constructions agree over
their restriction to a $(\udash m,{<\om},\Omega)^*$-strategy).

For stacks $\Ttvec$ on $M$ via $\Sigma^*$ of length $<\Omega$,
we will construct a corresponding normal tree $\Yy$, of successor
length, which will be via $\Sigma$ if all normal trees in $\Ttvec$
have length ${<\Omega}$, and which
``absorbs''
$\Ttvec$, and in particular, such that $M^\Ttvec_\infty$ embeds into
$M^\Yy_\infty$
(here, $M^{\Ttvec}_\infty$ will be well-defined as we will also verify that
$\Ttvec$ has only finitely many drops along its main branch,
by showing that drops in model in $\Ttvec$ correspond suitably to drops in
model in $\Yy$). In the case of a stack $(\Tt,\Uu)$ of length $2$
(with $\Tt,\Uu$ normal), $\Yy$ will be an inflation of $\Tt$, with the
$\Tt$-inflationary
extenders being
just copies of extenders used in $\Uu$. This easily yields a strategy for
finite
stacks of trees.
In the limit case, for a stack $\Ttvec$ of length $\eta$, we will have a
sequence of inflations
$\left<\Yy_\alpha\right>_{\alpha<\eta}$.
We will define $\Yy=\Yy_\eta$ as the comparison inflation of
$\{\Yy_\alpha\}_{\alpha<\eta}$.
The commutativity lemma \ref{lem:inflation_commutativity} is the key to seeing
that everything fits together appropriately.

Here is a more detailed sketch (cf.~Figure~\ref{fgr:infinite_stack_comm} on
page~\pageref{fgr:infinite_stack_comm},
where $O_n=M^{\Ttvec\rest n}_\infty$;
the figure incorporates more detail than given in this sketch). The trees
mentioned below are of successor length and the inflations are terminal. Given
a normal tree
$\Tt_0$ on $M$, via $\Sigma$, and a normal tree $\Tt_1$ on
$M^{\Tt_0}_\infty$, with $(\Tt_0,\Tt_1)$ via $\Sigma^*$, letting $\Yy_1=\Tt_0$,
we will define an
inflation $\Yy_2$ of $\Yy_1$, such that $M^{\Tt_1}_\infty$ embeds into
$M^{\Yy_2}_\infty$ (the reason for this misalignment of integers will become
clearer later). The
fact that $\Sigma$ has inflation condensation will ensure that this process
does not break down. Then, given a normal tree $\Tt_2$ on $M^{\Tt_1}_\infty$,
with $(\Tt_0,\Tt_1,\Tt_2)$ via $\Sigma^*$,
we will define an inflation $\Yy_3$ of $\Yy_2$,
such that
$M^{\Tt_2}_\infty$ embeds into $M^{\Yy_3}_\infty$. And so on for finite stacks.

Now let $\Ttvec=\left<\Tt_n\right>_{n<\om}$ be a stack of normal trees via
$\Sigma^*$.
We will have a sequence $\left<\Yy_n\right>_{n<\om}$ as above, where $\Yy_0$ is
the trivial tree on $M$.
So $\Yy_{l+2}$ is an
inflation of $\Yy_{l+1}$ is an inflation of $\Yy_l$. Using
\ref{lem:inflation_commutativity}, we will have that for $n_0<n_1<n_2$,
$\Yy_{n_2}$ is a inflation
of $\Yy_{n_1}$ is an inflation of $\Yy_{n_0}$, everything commutes (and all
these inflations are also terminal).
Let us assume for simplicity that all trees are
terminally non-dropping. Then for each $n_0<n_1$, $\Yy_{n_1}$ will be
$\Yy_{n_0}$-terminally-non-dropping,
and the iteration embeddings
\[ i^{\Tt_{n}}:M^{\Ttvec\rest n}_\infty\to
M^{\Ttvec\rest(n+1)}_\infty=M^{\Tt_n}_\infty\] and the
final inflation copy maps
\[
\pi_{n_0n_1}\eqdef\pi_\infty^{\Yy_{n_0}\inflatearrow\Yy_{n_1}}:M^{\Yy_{n_0}}
_\infty\to M^{\Yy_{n_1}}_\infty\]
will commute with the maps $\srsigma_{n_0},\srsigma_{n_1}$ where
\[ \srsigma_n:M^{\Ttvec\rest n}_\infty\to M^{\Yy_n}_\infty,\]
is the lifting map mentioned in the previous paragraph. Therefore the direct
limit $M^\Ttvec_\infty$ embeds into the direct limit of the models
$M^{\Yy_n}_\infty$ under the
maps $\pi_{n_0,n_1}$. We will set $\Yy_\om$ to be the comparison
inflation of $\{\Yy_n\}_{n<\om}$.
Then $\Yy_\om$ will be an $\Yy_n$-terminal inflation of $\Yy_n$ for each $n$,
and because of our extra assumptions here regarding (non-)dropping,
$\Yy_\om$ will be $\Yy_n$-terminally-non-dropping for each $n$.
Defining
\[ \pi_{n\om}=\pi_\infty^{\Yy_n\inflatearrow\Yy_\om}:M^{\Yy_n}_\infty\to
M^{\Yy_\om}_\infty, \]
then by \ref{lem:inflation_commutativity}, we have
\[ \pi_{n_0\om}=\pi_{n_1\om}\com\pi_{n_0n_1} \]
for $n_0<n_1<\om$. Therefore $M^{\Yy_\om}_\infty$ absorbs the direct
limit of the models $M^{\Yy_n}_\infty$, and so absorbs $M^\Ttvec_\infty$, and
in
particular,
$M^\Ttvec_\infty$ is wellfounded. The process then continues
through longer stacks in the same manner.

Note that our proof that the comparison inflation exists requires
that
$\Sigma$
be an $(\udash m,\Omega+1)$-strategy; thus, under the weaker assumption of
$(\udash m,\Omega)$-iterability
we do not see how to deal with limit stages, and so only obtain an $(\udash
m,{<\om},\Omega)^*$-strategy.
There are some  further details involved in dealing with dropping trees and
inflations,
but these are straightforward using \ref{lem:inflation_commutativity}.

We now proceed to the details.

\subsubsection{Stacks of length $2$}\label{subsubsec:stacks_lh_2}

Before we begin with the main construction,
we prove a fine structural lemma. The lemma,
however, is only needed in the proof of a detail
which the reader might prefer to ignore at a first pass. We
prove it only for $\lambda$-indexing, for notational
simplicity; the analogue also holds for MS-indexing, however
(see \cite[\S6]{fsfni} for related material).

\begin{dfn}\label{dfn:weak_cof}\index{weak
cofinality}\index{wcof}\index{$\wcof_{k+1}^M$} Let $k<\om$ and $S$ be a
$k$-sound
$\lambda$-indexed premouse.
Then $\wcof_{k+1}^S$ (for \emph{weak cofinality})
denotes the least $\tau$ such that
\[ \ex q\in S\ [\Hull_{\rSigma_{k+1}}^S(\tau\cup\{q\})\text{ is cofinal in
}\rho_k^S].\]
Note  this is the least $\tau\leq\rho_{k+1}^S$ such that either
$\tau=\rho_{k+1}^S$
or there is a $\bfrSigma_{k+1}^S$-function $f:\tau\to\rho_k^S$ which is
cofinal,
strictly increasing and continuous.\end{dfn}
\begin{lem}\label{lem:wcof_pres}
Let $R,S$ be $(k+1)$-sound $\lambda$-indexed premice and $\pi:R\to S$  a near
$(k+1)$-embedding. Then either:
\begin{enumerate}[label=--]
 \item $\wcof_{k+1}^R<\rho_{k+1}^R$
 and $\pi(\wcof_{k+1}^R)=\wcof_{k+1}^S$, or
 \item $\wcof_{k+1}^R=\rho_{k+1}^R$ and $\wcof_{k+1}^S=\rho_{k+1}^S$.
\end{enumerate}
\end{lem}
\begin{proof} Recall that either $\rho_k^R=\OR^R$ and
$\rho_k^S=\OR^S$, or
$\pi(\rho_k^R)=\rho_k^S$.
And $\pi(\rho_{k+1}^R)\geq\rho_{k+1}^S$ by
$\rSigma_{k+2}$-elementarity.\footnote{Here $\pi(\OR^R)$ denotes
$\OR^S$.}
Now given $\tau<\rho_{k+1}$ and some parameter $q$,
it is an $\rPi_{k+2}(\tau,q,\rho_k)$ assertion that
\[ \text{``}\Hull_{k+1}(\tau\cup\{q\})\text{ is cofinal in }\rho_k\text{''.}\]
And given $\tau<\rho_{k+1}$,
it is an $\rPi_{k+2}(\tau,\rho_k)$ assertion that
\[\text{``}\all\alpha<\tau\all q\ [\Hull_{k+1}^R(\alpha\cup\{q\})\text{ is
bounded in }\rho_k]\text{''.}\]
(For this can be expressed as ``For every $\alpha<\tau$ and $q$
and every $T\in T_{k+1}$ such that $T$ is a theory in parameters
$\alpha\cup\{q\}$,
there is some $T'\in T_k$ which codes witnesses to all $\rSigma_{k+1}$ formulas
in $T$'';
here \emph{coding a witness} is in the style described in \cite[\S2]{fsit}.)
Likewise,
it is an $\rPi_{k+2}(\rho_k)$ assertion that
\[ \text{``}\all\alpha<\rho_{k+1}\all q\ [\Hull_{k+1}^R(\alpha\cup\{q\})\text{
is bounded in }
\rho_k]\text{''.}\]
Since $\pi$ is a near $(k+1)$-embedding,
the lemma follows.
\end{proof}

\begin{figure}
\centering
\begin{tikzpicture}
 [mymatrix/.style={
    matrix of math nodes,
    row sep=0.35cm,
    column sep=0.4cm}
  ]
   \matrix(m)[mymatrix]{
  N_\beta& {}&                   {}&                {}&M^{\Xx^\beta}_\infty&
                        {}&                           \ & \ & \ &       {}&{}\\
       {}& {}&                   {}&                {}&                  {}&
                        {}&                           {}& {}& {}&       {}&{}\\
 N_\alpha& {}&M^{\Xx^\alpha}_\infty&                {}&          \Xx^\alpha&
                        {}&                           {}& {}& {}&\Xx^\beta&{}\\
       {}& {}&                   {}&                {}&                  {}&
                        {}&                           {}& {}& {}&       {}&{}\\
        N& {}&                  \Tt&                {}&                  {}&
                        {}&                           {}& {}& {}&       {}&{}\\
       {}& {}&                    M&                {}&                  {}&
                         M&                           {}& {}& {}&       {}&M\\};
\path[->,font=\scriptsize]
(m-6-3) edge node[below] {$\id$} (m-6-6)
(m-6-3) edge[dashed] node[below,pos=0.7] {$i^\Tt$} (m-5-1)
(m-6-6) edge node[below] {$\id$} (m-6-11)
(m-6-6) edge[dashed] node[right,pos=0.85] {$\ \ i^{\Xx^\alpha}$} (m-3-3)
(m-6-11) edge[dashed] node[right,pos=0.85] {$\ \ i^{\Xx^\beta}$} (m-1-5)
(m-5-1) edge node[left] {$i^\Uu_{0\alpha}$} (m-3-1)
(m-5-1) edge node[above] {$\psi_{0\alpha}\ \ $} (m-3-3)
(m-3-1) edge node[left] {$i^\Uu_{\alpha\beta}$} (m-1-1)
(m-3-1) edge node[above] {$\nrsigma_\alpha$} (m-3-3)
(m-1-1) edge node[above] {$\nrsigma_\beta$} (m-1-5)
(m-3-3) edge node[above,pos=0.35] {$\psi_{\alpha\beta}\ \ $} (m-1-5);
\draw [black] plot [smooth] coordinates {(-2.9,-1.7) (-3.1,-1.4) (-3.7,-1.2)
(-3.9,-0.9)};
\draw [black] plot [smooth] coordinates {(0.5,-1.7) (0.1,-0.7) (-1.2,0)
(-1.6,1)};
\draw [black] plot [smooth] coordinates {(4.9,-1.7) (4.2,-.3) (2.7,0.5)
(2,1.9)};
\end{tikzpicture}
\caption{Commutativity for maps relating to $\Upsilon^\Sigma_\Tt$
assuming $[0,\beta]_\Uu\inter\dropset^\Uu=\emptyset$
(see conditions \ref{item:interval_not_easy} and
\ref{item:interval_easy_no_mod_drop}).
The curved lines represent the iteration trees $\Tt$, $\Xx^\alpha$, $\Xx^\beta$.
The solid arrows commute.
The dashed arrows exist iff $b^\Tt\inter\dropset^\Tt=\emptyset$,
and when they exist, they commute with the other maps.}
\label{fgr:stack_2_comm}
\end{figure}

We now begin the  main proof for the case of realizing a stack of two
normal trees via a
single normal tree.
For this case we only assume in general that
$\Sigma$ is an $(\udash m,\Omega)$-strategy,
(not $(\udash m,\Omega+1)$). Let $\Tt$ be an
$\udash m$-maximal tree on $M$ of successor length $<\Omega$, via $\Sigma$. Let
$N=M^\Tt_\infty$ and $n=\udeg^\Tt(\infty)$.

We describe a\index{$\Upsilon^\Sigma_\Tt$}
\[ (\udash n,\Omega)\text{-iteration strategy }\Upsilon^\Sigma_\Tt\text{ for
}N.\]
In order to do this,
we lift $\udash n$-maximal trees $\Uu$ on $N$ via $\Upsilon^\Sigma_\Tt$ of
length $\leq\Omega$ to $\udash m$-maximal trees
$\Xx$ on $M$ via $\Sigma$. We write
$\Ww^\Sigma_\Tt(\Uu)$\index{$\Ww^\Sigma_\Tt$} for
$\Xx$.
\footnote{We
use the notation $\Ww^\Sigma_\Tt(\Uu)$
instead of $\Xx^\Sigma_\Tt(\Uu)$ for consistency with Steel's notation,
and because we will use $\Xx^\Sigma_\Tt(\Uu)$ in the future for (full)
normalization, as opposed to normal realization.
But for consistency with the rest of the paper, we continue to use the variable
$\Xx$.}
Here $\Xx$ will depend on $\Sigma$, $\Tt$ and the extenders used in $\Uu$,
but $\Xx$ will determine the branches chosen in $\Uu$. Moreover, for limits
$\eta<\lh(\Uu)$ we will have
\[ \Xx'\eqdef\Ww^\Sigma_\Tt(\Uu\rest\eta)\pins\Xx, \]
with $\lh(\Xx')$ a limit, and $\Sigma(\Xx')$
determines $[0,\eta)_\Uu$.
If $\Sigma$ extends to a $(\udash m,\Omega+1)$-strategy,
then so will $\Upsilon^\Sigma_\Tt$.
We will also define  $\Ww^\Sigma_\Tt(\Uu)$ when $\lh(\Uu)=\Omega+1$,
but this tree can have length ${>\Omega+\om}$, and so be not literally via
$\Sigma$.
For now we assume that $\lh(\Uu)\leq\Omega$, and then later consider the
extension to $\Omega+1$.

The tree $\Xx$ will be a non-$\Tt$-pending inflation of $\Tt$, via $\Sigma$,
with associated objects
\[
(t,C,\ldots,\lambda^\alpha,\Xx^\alpha,\ldots)=(t,C,\ldots,\lambda^\alpha,
\Xx^\alpha,\ldots)^{\Tt\inflatearrow\Xx}. \]
The $\Tt$-inflationary extenders $E^\Xx_{\zeta^\alpha}$ used in $\Xx$ will be
copies of extenders from $\Uu$
(and of course, the others are copied from $\Tt$).
We will define a lifting map\index{$\nrsigma_\alpha$}
\[ \nrsigma_\alpha:M^\Uu_\alpha\to M^{\Xx^\alpha}_\infty.\]
We say that $\alpha$ is \index{easy}\dfnemph{easy} iff
$\lambda^\alpha\notin(C^-)^\alpha$.

We will build $\Uu\rest\eta$,
$\left<\zeta^\alpha,E^\Xx_{\zeta^\alpha}\right>_{\alpha+1<\eta}$,
$\left<\lambda^\alpha,\Xx^\alpha,\nrsigma_\alpha\right>_{\alpha<\eta}$, etc,
thus determining
\[ \Xx\rest\sup_{\alpha<\eta}(\lambda^\alpha+1),\]
by induction on $\eta$,
maintaining the following conditions.
For $1\leq\eta\leq\Omega$, let \index{$\varphi(\eta)$}$\varphi(\eta)$ assert
that
these objects are defined
and the following conditions hold (\emph{N} is for \emph{normal}):

\begin{enumerate}[label=N\arabic*.,ref=N\arabic*]
 \item $\Xx\rest\sup_{\alpha<\eta}(\lambda^\alpha+1)$ is via $\Sigma$ and is an
inflation of $\Tt$,
 with the associated objects described above (in particular, for each
$\alpha<\eta$,
 $\Xx^\alpha$ is a $\Tt$-terminal inflation of $\Tt$ and
$\Xx\rest(\lambda^\alpha+1)=\Xx^\alpha\rest(\lambda^\alpha+1)$).
  \item\label{item:limit_Uu_rest_beta} Tree order:
$(<^\Uu)\rest\eta=(<^{\Xx/\Tt})\rest\eta$.
\item\label{item:main_embedding} For $\alpha<\eta$, we
have:\footnote{Remark~\ref{rem:fail_near_emb} shows that this condition cannot
in general be improved much.}
\begin{enumerate}[label=--]
\item $k\eqdef\udeg^\Uu(\alpha)\leq\udeg^{\Xx^\alpha}(\infty)$,
\item $\nrsigma_\alpha:M^\Uu_\alpha\to M^{\Xx^\alpha}_\infty$ is nice $\udash
k$-lifting,
\item $[0,\alpha]_\Uu\inter\dropset^\Uu=\emptyset$ iff $\lambda^\alpha\in
C^\alpha$.
\item If $[0,\alpha]_\Uu\inter\dropset^\Uu_{\deg}\neq\emptyset$ then:
\begin{enumerate}[label=--]
\item $\lambda^\alpha+1=\lh(\Xx^\alpha)$,
\item $\nrsigma_\alpha$ is a near $\udash k$-embedding,
\item if $[0,\alpha]_\Uu\inter\dropset^\Uu\neq\emptyset$ or $k+1<n$ then
$k=\udeg^{\Xx^\alpha}(\lambda^\alpha)$.
\end{enumerate}
\item If $\alpha$ is non-easy then
$[0,\alpha]_\Uu\inter\dropset_{\deg}^\Uu=\emptyset$.
\item If $[0,\alpha]_\Uu\inter\dropset_{\deg}^\Uu=\emptyset$ and  $\Tt$ is
terminally non-dropping then $\Xx^\alpha$ is terminally non-dropping and
$\nrsigma_\alpha$ is a $\udash m$-embedding.
\end{enumerate}
\item\label{item:copy_E^Uu_alpha} Let $\alpha<\beta<\eta$. Then:
\begin{enumerate}[label=--]
 \item  If $E^\Uu_\alpha=F^{M^\Uu_\alpha}$ then
$\zeta^\alpha+1=\lh(\Xx^\alpha)$ and
$E^{\Xx^{\beta}}_{\zeta^\alpha}=F^{M^{\Xx^\alpha}_\infty}$.
\item If
$E^\Uu_\alpha\neq F^{M^\Uu_\alpha}$ then
$E^{\Xx^{\beta}}_{\zeta^\alpha}=\nrsigma_\alpha(E^\Uu_\alpha)$,
and $\zeta^\alpha$ is the least $\zeta$ with
$\nrsigma_\alpha(E^\Uu_\alpha)\in\es(M^{\Xx^\alpha}_\zeta)$.
\end{enumerate}
\item\label{item:pi_beta(nu^Uu_alpha)} For $\alpha<\beta<\eta$, we have
$\nrsigma_\alpha\rest\In(E^\Uu_\alpha)\sub\nrsigma_\beta$ (so
$\nrsigma_\beta(\nutilde^\Uu_\alpha)=\nutilde^{\Xx^\beta}_{\zeta^\alpha}$), and
either
\begin{enumerate}[label=--]
 \item $\In(E^\Uu_\alpha)<\OR(M^\Uu_{\alpha+1})$ and
$\nrsigma_{\alpha+1}(\In(E^\Uu_\alpha))=\In(E^{\Xx^\beta}_{\zeta^\alpha})$, or
 \item $\In(E^\Uu_\alpha)=\OR(M^\Uu_{\alpha+1})$ and
$\In(E^{\Xx^\beta}_{\zeta^\alpha})=\OR(M^{\Xx^\beta}_{\zeta^\alpha+1})$ and
$M^\Uu_{\alpha+1},M^{\Xx^\beta}_{\zeta^\alpha+1}$ are active type 2 with
MS-indexing.
\end{enumerate}
\item\label{item:above_alpha_Uu-easy} Let $\alpha\leq\beta<\eta$ be such that
$\alpha$ is
easy (so \ref{lem:<^Xx/Tt_above_easy} applies). Then:
\begin{enumerate}
\item $\gamma\mapsto\lambda^\gamma$ restricts to an isomorphism
${(<^{\Uu\rest\beta+1})^{(\alpha)}}\to{(<^{\Xx^\beta})^{(\lambda^\alpha)}}$
preserving drop structure, and above drops in model, degree structure.
\item Let $\alpha\leq^\Uu\gamma\leq\beta$, so $\gamma$ is easy, so
$\lh(\Xx^\gamma)=\lambda^\gamma+1$ and
\[ \nrsigma_\gamma:M^\Uu_\gamma\to M^{\Xx^\gamma}_{\lambda^\gamma}.\]
Let $\gamma\leq^\Uu\xi\leq\beta$ with
$(\gamma,\xi]_\Uu\inter\dropset^\Uu=\emptyset$.
Let
\index{$\psi_{\alpha\beta}$}$\psi_{\gamma\xi}=i^{\Xx^\xi}_{
\lambda^\gamma\lambda^\xi } $. Then
\[ \nrsigma_\xi\com i^\Uu_{\gamma\xi}=\psi_{\gamma\xi}\com\nrsigma_\gamma, \]
and if $\gamma$ is a successor then letting
$\delta=\pred^\Uu(\gamma)$,\footnote{The fact that if
$\lambda^\gamma\in\dr^{\Xx^\gamma}$ then $\gamma\in\dr^\Uu$ depends on the fact
that $\alpha$ is
easy.}
\[ \gamma\in\dr^\Uu\ \iff\ \lambda^\gamma\in\dr^{\Xx^\gamma},\]
\[ \gamma\in\dr^\Uu\ \implies
\nrsigma_\delta(N^*_\gamma)=M^{*\Xx^\gamma}_{\lambda^\gamma}, \]
\[ \nrsigma_\gamma\com
i^{*\Uu}_{\gamma}=i^{*\Xx^\gamma}_{\lambda^\gamma}\com\nrsigma_\delta\rest
N^*_{\gamma}.\]
\end{enumerate}
\item\label{item:interval_not_easy} (Cf.~Figure \ref{fgr:stack_2_comm})
Let $\alpha\leq^\Uu\beta<\eta$ be such that $\beta$ is non-easy
(so $[0,\beta]_\Uu\inter\dropset_{\deg}^\Uu=\emptyset$ and
$\alpha$ is non-easy and $\Xx^\alpha,\Xx^\beta$ are
$\Tt$-terminally-non-dropping).
Let\index{$\psi_{\alpha\beta}$}
\[
\psi_{\alpha\beta}=\pi^{\alpha\beta}_{\lh(\Tt)-1}=\om^{\alpha\beta}_{\lh(\Tt)-1}
:M^{\Xx^\alpha}_\infty\to M^{\Xx^\beta}_\infty \]
(where $\pi^{\alpha\beta}_{\lh(\Tt)-1}=\om^{\alpha\beta}_{\lh(\Tt)-1}$ are
defined in \ref{dfn:Pi^alpha,beta_embs}
and are equal because  $\theta^\beta+1<\lh(\Tt)$ because $\beta$ is non-easy).
Then
$\psi_{\alpha\beta}\com\nrsigma_\alpha=\nrsigma_\beta\com
i^\Uu_{\alpha\beta}$.
\item\label{item:interval_easy_no_mod_drop}  (Cf.~Figure
\ref{fgr:stack_2_comm})
Let $\alpha\leq^\Uu\beta<\eta$ be such that $\beta$ is easy but
$\lambda^\beta\in C^\beta$.
Let\index{$\psi_{\alpha\beta}$}
\[ \psi_{\alpha\beta}=\om^{\alpha\beta}_{\lh(\Tt)-1}:M^{\Xx^\alpha}_\infty\to
M^{\Xx^\beta}_\infty. \]
Then $\psi_{\alpha\beta}\com\nrsigma_\alpha=\nrsigma_\beta\com
i^\Uu_{\alpha\beta}$.
\end{enumerate}

This completes the inductive hypotheses. Note that
\ref{item:interval_not_easy} and
\ref{item:interval_easy_no_mod_drop}
actually have the same conclusion.
We now begin the construction.

With $\Uu\rest 1=$ the trivial tree, $\Xx^0=\Tt$ and $\nrsigma_0=\id:N\to N$,
$\varphi(1)$ is
trivial.

Now suppose we are given $\Uu\rest\eta$ and the other related objects, and
$\varphi(\eta)$ holds;
we define $\Uu\rest\eta+1$, etc, and verify $\varphi(\eta+1)$. Suppose first
that
$\eta=\alpha+1$. So we have defined $\Xx^\beta,\nrsigma_\beta$, etc, for all
$\beta\leq\alpha$ and
$\zeta^\beta$ for all $\beta<\alpha$, and $\varphi(\alpha+1)$ holds. Let
$E=E^\Uu_\alpha$.

Now $\zeta^\alpha$ is determined by property \ref{item:copy_E^Uu_alpha};
let us observe that $\zeta^\alpha\geq\lambda^\alpha$. If $\alpha$ is a limit or
$E=F^{M^\Uu_\alpha}$ this is easy; suppose $\alpha=\gamma+1$ and $E\neq
F^{M^\Uu_\alpha}$. Then $\In(E^\Uu_\gamma)<\In(E)$, so by
\ref{item:pi_beta(nu^Uu_alpha)},
\[
\nrsigma_\alpha(\In(E))>\nrsigma_\alpha(\In(E^\Uu_\gamma))=\In(E^{\Xx^\alpha}_{
\zeta^\gamma}),\]
so $\zeta^\alpha\geq\zeta^\gamma+1=\lambda^\alpha$.

Now $\Xx^{\alpha+1}$ is determined by setting
$F=E^{\Xx^{\alpha+1}}_{\zeta^\alpha}$ according
to  \ref{item:copy_E^Uu_alpha}. By coherence,
$F$ is indeed $\Xx^\alpha\rest(\zeta^\alpha+1)$-normal, so we can do this.
(For the wcpm case, it is here that we use that $M$ is slightly coherent.
That is, by slight coherence and \ref{lem:slight_coherence_pres},
$\zeta_\alpha$
is the least $\zeta$ such that either $\lh(\Xx^\alpha)=\zeta+1$
or
$\varrho^{M^{\Xx^\alpha}_\zeta}(E^{\Xx^\alpha}_\zeta)\geq
\varrho^{M^{\Xx^\alpha}_\zeta}(F)$, so $F$ is
$\Xx^\alpha\rest(\zeta^\alpha+1)$-normal.)
This determines $\Xx^{\alpha+1}$ and ${<^{\Xx/\Tt}}\rest(\alpha+2)$;
note that because $\Sigma$ has inflation condensation,
$\Xx^{\alpha+1}$
is in fact via $\Sigma$, and in particular has wellfounded models.
It just remains to define $\nrsigma_{\alpha+1}$ and prove
$\varphi(\alpha+2)$.

Let $\kappa_E=\crit(E)$ and $\kappa_F=\nrsigma_\alpha(\kappa_E)=\crit(F)$. Let
$\beta=\pred^\Uu(\alpha+1)$ and $\xi=\pred^{\Xx^{\alpha+1}}(\zeta^\alpha+1)$.
So
for all
$\gamma<\beta$, we have
 $\nutilde^\Uu_\gamma\leq\kappa_E<\nutilde^\Uu_\beta$,
 so by  \ref{item:pi_beta(nu^Uu_alpha)},
 \[
\nutilde^{\Xx^{\alpha+1}}_{\zeta^\gamma}
=\nrsigma_\alpha(\nutilde^\Uu_\gamma)\leq\kappa_F<\nrsigma_\alpha
(\nutilde^\Uu_\beta)=\nutilde^{\Xx^{\alpha+1}}_{\zeta^\beta}, \]
so $\xi\in[\lambda^\beta,\zeta^\beta]$. Therefore
${<^\Uu}\rest(\alpha+2)={<^{\Xx/\Tt}}\rest(\alpha+2)$,
giving  \ref{item:limit_Uu_rest_beta}.

For the remaining properties we split into cases.

\begin{casetwo}
$\beta$ is easy or $\Uu$ drops in model or degree at $\alpha+1$.

The overall argument here is routine and left to the reader.
However, there are a couple of details which are new, and which we discuss.

We first show that $\alpha+1$ is easy (and establish some other useful facts).
If $\beta$ is easy this is immediate.
Suppose $\beta$ is non-easy but
$\alpha+1\in \dr^\Uu$. So $E^\Uu_\beta\neq
F^{M^\Uu_\beta}$, and in fact
\[ \kappa_E<\In(E^\Uu_\beta)<(\kappa_E^+)^{M^\Uu_\beta},\]
and as $\nrsigma_\beta$ is nice,
\[
\kappa_F<\In(E^{\Xx^{\alpha+1}}_{\zeta^\beta})<(\kappa_F^+)^{M^{\Xx^\beta}
_\infty}.\]
Now $\xi=\zeta^\beta$. For otherwise $\xi\in[\lambda^\beta,\zeta^\beta)$ and
$\kappa_F<\In(E^{\Xx^\beta}_\xi)$.
But $\In(E^{\Xx^{\beta}}_\xi)$ is a cardinal in $M^{\Xx^\beta}_\infty$,
and so
$(\kappa_F^+)^{M^\Uu_\beta}=(\kappa_F^+)^{\exit^{\Xx^\beta}_\xi}\leq\In(E^{\Xx^{
\alpha+1}}_{\zeta^\beta})$,
contradiction. Similarly,
\[ M^{*\Xx^{\alpha+1}}_{\zeta^\alpha+1}=\nrsigma_\beta(N^*_{\alpha+1})\pins
M^{\Xx^\beta}_{\zeta^\beta}=M^{\Xx^{\alpha+1}}_{\zeta^\beta},\]
\[ \text{ and if }\zeta^\beta+1<\lh(\Xx^\beta)\text{ then
}M^{*\Xx^{\alpha+1}}_{\zeta^\alpha+1}\pins
\exit^{\Xx^\beta}_{\zeta^\beta}, \]
(for the latter, use the fact that
$\In(E^{\Xx^{\alpha+1}}_{\zeta^\beta})<\In(E^{\Xx^\beta}_{\zeta^\beta})$ and
$\In(E^{\Xx^\beta}_{\zeta^\beta})$ is a cardinal of $M^{\Xx^\beta}_\infty$), so
$\zeta^\alpha+1\in\dr^{\Xx^{\alpha+1}}$ and $\lambda^{\alpha+1}\notin
C^{\alpha+1}$, hence $\alpha+1$ is easy.

Now suppose that $\beta$ is non-easy but $\Uu$ drops in degree, but not in
model, at $\alpha+1$. Then we claim that
$\xi+1=\lh(\Xx^\beta)$, and therefore $\alpha+1$ is easy (but
$\lambda^{\alpha+1}\in C^{\alpha+1}$). For because $\beta$ is non-easy, we have
$[0,\beta]_\Uu\inter\dropset_{\deg}^\Uu=\emptyset$
by property \ref{item:main_embedding}, so
\[ n=\udeg^\Uu(\beta)=\udeg^\Tt(\infty).\]
 So $\udash\rho_n(M^\Uu_\beta)\leq\kappa_E$.
Letting $\eps+1\in
b^\Tt$ and $G=E^\Tt_\eps$, then we have $\In(G)\leq\udash\rho_n(M^\Tt_\infty)$,
and $i^\Uu_{0\beta}$ is continuous at $\In(G)$,
so $i^\Uu_{0\beta}(\In(G))<\kappa_E$. But by properties
\ref{item:interval_not_easy} and
\ref{item:pi_beta(nu^Uu_alpha)}  we have
\[ \nrsigma_\beta\com
i^\Uu_{0\beta}(\In(G))=\omega^{\beta}_\infty(\In(G))\geq\In(E^{\Xx^\beta}_{
\delta^\beta_\eps}). \]
Therefore $\nrsigma_\beta(\kappa)>\In(E^{\Xx^\beta}_{\delta^\beta_\eps})$.
This holds for every $\eps+1\in b^\Tt$, and it follows that
$\xi+1=\lh(\Xx^\beta)$.

To see that if $\Tt$ is terminally non-dropping and
$[0,\alpha+1]_\Uu$ does not drop in model or degree, then $\nrsigma_{\alpha+1}$
is a $\udash m$-embedding, use
the cofinality of the relevant maps at $\udash\rho_m$.

We now consider the verification that $\nrsigma_{\alpha+1}$ is a near $\udash
k$-embedding,
where $k=\udeg^\Uu(\alpha+1)$,
given that $[0,\alpha+1]_\Uu\inter\dropset_{\deg}^\Uu\neq\emptyset$.
The reader can safely skip this proof on a first pass, if they are so inclined,
moving to Case \ref{case:beta_non-easy,Uu_not_drop_alpha+1} below;
it is just a detail which is not central to our considerations.
We officially assume that $M$ is $\lambda$-indexed for the proof,
and thus can drop the prefix ``$\udash$''.
The proof is mostly like in that of Lemma \ref{lem:tree_embedding_copy}
(which was a slight variant of that in
\cite{fs_tame}), so we leave most of the details to the reader.
However, it  requires one extra observation. Fix $\delta\leq^\Uu\beta$
largest such that
$[0,\delta]_\Uu$ does not drop in model or degree.
So $\deg^\Uu(\delta)=n$ and $\nrsigma_\delta:M^\Uu_\delta\to
M^{\Xx^\delta}_\infty$ is an $n$-lifting embedding. Let
\[ X=\{\gamma\leq\alpha+1\bigm|  \delta<^\Uu\gamma\text{ and
}\successor^\Uu(\delta,\gamma)\in\dropset_{\deg}^\Uu\}\]
and $X'=\{\lambda^\gamma\bigm| \gamma\in X\}$. Note that
$\gamma\mapsto\lambda^\gamma$
is an isomorphism between ${<^\Uu}\rest X$ and
${<^{\Xx^{\alpha+1}}}\rest X'$. For $\chi$ such that $\chi+1\in X$, we define
\dfnemph{strong closeness at $\chi$}
(relating the definability of the measures of $E^\Uu_\chi$ to that of their
lifts, measures of $E^{\Xx^{\alpha+1}}_{\zeta^\chi}$),
and for $\eps\in X$, we define \dfnemph{translatability at $\eps$} (which,
given $\gamma+1\in X$ with $\gamma+1\leq^\Uu\eps$ and
$(\gamma+1,\eps)_\Uu\inter\dropset^\Uu=\emptyset$,
allows us to translate definitions of subsets of
$\crit(i^{*\Uu}_{\gamma+1,\eps})$ over $M^\Uu_\eps$,
to definitions over $M^{*\Uu}_{\gamma+1}$, in a manner which reflects up to
$M^{\Xx^{\alpha+1}}_{\lambda^\eps}$
and $M^{*\Xx^{\alpha+1}}_{\lambda^{\gamma+1}}$). One proves these properties
hold inductively,
basically as in  \cite{fs_tame} (simultaneously showing that $\nrsigma_\gamma$
is a near
$\deg^\Uu(\gamma)$-embedding
for each $\gamma\in X$). However, there is a wrinkle in verifying that
$\nrsigma_{\alpha+1}$
is a near $k$-embedding for example when:
\begin{enumerate}[label=--]
 \item $[0,\alpha+1]_\Uu\inter\dropset^\Uu=\emptyset$ (but
$[0,\alpha+1]_\Uu\inter\dropset_{\deg}^\Uu\neq\emptyset$, so
$\xi+1=\lh(\Xx^\beta)$ and $\lambda^{\alpha+1}\notin\dropset^{\Xx^{\alpha+1}}$
and
$M^{*\Xx^{\alpha+1}}_{\lambda^{\alpha+1}}=M^{\Xx^\beta}_\xi=M^{\Xx^\beta}
_\infty$),
  \item $k+1=n$ (so $\rho_{k+1}^{M^\Uu_\beta}\leq\kappa_E<\rho_k^{M^\Uu_\beta}$
where $E=E^\Uu_\alpha$),
 \item $\deg^{\Xx^{\alpha+1}}(\lambda^{\alpha+1})=k+1$ (so
$\nrsigma_\beta(\rho_{k+1}^{M^\Uu_\beta})\leq\nrsigma_\beta(\kappa_E)<\rho_{k+1}
(M^{\Xx^\beta}_\xi)$).
\end{enumerate}
For $i\in\{k,k+1\}$ let $U_i=\Ult_{i}(M^{\Xx^\beta}_\xi,F)$ where
$F=E^{\Xx^{\alpha+1}}_{\zeta^\alpha}$.
Let
$j_i:M^{\Xx^\beta}_\xi\to U_i$
be the ultrapower map.
By induction, $\nrsigma_\beta$ is a near $k$-embedding,
and letting
\[ \nrsigmabar:M^\Uu_{\alpha+1}=\Ult_{k}(M^\Uu_\beta,E^\Uu_\alpha)\to U_k \]
be given by the Shift Lemma, then as above (using the argument of
\cite{fs_tame}),
$\nrsigmabar$ is a near $k$-embedding.
Now we have $\deg^{\Xx^{\alpha+1}}(\lambda^{\alpha+1})=k+1$,
and $\nrsigma_{\alpha+1}=\nrsigma'\com\nrsigmabar$ where
$\nrsigma':U_k\to U_{k+1}$
is the natural factor map.
So it suffices to see that, in fact, $U_k=U_{k+1}$ and $\nrsigma'=\id$; this
completes the proof.

To see this, it suffices to see that
$\nrsigma'``\rho_k^{U_k}$
is cofinal in $\rho_k^{U_{k+1}}$.
For   suppose this holds.
Note $\nrsigma'(\pvec_{k+1}^{U_k})=\pvec_{k+1}^{U_{k+1}}$,
and by
\cite[Lemma 2.4]{premouse_inheriting},  it follows that
$\nrsigma'$ is a $k$-embedding. But $U_k,U_{k+1}$
are $(k+1)$-sound (see the proof of \cite[Corollary 2.24]{extmax}, for example)
and $\rho_{k+1}^{U_k}=\rho_{k+1}^{U_{k+1}}$
and $\nrsigma'\rest\rho_{k+1}^{U_k}=\id$. It follows
that $U_{k+1}\sub\rg(\nrsigma')$, which suffices.

To see the desired cofinality of $\nrsigma'$,
it suffices to see that
$j_{k+1}$ is continuous at $\rho_k(M^{\Xx^\beta}_\xi)$,
since $j_{k+1}=\nrsigma'\com j_k$ and $\rho_k^{U_k}=\sup
j_k``\rho_k(M^{\Xx^\beta}_\xi)$.

Now let $\mu^{\Xx^\beta}=\wcof^{M^{\Xx^\beta}_\xi}_{k+1}$
(see Definition \ref{dfn:weak_cof}). We have
$\deg^{\Xx^{\alpha+1}}(\lambda^{\alpha+1})=k+1$,
so $\crit(F)<\rho_{k+1}(M^{\Xx^\beta}_\xi)$. So $j_{k+1}$ is continuous at
$\rho_k(M^{\Xx^\beta}_\xi)$
iff $\crit(F)\neq\mu^{\Xx^\beta}$.
So we must see $\crit(F)\neq\mu^{\Xx^\beta}$.

Let $\mu^\Tt=\wcof^{M^\Tt_\infty}_{k+1}$.
Then because $\psi_{0\beta}=\omega^{0\beta}_\infty$ is a near $(k+1)$-embedding
(as
$\deg^\Tt(\infty)=k+1=\deg^{\Xx^{\alpha+1}}(\lambda^{\alpha+1})=\deg^{\Xx^\beta}
(\xi)$)
and by Lemma \ref{lem:wcof_pres}, either:
\begin{enumerate}[label=--]
 \item $\mu^\Tt<\rho_{k+1}(M^\Tt_\infty)$ and
$\mu^{\Xx^\beta}=\psi_{0\beta}(\mu^\Tt)$, or
 \item $\mu^\Tt=\rho_{k+1}(M^\Tt_\infty)$ and
$\mu^{\Xx^\beta}=\rho_{k+1}(M^{\Xx^\beta}_\xi)$.
\end{enumerate}
But $\crit(F)<\rho_{k+1}(M^{\Xx^\beta}_\xi)$.
So suppose $\mu^\Tt<\rho_{k+1}(M^\Tt_\infty)$ and (by commutativity)
\[
\mu^{\Xx^\beta}=\psi_{0\beta}(\mu^\Tt)=\nrsigma_\beta(i^\Uu_{0\beta}(\mu^\Tt)).
\]
Then since  $\deg^\Uu(0)=\deg^\Uu(\beta)=k+1$ and $\deg^\Uu(\alpha+1)=k$,
\[ i^\Uu_{0\beta}(\mu^\Tt)<\rho_{k+1}(M^\Uu_\beta)\leq\crit(E), \]
so
$\mu^{\Xx^\beta}<\nrsigma_\beta(\crit(E))=\crit(F)$,
completing the proof that $\nrsigma_{\alpha+1}$ is a near $k$-embedding.

We do not need this kind of argument for degrees $<k$,
because if $\nrsigma:R\to S$ is a near $k$-embedding where $k>0$,
then $\nrsigma(\rho_k^R)\geq\rho_k^S$.
We leave the remaining details in this case to the reader.
\end{casetwo}

\begin{casetwo}\label{case:beta_non-easy,Uu_not_drop_alpha+1} $\beta$ is
non-easy,
and $\Uu$ does not drop in model or degree at
$\alpha+1$.

So $\xi\in C^\beta$. Let $\theta^*=f^\beta(\xi)$.

\begin{scasetwo} $\xi+1=\lh(\Xx^\beta)$.

Then
$\nrsigma_\beta:M^\Uu_\beta\to M^{\Xx^\beta}_\xi$,
and everything is routine. We have $\lambda^{\alpha+1}\in C^{\alpha+1}$
 but $\theta^*+1=\theta^{\alpha+1}+1=\lh(\Tt)$, so $\alpha+1$ is easy.
\end{scasetwo}

\begin{scasetwo}\label{scase:stack_2_main} $\xi+1<\lh(\Xx^\beta)$.

So $\theta^*+1<\lh(\Tt)$.
Now $E$ is total over $M^\Uu_\beta$ and
$(\kappa_E^+)^{M^\Uu_\beta}\leq\nutilde^\Uu_\beta$ (for
$\kappa_E<\nutilde^\Uu_\beta$, so if
$(\kappa_E^+)^{M^\Uu_\beta}>\nutilde^\Uu_\beta$ then
$E^\Uu_\beta=F^{M^\Uu_\beta}$ and
$\kappa_E=\lgcd(M^\Uu_\beta)$, but then
$E^{\Xx^\beta}_{\zeta^\beta}=F(M^{\Xx^\beta}_\infty)$ and
$\kappa_F=\lgcd(M^{\Xx^\beta}_\infty)$, so $\xi+1=\lh(\Xx^\beta)$,
contradiction). Therefore
$(\kappa_F^+)^{M^{\Xx^\beta}_\infty}\leq\nutilde^{\Xx^{\alpha+1}}_{\zeta^\beta}$
, so $F$ is total over
$M^{\Xx^\beta}_\infty$, so $F$ is total over $\exit^{\Xx^\beta}_\xi$ (and
$E^{\Xx^\beta}_\xi$ is the copy of $E^\Tt_{\theta^*}$), so
$\exit^\beta_\xi\ins M^{*\Xx^{\alpha+1}}_{\zeta^\alpha+1}$.
So $\lambda^{\alpha+1}=\zeta^\alpha+1\in C^{\alpha+1}$ and
$\theta^{\alpha+1}=\theta^*$.

\begin{figure}
\centering
\begin{tikzpicture}
 [mymatrix/.style={
    matrix of math nodes,
    row sep=0.9cm,
    column sep=1cm}
  ]
   \matrix(m)[mymatrix]{
  N_{\alpha+1}& \bar{M}_\infty & M^{\Xx^{\alpha+1}}_\infty\\
  N_\beta & M^{\Xx^\beta}_\infty & {}\\
   };
\path[->,font=\scriptsize]
(m-1-1) edge[bend left] node[above] {$\nrsigma_{\alpha+1}$} (m-1-3)
(m-1-1) edge node[above] {$\nrsigmabar$} (m-1-2)
(m-1-2) edge node[above] {$\nrsigma'$} (m-1-3)
(m-2-1) edge node[above] {$\nrsigma_\beta$} (m-2-2)
(m-2-1) edge node[left] {$i^\Uu_{\beta,\alpha+1}$} (m-1-1)
(m-2-2) edge node[right,pos=0.25] {$\ \ \psi_{\beta,\alpha+1}$} (m-1-3)
(m-2-2) edge node[left] {$\psibar$} (m-1-2);
\end{tikzpicture}
\caption{The diagram commutes, in Subcase
\ref{scase:stack_2_main}.}\label{fgr:stack_2_comm_succ}
\end{figure}

See Figure \ref{fgr:stack_2_comm_succ}. Let $\psi=\psi_{\beta,\alpha+1}$ (see
\ref{item:interval_not_easy}) and
\[
\varsigma=\om^{\beta,\alpha+1}_{\theta^{\alpha+1}}=i^{*\Xx^{\alpha+1}}_{
\zeta^\alpha+1}\rest Q^\beta_\xi.\]
By  \ref{lem:Pi^alpha,beta_is_tree_embedding}
(and recall property \ref{item:model_embeddings}), $\psi$ is a near
$\udash n$-embedding,
and by \ref{lem:int_tree_embs_int_comm_extra_agmt}, $\varsigma\rest
Q^\beta_\xi\sub\psi$.
So $F\rest\nutilde(F)$ is the $(\kappa_F,\nutilde(F))$-extender derived from
$\psi$,
and note
\[
Q^\beta_\xi||(\kappa_F^+)^{Q^\beta_\xi}=M^{\Xx^\beta}_\infty||(\kappa_F^+)^{M^{
\Xx^\beta}_\infty}.\]
Let $\Mbar_\infty,\psibar,\nrsigmabar,\nrsigma',\nrsigma_{\alpha+1}$ be defined
as follows:
\begin{enumerate}[label=--]
 \item $\Mbar_\infty=\Ult_{\udash n}(M^{\Xx^\beta}_\infty,F)$,
 \item $\psibar:M^{\Xx^\beta}_\infty\to\Mbar_\infty$ is the associated
ultrapower map $i^{M^{\Xx^\beta}_\infty,\udash n}_F$,
 \item $\nrsigmabar:M^\Uu_{\alpha+1}\to\Mbar_\infty$
is given by the Shift Lemma from $\nrsigma_\beta$ and
$\nrsigma_\alpha\rest\exit^\Uu_\alpha$,
 \item $\nrsigma':\Mbar_\infty\to M^{\Xx^{\alpha+1}}_\infty$
is the natural factor map (for example if $n=0$ then
\[ \nrsigma'([a,f]^{M^{\Xx^\beta}_\infty}_F)=\psi(f)(a), \]
and if $n>0$ use the obvious generalization; this is well-defined by the
remarks
above),
 \item $\nrsigma_{\alpha+1}=\nrsigma'\com\nrsigmabar:M^\Uu_{\alpha+1}\to
M^{\Xx^{\alpha+1}}_\infty$.
\end{enumerate}

Let $\mu=i_F(\kappa_F)$. Then
$\nrsigma'$ is nice $\udash n$-lifting,
 $\nrsigma'\com\psibar=\psi$,
and
\[
\crit(\nrsigma')>(\mu^+)^{\Mbar_\infty}=(\mu^+)^{M^{\Xx^{\alpha+1}}_\infty},\]
so $\nrsigma'$ fixes  $\nutilde^{\Xx^{\alpha+1}}_{\zeta^\alpha}$ and
$\In(E^{\Xx^{\alpha+1}}_{\zeta^\alpha})=\In(F)$.
(The latter holds as $\nrsigma'\rest\mu=\id$
and $\psi(\kappa_F)=\mu$
and
$(\mu^+)^{\Mbar_\infty}=(\mu^+)^{Q^{\alpha+1}_{\theta^{\alpha+1}}}=(\mu^+)^{M^{
\Xx^{\alpha+1}}_\infty}$.)

Also note
that $\nrsigmabar$ is nice $\udash n$-lifting,
  $\nrsigma_\alpha\rest\exit^\Uu_\alpha\sub\nrsigmabar$,
$\nrsigmabar(\nutilde^\Uu_\alpha)=(\nutilde^{\Xx^{\alpha+1}}_{
\zeta^\alpha})$, $\nrsigmabar(\In(E))=\In(F))$,
and $\nrsigmabar\com i^{\Uu}_{\beta,\alpha+1}=\psibar\com\nrsigma_\beta$.
(We have $\In(E)\in M^\Uu_{\alpha+1}$ because if
$\In(E)=\OR^{M^\Uu_{\alpha+1}}$,
so $\kappa_E=\lgcd(M^\Uu_\beta)$ and $M^\Uu_\beta$ is active type 2,
then $\kappa_F=\lgcd(M^{\Xx^\beta}_\infty)$, but then $\xi+1=\lh(\Xx^\beta)$,
contradiction).

Therefore
\begin{enumerate}[label=--]
 \item $\nrsigma_{\alpha+1}$ is nice $\udash n$-lifting,
\item  $\nrsigma_\alpha\rest\exit^\Uu_\alpha\sub\nrsigma_{\alpha+1}$,
\item
$\nrsigma_{\alpha+1}(\nutilde^\Uu_\alpha)=\nutilde^{\Xx^{
\alpha+1}}_{\zeta^\alpha}$
and $\nrsigma_{\alpha+1}(\In(E^\Uu_\alpha))=\In(E^{\Xx^{\alpha+1}}_{
\zeta^\alpha}))$,
\item $\psi_{\beta,\alpha+1}\com\nrsigma_\beta=\nrsigma_{\alpha+1}\com
i^\Uu_{\beta,\alpha+1}$.
\end{enumerate}
It is now easy to see that $\varphi(\alpha+2)$ holds.
\end{scasetwo}
\end{casetwo}

Now suppose $\eta<\Omega$ is a limit. We have
\[
\Xx^\eta\rest\lambda^\eta=\bigcup_{\alpha<\eta}\Xx^\alpha\rest(\zeta^\alpha+1)
\]
and $[0,\lambda^\eta)_{\Xx^\eta}=\Sigma(\Xx^\eta\rest\lambda^\eta)$, giving
$\Xx^\eta\rest(\lambda^\eta+1)$.
Since ${<^\Uu}\rest\eta=({<^{\Xx^\eta/\Tt}})\rest\eta$,
we can and do define a $\Uu\rest\eta$-cofinal branch by setting
\[ [0,\eta)_\Uu=[0,\eta)_{\Xx^\eta/\Tt},\]
maintaining property
\ref{item:limit_Uu_rest_beta}.
Note that $\Xx^\eta$ exists (and is according to $\Sigma$), by inflation
condensation and as $\eta<\Omega$.
We now define
\[ \nrsigma_\eta:M^\Uu_\eta\to M^{\Xx^\eta}_\infty;\]
it will then be easy to see that $\varphi(\eta+1)$ holds.

\begin{casethree}
There is $\alpha<\eta$ such that $\alpha$ is easy and
$\lambda^\alpha<^{\Xx^\eta}\lambda^\eta$.

Then $\beta$ is easy for every $\beta\in[\alpha,\eta]_{\Xx^\eta/\Tt}$, and
using
the inductive hypotheses, $[0,\eta)_\Uu$ has only finitely many drops
and we can define $\nrsigma_\eta$ commuting with earlier maps in a routine
manner.
\end{casethree}

\begin{casethree} Otherwise.\footnote{In this case, $\eta$ itself can be easy,
but this is not relevant.}

By \ref{lem:int_tree_embs_int_comm_extra_agmt},
$\psi_{\alpha\eta}=\psi_{\beta\eta}\com\psi_{\alpha\beta}$
for all $\alpha\leq^\Uu\beta<^\Uu\eta$.
So by the commutativity given by property \ref{item:interval_not_easy} we can
and do define $\nrsigma_\eta$
in a unique manner preserving commutativity.
That is,
\[ \nrsigma_\eta\com i^\Uu_{\alpha\eta}=\psi_{\alpha\eta}\com\nrsigma_\alpha\]
for all $\alpha<^\Uu\eta$.
\end{casethree}

This completes the definition of $\Upsilon^\Sigma_\Tt$; clearly it is a
$(\udash
n,\Omega)$-strategy for $N$, as desired. If $\lh(\Uu)=\alpha+1$
then we finally set $\Xx=\Xx^\alpha$,
and if $\lh(\Uu)$ is a limit $\eta$ we set
$\Xx=\bigcup_{\alpha<\eta}\Xx^\alpha\rest(\lambda^\alpha+1)$.
So if $\lh(\Uu)$ is a successor then $\Xx$ is a $\Tt$-terminal inflation of
$\Tt$.

Finally suppose that $\Sigma$ extends to a $(\udash m,\Omega+1)$-strategy
$\Sigma'$ for $M$.
Then $\Upsilon^\Sigma_\Tt$ extends to a $(\udash n,\Omega+1)$-strategy
$\Upsilon^{\Sigma'}_\Tt$ for $N$.
For given $\Uu$ via $\Upsilon^\Sigma_\Tt$ of length $\Omega$,
note that $\Xx$ also has length $\Omega$, and $\varphi(\Omega)$
holds. But then just as in the limit case above, we get a $\Uu$-cofinal branch
$b$,
and $M^\Uu_b$ is well-defined and wellfounded as $\cof(\Omega)>\om$, so player
II has won.
We don't actually need $\Xx^\Omega$ here, but we can and do define it by
copying the remainder of $\Tt$ following $\Xx^\Omega\rest(\Omega+1)$.
Of course if $\lh(\Xx^\Omega)>\Omega+1$ then $\Xx^\Omega$ is not literally via
$\Sigma$,
but note that its models are wellfounded, because $\Omega>\om$ is regular.
We then define $\psi_{\alpha\Omega}$ and $\nrsigma_\Omega$ as before.

\begin{dfn}\label{dfn:W(T,U)_etc}\index{$\Ww^\Sigma_\Tt$}\index{
$\nrsigma^\Sigma_\Tt$}\index{$\Upsilon^\Sigma_\Tt$}
Given the objects above,
let
\[ \Ww^\Sigma_\Tt(\Uu)=\Ww^\Sigma(\Tt,\Uu)=\Xx,\]
and if $\lh(\Uu)$ is also a successor, let
\[ \nrsigma^\Sigma_\Tt(\Uu)=\nrsigma^\Sigma(\Tt,\Uu)=\nrsigma_{\lh(\Uu)-1}.\]
And $\Upsilon^\Sigma_\Tt$ denotes the $\udash\deg^\Tt(\infty)$-strategy for
$N=M^\Tt_\infty$ defined above.
\end{dfn}

\begin{rem}\label{rem:b_determines_c}
The following observation, which is natural, but not actually important for our
construction,
is mostly due to Steel:\footnote{
 Our construction uses only
 the fact that the branches of $\Xx$ determine those of $\Uu$, so Steel's
observation
 is not important for us here, and the author did not initially consider it.
 It is, however, relevant to Steel's construction, as he proceeds in the other
direction.
 After we had developed most of our construction, Steel pointed out that for
each $(b,d)$
  such that $b$ is a $\Uu\rest\eta$-cofinal branch
  and $d$ is either a node in $\Tt$ or a $\Tt$-maximal branch,
  there is at most one corresponding $\wt{\Xx}$-cofinal branch $c_{b,d}$.
  The author later noticed that $b$ in fact determines $d$, given that we are
seeking an inflation of $\Tt$.}
  One could actually drop the superscript ``$\Sigma$'' in the notation
$\Ww^\Sigma_\Tt$ and $\nrsigma^\Sigma_\Tt$,
  without ambiguity.

 For consider the limit stage $\eta$ in the preceding construction.
 Let $\wt{\Xx}=\Xx\rest\lambda^\eta$ be defined as above.
 We observe that $[0,\eta)_\Uu$ determines $[0,\lambda_\eta)_{\Xx^\eta}$,
subject to the requirement
 that $\Xx^\eta$ be an inflation of $\Tt$.
 In fact, for any $\Uu\rest\eta$-cofinal branch $b$ there is a unique
$\wt{\Xx}$-cofinal branch
 $c=c_b$ such that
 \begin{enumerate}[label=--]
  \item $c$ induces $b$ in the same manner that $[0,\lambda^\eta)_{\Xx^\eta}$
induces $[0,\eta)_\Uu$, and
  \item if $b\inter\dropset^\Uu=\emptyset$  then $\wt{\Xx}\conc
c$ is a
\emph{putative inflation of $\Tt$},
  meaning that all requirements of inflations are met, excluding the requirement
  that $M^{\wt{\Xx}}_c$ be well-defined and wellfounded.
 \end{enumerate}

 For write $C=C^{\Tt\inflatearrow\wt{\Xx}}$, etc, and
$C'=C^{\Tt\inflatearrow(\wt{\Xx},c)}$, etc
 (for a candidate $c$).
 If $\lambda^\beta\notin C$ for some $\beta\in b$,
 this is immediate (and $\lambda^\eta\notin C'$).
 So suppose otherwise and let
   $\theta=\sup_{\beta\in b}f(\lambda^\beta)$. Then
   $\lambda^\eta\in C'$ and $f'(\lambda^\eta)=\theta$.
 If $\theta=f(\lambda^\beta)$ for some $\beta\in b$
 (hence $\theta=f(\lambda^\beta)$ for all sufficiently large $\beta\in b$)
 then everything is clear.
 So suppose otherwise; then $\theta$ is a limit.
 Note that for $c$ as desired to exist, we must have $\theta<\lh(\Tt)$.
 Note that for $\alpha<\theta$,
 $\gamma_\alpha\eqdef\lim_{\beta\in b}\gamma^\beta_\alpha$ exists,
$c=\bigcup_{\alpha<^\Tt\theta}[0,\gamma_\alpha)_{\wt{\Xx}}$
is an $\wt{\Xx}$-cofinal branch, $(\wt{\Xx},c)$ is a putative inflation of
$\Tt$,
$c$ determines $b$, and moreover, $c$ is the unique such branch.
\end{rem}

\begin{rem}\label{rem:fail_near_emb}
Consider condition \ref{item:main_embedding} of the preceding construction.
By this condition,
$\nrsigma_\alpha$ is a $\udash k$-lifting embedding, and if
$[0,\alpha]_\Uu\inter\dropset_{\deg}\neq\emptyset$
then $\nrsigma_\alpha$ is a near $\udash k$-embedding.
Also by this condition, if $\Tt$ is terminally non-dropping and
$[0,\alpha]_\Uu\inter\dropset_{\deg}^\Uu=\emptyset$
then $\nrsigma_\alpha$ is a near $\udash k$-embedding (in fact a $\udash
k$-embedding).
But $\nrsigma_\alpha$ can fail to be a near $\udash n$-embedding when $\Tt$ is
terminally dropping and $[0,\alpha]_\Uu\inter\dropset_{\deg}^\Uu=\emptyset$.
Moreover, it can also be that $M$ has $\lambda$-indexing and:
\begin{enumerate}[label=--]
 \item  $n=k+1=\deg^\Tt(\infty)>0$,
 \item $\rho_{k+1}(M^\Uu_\alpha)<\OR(M^\Uu_\alpha)$,
 \item
$\nrsigma_\alpha(\rho_{k+1}(M^\Uu_\alpha))<\rho_{k+1}(M^{\Xx^\alpha}_\infty)$,
 \item  $M^\Uu_\alpha$ has a measurable $\gamma\geq\rho_{k+1}(M^\Uu_\alpha)$
such that $\nrsigma_\alpha(\gamma)<\rho_{k+1}(M^{\Xx^\alpha}_\infty)$,
 \item $E^\Uu_\alpha$ is $M^\Uu_\alpha$-total with
$\crit(E^\Uu_\alpha)=\gamma$
and $\zeta^\alpha+1=\lh(\Xx^\alpha)$ (and $E^{\Xx^\alpha}_{\zeta^\alpha}$
is
$M^{\Xx^\alpha}_\infty$-total),
 so
 \[ \deg^\Uu(\alpha+1)=k\text{ but }\deg^{\Xx^{\alpha+1}}(\infty)={k+1} \]
 (but as we saw, even in this case, $\nrsigma_{\alpha+1}$ is a near
$k$-embedding).
\end{enumerate}
For here is an example, with $k=0$.
Suppose that $M$ is $\lambda$-indexed, $2$-sound and $\rho_2^M<\rho_1^M<\OR^M$
and $\cof^M(\rho_1^M)=\kappa$
where $\kappa<\rho_2^M$ is $M$-measurable and $\rho_1^M$ is a limit of
$M$-measurables.
Let $\mu\in[\rho_2^M,\rho_1^M)$ be $M$-measurable and $E\in\es^M$ be a measure
on $\mu$.
Let $\Tt$ be the $2$-maximal tree on $M$ using only $E$. So
$N=M^\Tt_1$, $n=\deg^\Tt(1)=1$,
and $\Uu$ will be a $1$-maximal tree.
Let $F\in\es^M\inter\es^N$ be the order $0$ measure on
$\kappa$.
Let $E^\Uu_0=F$ ($\Uu$ will use two extenders; $E^\Uu_1$ will be defined in a
moment).
This determines $\Xx^0=\Tt$ and $\Xx^1$.
Note that (so far) there is no dropping in model in any of our trees.
We have  $\nrsigma_1:M^\Uu_1\to M^{\Xx^1}_\infty$.
We claim
$\deg^\Uu(1)=1=\deg^{\Xx^1}(\infty)$ but
$\nrsigma_1(\rho_1(M^\Uu_1))<\rho_1(M^{\Xx^1}_\infty)$,
and therefore $\nrsigma_1$ is not a near $1$-embedding. To see this, use
routine
calculations to verify the following:
\begin{enumerate}[label=--]
 \item $\deg^\Tt(1)=1=\deg^\Uu(0)=\deg^\Uu(1)$,
 \item $\lambda_0=\zeta_0=0$ (so $\lambda^1=1$),
 \item $\lh(\Xx^1)=3$ and $E^{\Xx^1}_0=F$ and $E^{\Xx^1}_1=i^{\Xx^1}_{01}(E)$,
 \item $\deg^{\Xx^1}(1)=2$ and $\deg^{\Xx^1}(2)=1$,
 \item $\rho_1^N=\sup i^\Tt``\rho_1^M=i^\Tt(\rho_1^M)$,
 \item $\rho_1(M^\Uu_1)=\sup
i^\Uu_{01}``\rho_1^N<i^\Uu_{01}(\rho_1^N)=i^\Uu_{01}\com i^\Tt(\rho_1^M)$,
 \item $\sup
i^{\Xx^1}_{01}``\rho_1^M<\rho_1(M^{\Xx^1}_1)=i^{\Xx^1}_{01}(\rho_1^M)$,
 \item $\rho_1(M^{\Xx^1}_2)=\sup
i^{\Xx^1}_{12}``\rho_1(M^{\Xx^1}_1)=i^{\Xx^1}_{12}(\rho_1(M^{\Xx^1}_1))=i^{\Xx^1
}_{02}(\rho_1^M)$,
 \item $\nrsigma_1\com i^\Uu_{01}\com i^\Tt=i^\Xx_{02}$, and hence,
$\nrsigma_1(i^\Uu_{01}(i^\Tt(\rho_1^M)))=i^\Xx_{02}(\rho_1^M)=\rho_1(M^{\Xx^1}
_2)$.
\end{enumerate}
The claim follows from these facts, and gives the desired example.
\end{rem}

\subsubsection{Stacks of limit length}\label{subsubsec:limit_length_stacks}

From now on we assume that $\Sigma$ is a $(\udash
m,\Omega+1)$-strategy for
$M$
with inflation condensation.
We will define an optimal-$(\udash m,\Omega,\Omega+1)^*$-strategy $\Sigma^*$
for $M$.
Let us say that a \dfnemph{round} of the iteration game consists of a single
normal tree.
Given $\alpha<\Omega$, at the start of round $\alpha$, with player $\Two$ not
yet having lost, we
will
have defined sequences
$\left<\Tt_\beta\right>_{\beta<\alpha}$,
$\left<O_\beta,n_\beta,\Yy_\beta,\srsigma_\beta\right>_{\beta\leq\alpha}$ with
the following
properties (\emph{S} is for \emph{stack}; see Figure
\ref{fgr:infinite_stack_comm}):

\begin{figure}
\centering
\begin{tikzpicture}
 [mymatrix/.style={
    matrix of math nodes,
    row sep=0.7cm,
    column sep=0.6cm}
  ]
   \matrix(m)[mymatrix]{
   {}&              {}&              {}&                          {}&
             {}&                      {}&                          {}&\Yy_\eta\\
   {}&              {}&              {}&                          {}&
             {}&                      {}&
\Yy_\delta&M^{\Yy_\eta}_\infty\\
   {}&              {}&              {}&                          {}&
             {}&           \Yy_{\eps+1}&       M^{\Yy_\delta}_\infty&{}\\
   {}&              {}&              {}&                          {}&
      \Yy_\eps&M^{\Yy_{\eps+1}}_\infty&                          {}&{}\\
   {}&              {}&              {}&                  \Yy_\alpha&
M^{\Yy_\eps}_\infty&  M^{\Uu_\eps}_\infty{}&                          {}&{}\\
   {}&              {}&           \Yy_2&       M^{\Yy_\alpha}_\infty&
             {}&                      {}&                          {}&{}\\
   {}&           \Yy_1&M^{\Yy_2}_\infty&                          {}&
             {}&                      {}&                          {}&{}\\
\Yy_0&M^{\Yy_1}_\infty&M^{\Uu_1}_\infty&                          {}&
             {}&                      {}&                          {}&{}\\
    M&O_1&O_2&O_\alpha&O_\eps&O_{\eps+1}&O_\delta&O_\eta\\
};

\path[->, dash pattern=on 4pt off 2pt,decoration={zigzag,segment
length=4,amplitude=.9, post=lineto,post length=2pt},font=\scriptsize,line
join=round]
(m-8-1) edge[decorate] node[auto] {} (m-7-2)
(m-7-2) edge[decorate] node[auto] {} (m-6-3)
(m-6-3) edge[decorate] node[auto] {} (m-5-4)
(m-5-4) edge[decorate] node[auto] {} (m-4-5);
\path[->,decoration={zigzag,segment length=4,amplitude=.9, post=lineto,post
length=2pt},font=\scriptsize,line join=round]
(m-4-5) edge[decorate] node[auto] {} (m-3-6)
(m-3-6) edge[decorate] node[auto] {} (m-2-7)
(m-2-7) edge[decorate] node[auto] {} (m-1-8);
\path[dotted,->,font=\scriptsize]
(m-9-1) edge node[below] {$\Tt_0$} (m-9-2)
(m-9-2) edge node[below] {$\Tt_1$} (m-9-3)
(m-9-3) edge node[below] {} (m-9-4)
(m-9-4) edge node[below] {} (m-9-5)
(m-9-1) edge node[auto] {} (m-8-2)
(m-8-2) edge node[auto] {} (m-7-3)
(m-7-3) edge node[auto] {} (m-6-4)
(m-6-4) edge node[auto] {} (m-5-5)
(m-8-2) edge node[below] {$\Uu_1$} (m-8-3);
   \path[->,font=\scriptsize]
(m-5-5) edge node[left] {$\pi_{\eps,\eps+1}$} (m-4-6)
(m-4-6) edge node[left] {$\pi_{\eps+1,\delta}$} (m-3-7)
(m-3-7) edge node[left] {$\pi_{\delta\eta}$} (m-2-8)
(m-9-5) edge node[auto] {$i_\eps$} node[below] {$\Tt_\eps$} (m-9-6)
(m-9-6) edge node[auto] {$i_{\eps+1,\delta}$} (m-9-7)
(m-9-7) edge node[auto] {$i_{\delta\eta}$} (m-9-8)
(m-5-5) edge node[below] {$\Uu_\eps$} node[above] {$j_\eps$} (m-5-6)
(m-9-2) edge  node[right] {$\srsigma_1$} (m-8-2)
(m-9-3) edge  node[left] {$\srsigma_2^0$} (m-8-3)
(m-8-3) edge  node[left,near start] {$\srsigma_2^1$} (m-7-3)
(m-9-3) edge [bend right] node[right] {$\srsigma_2$} (m-7-3)
(m-9-4) edge  node[right,pos=0.685] {$\srsigma_\alpha$} (m-6-4)
(m-9-5) edge  node[right,pos=0.49] {$\srsigma_\eps$} (m-5-5)
(m-9-6) edge  node[left,pos=0.495] {$\srsigma^0_{\eps+1}$} (m-5-6)
(m-5-6) edge  node[left,near start] {$\srsigma^1_{\eps+1}$} (m-4-6)
(m-9-6) edge [bend right=24] node[right,pos=0.375] {$\srsigma_{\eps+1}$} (m-4-6)
(m-9-7) edge  node[right,pos=0.311] {$\srsigma_\delta$} (m-3-7)
(m-9-8) edge  node[right,pos=0.2615] {$\srsigma_\eta$} (m-2-8);
\end{tikzpicture}
\caption{Commutative diagram for an infinite stack. Note that
$\Uu_\beta,\srsigma^0_{\beta+1},\srsigma^1_{\beta+1}$
are not mentioned in conditions
\ref{item:lim_stack_N_0}--\ref{item:inflation_iteration_commutativity}.
Note that $\Yy_0$ is trivial,  $O_2=M^{\Tt_1}_\infty$ and
$N_{\eps+1}=M^{\Tt_\eps}_\infty$.
The squiggly arrows indicate inflations $\Yy\inflatearrow\Zz$;
a dashed squiggly arrow indicates that $\Zz$ is possibly
$\Yy$-terminally-model-dropping,
whereas a solid squiggly arrow indicates that $\Zz$ is
$\Yy$-terminally-non-model-dropping.
The solid horizontal arrows are iteration maps;
$i_{\beta\gamma}=i^{\Ttvec\rest[\beta,\gamma)}$
(and  $b^{\Ttvec\rest[\beta,\gamma)}$ does not drop
in model where they appear in the diagram)
$i_\eps=i_{\eps,\eps+1}$ and $j_\eps=i^{\Uu_\eps}$. Dotted horizontal arrows
represent iteration trees possibly dropping on their main branches. Solid
diagonal arrows are final inflation copy maps
$\pi_\infty^{\Yy\inflatearrow\Zz}$ (such exist where they appear in
the diagram). Dotted diagonal arrows represent inflations
$\Yy\inflatearrow\Zz$ which are possibly $\Yy$-terminally-model-dropping.
Vertical arrows are the lifting maps $\srsigma_\delta$. All solid arrows
commute.} \label{fgr:infinite_stack_comm}
\end{figure}

\begin{enumerate}[label=S\arabic*.,ref=S\arabic*]
 \item\label{item:lim_stack_N_0} $O_0=M$ and $n_0=m$ and $\Yy_0$ is the trivial
tree on $M$ and $\srsigma_0:M\to M$ is the
identity.
 \item $n_\beta\leq\om$ and $O_\beta$ is a $\udash n_\beta$-sound
segmented-premouse and if $\beta<\alpha$ then
$\Tt_\beta$ is a $\udash n_\beta$-maximal tree on $O_\beta$ of successor length
$<\Omega$.
 \item $O_{\beta+1}=M^{\Tt_\beta}_\infty$ and
$n_{\beta+1}=\udeg^{\Tt_\beta}(\infty)$.
 \item For each limit $\beta\leq\alpha$, there is $\gamma<\beta$ such that for
all
$\eps\in[\gamma,\beta)$, $b^{\Tt_\eps}$ does not drop in model or degree,
$O_\beta=M^{\Ttvec\rest\beta}_\infty$ and
$n_\beta=\udeg^{\Ttvec\rest\beta}(\infty)$.
\footnote{That is, $O_\beta$ is the direct limit of the the $O_\eps$ for
$\eps\in[\gamma,\beta)$,
under the iteration maps, and $n_\beta=\lim_{\eps\to\beta}\udeg^{\Tt_\eps}(0)$.}
 \item $\Yy_\beta$ is a $\udash m$-maximal tree on $M$, via $\Sigma$, of
successor length $<\Omega$, and
$\udeg^{\Yy_\beta}(\infty)\geq n_\beta$.
 \item\label{item:lim_stack_tau_beta}
$\srsigma_\beta:O_\beta\to
M^{\Yy_\beta}_\infty$ is a
nice $\udash n_\beta$-lifting
embedding.\setcounter{footnote}{0}\footnote{In the proof,
for $\beta>0$, $\srsigma_\beta$ will be
defined as the composition $\srsigma_\beta^1\com\srsigma_\beta^0$.}
 \item\label{item:standard_terminal_inflation} For each
$\gamma<\beta\leq\alpha$, $\Yy_\beta$ is an
$\Yy_\gamma$-terminal inflation of $\Yy_\gamma$.
\item\label{item:Ttvec_model_drop} For each $\gamma<\beta\leq\alpha$,
$\Ttvec\rest[\gamma,\beta)$ drops in model\footnote{That is, $b^{\Tt_\eps}$
drops in model
for some $\eps\in[\gamma,\beta)$.} iff $\Yy_\beta$
is $\Yy_\gamma$-terminally-model-dropping.
\item\label{item:eventually_terminally-non-dropping} For each limit
$\beta\leq\alpha$ there is
$\eps<\beta$ such that
\[ \all\delta_0,\delta_1\ [\text{if }\eps\leq\delta_0<\delta_1\leq\beta\text{
then }\Yy_{\delta_1}\text{ is
}\Yy_{\delta_0}\text{-terminally-non-dropping}].\]
\item\label{item:inflation_iteration_commutativity} If $\gamma<\beta\leq\alpha$
and $\Yy_\beta$ is
$\Yy_\gamma$-terminally-non-model-dropping then letting
\[ \pi_{\gamma\beta}:M^{\Yy_\gamma}_\infty\to
M^{\Yy_\beta}_\infty \]
be $(\pi_\infty)^{\Yy_\gamma\inflatearrow\Yy_\beta}$ (see
\ref{dfn:inflation_notation}), we have
$\pi_{\gamma\beta}\com\srsigma_\gamma=\srsigma_\beta\com
i^{\Ttvec\rest[\gamma,\beta)}$.
\end{enumerate}

Given these inductive hypotheses, player $\Two$ plays out round
$\alpha$ as follows. We have
the nice $\udash n_\alpha$-lifting embedding
\[ \srsigma_\alpha:O_\alpha\to M^{\Yy_\alpha}_\infty,\]
and $n_\alpha\leq y_\alpha=\udeg^{\Yy_\alpha}(\infty)$
and $\lh(\Yy_\alpha)<\Omega$. We have the $(y_\alpha,\Omega+1)$-strategy
$\Upsilon^\Sigma_{\Yy_\alpha}$
for $M^{\Yy_\alpha}_\infty$ defined in \ref{dfn:W(T,U)_etc}.
Let $\bar{\Upsilon}$ be the $(n_\alpha,\Omega+1)$-strategy
for $O_\alpha$ which is the $\srsigma_\alpha$-pullback of
$\Upsilon^\Sigma_{\Yy_\alpha}$.
Then player II uses $\bar{\Upsilon}$ to play round $\alpha$ (forming
$\Tt_\alpha$).
So player II does not lose in round $\alpha$.

Now suppose that $\lh(\Tt_\alpha)<\Omega$, so the game continues.
So $O_{\alpha+1}=M^{\Tt_\alpha}_\infty$ and
$n_{\alpha+1}=\udeg^{\Tt_\alpha}(\infty)$.
We must define $\Yy_{\alpha+1}$ and $\srsigma_{\alpha+1}$ and verify the
inductive hypotheses.

Let $\Uu_\alpha=\srsigma_\alpha\Tt_\alpha$ be the $\srsigma_\alpha$-copy of
$\Tt_\alpha$ to a $\udash y_\alpha$-maximal tree
on $M^{\Yy_\alpha}_\infty$. Let $n'_{\alpha+1}=\udeg^{\Uu_\alpha}(\infty)$. Then
$n_{\alpha+1}\leq n'_{\alpha+1}$. Let
\[ \srsigma_{\alpha+1}^0:O_{\alpha+1}\to M^{\Uu_\alpha}_\infty \]
be the final copy map, so $\srsigma_{\alpha+1}^0$ is nice $\udash
n_{\alpha+1}$-lifting.

Now $\Uu_\alpha$ is via $\Upsilon^\Sigma_{\Yy_\alpha}$ and
$\lh(\Uu_\alpha)<\Omega$.
Using \ref{dfn:W(T,U)_etc}, we define
\[ \Yy_{\alpha+1}=\Ww^\Sigma_{\Yy_\alpha}(\Uu_\alpha), \]
\[
\srsigma_{\alpha+1}^1=\nrsigma^\Sigma_{\Yy_\alpha}(\Uu_\alpha):M^{\Uu_\alpha}
_\infty\to M^{\Yy_{\alpha+1}}_\infty.\]
So $\srsigma_{\alpha+1}^1$ is nice $\udash n'_{\alpha+1}$-lifting,
$\lh(\Yy_{\alpha+1})<\Omega$ and
$n'_{\alpha+1}\leq\udeg^{\Yy_{\alpha+1}}(\infty)$.

Composing, we define
$\srsigma_{\alpha+1}=\srsigma_{\alpha+1}^1\com\srsigma_{\alpha+1}^0$,
also nice $\udash n_{\alpha+1}$-lifting.

We have verified  properties
\ref{item:lim_stack_N_0}--\ref{item:lim_stack_tau_beta}
and \ref{item:eventually_terminally-non-dropping} (some are trivial by
induction).
It just remains to establish
\ref{item:standard_terminal_inflation}, \ref{item:Ttvec_model_drop}
and \ref{item:inflation_iteration_commutativity} for
$\beta=\alpha+1$.

Suppose first $\gamma=\alpha<\alpha+1=\beta$. Property
\ref{item:standard_terminal_inflation}
is directly by \S\ref{subsubsec:stacks_lh_2} (note $\lh(\Uu_\alpha)<\Omega$
and is a successor).
For property \ref{item:Ttvec_model_drop}, we have that $b^{\Tt_\alpha}$ drops
in
model
iff $b^{\Uu_\alpha}$ drops in model iff (by \S\ref{subsubsec:stacks_lh_2})
$\Yy_{\alpha+1}$ is $\Yy_\alpha$-terminally-model-dropping.
And if $\Yy_{\alpha+1}$ is $\Yy_\alpha$-terminally-non-model-dropping,
so $b^{\Tt_\alpha},b^{\Uu_\alpha}$ do not drop in model (but possibly in
degree), then again by \S\ref{subsubsec:stacks_lh_2}, we have
\[
\pi_{\alpha,\alpha+1}:M^{\Yy_\alpha}_\infty\to
M^{\Yy_{\alpha+1}}_\infty \]
(defined in \ref{item:inflation_iteration_commutativity}),
and $\pi_{\alpha,\alpha+1}=\srsigma^1_{\alpha+1}\com i^{\Uu_\alpha}$, so
\begin{equation}\label{eqn:inflation_lift_comm}
\pi_{\alpha,\alpha+1}\com\srsigma_\alpha=\srsigma_{\alpha+1}\com
i^{\Tt_\alpha},
\end{equation}
as required for property \ref{item:inflation_iteration_commutativity}.

Finally suppose that $\gamma<\alpha<\alpha+1=\beta$. Properties
\ref{item:standard_terminal_inflation} and \ref{item:Ttvec_model_drop} follow
easily by induction,
the facts established above regarding $\Yy_{\alpha+1}$, and
Lemma \ref{lem:inflation_commutativity} (commutativity of inflation).
For example for property \ref{item:Ttvec_model_drop}:
By \ref{lem:inflation_commutativity}, we have that $\Yy_{\alpha+1}$ is
$\Yy_\gamma$-terminally-model-dropping
iff either $\Yy_{\alpha+1}$ is $\Yy_\alpha$-terminally-model-dropping or
$\Yy_\alpha$ is $\Yy_\gamma$-terminally-model-dropping,
which by induction and the previous paragraph, suffices.
Consider property
\ref{item:inflation_iteration_commutativity}; suppose
$\Yy_{\alpha+1}$ is $\Yy_\gamma$-terminally-non-model-dropping.
So $b^{\Ttvec\rest[\gamma,\alpha+1)}$ does not drop in model, $\Yy_{\alpha+1}$
is
$\Yy_\alpha$-terminally-non-model-dropping and $\Yy_{\alpha}$ is
$\Yy_\gamma$-terminally-non-model-dropping,
and by Lemma \ref{lem:inflation_commutativity},
$\pi_{\gamma,\alpha+1}=\pi_{\alpha,\alpha+1}\com\pi_{\gamma\alpha}$.
Property \ref{item:inflation_iteration_commutativity} now follows by induction
and line (\ref{eqn:inflation_lift_comm}).

This verifies all the properties at the end of round $\alpha$.

Now let $\eta<\Omega$ be a limit ordinal, and suppose we have defined
\[
\left<\Tt_\beta,O_\beta,n_\beta,\Yy_\beta,\srsigma_\beta\right>_{\beta<\eta},\]
and maintained the inductive hypotheses
through all $\alpha<\eta$. We need to define $\Yy_\eta$ and $\srsigma_\eta$ and
see that the
inductive hypotheses hold at $\alpha=\eta$ (of course, $O_\eta$ are $n_\eta$
will be determined).

We set $\Yy_\eta$ to be the comparison inflation of
$\mathscr{T}=\{\Yy_\alpha\}_{\alpha<\eta}$
(see Definition \ref{dfn:min_inf}).
This exists and $\lh(\Yy_\eta)$ is a successor $\xi+1<\Omega$,
by Lemma \ref{lem:min_inf} and because $\eta<\Omega$ and each
$\lh(\Yy_\alpha)<\Omega$.
Also by Lemma \ref{lem:min_inf}, there is $\eps<\eta$ such that $\Yy_\eta$ is
$\Yy_\eps$-terminally-non-dropping;
let $\eps_0$ be the least such $\eps$.
By Lemma \ref{lem:inflation_commutativity} then,
$\Yy_\eta$ is $\Yy_\delta$-terminally-non-dropping for all
$\delta\in[\eps_0,\eta)$.
This gives property \ref{item:eventually_terminally-non-dropping}.
Now let $\eps$ be least such that $\Yy_\eta$ is
$\Yy_\eps$-terminally-non-\emph{model}-dropping.
Again by Lemma \ref{lem:inflation_commutativity}, $\Yy_\eta$ is
$\Yy_\delta$-terminally-non-model-dropping
for each $\delta\in[\eps,\eta)$, and $\Yy_{\delta_1}$ is
$\Yy_{\delta_0}$-terminally-non-model-dropping
for all $\delta_0,\delta_1$ such that $\eps\leq\delta_0<\delta_1<\eta$.
So by induction and property \ref{item:Ttvec_model_drop}, for all such
$\delta_i$,
$b^{\Tt_{\delta_0}}$ does not drop
in model and
$\pi_{\delta_0\delta_1}\com\srsigma_{\delta_0}=\srsigma_{\delta_1}\com
i^{\Ttvec\rest[\delta_0,\delta_1)}$.
Also, by Lemma \ref{lem:inflation_commutativity}, for all such $\delta_i$,
\[
\pi_{\delta_0\eta}=\pi_{\delta_1\eta}\com\pi_{\delta_0\delta_1}:M^{\Yy_{\delta_0
}}_\infty\to M^{\Yy_\eta}_\infty. \]

Therefore $O_\eta=M^{\Ttvec\rest\eta}_\infty$ and
$n_\eta=\udeg^{\Ttvec\rest\eta}(\infty)$
are well-defined, and we (can and do) define
\[ \srsigma_\eta:O_\eta\to M^{\Yy_\eta}_\infty \]
in the unique manner preserving commutativity, that is,
\[ \srsigma_\eta\com
i^{\Ttvec\rest[\delta,\eta)}=\pi_{\delta,\eta}\com\srsigma_\delta \]
for all $\delta\in[\eps,\eta)$. Then $\srsigma_\eta$ is a nice $\udash
n_\eta$-lifting
 embedding, and $O_\eta$ is wellfounded. It is now easy to verify properties
\ref{item:lim_stack_N_0}--\ref{item:inflation_iteration_commutativity}.

Finally suppose we have defined $\left<\Tt_\alpha\right>_{\alpha<\Omega}$.
Then because $\cof(\Omega)>\om$, we get that  for all sufficiently large
$\alpha<\Omega$,
$b^{\Tt_\alpha}$ does not drop,
and $M^\Ttvec_\infty$ is wellfounded, so player II has won.

This completes the proof of Theorem \ref{thm:stacks_iterability}.
\hfill\qedsymbol

\begin{dfn}
\label{dfn:W(Tvec)_etc}\index{$\srsigma^\Sigma$}
\index{$\Ww^\Sigma$}\index{$\Sigma^\stk$}
Given $\Ttvec=\left<\Tt_\beta\right>_{\beta<\alpha}$, etc, satisfying
\ref{item:lim_stack_N_0}--\ref{item:inflation_iteration_commutativity}, with
$\lh(\Ttvec)=\alpha<\Omega$,
we define
$\Ww^\Sigma(\Ttvec)=\Yy_\alpha$
and
$\srsigma^\Sigma(\Ttvec)=\srsigma_\alpha:M^\Ttvec_\infty=O_\alpha\to
M^{\Yy_\alpha}_\infty$, with notation as above.
(We don't try to define these things if $\lh(\Ttvec)=\Omega$; there seems to be
no
clear manner in which to define
$\Yy_\Omega$, because $\Sigma$ is not sufficiently powerful.)
Given also a tree $\Tt$ of length $\leq\Omega$, according to the strategy
$\bar{\Upsilon}$ for
round $\alpha$ defined above, we define $\Ww^\Sigma(\Ttvec\conc\Tt)$
to be the corresponding $\udash m$-maximal tree $\Yy$ on $M$
(which, in the construction of $\bar{\Upsilon}$,
was denoted $\Xx$ or $\Xx^\Omega$).\end{dfn}
\begin{dfn}\label{dfn:Sigma^stk}\index{$\Sigma^\stk$}
We write $\Sigma^{\stk}$ for the stacks strategy $\Sigma^*$ induced by
$\Sigma$,
defined above.
\end{dfn}

\subsubsection{Length $\om$ stacks of finite
trees}\label{subsubsec:stacks_of_finite_trees}

\begin{proof}[Sketch of Proof of Theorem \ref{thm:stacks_of_finite_trees}]
We just consider the fine version. As in Lemma
\ref{lem:sub-optimal_reduces_to_optimal} we can naturally derive
the full strategies from strategies for optimal stacks.
So  we can restrict our attention to
 optimal stacks.

Let $M$ be $\udash m$-sound and $\Sigma$ be an $(\udash
m,\Omega+1)$-strategy for $M$. Then player $\Two$
wins $\Gg_{\fin,\mathrm{opt}}(M,\udash m,\Omega+1)$\footnote{This
game is just like $\Gg_\fin$ but player I may not make artificial drops.} by
using the strategy
defined for player $\Two$ in
the iteration game
for stacks of length ${<\om}$ in the previous proof.
Because the normal trees in the stack are finite, there are no branches (of the
first tree $\Tt$ in a stack of length $2$) to consider, so
no condensation of $\Sigma$ is required to keep the process going. And given a
stack
$\Ttvec=\left<\Tt_n\right>_{n<\om}$, consisting of finite trees, the desired
conclusions regarding
$\Ttvec$ also follow from the limit case of the previous proof, again because
we
are only inflating
finite trees $\Tt_n$. (In a stack $(\Tt,\Uu)$ of length $2$, $\Uu$ could have
arbitrary length $\leq\Omega+1$,
and also the comparison inflation $\Tt_\om$ of the stack
$\left<\Tt_n\right>_{n<\om}$
could seemingly have arbitrary length ${<\Omega}$, but this is no problem.)

Now suppose that $M$ is MS-indexed, $m$-sound, and $\Sigma$ is an
$(m,\Omega+1)$-strategy for $M$.
If $M$ is type 3 and $m<\om$ let $m'=m+1$; otherwise let $m'=m$.
Given a finite $m$-maximal tree $\Tt$ on $M$, let $\Tt'$ be the corresponding
$\udash m'$-maximal tree on $M$ (with $\infty$
non-$\Tt'$-special),
so $(M^{\Tt'}_\infty)^\pm=M^\Tt_\infty$. Thus, if $M^\Tt_\infty$
is
type 3 and $\deg^\Tt(\infty)<\om$
then $\udeg^{\Tt'}(\infty)=\deg^\Tt(\infty)+1$; otherwise
$\udeg^{\Tt'}(\infty)=\deg^\Tt(\infty)$.
Moreover, $b^{\Tt'}$ drops iff $b^\Tt$ drops, and if non-dropping then
\[ i^{\Tt}=i^{\Tt'}\rest M^\sq.\]
This generalizes to $m$-maximal finite stacks $\Ttvec$ on $M$ consisting of
finite trees $\Tt_n$,
giving a $\udash m'$-maximal finite stack $\Ttvec'$ with analogous
correspondence.

So given such a finite stack $\Ttvec$,
the $\udash$iteration strategy for $M^{\Ttvec'}_\infty$ given above
induces a standard iteration strategy for $M^\Ttvec_\infty$; so player II wins
$\Gg_{\fin,\mathrm{opt}}(M,m,\Omega+1)$.
Now let $\Ttvec$ have length $\om$, consisting of finite trees, and $\Ttvec'$
be its translation to a
$\udash m'$-maximal stack.
Let $O_0=M$ and $O_n=M^{\Tt_{n-1}}_\infty$ for $n>0$; likewise for $O'_n$.
Since for all large $n$, $b^{\Tt'_n}$ does not drop, neither does $b^{\Tt_n}$,
and because $i^{\Tt_n}=i^{\Tt'_n}\rest(O_n^\sq)$,
we get that $M^{\Ttvec}_\infty$ is wellfounded. (If $O_n$ is non-type 3 for
large $n$, this is trivial as $M^{\Ttvec'}_\infty=M^{\Ttvec}_\infty$.
Otherwise, because $(O_n')^\pm=O_n$
and the iteration maps correspond, we have
 $(M^{\Ttvec'}_\infty)^\pm=M^{\vec{\Tt}}_\infty$,
 which is
wellfounded.)
\end{proof}

\begin{rem}
We are not sure whether one might extend the preceding theorem to stacks
$\left<\Tt_\alpha\right>_{\alpha<\lambda}$
of finite trees $\Tt_\alpha$ of arbitrary transfinite length $\lambda$.
The method used so far runs into difficulties when $\lambda=\om+1$,
because $\Yy_\om$ can be infinite, so that, at least superficially, one seems
to
need inflation condensation
in order to continue. However, the stack $\left<\Tt_n\right>_{n<\om}$ is only a
linear stack of finite iterations,
so the possible branch choices might be much more limited. Of course in some
situations one can
just use an absoluteness argument to show that $M^\Ttvec_\infty$ is wellfounded
for any $\lambda$
(when each $\Tt_\alpha$ is finite). However particularly when $M$ is active,
this is not so easy.
\end{rem}

\subsubsection{Variants for partial
strategies}\label{subsubsec:variants_partial}\index{partial strategy}

 We now state a version of Theorem \ref{thm:stacks_iterability} for partial
strategies.
 Typical examples would be a normal strategy $\Sigma$ for $M$
for nice iteration trees
 (that is, in which all extenders $E=E^\Tt_\alpha$ are total,
 with $\nu_E=\strength(E)$ is inaccessible in $M^\Tt_\alpha$),
 or for trees which are based on $M|\delta$, where $\delta$ is some
$M$-cardinal. In this section we restrict our attention to optimal stacks, but
this is only for simplicity, and one can of course consider stacks with
artificial drops.

\begin{dfn}\label{dfn:partial_Sigma^*}
\index{$D$-$(\udash m,\theta)$-strategy}
\index{good}
\index{$(\Sigma,D)$-good}
\index{weakly $(\Sigma,D)$-good}
\index{optimal stack}
Let $M,m,\mathscr{T}$ be as in
Definition \ref{dfn:strategy_classes}\ref{item:pre-inf_wcpm} or
\ref{dfn:strategy_classes}\ref{item:pre-inf_rm-sound}
(so in particular, $\mathscr{T}$ is a class of putative trees).
Let $D\sub\mathscr{T}$ be
closed under initial segment.
 Let $\Omega>\om$ be regular.
We say $\Sigma$ is a \dfnemph{$D$-$(\udash
m,\Omega+1)$-strategy}
iff $\Sigma$ is a function
such that $\dom(\Sigma)$ is exactly the set of
 trees in $D$ of limit length
$\leq\Omega$ which are via $\Sigma$,
for all $\Tt\in\dom(\Sigma)$, we have
$\Tt\conc\Sigma(\Tt)\in D$, and all putative trees in $D$
via $\Sigma$ have
wellfounded well-defined models.
We define \dfnemph{$D^*$-optimal-$(\udash
m,\Omega,\Omega+1)^*$-strategy} analogously
(for some class $D^*$ of putative optimal $\udash m$-maximal stacks on $M$).
Likewise for $M,m,\mathscr{T}$ as in
\ref{dfn:strategy_classes}\ref{item:pre-inf_m-sound_MS}.

 Let $\Sigma$ be a conveniently inflationary partial
$\mathscr{T}$-strategy for $M$.
Suppose $\Sigma$ is a $D$-$(\udash m,\Omega+1)$-strategy.
 Define $D'$ and a $D'$-optimal-$(\udash m,\Omega,\Omega+1)^*$-strategy
$\Sigma^*$ for $M$ inductively as
follows.
 Let $\Ttvec=\left<\Tt_\alpha\right>_{\alpha<\lambda}$ be an optimal stack on
$M$, where $\lambda<\Omega$.
Say $\Ttvec$ is \dfnemph{weakly $(\Sigma,D)$-good} iff there is
 a tree $\Yy=\Ww^\Sigma(\Ttvec)$, defined as in
 Definition \ref{dfn:W(Tvec)_etc} (so as in the
proof
of Theorem \ref{thm:stacks_iterability}),
with $\Yy\rest(\Omega+1)$ in $D$ and via $\Sigma$
(recall that if $\lambda=\alpha+1$ and $\lh(\Tt_\alpha)=\Omega+1$,
we can have $\lh(\Yy)>\Omega+1$).
 We say that $\Ttvec$ is \dfnemph{$(\Sigma,D)$-good} iff $\Ttvec$ is weakly
$(\Sigma,D)$-good
 and $\Ttvec\rest\eta\conc\left<\Tt_\eta\rest\beta\right>$
 is weakly $(\Sigma,D)$-good for every $\eta<\lambda$ and
$\beta\leq\lh(\Tt_\eta)$.
 Note that if $\Ttvec$ is $(\Sigma,D)$-good then $\Yy=\Ww^\Sigma(\Ttvec)$
 is uniquely determined just as before,
 and $M^{\Ttvec}_\infty$ exists and is wellfounded
 and  embedded into $M^\Yy_\infty$.

Now define $\Sigma^*$ as the partial strategy such that
$\Ttvec\conc\Uu\in\dom(\Sigma^*)$
 (where $\Ttvec$ is a stack of normal trees and $\Uu$ is normal)
 iff there is $b$ such that $\Ttvec\conc(\Uu\conc b)$ is $(\Sigma,D)$-good;
 in this case set $\Sigma^*(\Ttvec\conc\Uu)=b$. (And define
$D'=\dom(\Sigma^*)$.)

Now suppose instead that $M,m,\mathscr{T}$ are as in
Definition \ref{dfn:strategy_classes}\ref{item:pre-inf_m-sound_MS},
and let $C\sub\mathscr{T}$ be closed under initial segment. Let
$\Gamma$ be an inconveniently inflationary partial $\mathscr{T}$-strategy,
and suppose $\Gamma$ is a $C$-$(m,\Omega+1)$-strategy.
Let $\Sigma$ be the partial $\udash m'$-maximal strategy
corresponding to $\Gamma$ (as in Remark
\ref{rem:partial_strategy_conversion}\footnote{Because $\Gamma$ is a
partial
$\mathscr{T}$-strategy,
all trees via $\Gamma$ are $M$-$\udash$wellfounded
by Footnote \ref{ftn:Sigma_partial_pre-inf}, so
 Remark \ref{rem:partial_strategy_conversion} applies.}), and $D$ the class of
initial segments unravellings of
 trees via $\Sigma$ (note these unravellings exist), so $\Sigma$ is a
$D$-$(\udash m,\Omega+1)$-strategy.
Note that all successor length trees via $\Sigma$
can be arbitrarily finitely extended (with $\udash m'$-maximal
extensions), because trees via $\Gamma$ can be
finitely extended
and by Lemma \ref{lem:tree_conversion}.
Let $\Sigma^*$ be as above.
Then we extend $\Gamma$ to the partial optimal stacks strategy $\Gamma^*$,
which is just the partial $\udash$strategy determined by
$\Sigma^*$.\footnote{All stacks
$\vec{\Tt}=\left<\Tt_\alpha\right>_{\alpha<\lambda}$
via $\Sigma^*$ such that $\Tt_\alpha$ is unravelled for all $\alpha+1<\lambda$,
are everywhere unravelable, since all successor length trees via $\Sigma$ are
arbitrarily
finitely extendible.}
\end{dfn}

Now for example we have:

\begin{tm}\label{thm:stacks_iterability_partial}
Let $\Omega>\om$ be regular.
 Let $\Sigma$ be an inflationary\footnote{Recall this allows both
 conveniently and inconveniently inflationary.} partial
strategy for $M$.
Suppose either
 \begin{enumerate}[label=\tu{(}\roman*\tu{)}]
  \item\label{item:nice_trees} $M\sats\ZFC$ and $D$ is the the class of
normal
nice putative trees on $M$, or
  \item\label{item:based_on_M|delta} $\Sigma$ is convenient \tu{(}resp.,
inconvenient\tu{)} and there is some $M$-cardinal $\delta<\rho_0^M$ such that
$D$ is the
class of
$\udash m$-maximal \tu{(}resp., $m$-maximal\tu{)}
putative  trees on $M$ which are based on $M|\delta$.
  \end{enumerate}
Suppose $\Sigma$ is a $D$-$(\udash
m,\Omega+1)$-strategy \tu{(}resp., $D$-$(m,\Omega+1)$\tu{)}.

Then $\Sigma^*$ is a $D^*$-optimal-$(\udash m,\Omega,\Omega+1)^*$-strategy
\tu{(}resp., $(m,\Omega,\Omega+1)^*$\tu{)} for
$M$,
where either:
\begin{enumerate}[label=\tu{(}\roman*\tu{)}']
 \item $D^*$ is the class of optimal putative
 stacks of normal nice (putative) trees on
$M$, or
 \item $D^*$ is the class of optimal $\udash m$-maximal \tu{(}resp.,
$m$-maximal\tu{)}
putative stacks on $M$ which are based
on
$M|\delta$,
\end{enumerate}
according to whether \ref{item:nice_trees} or \ref{item:based_on_M|delta} above
holds.
\end{tm}
\begin{proof}
 This is a corollary to the proof of Theorem \ref{thm:stacks_iterability}.
 One simply notes that $\Yy=\Ww^\Sigma(\Ttvec)$ is in $D$ (for the relevant
$\Ttvec$).
 In the nice tree case, this is because all extenders used are copied from some
extender
 used in some $\Ttvec$, which is therefore nice in the model it is taken from
(and note that all trees are nowhere dropping in this case).
 In the other case, it uses such copying, and also the commutativity and
correspondence of drops
 described in the proof.
\end{proof}

\section{Properties of $\Sigma^\stk$ and (weak) Dodd-Jensen}\label{sec:npc}

In this final section we show that if $\Sigma$ has certain extra properties,
then the stacks strategy $\Sigma^\stk$ inherits certain extra properties itself.
We then give a couple of applications of the theorems to absoluteness of
iterability
and constructing normal strategies with weak Dodd-Jensen
without $\DC$.

\subsection{Normal pullback consistency for $\Sigma^\stk$}

We record some notation for some standard notions:
\begin{dfn}\label{dfn:npc}\index{$\Sigma^{\mathrm{nm}}$}
Let $\Omega>\om$ be regular.
Let $\Gamma$ be a $(\udash m,\Omega,\Omega+1)^*$-strategy,
with first round $\Gamma^{\mathrm{nm}}$ ($\mathrm{nm}$
for \emph{normal}), so $\Gamma^{\mathrm{nm}}$ is a $(\udash
m,\Omega+1)$-strategy).
Let $\Ttvec$ be a stack via $\Gamma$ of length ${<\Omega}$, with each component
normal tree of length ${<\Omega}$.
Let $N=M^\Ttvec_\infty$ and $n=\udeg^\Ttvec(\infty)$.
Then $\Gamma^{\mathrm{nm}}_\Ttvec$ denotes the induced $(\udash
n,\Omega+1)$-strategy
for $N$; that is,
$\Gamma^{\mathrm{nm}}_\Ttvec(\Uu)=\Gamma(\Ttvec\conc\Uu)$
for $\udash n$-maximal  trees  $\Uu$ of length $\leq\Omega$.
If $b^\Ttvec$ does not drop in model or degree then
$\Gamma^{\mathrm{nm}}_{\leftarrow\Ttvec}$
denotes the $i^\Ttvec$-pullback of $\Gamma^{\mathrm{nm}}_\Ttvec$,
a $(\udash m,\Omega+1)$-strategy for $M$.
We say that $\Gamma$ is \dfnemph{normally pullback consistent}\index{normal
pullback consistency}\index{pullback consistency}
iff for all such $\Ttvec$, if $b^\Ttvec$ does not drop in model or degree
then
$\Gamma^{\mathrm{nm}}_{\leftarrow\Ttvec}=\Gamma^{\mathrm{nm}}$.
\end{dfn}

\begin{rem}
Given sufficient condensation properties of $\Sigma$,
the author expects that one should be able to deduce good condensation
properties of $\Sigma^\stk$,
such as pullback consistency (not just normal pullback consistency).
In the proof to follow, of the fact that $\Sigma^\stk$ is normally pullback
consistent,
assuming sufficient condensation for $\Sigma$,
we consider a normal tree $\Tt$ via $\Sigma$, such that $b^\Tt$ does not drop
in
model or degree,
and letting $N=M^\Tt_\infty$, we lift a normal tree $\bar{\Uu}$ on $M$ to
$\Uu=i^\Tt\bar{\Uu}$ on $N$,
with $\Uu$ according to $(\Sigma^\stk)^{\mathrm{nm}}_\Tt$. Na\"ively, one would
like to exhibit a tree embedding $\Pi$
from $\bar{\Uu}$ into $\Xx=\Ww^\Sigma_\Tt(\Uu)$. The natural na\"ive candidate
for $\Pi$ would be that with $\gamma_\alpha=\lambda^\alpha$
and $\delta_\alpha=\zeta^\alpha$, where
$\left<\lambda^\alpha,\zeta^\alpha\right>$ arise
from the inflation $\Tt\inflatearrow\Xx$.
It is easy to see that this $\Pi$ can fail to be a bounding tree embedding,
so inflation condensation does not seem to suffice.
In fact, it can fail to be a tree embedding at all, because the requirement
that
$\gamma_\alpha\leq^\Xx\delta_\alpha$
can fail. But this can only fail in a special manner, and by slightly
generalizing the definition
of \emph{tree embedding}, and demanding condensation of $\Sigma$ with respect
to
this more general notion,
our proof goes through. We now describe the generalization.
In the end it is actually more convenient to generalize the demands of
\emph{normality} for the larger tree $\Xx$,
and retain the demand that $\gamma_\alpha\leq^\Xx\delta_\alpha$, so this is how
we proceed.
\end{rem}

\begin{dfn}\index{essential}
Let $\Xx$ be an iteration tree on a seg-pm $M$.
We say that $\Xx$ is \dfnemph{essentially $\udash m$-maximal}
iff $\Xx$ satisfies the requirements of $\udash m$-maximality except that we
replace the requirement
\[ \In(E^\Xx_\alpha)\leq\In(E^\Xx_\beta)\text{ for all
}\alpha+1<\beta+1<\lh(\Xx) \]
with the requirement that for all $\alpha+1<\beta+1<\lh(\Xx)$, either
\begin{enumerate}[label=--]
 \item  $\In(E^\Xx_\alpha)\leq\In(E^\Xx_\beta)$, or
 \item $E^\Xx_\alpha$ is of superstrong type and
$\lambda(E^\Xx_\alpha)<\In(E^\Xx_\beta)$.
\end{enumerate}
Recall from \S\ref{subsec:terminology} that if $M$ is $\lambda$-indexed
then every extender in $\es_+(M)$ has superstrong type,
so in this case, we just require in general that
$\lambda(E^\Xx_\alpha)<\In(E^\Xx_\beta)$.
\end{dfn}

\begin{rem}
 It is easy to see that a $\udash m$-maximal strategy $\Sigma$
 yields a corresponding essentially $\udash m$-maximal strategy
\index{$\Sigma_{\mathrm{ess}}$} $\Sigma_{\mathrm{ess}}$;
 trees $\Xx$ via $\Sigma_{\mathrm{ess}}$
are those for which there is  $\Xx'$ via $\Sigma$
which uses exactly
 those extenders $E$ such that $E=E^\Xx_\alpha$
 for some $\alpha$ such that $\In(E^\Xx_\alpha)\leq\In(E^\Xx_\beta)$
 for all $\beta>\alpha$, and which has corresponding branches.
\end{rem}

\begin{dfn}\label{dfn:plus-strong_hc}\index{essential}\index{
$\Pi:\Tt\hookrightarrow_{\mathrm{ess}}\Xx$}
 Let $\Tt$ be $\udash m$-maximal and $\Xx$ be essentially $\udash
m$-maximal.\footnote{Note that
 while $\Xx$ is only \emph{essentially} $\udash m$-maximal, we still demand
that
$\Tt$
 be (fully) $\udash m$-maximal.}
 An \dfnemph{essential tree embedding}
$\Pi:\Tt\hookrightarrow_{\mathrm{ess}}\Xx$
 is a system $\Pi=\left<I_\alpha\right>_{\alpha<\lh(\Tt)}$
 satisfying the requirements of a tree embedding, and with corresponding
notation,
 such that whenever $\xi<\eta$ but $\In(E^\Xx_\xi)>\In(E^\Xx_\eta)$,
 then:
 \begin{enumerate}[label=--]
  \item $\eta=\xi+1$, and
  \item there is $\alpha+1<\lh(\Tt)$ such
that $\gamma_\alpha<\delta_\alpha=\xi+1$
 (so $E^\Xx_{\xi+1}=E^\Xx_{\delta_\alpha}$
 is copied from $\Tt$).
 \end{enumerate}

 We say that a  $\udash m$-maximal iteration strategy $\Sigma$ has
\dfnemph{plus-strong hull condensation}\index{plus-strong hull condensation}
 iff whenever $\Xx$ is via $\Sigma_{\mathrm{ess}}$ and
$\Pi:\Tt\hookrightarrow_{\mathrm{ess}}\Xx$,
 then $\Tt$ is via $\Sigma$.
\end{dfn}

\begin{rem}
 We pause to give a simple example of an essential tree embedding which is not
a
tree embedding,
 and which gives a fairly general illustration of how these arise in the proof.

 Let $M\sats\ZFC$ be a mouse. Let $E\in\es^M$ and $\mu,\kappa$ be such that
 \[ \crit(E)<\mu<\kappa<\nu_E \]
 and $\nu_E$ is an $M$-cardinal, and letting $U=\Ult(M,E)$, such that there is
$F\in\es^U$
 with $\In(E)<\In(F)$ and $\crit(F)=\kappa$ and $F$ has superstrong type.
 Suppose also that $\mu$ is $M$-measurable and there is $G\in\es^M$ with
$\kappa<\In(G)<(\kappa^+)^M$.

 Let $\Tt$ be the normal tree using $E^\Tt_0=E$ and $E^\Tt_1=F$,
 so $0=\pred^\Tt(2)$ and $N\eqdef M^\Tt_2=\Ult(M,F)$.

 Now let $D\in\es^M$ be a normal measure with $\crit(D)=\mu$, so
$\In(D)<\kappa$.
 Let $\bar{\Uu}$ be the tree on $M$ using $E^{\bar{\Uu}}_0=D$ and
$E^{\bar{\Uu}}_1=i^M_D(G)$.

 Let $\Uu=i^\Tt\bar{\Uu}$. So $\Uu$ is the tree on $N$ using $E^\Uu_0=D$ (as
$\crit(i^\Tt)=\kappa>\mu$)
 and $E^\Uu_1=i^N_D(i^\Tt(G))$. Write $N_\alpha=M^\Uu_\alpha$.

 Now let $\Xx=\Ww_\Tt(\Uu)$. Write $\left<\lambda^\alpha,\zeta^\alpha\right>$
 for those ordinals arising from the inflation $\Tt\inflatearrow\Xx$.
 We have $\Xx^0=\Tt$, with $N=M^{\Xx^0}_\infty$, and $\tau_0:N\to
M^{\Xx^0}_\infty$ is $\tau_0=\id$.
 Since $\In(D)<\In(E^\Tt_0)$,
 we have $\lambda^0=\zeta^0=0$ and $E^\Xx_{\zeta^0}=D$.
 So $\lambda^1=1$.
 Then $\Xx^1$ is the tree with $E^{\Xx^1}_0=D$, followed by copying $\Tt=\Xx^0$.
 So $E^{\Xx^1}_1=i^M_D(E)$,
 and since $\crit(E)<\mu=\crit(D)$, we have $0=\pred^{\Xx^1}(2)$ (so note that
$1=\lambda^1\not\leq^{\Xx^1} 2$)
 and $M^{\Xx^1}_2=\Ult(M,E^{\Xx^1}_1)$, and letting
 \[ \psi_1:U=M^{\Xx^0}_1\to M^{\Xx^1}_2 \]
 be the copy map, then $E^{\Xx^1}_2=\psi_1(E^{\Xx^0}_1)=\psi_1(F)$.
 Then since $\kappa=\crit(F)$ and $\In(D)<\kappa\leq\psi_1(\kappa)$ (actually,
in this particular example, $\psi_1(\kappa)=\kappa$),
 and $\psi_1(\kappa)<\nu(E^{\Xx^1}_1)$, therefore
$\pred^{\Xx^1}(3)=1=\lambda^1$.
 So $M^{\Xx^1}_3=\Ult(M^{\Xx^1}_1,\psi_1(F))$, and $\lh(\Xx^1)=3+1$, so this
completes $\Xx^1$.

 We have $\psi_{01}:N=M^{\Xx^0}_2\to M^{\Xx^1}_3$ is the final copy map,
 and $\nrsigma_1:N_1\to M^{\Xx^1}_3$ is as defined in the construction of
$\Sigma^\stk$,
 and $\tau_1\com i^\Uu_{01}=\psi_{01}$.
 Note that
 \[
E^\Xx_{\zeta^1}=\nrsigma_1(E^\Uu_1)=\nrsigma_1(i^\Uu_{01}(i^\Tt(G)))=\psi_{01}
(i^\Tt(G))=i^{\Xx^1}(G)\]
and so
 \[ \lambda(E^{\Xx^1}_2)<\In(E^\Xx_{\zeta^1})<\In(E^{\Xx^1}_2).\]
 So $\zeta^1=2$, so $\lambda^1\not\leq^{\Xx^1}\zeta^1$. So if we set
$\gamma_1=\lambda^1$
 and $\delta_1=\zeta^1$, we wouldn't have a tree embedding
$\bar{\Uu}\hookrightarrow\Xx$.
 However, by replacing $\Xx$ with the essentially normal tree $\wt{\Xx}$
 where $E^{\wt{\Xx}}_2=E^{\Xx^1}_2$ and then $E^{\wt{\Xx}}_3=E^\Xx_{\zeta^1}$,
 we do get an essential tree embedding
$\bar{\Uu}\hookrightarrow_{\mathrm{ess}}\wt{\Xx}$.

 In the  proof below we will actually index the tree $\wt{\Xx}$ differently to
this, however.
 In the situation above we would include two indices $(\zeta^1,0)$ and
$\zeta^1$, with $(\zeta^1,0)<\zeta^1$,
 and set $E^{\wt{\Xx}}_{(\zeta^1,0)}=E^{\Xx^1}_2$ and
$E^{\wt{\Xx}}_{\zeta^1}=E^{\Xx}_{\zeta^1}$.
\end{rem}

\begin{tm}\label{thm:Sigma_shc_implies_stacks_npc}
Let $\Sigma,\Omega$ be as in Theorem \ref{thm:stacks_iterability} \tu{(}so
$\Sigma$ is
regularly $(\Omega+1)$-total\tu{)}.
Suppose that $\Sigma$ has plus-strong hull condensation \tu{(}see
\ref{dfn:plus-strong_hc}\tu{)}.
Then $\Sigma^\stk$ is normally pullback consistent.
\end{tm}

\begin{proof}
Note that  in the definition of normal pullback consistency,
we assume that $b^{\Ttvec}$ does not drop in model or degree,
and in particular, $\Ttvec$ is optimal. Therefore we only
need consider optimal stacks in the present proof, and that
aspect of the construction of $\Sigma^{\stk}$.
It also easily suffices to consider the case that $\Sigma$ is a convenient
strategy.

Let $\Ttvec=\left<\Tt_\alpha\right>$ be a stack via $\Sigma^\stk$, such that
$b^\Ttvec$ exists
and does not drop in model or degree, $\lh(\Ttvec)<\Omega$ and
$\lh(\Tt_\alpha)<\Omega$ for each $\alpha$.
Let $\Xx=\Ww^\Sigma(\Ttvec)$.
Then $b^\Xx$ exists and does not drop in model or degree, $\lh(\Xx)<\Omega$,
and $i^\Xx=\srsigma\com i^{\Ttvec}$
where $\srsigma=\srsigma^\Sigma(\Ttvec)$;
recall $\srsigma:M^\Ttvec_\infty\to
M^\Xx_\infty$
(Definition \ref{dfn:W(Tvec)_etc}).
Now $\Sigma^{\mathrm{nm}}_\Ttvec$
is the $\srsigma$-pullback of $\Upsilon^\Sigma_\Xx$
(Definition \ref{dfn:W(T,U)_etc}).
Since $i^\Xx=\srsigma\com i^{\Ttvec}$, we may assume that
$\Ttvec$ consists of a single normal tree $\Tt$ (so then $\Xx=\Tt$).

So let $\Tt$, via $\Sigma$, have length ${<\Omega}$, and such that $b^\Tt$ does
not drop in model or degree.
Let $\bar{\Uu}$ be a $\udash m$-maximal tree on $M$, via
$\Sigma^{\mathrm{nm}}_{\leftarrow\Tt}$.
Let $\Uu=i^\Tt\bar{\Uu}$, so $\Tt\conc\Uu$ is via $\Sigma^\stk$.
Let $\Xx=\Ww_\Tt(\Uu)$. So
${<^{\bar{\Uu}}}$, ${<^\Uu}$ and ${<^{\Xx/\Tt}}$ are identical.

We will define an essentially $\udash m$-maximal tree $\widetilde{\Xx}$
whose corresponding $\udash m$-maximal tree is $\Xx$
(so $\widetilde{\Xx}$ is via $\Sigma_{\mathrm{ess}}$), and
exhibit an an essential tree embedding
$\Pi:\bar{\Uu}\hookrightarrow_{\mathrm{ess}}\widetilde{\Xx}$;
therefore (by plus-strong hull condensation)
$\bar{\Uu}$ is via $\Sigma$,
giving normal pullback consistency for this case.

Let $\iota=\lh(\bar{\Uu})$ and
$\left<\lambda^\eta,\zeta^\eta\right>_{\eta<\iota}$
be determined by the inflation $\Tt\inflatearrow\Xx$.

For convenience, we index $\widetilde{\Xx}$ with a set $D$ such that
\[ \lh(\Xx)\sub D\sub\lh(\Xx)\cup(\lh(\Xx)\cross\{0\}).\]
For $(\zeta,0)\in D$, we set $(\zeta,0)<\zeta$,
and $(\beta,0)<\beta<(\zeta,0)$ for all $\beta<\zeta$.
We will have $E^{\widetilde{\Xx}}_\zeta=E^\Xx_\zeta$
for every $\zeta+1<\lh(\Xx)$. The consecutive pairs $x<x'\in D$
such that $\In(E^{\widetilde{\Xx}}_x)>\In(E^{\widetilde{\Xx}}_{x'})$
will be exactly those of the form $x=(\zeta,0),x'=\zeta$ where $(\zeta,0)\in D$.

For each $\zeta<\lh(\Xx)$,
we put $(\zeta,0)\in D$  iff there is $\alpha+1<\lh(\bar{\Uu})$
such that
\[ \zeta=\zeta^\alpha\text{ and }\zeta^\alpha+1<\lh(\Xx^\alpha)
\text{ and }
\lambda(E^{\Xx^\alpha}_{\zeta^\alpha})<\In(E^\Xx_{\zeta^\alpha}).\]
Of course, whenever $\zeta^\beta+1<\lh(\Xx^\beta)$,
we have $\In(E^\Xx_{\zeta^\beta})<\In(E^{\Xx^\beta}_{\zeta^\beta})$.
So if $(\zeta^\alpha,0)\in D$ then $E^{\Xx^\alpha}_{\zeta^\alpha}$ has
superstrong type
(and note that this extender is copied from $\Tt$,
via the inflation $\Tt\inflatearrow\Xx$).
If $(\zeta^\alpha,0)\in D$ then define $\zeta_*^\alpha=\zeta^\alpha+1$;
otherwise define $\zeta_*^\alpha=\zeta^\alpha$.

If $(\zeta^\alpha,0)\in D$ then we set
$E^{\widetilde{\Xx}}_{(\zeta^\alpha,0)}=E^{\Xx^\alpha}_{\zeta^\alpha}$
and $E^{\widetilde{\Xx}}_{\zeta^\alpha}=E^\Xx_{\zeta^\alpha}$, so
$E^{\widetilde{\Xx}}_{(\zeta^\alpha,0)}$ has superstrong type and
\[
\lambda(E^{\widetilde{\Xx}}_{(\zeta^\alpha,0)})<\In(E^{\widetilde{\Xx}}_{
\zeta^\alpha})<\In(E^{\widetilde{\Xx}}_{(\zeta^\alpha,0)});\]
we will also have here that
$\lambda^\alpha<^{\Xx^\alpha}\zeta^\alpha+1\in b^{\Xx^\alpha}$
and $\lambda^\alpha<^{\widetilde{\Xx}}\zeta^\alpha$.

With this notation, the desired essential tree embedding
$\Pi:\bar{\Uu}\hookrightarrow_{\mathrm{ess}}\wt{\Xx}$ is that with
\[
I_\alpha=[\gamma_\alpha,\delta_\alpha]_{\widetilde{\Xx}}=[\lambda^\alpha,
\zeta^\alpha]_{\widetilde{\Xx}}.\]
We will verify that this does indeed work. We write
$\pi_\alpha:M_\alpha\to M^{\widetilde{\Xx}}_{\lambda^\alpha}$
for the associated embedding.

Adopt the notation from the construction of $\Xx=\Ww^\Sigma_\Tt(\Uu)$
(\S\ref{subsubsec:stacks_lh_2}).
So $N_\eta=M^\Uu_\eta$, etc. Write $M_\eta=M^{\bar{\Uu}}_\eta$ and
$\bar{E}_\eta=E^{\bar{\Uu}}_\eta$
and $E_\eta=E^\Uu_\eta$ and $F_\eta=E^{\wt{\Xx}}_{\zeta^\eta}$.
Let $\varphi_\eta:M_\eta\to N_\eta$
be the copy map. We have $\nrsigma_\eta:N_\eta\to M^{\Xx^\eta}_\infty$. So
\[
F_\eta=E^{\wt{\Xx}}_{\zeta^\eta}=E^\Xx_{\zeta^\eta}
=\nrsigma_\eta(E_\eta)=\nrsigma_\eta(\varphi_\eta(\bar{E}_\eta)).\]
Note that $\Uu$ and $\bar{\Uu}$ have matching drop and
degree structure
(using
that copying propagates near embeddings, by the argument of \cite{fs_tame}),
and when  $[0,\eta]_{\bar{\Uu}}$ drops in model or degree, we have
$\lambda^\eta+1=\zeta^\eta+1=\lh(\Xx^\eta)$
and $(\zeta^\eta,0)\notin D$ and
$i^{\Xx^\eta}_{\lambda^\eta\zeta^\eta}=\id=i^{\wt{\Xx}}_{\lambda^\eta\zeta^\eta}
$,
and so in this case various things stated below simplify or trivialize.

We will prove the following facts, by induction on $\eta<\lh(\bar{\Uu})$:
\begin{enumerate}[label=\arabic*.,ref=\arabic*]
\item\label{item:partial_ess_tree_emb}
$\left<I_\alpha\right>_{\alpha<\eta}\conc\left<[\lambda^\eta,\lambda^\eta]
\right>$
is an essential tree embedding
$\bar{\Uu}\rest(\eta+1)\hookrightarrow\widetilde{\Xx}$.
\item\label{item:zeta_in_branch} If $\alpha<\eta$ then
$\lambda^\alpha\leq^{\widetilde{\Xx}}\zeta^\alpha$
and
$(\lambda^\alpha,\zeta^\alpha]_{\widetilde{\Xx}}\inter\dropset_{\deg}^{
\widetilde{\Xx}}=\emptyset$, so
\[
F_\alpha=i^{\widetilde{\Xx}}_{\lambda^\alpha\zeta^\alpha}(\pi_\alpha(\bar{E}
_\alpha)),\]
by condition \ref{item:partial_ess_tree_emb}.
Moreover:
\begin{enumerate}
\item\label{item:(zeta,0)_notin_D} If $(\zeta^\alpha,0)\notin D$ then
$\lambda^\alpha\leq^{\Xx^\alpha}\zeta^\alpha\in b^{\Xx^\alpha}$ and
$\crit(i^{\Xx^\alpha}_{\zeta^\alpha\infty})>\In(E^{\Xx}_{\zeta^\alpha})$.
\item\label{item:(zeta,0)_in_D} If $(\zeta^\alpha,0)\in D$
then $\lambda^\alpha<^{\Xx^\alpha}\zeta^\alpha+1\in b^{\Xx^\alpha}$ and
$\crit(i^{\Xx^\alpha}_{\zeta^\alpha+1,\infty})>\In(E^{\Xx}_{\zeta^\alpha})$ and
$\lambda^\alpha<^{\widetilde{\Xx}}\zeta^\alpha$.
\end{enumerate}
(Recall that $(\zeta^\alpha,0)\in D$ iff [$\zeta^\alpha+1<\lh(\Xx^\alpha)$ and
$\lambda(E^{\Xx^\alpha}_{\zeta^\alpha})<\In(E^\Xx_{\zeta^\alpha})$].)
\item\label{item:lambda_in_branch} $\lambda^\eta\in b^{\Xx^\eta}$ and
$(\lambda^\eta,\infty]_{\Xx^\eta}$
does not drop in model or degree.
\item\label{item:npc_single_comm}
$\nrsigma_\eta\com\varphi_\eta=i^{\Xx^\eta}_{\lambda^\eta\infty}\com\pi_\eta$
(cf.~Figure \ref{fgr:npc_comm}).
\item\label{item:npc_branch_comm} If $\alpha<^{\bar{\Uu}}\eta$ and
$(\alpha,\eta]_{\bar{\Uu}}$
does not drop in model
then the diagram in Figure \ref{fgr:npc_comm} commutes.
\begin{figure}
\centering
\begin{tikzpicture}
 [mymatrix/.style={
    matrix of math nodes,
    row sep=0.7cm,
    column sep=0.8cm}
  ]
   \matrix(m)[mymatrix]{
                     {}&                            {}&
 {}&              M^{\Xx^\eta}_\infty\\
                     {}&                            {}&
 {}&                                    {}\\
                     {}&                      N_{\eta}&
 {}&                                    {}\\
               M_{\eta}&                            {}&
 {}&           M^{\wt{\Xx}}_{\lambda^{\eta}}\\
                     {}&                            {}&
M^{\Xx^\alpha}_\infty&                                    {}\\
                     {}&                      N_\alpha&
 {}&                                    {}\\
               M_\alpha&
{}&M^{\wt{\Xx}}_{\lambda^\alpha}&{}\\
};
\path[->,font=\scriptsize]
(m-7-1) edge node[left] {$i^{\bar{\Uu}}_{\alpha\eta}$} (m-4-1)
(m-7-1) edge node[below] {$\pi_\alpha$} (m-7-3)
(m-7-1) edge node[auto] {$\varphi_\alpha$} (m-6-2)
(m-6-2) edge node[left,pos=0.4] {$i^\Uu_{\alpha\eta}$} (m-3-2)
(m-6-2) edge node[auto] {$\nrsigma_\alpha$} (m-5-3)
(m-7-3) edge node[left] {$i^{\Xx^\alpha}_{\lambda^\alpha\infty}$} (m-5-3)
(m-7-3) edge node[right] {$i^{\wt{\Xx}}_{\lambda^\alpha\lambda^{\eta}}$} (m-4-4)
(m-5-3) edge node[left, pos=0.6] {$\psi_{\alpha\eta}$} (m-1-4)
(m-4-1) edge  node[above=0.1cm, near start] {$\varphi_{\eta}$} (m-3-2)
(m-4-1) edge node[above] {$\pi_{\eta}$} (m-4-4)
(m-3-2) edge node[left] {$\nrsigma_{\eta}$} (m-1-4)
(m-4-4) edge  node[right] {$i^{\Xx^{\eta}}_{\lambda^{\eta}\infty}$} (m-1-4);
\end{tikzpicture}
\caption{The diagram commutes, where $(\alpha,\eta]_{\bar{\Uu}}$ does not drop
in model.} \label{fgr:npc_comm}
\end{figure}
\end{enumerate}

Note that if $[0,\eta]_{\bar{\Uu}}$ drops in model or degree
then $\lambda^\eta+1=\lh(\Xx^\eta)$ and
$M^{\wt{\Xx}}_{\lambda^\eta}=M^{\Xx^\eta}_{\lambda^\eta}$
and $i^{\Xx^\eta}_{\lambda^\eta\infty}=\id$, so
condition \ref{item:npc_single_comm} becomes
$\nrsigma_\eta\com\varphi_\eta=\pi_\eta$,
and if $\alpha<^{\bar{\Uu}}\eta$ and $[0,\alpha]_{\bar{\Uu}}$  drops in
model or degree
then the diagram in Figure \ref{fgr:npc_comm} simplifies to become
that in Figure \ref{fgr:npc_comm_drop}.

\begin{figure}
\centering
\begin{tikzpicture}
 [mymatrix/.style={
    matrix of math nodes,
    row sep=0.9cm,
    column sep=1cm}
  ]
   \matrix(m)[mymatrix]{
               M_{\eta}&                        N_\eta&
M^{\wt{\Xx}}_{\lambda^{\eta}}\\
               M_\alpha&
N_\alpha&M^{\wt{\Xx}}_{\lambda^\alpha}\\
};
\path[->,font=\scriptsize]
(m-2-1) edge node[left] {$i^{\bar{\Uu}}_{\alpha\eta}$} (m-1-1)
(m-2-1) edge[bend right] node[below] {$\pi_\alpha$} (m-2-3)
(m-2-1) edge node[above] {$\varphi_\alpha$} (m-2-2)
(m-2-2) edge node[left] {$i^\Uu_{\alpha\eta}$} (m-1-2)
(m-2-2) edge node[above] {$\nrsigma_\alpha$} (m-2-3)
(m-2-3) edge node[right]
{$\psi_{\alpha\eta}=i^{\wt{\Xx}}_{\lambda^\alpha\lambda^{\eta}}$} (m-1-3)
(m-1-1) edge  node[below] {$\varphi_{\eta}$} (m-1-2)
(m-1-1) edge[bend left] node[above] {$\pi_{\eta}$} (m-1-3)
(m-1-2) edge node[below] {$\nrsigma_{\eta}$} (m-1-3);
\end{tikzpicture}
\caption{The simplification of Figure \ref{fgr:npc_comm} when
$[0,\alpha]_{\bar{\Uu}}$ drops in model or degree.} \label{fgr:npc_comm_drop}
\end{figure}

\begin{figure}
\centering
\begin{tikzpicture}
 [mymatrix/.style={
    matrix of math nodes,
    row sep=0.9cm,
    column sep=0.9cm}
  ]
   \matrix(m)[mymatrix]{
                     {}&                            {}&
  {}&                M^{\Xx^{\beta+1}}_\infty&                         {}&
              {}&                          {}\\
                     {}&
{}&\bar{M}^{\Xx^{\beta+1}}_\infty&                                      {}&
                   {}&                     {}&
M^{\Xx^\beta}_\infty\\
                     {}&                   N_{\beta+1}&
  {}&                                      {}&                         {}&
         N_\beta&                          M^{\Xx^\beta}_{\zeta_*^\beta}\\
            M_{\beta+1}&                            {}&
  {}&        M^{\wt{\Xx}}_{\lambda^{\beta+1}}&                    M_\beta&
              {}&                          M^{\wt{\Xx}}_{\lambda^\beta}\\
                     {}&                            {}&
M^{\Xx^\alpha}_\infty&                                      {}&
       {}&                     {}&                          {}\\
                     {}&                      N_\alpha&
  {}&M^{\wt{\Xx}}_{\xi=\gamma_{\alpha\kappa}}&                         {}&
              {}&                          {}\\
               M_\alpha&                            {}&
M^{\wt{\Xx}}_{\lambda^\alpha}&                                      {}&
               {}&                     {}&                          {}\\
};
\path[->,font=\scriptsize]
(m-7-1) edge node[left] {$i^{M_\alpha}_{\bar{E}_\beta}$} (m-4-1)
(m-7-1) edge node[below] {$\pi_\alpha$} (m-7-3)
(m-7-1) edge node[auto] {$\varphi_\alpha$} (m-6-2)
(m-6-2) edge node[left,pos=0.4] {$i^{N_\alpha}_{E_\beta}$} (m-3-2)
(m-6-2) edge node[auto] {$\nrsigma_\alpha$} (m-5-3)
(m-7-3) edge node[right] {$i^{\Xx^\alpha}_{\lambda^\alpha\infty}$} (m-5-3)
(m-7-3) edge node[below] {$\ \ \ \ i^{\wt{\Xx}}_{\lambda^\alpha\xi}$} (m-6-4)
(m-6-4) edge  node[above] {$\ \ i^{\Xx^\alpha}_{\xi\infty}$} (m-5-3)
(m-6-4) edge  node[right] {$i^{*\wt{\Xx}}_{\lambda^{\beta+1}}$} (m-4-4)
(m-5-3) edge  node[left, pos=0.2] {$i^{M^{\Xx^\alpha}_\infty}_{F_\beta}$}
(m-2-3)
(m-5-3) edge node[left, pos=0.6] {$\psi_{\alpha,\beta+1}$} (m-1-4)
(m-4-1) edge  node[above=0.1cm, near start] {$\varphi_{\beta+1}$} (m-3-2)
(m-4-1) edge[bend left=15] node[above] {$\pi_{\beta+1}$} (m-4-4)
(m-3-2) edge  node[above=0.1cm] {$\bar{\nrsigma}_{\beta+1}$} (m-2-3)
(m-3-2) edge[bend left=35] node[left] {$\nrsigma_{\beta+1}$} (m-1-4)
(m-2-3) edge  node[above=0.1cm, near start] {$\nrsigma'_{\beta+1}$} (m-1-4)
(m-4-4) edge  node[right] {$i^{\Xx^{\beta+1}}_{\lambda^{\beta+1}\infty}$}
(m-1-4)
(m-4-5) edge  node[below] {$\pi_\beta$} (m-4-7)
(m-4-5) edge  node[auto] {$\varphi_\beta$} (m-3-6)
(m-3-6) edge  node[auto] {$\nrsigma_\beta$} (m-2-7)
(m-3-7) edge node[right] {$i^{\Xx^\beta}_{\zeta^\beta_*\infty}$} (m-2-7)
(m-4-7) edge  node[right] {$i^{\Xx^\beta}_{\lambda^\beta\zeta^\beta_*}$}
(m-3-7);
\end{tikzpicture}
\caption{The diagrams commute (here $\beta+1\notin\dropset^{\bar{\Uu}}$). We
have $\alpha=\pred^{\bar{\Uu}}(\beta+1)$ and
$\xi=\gamma_{\alpha\kappa}=\gamma_{\Pi\alpha\kappa}$,
where $\kappa=\crit(\bar{E}_\beta)$ and $\Pi$ is the essential
tree embedding under construction.
Note that $M^{\Xx^\beta}_{\zeta^\beta_*}=M^{\wt{\Xx}}_{\zeta^\beta}$
and
$i^{\Xx^\beta}_{\lambda^\beta\zeta_*^\beta}=i^{\wt{\Xx}}_{
\lambda^\beta\zeta^\beta}$.}\label{fgr:npc_successor_comm}
\end{figure}

Consider first the case that $\eta=0$.
Conditions \ref{item:zeta_in_branch} and \ref{item:npc_branch_comm} are trivial.
The essential tree embedding referred to in condition
\ref{item:partial_ess_tree_emb}
is just the trivial one.
Recall that $\Tt=\Xx^0$ does not drop in model or degree on $b^\Tt$,
and of course $\lambda^0=0$.
So condition \ref{item:lambda_in_branch} is immediate.
And condition \ref{item:npc_single_comm} holds because
$\varphi_0=i^\Tt$ and $\nrsigma_0=\id$ and $\pi_0=\id$ and
$i^{\Xx^0}_{0\infty}=\varphi_0$.

For limit $\eta$, condition \ref{item:zeta_in_branch} is trivial by induction,
and the other conditions follow by induction
using the commutativity of the various maps discussed in \S\ref{sec:factor_tree}
and the construction of $\Sigma^\stk$. We leave the details to the reader.

So consider the case that $\eta=\beta+1$.

Condition \ref{item:zeta_in_branch}: Consider the case that $\alpha=\beta$.
Since $\Uu=i^\Tt\bar{\Uu}=\varphi_0\bar{\Uu}$ and by induction (conditions
\ref{item:lambda_in_branch} and \ref{item:npc_single_comm}), $\lambda^\beta\in
b^{\Xx^\beta}$ and
$(\lambda^\beta,\infty]_{\Xx^\beta}\inter\dropset^{\Xx^\beta}_{\deg}=\emptyset$
and
\[
E^\Xx_{\zeta^\beta}=\nrsigma_\beta(\varphi_\beta(E^{\bar{\Uu}}_\beta))=i^{
\Xx^\beta}_{\lambda^\beta\infty}(\pi_\beta(E^{\bar{\Uu}}_\beta)), \]
so $E^\Xx_{\zeta^\beta}\in\rg(i^{\Xx^\beta}_{\lambda^\beta\infty})$.
And $\zeta^\beta$ is the least $\zeta$ such that
$E^\Xx_{\zeta^\beta}\in\es_+(M^{\Xx^\beta}_\zeta)$.
So\footnote{An earlier draft of this paper
had the first inequality here
as strict $<$, but it seems we might need to allow $\leq$.
However, this has little bearing on the proof.}
\[ \In(E^{\Xx^\beta}_\xi)\leq\In(E^\Xx_{\zeta^\beta})<\In(E^{\Xx^\beta}_{\xi'})
\]
for all $\xi<\zeta^\beta$ and $\xi'\geq\zeta^\beta$ with
$\xi'+1<\lh(\Xx^\beta)$.

Now suppose that $(\zeta^\beta,0)\notin D$;
we must verify condition \ref{item:(zeta,0)_notin_D}. So if
$\zeta^\beta+1<\lh(\Xx^\beta)$
then $\In(E^\Xx_{\zeta^\beta})<\lambda(E^{\Xx^\beta}_{\zeta^\beta})$.
We may easily assume that $\zeta^\beta+1<\lh(\Xx^\beta)$.
Suppose $\zeta^\beta\notin b^{\Xx^\beta}$,
and let $\zeta\geq\zeta^\beta$ be least such that $\zeta+1\in b^{\Xx^\beta}$.
Then
\[
\lambda^\beta\leq^{\Xx^\beta}\xi\eqdef\pred^{\Xx^\beta}
(\zeta+1)<\zeta^\beta<\zeta+1,\]
so
\[
\crit(E^{\Xx^\beta}_\zeta)<\nutilde(E^{\Xx^\beta}_\xi)<\lambda(E^{\Xx^\beta}_{
\zeta^\beta})\leq\lambda(E^{\Xx^\beta}_\zeta), \]
and as $E^\Xx_{\zeta^\beta}\in\rg(i^{\Xx^\beta}_{\lambda^\beta\infty})$
and $\In(E^\Xx_{\zeta^\beta})<\lambda(E^{\Xx^\beta}_{\zeta^\beta})$,
therefore
$\In(E^\Xx_{\zeta^\beta})<\crit(E^{\Xx^\beta}_\zeta)$.
But then $E^\Xx_{\zeta^\beta}\in\es(M^{\Xx^\beta}_\xi)$, contradicting the
minimality of $\zeta^\beta$.
So $\lambda^\beta\leq^{\Xx^\beta}\zeta^\beta\in b^{\Xx^\beta}$. Let
$\zeta+1=\successor^{\Xx^\beta}(\zeta^\beta,\infty)$.
If
\[
\crit(E^{\Xx^\beta}_\zeta)=\crit(i^{\Xx^\beta}_{\zeta^\beta\infty})<\In(E^\Xx_{
\zeta^\beta}) \] then
\[
\crit(E^{\Xx^\beta}_\zeta)<\In(E^\Xx_{\zeta^\beta})<\lambda(E^{\Xx^\beta}_{
\zeta^\beta }
)\leq\lambda(E^{\Xx^\beta}_\zeta), \]
again contradicting the fact that
$E^\Xx_{\zeta^\beta}\in\rg(i^{\Xx^\beta}_{\lambda^\beta\infty})$.
This gives \ref{item:(zeta,0)_notin_D}.

Now suppose  $(\zeta^\beta,0)\in D$; we verify
 \ref{item:(zeta,0)_in_D}. So $\zeta^\beta+1<\lh(\Xx^\beta)$
 and
 \begin{equation}\label{eqn:lambda<lh<lh}
\lambda(E^{\Xx^\beta}_{\zeta^\beta})<\In(E^\Xx_{\zeta^\beta})<\In(E^{\Xx^\beta}_
{\zeta^\beta}).\end{equation}
Let $\zeta\geq\zeta^\beta$ be least such that $\zeta+1\in b^{\Xx^\beta}$, and
$\xi=\pred^{\Xx^\beta}(\zeta+1)$. Then
$\lambda^\beta\leq^{\Xx^\beta}\xi\leq\zeta^\beta$, so
\[
\crit(E^{\Xx^\beta}_\zeta)<\nutilde(E^{\Xx^\beta}_\xi)\leq\lambda(E^{\Xx^\beta}_
{\zeta^\beta})\leq\lambda(E^{\Xx^\beta}_\zeta), \]
and since $E^\Xx_{\zeta^\beta}\in\rg(i^{\Xx^\beta}_{\lambda^\beta\infty})$ and
by line (\ref{eqn:lambda<lh<lh}),
 it follows that $\zeta=\zeta^\beta$, so $\zeta^\beta+1\in b^{\Xx^\beta}$,
 so $\lambda^\beta<^{\Xx^\beta}\zeta^\beta+1$.
 The fact that $\lambda^\beta<^{\widetilde{\Xx}}\zeta^\beta$
 follows immediately by definition; note that the role of $\zeta^\beta$ in
$\widetilde{\Xx}$ corresponds to $\zeta^\beta+1$ in $\Xx^\beta$,
 as $E^{\widetilde{\Xx}}_{(\zeta^\beta,0)}=E^{\Xx^\beta}_{\zeta^\beta}$.
 Similarly,
$\In(E^\Xx_{\zeta^\beta})<\crit(i^{\Xx^\beta}_{\zeta^\beta+1,\infty})$.
 This gives \ref{item:(zeta,0)_in_D}.

We can now complete the proof of condition \ref{item:zeta_in_branch}.
We have:
\begin{enumerate}[label=--]
 \item  $\Xx^\beta\rest(\zeta_*^\beta+1)$
is the $\udash m$-maximal tree corresponding to
$\widetilde{\Xx}\rest(\zeta^\beta+1)$.\footnote{
Recall that $\widetilde{\Xx}$ is only essentially $\udash m$-maximal,
and note that if $(\zeta^\beta,0)\in D$,
so $\zeta^\beta_*=\zeta^\beta+1$, then both
$\widetilde{\Xx}\rest(\zeta^\beta+1)$ and $\Xx^\beta\rest\zeta^\beta_*+1$ use
last extender
$E^{\widetilde{\Xx}}_{(\zeta^\beta,0)}=E^{\Xx^\beta}_{\zeta^\beta}$.}
\item $\lambda^\beta\leq^{\Xx^\beta}\zeta_*^\beta\in b^{\Xx^\beta}$ and
$(\lambda^\beta,\infty]_{\Xx^\beta}$ does not drop in model or degree,
\item $(\lambda^\beta,\zeta^\beta]_{\widetilde{\Xx}}$ does not drop in model or
degree,
\item  $M^{\widetilde{\Xx}}_{\zeta^\beta}=M^{\Xx^\beta}_{\zeta_*^\beta}$ and
$i^{\widetilde{\Xx}}_{\lambda^\beta\zeta^\beta}=i^{\Xx^\beta}_{
\lambda^\beta\zeta_*^\beta}$,
\item
$E^{\widetilde{\Xx}}_{\zeta^\beta}=i^{\Xx^\beta}_{\lambda^\beta\infty}
(\pi_\beta(E^{\bar{\Uu}}_\beta))$ and
$\In(E^{\widetilde{\Xx}}_{\zeta^\beta})<\crit(i^{\Xx^\beta}_{\zeta_*^\beta\infty
})$, so
\item
$E^{\widetilde{\Xx}}_{\zeta^\beta}=i^{\widetilde{\Xx}}_{\lambda^\beta\zeta^\beta
}(\pi_\beta(E^{\bar{\Uu}}_\beta))$.
\end{enumerate}

 Condition \ref{item:partial_ess_tree_emb}:
By induction,
$\left<I_\alpha\right>_{\alpha<\beta}\conc\left<[\lambda^\beta,\lambda^\beta]
\right>$
is an essential tree embedding
$\bar{\Uu}\rest(\beta+1)\hookrightarrow\widetilde{\Xx}$.
But then by condition \ref{item:zeta_in_branch} and as
$\lambda^{\beta+1}=\zeta^\beta+1$,
\[
\left<I_\alpha\right>_{\alpha\leq\beta}\conc\left<[\lambda^{\beta+1},\lambda^{
\beta+1}]\right> \]
is also an essential tree embedding
$\bar{\Uu}\rest(\beta+2)\hookrightarrow\widetilde{\Xx}$.

Conditions \ref{item:lambda_in_branch}, \ref{item:npc_single_comm},
\ref{item:npc_branch_comm}:
We consider the case that
$[0,\beta+1]_{\bar{\Uu}}\inter\dropset_{\deg}^{\bar{\Uu}}=\emptyset$,
and leave the other case to the reader.
By induction, it suffices to
verify condition \ref{item:lambda_in_branch} for $\eta=\beta+1$
and to verify that the diagram on the left of Figure
\ref{fgr:npc_successor_comm} commutes, for
the current $\beta$ and $\alpha=\pred^{\bar{\Uu}}(\beta+1)$
(by induction and condition \ref{item:zeta_in_branch},
the diagram on the right of Figure \ref{fgr:npc_successor_comm} commutes).
Note  the embeddings $i^{M_\alpha}_{\bar{E}_\beta}$,
$i^{N_\alpha}_{E_\beta}$ and $i^{M^{\Xx^\alpha}_\infty}_{F_\beta}$
are  the ultrapower embeddings associated to $\Ult_{\udash
m}(M_\alpha,\bar{E}_\beta)$, etc.

As in the figure, let $\kappa=\crit(\bar{E}_\beta)$
and $\xi=\gamma_{\alpha\kappa}=\gamma_{\Pi\alpha\kappa}$,
so $\xi=\pred^{\widetilde{\Xx}}(\lambda^{\beta+1})$
and $\xi\in I_\alpha=[\lambda^\alpha,\zeta^\alpha]_{\widetilde{\Xx}}$.
Let
\[
\mu=i^{\wt{\Xx}}_{\lambda^\alpha\xi}(\pi_\alpha(\kappa))=i^{\wt{\Xx}}_{
\lambda^\alpha\zeta^\alpha}(\pi_\alpha(\kappa)).\]
Then $\xi\in[\lambda^\alpha,\zeta_*^\alpha]_{\Xx^\alpha}$
and
\[
\mu=i^{\Xx^\alpha}_{\lambda^\alpha\xi}(\pi_\alpha(\kappa))=i^{\Xx^\alpha}_{
\lambda^\alpha\infty}(\pi_\alpha(\kappa))<\crit(i^{\Xx^\alpha}_{\xi\infty}) \]
and either:
\begin{enumerate}[label=--]
 \item $(\zeta^\alpha,0)\notin D$ and
$[\lambda^\alpha,\zeta^\alpha]_{\wt{\Xx}}=[\lambda^\alpha,\zeta_*^\alpha]_{
\Xx^\alpha}$, or
 \item $(\zeta^\alpha,0)\in D$ and
$\xi\in[\lambda^\alpha,\eps]_{\wt{\Xx}}=[\lambda^\alpha,\eps]_{\Xx^\alpha}$
 where
 \[ \eps=\pred^{\Xx^\alpha}(\zeta^\alpha+1)=\pred^{\wt{\Xx}}(\zeta^\alpha).\]
 \end{enumerate}

 For suppose $(\zeta^\alpha,0)\notin D$. Then
 \[ \zeta^\alpha=\zeta^\alpha_*\text{ and }
\xi\in[\lambda^\alpha,\zeta^\alpha]_{\wt{\Xx}}=[\lambda^\alpha,\zeta^\alpha_*]_{
\Xx^\alpha}\text{ and }
i^{\wt{\Xx}}_{\lambda^\alpha\zeta^\alpha}=i^{\Xx^\alpha}_{
\lambda^\alpha\zeta^\alpha}.\]
Also $\kappa<\nutilde(E^{\bar{\Uu}}_\alpha)$, so
$\mu<\nutilde(E^{\wt{\Xx}}_{\zeta^\alpha})<\In(E^{\wt{\Xx}}_{\zeta^\alpha}
)<\crit(i^{\Xx^\alpha}_{\zeta^\alpha\infty})$,
which easily suffices.

Now suppose instead that $(\zeta^\alpha,0)\in D$.
Then $M^{\wt{\Xx}}_{\zeta^\alpha}||\In(E^{\wt{\Xx}}_{(\zeta^\alpha,0)})$ has
largest cardinal
$\lambda=\lambda(E^{\wt{\Xx}}_{(\zeta^\alpha,0)})$ and
\begin{equation}\label{eqn:again_lambda<lh<lh}
\lambda<\In(F_\alpha)<\In(E^{\wt{\Xx}}_{(\zeta^\alpha,0)}). \end{equation}
We have $\beta+1\notin\dropset^{\bar{\Uu}}$,
so $(\kappa^+)^{M^{\bar{\Uu}}_\alpha}\leq\In(\bar{E}_\alpha)$,
so $(\mu^+)^{M^{\wt{\Xx}}_{\zeta^\alpha}}\leq\In(F_\alpha)$,
so by (\ref{eqn:again_lambda<lh<lh}), $\mu<\lambda$.
Since $\mu\in\rg(i^{\wt{\Xx}}_{\lambda^\alpha\zeta^\alpha})$, it follows that
$\mu<\crit(E^{\wt{\Xx}}_{(\zeta^\alpha,0)})$,
so $\xi\leq\eps=\pred^{\wt{\Xx}}(\zeta^\alpha)$, as required. The rest is now
clear.

From the preceding discussion, it follows that $\lambda^{\beta+1}\in
b^{\Xx^{\beta+1}}$,
$(\lambda^{\beta+1},\infty]_{\Xx^{\beta+1}}$ does not drop in model or degree
 and
\[\crit(i^{\Xx^{\beta+1}}_{\lambda^{\beta+1}\infty})=i^{*\wt{\Xx}}_{\lambda^{
\beta+1}}(\crit(i^{\Xx^\alpha}_{\xi\infty})). \]
So the left diagram in Figure \ref{fgr:npc_successor_comm} is at least
plausible.

It remains to verify that the diagram commutes. The diagram which results
if we remove
$\pi_{\beta+1}$ from Figure \ref{fgr:npc_successor_comm}, is already known to
commute,
by induction and facts about $\Sigma^\stk$.
We have
\[ \pi_{\beta+1}\com
i^{M_\alpha}_{\bar{E}_\beta}=i^{*\wt{\Xx}}_{\lambda^{\beta+1}}\com
i^{\wt{\Xx}}_{\lambda^\alpha\xi}\com\pi_\alpha \]
by properties of essential tree embeddings.

For each $\eps$, write $\varrho_\eps=\nrsigma_\eps\com\varphi_\eps$.
Let $j=i^{\Xx^{\beta+1}}_{\lambda^{\beta+1}\infty}$.
It just remains to see that
\[
\varrho_{\beta+1}=\nrsigma_{\beta+1}\com\varphi_{\beta+1}=j\com\pi_{\beta+1}.
\]
For simplicity let us assume that $m=0$; for $m>0$ it is analogous.

Let $x=i^{M_\alpha}_{\bar{E}_\beta}(f)(a)\in M_{\beta+1}$, where
$a\in[\nu(\bar{E}_\beta)]^{<\om}$
and $f\in M_\alpha$. Then
\begin{equation}\label{eqn:psi_com_rho(x)}
\varrho_{\beta+1}(x)=\psi_{\alpha,\beta+1}(\varrho_\alpha(f))(\varrho_\beta(a)),
\end{equation}
since $\nrsigma_{\beta+1}=\nrsigma'_{\beta+1}\com\bar{\nrsigma}_{\beta+1}$ and
\[
\bar{\nrsigma}_{\beta+1}(\varphi_{\beta+1}(x))=i^{M^{\Xx^\alpha}_\infty}_{
F_\beta}(\varrho_\alpha(f))(\varrho_\beta(a)), \]
and as discussed earlier,
$\crit(\nrsigma'_{\beta+1})>\lambda(E^{\wt{\Xx}}_{\zeta^\beta}
)>\max(\varrho_\beta(a))$.

On the other hand,
\begin{equation}\label{eqn:it_com_pi(x)} j(\pi_{\beta+1}(x))=j(g(c))=j(g)(c)
\end{equation}
where
$g=i^{*\wt{\Xx}}_{\lambda^{\beta+1}}\com
i^{\wt{\Xx}}_{\lambda^\alpha\xi}\com\pi_\alpha(f)$
and $c=i^{\wt{\Xx}}_{\lambda^\beta\zeta^\beta}(\pi_\beta(a))=j(c)$,
since $\max(c)<\nu(F_\beta)\leq\crit(j)$.

So it suffices to show that $j(g)=\psi_{\beta,\alpha+1}(\varrho_\alpha(f))$
and $c=\varrho_\beta(a)$, as then the objects in lines
(\ref{eqn:psi_com_rho(x)}) and (\ref{eqn:it_com_pi(x)})
are equal, as desired. But $j(g)=\psi_{\beta,\alpha+1}(\varrho_\alpha(f))$
by the commutativity already known in the left diagram of Figure
\ref{fgr:npc_successor_comm};
and by its right diagram and since
$\max(c)<\In(F_\beta)<\crit(i^{\Xx^\beta}_{\zeta^\beta_*\infty})$, we have
\[ \varrho_\beta(a)=i^{\Xx^\beta}_{\lambda^\beta\infty}(\pi_\beta(a))=
i^{\Xx^\beta}_{\lambda^\beta\zeta^\beta_*}(\pi_\beta(a))=
i^{\wt{\Xx}}_{\lambda^\beta\zeta^\beta}(\pi_\beta(a))=c, \]
completing the proof of the theorem.
\end{proof}

\subsection{Dodd-Jensen and $\Sigma^\stk$}\label{subsec:DJ_and_Sigma^stk}

\begin{dfn}\index{lifting (weak) DJ}
\dfnemph{Lifting Dodd-Jensen (DJ)} is defined just like the DJ property,
but with the class of $n$-lifting embeddings replacing near $n$-embeddings
(when
at degree $n$).
Likewise for \dfnemph{lifting weak DJ}.
\end{dfn}
\begin{rem}
Assuming $\DC$, given a sufficiently iterable countable premouse $M$ and an
enumeration $e$ of $M$ in ordertype $\om$,
we can construct a strategy $\Sigma$ for $M$ with lifting
weak DJ,
completely analogously to the construction of one with (standard) weak DJ.
Clearly lifting (weak) DJ implies weak DJ,
because every near $n$-embedding is $n$-lifting.
\end{rem}

\begin{tm}\label{tm:stacks_strat_inherits_DJ} Let $\Sigma,\Omega,M$ be as in
Theorem \ref{thm:stacks_iterability},
with $M$ a premouse,
and $\Sigma$ an $(m,\Omega+1)$-strategy for $M$,
and suppose that $\card(M)<\Omega$.
If $\Sigma$ has lifting DJ then so does $\Sigma^\stk$.
If $M$ is countable and $e$ is an enumeration of $M$ in ordertype $\om$,
then likewise for lifting weak DJ with respect to $e$.
\end{tm}
\begin{proof}
Suppose $M$ is $\lambda$-indexed. We literally give the proof for lifting DJ,
but for lifting weak DJ it is essentially the same.
Let $\Ttvec$ be according to $\Sigma^\stk$, with $N=M^\Ttvec_\infty$
and $n=\deg^\Ttvec(\infty)$,
and let $Q,\pi$ be such that $(Q,m)\ins(N,n)$
and $\pi:M\to Q$ is $m$-lifting.

We may assume that $\lh(\Ttvec)<\Omega$ and each normal tree in $\Ttvec$
has length ${<\Omega}$, because $\card(M)<\Omega$ and $\Omega$ is regular.
Let $\Uuvec$ be the corresponding optimal $m$-maximal stack on $M$ given
by the proof of Lemma \ref{lem:sub-optimal_reduces_to_optimal}.
Let $N'=M^{\Uuvec}_\infty$ and $n'=\deg^{\Uuvec}(\infty)$.
Letting $Q'\ins N'$ be the resulting lift of $Q$
and $\sigma:Q\to Q'$ the restricted copy map,
note that $(Q',m)\ins(N',n')$ and
$\pi'=\sigma\com\pi:M\to Q'$ is $m$-lifting.

Now $\Xx=\Ww^\Sigma(\Uuvec)$ is via $\Sigma$, of length ${<\Omega}$,
and we have the $n'$-lifting
\[ \srsigma=\srsigma^\Sigma(\Ttvec):M^\Ttvec_\infty\to M^\Xx_\infty.\]
Let $Q''=\srsigma(Q')$ if $Q'\pins N'$, and $Q''=M^\Xx_\infty$ otherwise. Let
$\pi''=\srsigma\com\pi'$.
Then we can apply lifting DJ for $\Sigma$ to $\Xx,Q'',\pi''$.
Therefore $Q''=M^\Xx_\infty$ (so $Q'=N'$) and $b^\Xx$ does not drop in model or
degree, so $n'=m$,
and for each $\alpha\in\OR^M$, we have $i^\Xx(\alpha)\leq\pi''(\alpha)$.
Therefore $b^{\Uuvec}$ does not drop in model or degree,
so $i^\Xx=\pi''\com i^{\Uuvec}$. Therefore $i^\Uuvec(\alpha)\leq\pi'(\alpha)$
for each $\alpha<\OR^M$. But then, similarly, $b^\Ttvec$ also does not drop in
model or degree, $i^{\Uuvec}=\sigma\com i^\Ttvec$ and
$i^\Ttvec(\alpha)\leq\pi(\alpha)$
for each $\alpha$, so we are done.

If instead $M$ is MS-indexed then combine the preceding argument with that in
the proof of Claim \ref{clm:DJ_for_Sigma'}
of the proof of Theorem \ref{thm:strat_with_cond_extends_to_generic_ext}.
\end{proof}
\begin{rem}
One would like to be able to prove a version of the preceding theorem
for standard (weak) DJ.
We can prove this in certain cases, but do not see how to in general.
This is because (considering $\Gamma^\stk$, where $\Gamma$ is the
$\udash$strategy corresponding to $\Sigma$)
the lifting map
$\srsigma:M^\Ttvec_\infty\to M^\Xx_\infty$
need not be a near $\udash n$-embedding where $n=\udeg^\Ttvec(\infty)$.
However, if either
(i) $m>0$, or
(ii) $M$ is passive, or
(iii) $M$ is MS-indexed type 1 or 3,
then we do get (weak) DJ for $\Sigma^\stk$.
This is due to the following easy consequence of condensation.

Let $M,N$ be $n$-sound and $\nrsigma:M\to N$ be $n$-lifting
$\pvec_n$-preserving.
Suppose that $\nrsigma$ is not an $n$-embedding, and either (i) $n>0$, (ii) $M$
is passive,
or (iii) $M$ is MS-indexed type 1 or 3. Then there is some $Q$ such that
\[ \text{either }Q\pins N\text{, or }
Q\pins\Ult(N|\rho,F^{N|\rho})\text{ for some }\rho, \]
and an $n$-embedding
$\pi':M\to Q$.

For let $\rho=\sup\pi``\rho_n^M$.
We have $\rho<\rho_n^N$ because $\pi$ is not an $n$-embedding.

If $n=0$ and $M$ is passive then clearly
$Q=N||\rho$ and $\pi'=\pi$ works (but note maybe $N|\rho$ is active).

If $n=0$ and $M$ is MS-indexed type 3 then note that $\rho$ is a limit of
generators
of $F^N$, and let $Q\pins N$ or $Q\pins\Ult(N|\rho,F^{N|\rho})$
be such that $F^N\rest\rho=F^Q$, and $\pi'=\pi$ (note that in this case,
$\dom(\pi)=M^\sq$).

Suppose $n=0$ and $M$ is MS-indexed type 1. Let $\mu=\crit(F^M)$ and
$\kappa=\crit(F^N)=\pi(\mu)$.
Let
\[ Q=\cHull_0^{N}(\kappa\cup\rg(\pi)) \]
and $\sigma:Q\to N$ be the uncollapse and
$\pi':M\to Q$ be such that $\sigma\com\pi'=\pi$.
Then
\[ \sup\sigma``\OR^Q=\sup\pi``\OR^M \]
and $Q$ is a type 1 premouse by standard arguments,
and $\pi'$ is $\rSigma_1$-elementary. We have $Q\in N$ and
$(\kappa^+)^Q<(\kappa^+)^N$
and
\[ F^Q\rest(\kappa^+)^Q=F^N\rest(\kappa^+)^Q.\]
So basically by \cite[\S4]{mim}, either
$Q\pins N$ or letting $\alpha=(\kappa^+)^N$, $N|\alpha$ is active
and $Q\pins\Ult(N|\alpha,F^{N|\alpha})$, so we are done.

Now suppose $n>0$. Let
$Q=\cHull_n^N(\rho\cup\pvec_n^N)$
and $\sigma:Q\to N$ be the uncollapse.
Note that $\rg(\pi)\sub\rg(\sigma)$ and let
$\pi':M\to Q$ be such that $\sigma\com\pi'=\pi$.
Note that $Q$ is $(n-1)$-sound
and $\pi'$ is a near $(n-1)$-embedding,
Note that $\pi'(p_n^M)$ is $n$-solid for $Q$ and
\[ Q=\Hull^Q_n(\rho\cup\pi'(\pvec_n^M)), \]
so $\pi'(p_n^M)=p_n^Q\cut\rho$,
but also because $\pi$ is $n$-lifting, therefore $\rho_n^Q=\rho$.
So $Q$ is $n$-sound. Also, $Q\in N$.
 By condensation, either $Q\pins N$
or $N|\rho$ is active and $Q\pins\Ult(N|\rho,F^{N|\rho})$. Moreover,
 because $\pi$ is $n$-lifting and $\rho_n^Q=\rho$ and
$\pi'(\pvec_n^M)=\pvec_n^Q$,
$\pi'$ is in fact an $n$-embedding, which suffices.
\end{rem}

We conclude this segment with some simple corollaries pertaining to generic
absoluteness of iterability under choice, and also
constructing strategies with weak DJ in choiceless contexts.

\begin{cor}\label{cor:stack_it_absoluteness}
Let $\Omega>\om$ be regular.
Let $\PP$ be an $\Omega$-cc forcing and let $G$ be $V$-generic for $\PP$.
Let $M$ be a countable $n$-sound premouse.
Then:
\begin{enumerate}[label=--]
\item If $V\sats\DC+$``$M$ is $(n,\Omega,\Omega+1)^*$-iterable'' then
$V[G]\sats$``$M$ is $(n,\Omega,\Omega+1)^*$-iterable''.
\item If $V[G]\sats\DC+$``$M$ is $(n,\Omega,\Omega+1)^*$-iterable''
then $V\sats$``$M$ is $(n,\Omega,\Omega+1)^*$-iterable''.
\end{enumerate}
\end{cor}
\begin{proof}

Assume $\DC$ and suppose $M$ is $(n,\Omega,\Omega+1)^*$-iterable. Then
there is an $(n,\Omega+1)$-strategy $\Sigma$ for $M$ with weak DJ (note that
the
construction
of such a strategy only uses $(n,\Omega,\Omega+1)^*$-iterability, not
$(n,\Omega,\Omega+1)$-iterability),
which by \ref{tm:wDJ_implies_cond} has strong hull condensation. Therefore by
\ref{thm:strat_with_cond_extends_to_generic_ext},
$V[G]$ has an $(n,\Omega+1)$-strategy $\Sigma'$ for $M$ with strong hull
condensation.
So by Theorem \ref{thm:stacks_iterability}, $M$ is
$(n,\Omega,\Omega+1)^*$-iterable in $V[G]$.

Now suppose instead that $V[G]\sats\DC+$``$M$ is
$(n,\Omega,\Omega+1)^*$-iterable''.
Then in $V[G]$ there is an $(n,\Omega+1)$-strategy $\Sigma'$ with weak DJ with
respect to
some enumeration $e\in V$ of $M$. By \ref{cor:wDJ_absoluteness},
$\Sigma=\Sigma'\rest V\in V$
and $\Sigma$ has weak DJ in $V$. So by \ref{thm:stacks_iterability},
$\Sigma$ extends to an $(n,\Omega,\Omega+1)^*$-strategy in $V$.
\end{proof}

Note that in the following corollary, $M$ is a premouse, not a
wcpm.
 \begin{cor}\label{cor:ZFC_iter_forcing_ab}
 Assume $\ZFC$. Let $M$ be a countable $m$-sound premouse and $e$ be an
enumeration of $M$ in ordertype $\om$. Let $m<\om$. Let $\Omega>\om$ be regular.
 Let $\PP$ be an $\Omega$-cc forcing and $G$ be $V$-generic for $\PP$. Then the
following are equivalent:
 \begin{enumerate}[label=--]
 \item There is an $(m,\Omega+1)$-strategy for $M$ with strong hull
condensation.
 \item $M$ is $(m,\Omega,\Omega+1)^*$-iterable.
 \item There is an $(m,\Omega+1)$-strategy for $M$ with weak DJ with respect to
$e$.
 \item $V[G]$ satisfies one of the preceding statements.
 \end{enumerate}
 \end{cor}

 \begin{proof}[Proof of Corollary \ref{cor:ZFC_iter_forcing_ab}]
Both $V$ and $V[G]$ satisfy $\ZFC$, so the previous corollary and its arguments
apply (note that $e\in V$), which easily yields \ref{cor:ZFC_iter_forcing_ab}.
\end{proof}

\subsection{Weak DJ without $\DC$}\label{subsec:weak_DJ_without_DC}

We now discuss some choiceless constructions of strategies with weak DJ.
The main result is Corollary \ref{cor:choiceless_wDJ_HOD},
and the basic idea for that may have originated from some observations of
Dominik Adolf (that is, using Theorem
\ref{thm:strat_with_cond_extends_to_generic_ext} to extend
a strategy in $\HOD_X$ via Vopenka forcing). However, we start with Corollary
\ref{cor:choiceless_wDJ_forcing_CH} below, which is actually less general,
but the two proofs are different, and both seem of interest.
\begin{cor}\label{cor:choiceless_wDJ_forcing_CH}
 Let $\Omega>\om$ be regular and suppose that for no $\alpha<\Omega$ is
 $\Omega$ the surjective image of $V_\alpha$. Let $M$ be a countable $m$-sound
premouse.
 Let $\Sigma$ be an $(m,\Omega+1)$-strategy for $M$ with strong hull
condensation.
 Let $e$ be an enumeration of $M$ in ordertype $\om$.
 Then there is an $(m,\Omega+1)$-strategy for $M$ with weak DJ with respect to
$e$.
\end{cor}

\begin{rem}
Before proving the corollary, we sketch another proof scenario,
which other than the extension of $\Sigma$ to the stacks strategy $\Sigma^\stk$,
would only use standard techniques if it could be made to work,
and point out where the scenario seems to run into problems.
First extend $\Sigma$ to  $\Sigma^\stk$,
and then attempt a choiceless variant
of the construction of a strategy with weak DJ from $\Sigma^\stk$.
(Note that we are not assuming $\DC$, which the
usual construction uses.)
A natural attempt for the latter is as follows.

Let $\alpha_0$ be the least $\alpha$ such that there is $\Ttvec\in V_\alpha$
with $\Ttvec$ according to $\Sigma^\stk$,
and some $Q\ins M^\Ttvec_\infty$ and $\pi:M\to Q$ violating weak DJ.
Let $A_0$ be the set of all such pairs $(\Ttvec,Q)$ where $\Ttvec\in
V_{\alpha_0}$.
For $(\Ttvec,Q)\in A_0$, let $\Sigma^\stk_{\Ttvec,Q}$ by the strategy for $Q$
given by the tail of $\Sigma^\stk$.

Now for each such $(\Ttvec,Q)\in A_0$, define a tree $\Uu_{\Ttvec,Q}$ on $Q$,
via $\Sigma^\stk_{\Ttvec,Q}$,
with these trees resulting from the simultaneous comparison
of the $Q$'s (but note that for a given $Q$, there could be multiple
corresponding trees $\Ttvec$,
and so multiple corresponding $\Uu_{\Ttvec,Q}$'s). This comparison terminates
in
$<\Omega$ stages,
because if we reached stage $\Omega+1$, then working in $L(X,V_{\alpha_0})$,
where
$X$ is a subset of $V_{\alpha_0}\cross\OR$ coding
the comparison, including final branches, we can form a hull of $V$ and reach
the usual contradiction.
Now for some $(\Ttvec,Q)\in A_0$, $b^{\Uu_{\Ttvec,Q}}$ does not drop in model
or
degree.
Choosing such a $(\Ttvec,Q)$ with $\OR(M^{\Uu_{\Ttvec,Q}}_\infty)$ least
possible,
let $Q'=M^{\Uu_{\Ttvec,Q}}_\infty$. Then there is some $\pi':M\to Q'$
witnessing a
failure of weak DJ,
and note that we have defined $Q'$ outright from $\Sigma^\stk$. We can also set
$\pi'$ to be the $e$-lexicographically
least witness.

However, there could be multiple pairs $(\Ttvec,Q)$ with
$Q'=M^{\Uu_{\Ttvec,Q}}_\infty$.
Thus, we don't seem to have a uniquely specified tail of $\Sigma^\stk$ for
iterating $Q'$.
We do have only $V_{\alpha_0}$-many such pairs,
so only $V_{\alpha_0}$-many strategies for $Q'$.
So we might continue by looking for failures of weak DJ arising from each of
these strategies,
comparing these and so on. But after repeating this process $\om$-many times,
we
seem to need $\DC$ in order
to choose some bad stack via some specific strategy, in order to reach a
contradiction. Thus,
we do not see how to complete the proof in this scenario.

We now give a proof that does work. We first need a forcing lemma.
\end{rem}

\begin{lem}\label{lem:force_CH}
 Let $\Omega>\om$ be regular and suppose that for no $\alpha<\Omega$ is
 $\Omega$ the surjective image of $V_\alpha$. Then there is a homogeneous
$\Omega$-cc forcing $\PP$ which forces $\mathsf{CH}$,
 in the strong sense that
$\Omega=\om_1^{V^\PP}=(2^{\aleph_0})^{V^\PP}=\card^{V^\PP}(\HC^{V^\PP})$.
\end{lem}
\begin{proof}
 Let $\PP$ be the forcing whose conditions are functions $p$ with $\dom(p)$
 a finite set $\sub(0,\Omega)\cross\om$ and
 $p(\alpha,n)\in V_\alpha$
 for each $(\alpha,n)\in\dom(p)$, and with ordering $p\leq q$ iff $q\sub p$.
 We claim that $\PP$ works.

 For clearly $\PP$ is homogeneous. Let $G$ be $V$-generic and $g=\bigcup G$.
Clearly  and
 \[ g:(0,\Omega)\cross\om\to V_\Omega \] is a surjection;
 in fact, for each $\alpha\in(0,\Omega)$, the function $n\mapsto g(\alpha,n)$
is
a surjection $\om\to V_\alpha$.
 So $\Omega\leq\om_1^{V[G]}$ and it suffices to see that $\PP$ is $\Omega$-cc
 and for each $x\in\HC^{V[G]}$ there is a $\PP$-name $\dot{x}\in V_\Omega$
 such that $\dot{x}^G=x$.

 \begin{clmthree}$\PP$ is $\Omega$-cc.
 \end{clmthree}
\begin{proof}
Let $\lambda\in\OR$ and $\left<A_\alpha\right>_{\alpha<\lambda}$ be a
$\lambda$-pre-antichain of $\PP$.
We must see that $\lambda<\Omega$. So suppose $\lambda=\Omega$. The proof is
just a choiceless variant of the usual $\Delta$-system argument.

For each $p\in\PP$, $\dom(p)$ is just a finite set of pairs of ordinals.
So by reducing each $A_\alpha$ if necessary, may assume that we have
$\left<d_\alpha\right>_{\alpha<\Omega}$
such that $\dom(p)=d_\alpha$ for all $p\in A_\alpha$, for all $\alpha$.
In $L[\left<d_\alpha\right>_{\alpha<\Omega}]$, where we have $\ZFC$
(and $\Omega$ is regular)
we can use the $\Delta$-system lemma. So we may assume that
we have some fixed finite $d\sub\Omega\cross\om$ such that $d_\alpha\inter
d_\beta=d$
for all $\alpha<\beta<\Omega$.
Let $\gamma<\Omega$ be such that $d\sub\gamma\cross\om$. Then for each
$\alpha$ and $p\in A_\alpha$, we have $p\rest d\in V_{\gamma+\om}$.

Let $X=\{p\rest d\bigm|  p\in A_\alpha\text{ and }\alpha<\Omega\}$.
So $X\sub V_{\gamma+\om}$. For $x\in X$, let $\alpha_x$ be the least $\alpha$
such that $x=p\rest d$ for some $p\in A_\alpha$.
Then since there is no surjection $V_{\gamma+\om}\to\Omega$ and $\Omega$ is
regular,
we may fix $\beta>\sup_{x\in X}\alpha_x$. Let $q\in A_\beta$. Let $x=q\rest d$.
Then $x\in X$. Let $\alpha=\alpha_x$. Then $\alpha<\beta$. Let $p\in A_\alpha$
be such that $x=p\rest d$. Then we have $p\rest d=x=q\rest d$, but since
$d=\dom(p)\inter\dom(q)$,
it follows that $p,q$ are compatible, a contradiction.
\end{proof}

\begin{clmthree}\label{clm:small_names}For each $x\in\pow(\om)^{V[G]}$
there is a $\PP$-name $\dot{x}\in V_\Omega$
 such that $\dot{x}^G=x$.
 \end{clmthree}
 \begin{proof}
 Let $\tau$ be a $\PP$-name for $x$. For $n<\om$, let
 $B_n=\{p\in\PP\bigm|  p\forces\check{n}\in\tau\}$.
 Let $\left<d_\alpha\right>_{\alpha<\Omega}$ enumerate
 $[\Omega\cross\om]^{<\om}$. For $p\in\PP$, let $\alpha_p$ be the $\alpha$
 such that $\dom(p)=d_\alpha$.
 Define a set $C_n\sub B_n$, determining whether $p\in C_n$ recursively
 on $\alpha_p$, as
follows:
 given $p\in B_n$, put $p\in C_n$ iff $p\incompat q$ for all $q\in C_n$
 such that $\alpha_q<\alpha_p$.
Note that $C_n\in V_\Omega$, as otherwise we easily get an
$\Omega$-pre-antichain. And $C_n$ is pre-dense in $B_n$,
because if  $p\in B_n\cut C_n$ then $p\compat q$ for some $q\in C_n$
with $\alpha_q<\alpha_p$. So let $\dot{x}$ be the $\PP$-name
consisting of all pairs $(p,\check{n})$ such that $n<\om$ and $p\in C_n$.
It follows that $\dot{x}\in V_\Omega$
and $\dot{x}_G=x$, as desired.
 \end{proof}
 This completes the proof of the lemma.
 \end{proof}

\begin{proof}[Proof of Corollary \ref{cor:choiceless_wDJ_forcing_CH}]
 Let $\PP$ be the forcing of \ref{lem:force_CH} and $G$ be $V$-generic for
$\PP$.
 So $\PP$ is homogeneous, $\Omega$-cc and $V[G]$ has a bijection
$f:\Omega=\aleph_1^{V[G]}\to\HC^{V[G]}$.
 Let $\Sigma'$ be the extension of $\Sigma$ to $V[G]$ given by
\ref{thm:strat_with_cond_extends_to_generic_ext}.

 Work in $V[G]$. So $\Sigma'$ is an $(m,\om_1+1)$-strategy with strong hull
condensation,
 and $\om_1$ is regular.
 Using $(\Sigma')^\stk$ and the bijection $f$, we can run the usual construction
 of an $(m,\omega_1+1)$-strategy $\Lambda'$ for $M$ with weak DJ with respect to
$e$.
 As mentioned in \ref{rem:wDJ_implies_cond},
 $\Lambda'$ is the unique such strategy for $M$.

 But then $\Lambda\eqdef\Lambda'\rest V\in V$ (because $\PP$ is homogeneous and
$\Lambda'$ is unique;
 alternatively, use \ref{cor:wDJ_absoluteness}), and $\Lambda$ has weak DJ with
respect to $e$ in $V$.
\end{proof}

We now slightly improve on \ref{cor:choiceless_wDJ_forcing_CH}.
But this time, the proof works by executing the $\AC$ part of the argument in an
inner model of choice,
instead of a forcing extension. As mentioned above, the idea
of Using Theorem \ref{thm:strat_with_cond_extends_to_generic_ext} to extend a
strategy of $\HOD_X$ via Vopenka
may have come from observations of Dominik Adolf.

\begin{cor}\label{cor:choiceless_wDJ_HOD}
 Let $\Omega>\om$ be regular and such that for no $\alpha<\Omega$
 is $\Omega$ the surjective image of
 $\pow(\alpha)$.
 Let $M$ be a countable $m$-sound $(m,\Omega,\Omega+1)^*$-iterable premouse and
$e$ be an enumeration of $M$ in ordertype $\om$.
 Then there is an $(m,\Omega+1)$-iteration strategy for $M$ with weak DJ with
respect to $e$.
\end{cor}
\begin{proof}
 Note that $\Omega$ is inaccessible in every proper class inner model $H$ of
$\ZFC$.
 When we mention weak DJ below, we mean with respect to $e$.

 Let $\Sigma$ be an $(m,\Omega,\Omega+1)^*$-strategy for $M$.
 Let $H=\HOD_{\Sigma,M,e}$ and $\Lambda=\Sigma\rest H$.
 So $\Lambda,M,e\in H$
 and
 \[ H\sats\ZFC+\text{``}\Omega\text{ is inaccessible and }\Lambda\text{ is an
}(m,\Omega,\Omega+1)^*\text{-strategy for }M\text{''.}\]
So there is (a unique) $\Psi\in H$ such that
\[ H\sats\text{``}\Psi\text{ is an }(m,\Omega+1)\text{-strategy for }M\text{
with weak DJ''.}\]

 For each $\alpha<\Omega$ and $X\sub\alpha$, let $G_X$ be the Vopenka
 generic for adding $X$ to $H$.
This Vopenka forcing has the $\Omega$-cc in $H$,
because $\Omega$ is
not the surjective image of $\pow(\alpha)$ in $V$.
Also, $H[G_X]=\HOD_{\Sigma,M,e,X}$.
Given $\beta<\Omega$ and $Y\sub\beta$,
let $G^X_Y$ be the Vopenka generic for adding $Y$ to $H[G_X]$.
This forcing is $\Omega$-cc in $H[G_X]$.

 So by \ref{thm:strat_with_cond_extends_to_generic_ext},
 there is a unique $\Psi_X\in H[G_X]$
 such that
 \[H[G_X]\sats\text{``}\Psi_X\text{ is an }(m,\Omega+1)\text{-strategy
 with weak DJ'';}\]
 moreover, $\Psi\sub\Psi_X$. Similarly,
  there is a unique $\Psi^X_Y$
 for $H[G_X][G^X_Y]$, and $\Psi_X\sub\Psi^X_Y$.
 Note that
 \[ H[G_X][G^X_Y]=\HOD_{\Sigma,M,e,X,Y}=H[G^Y][G^Y_X], \]
 so $\Psi^X_Y=\Psi^Y_X$, so $\Psi_X$ is compatible with $\Psi_Y$.

 Let $\Psi^\Omega_X$ be the restriction of $\Psi_X$ to an $(m,\Omega)$-strategy
of $H[G_X]$,
 and let
 $\Psi^\Omega$ be the union of all $\Psi^\Omega_X$ (over all bounded subsets
$X$ of $\Omega$).

 Then clearly $\Psi^\Omega$ is an $(m,\Omega)$-strategy with weak DJ.
 In fact, $\Psi^\Omega$ is the unique such strategy, because otherwise we can
run the usual phalanx comparison argument
 working inside some inner model of $\ZFC$, using the fact that $\Omega$ is
inaccessible
 there, to see that the comparison terminates.

 We claim that $\Psi^\Omega$ extends (uniquely) to an $(m,\Omega+1)$-strategy.
 For given any tree $\Tt$ via $\Psi^\Omega$ of length $\Omega$, we can argue as
above
 with
 $H_\Tt=\HOD_{\Sigma,M,e,\Tt}$
 replacing $H$.
 Let $\Psi_\Tt\in H_\Tt$ be the resulting $(m,\Omega+1)$-strategy of $H_\Tt$,
 and $\Psi^\Omega_\Tt\in V$ the resulting $(m,\Omega)$-strategy of $V$.
 Then $\Psi^\Omega=\Psi^\Omega_\Tt$ by the uniqueness mentioned of
$\Psi^\Omega$.
 But $\Psi_\Tt$ is compatible with $\Psi^\Omega_\Tt=\Psi^\Omega$,
 so $\Tt$ is via $\Psi_\Tt$, and since $\Tt\in H\sats$``$\Psi_\Tt$ is an
$(m,\Omega+1)$-strategy'',
 therefore $\Psi_\Tt(\Tt)$ is a $\Tt$-cofinal branch, as desired.

 So let $\Psi^+$ be this extension of $\Psi^\Omega$. Then $\Psi^+$ has weak DJ,
completing the proof.
 For if we have some counterexample to weak DJ given by a tree $\Tt$ of length
$\Omega+1$,
 note that by the regularity of $\Omega$, there is some $\alpha\in b^\Tt$ such
that $\Psi\rest(\alpha+1)$ is also a counterexample,
 contradicting weak DJ for $\Psi^\Omega$.
\end{proof}

\printindex
\bibliographystyle{plain}
\bibliography{stacks_from_normal_iterability_fork_arxiv}

\end{document}